\documentclass{article}

\usepackage[utf8]{inputenc}
\usepackage[width=15cm,height=20cm]{geometry}
\usepackage{bbm}
\usepackage{graphicx}
\usepackage{url}

\usepackage{xcolor}

\usepackage[titles]{tocloft}
\usepackage{amsmath,amsfonts,amssymb,amsthm,mathrsfs}

\usepackage[hidelinks,pdfusetitle]{hyperref}
\hypersetup{bookmarksdepth=2}
\usepackage[capitalise]{cleveref}

\usepackage{enumerate}

\usepackage[intoc,refpage]{nomencl}
\usepackage{ifthen,etoolbox}
\makenomenclature

\newif\ifnomentry
\renewcommand\nomgroup[1]{%
	\nomentryfalse
	\ifstrequal{#1}{O}{\item[Other]}{%
		\ifstrequal{#1}{F}{\item[Functional spaces]}{%
			{}}}%
	\nomentrytrue
}

\newtheorem{thm}{Theorem}
\newtheorem{prop}{Proposition}[section]
\newtheorem{defi}[prop]{Definition}
\newtheorem{coro}[prop]{Corollary}
\newtheorem{lem}[prop]{Lemma}
\newtheorem{rmk}[prop]{Remark}
\numberwithin{equation}{section}

\newcommand{\p}{\partial}
\newcommand{\gl}{\Lambda}
\newcommand{\gk}[1]{\gl_{#1}}
\newcommand{\gnot}{\gl_{0}}

\newcommand{\N}{\mathbb{N}}
\newcommand{\R}{\mathbb{R}}
\newcommand{\fu}{\mathbbm{u}}
\newcommand{\fv}{\mathbbm{v}}
\newcommand{\eps}{\varepsilon}
\newcommand{\Om}{\Omega}
\newcommand{\dd}{\:\mathrm{d}}
\newcommand{\Z}{\mathbb{Z}}
\newcommand{\cA}{\mathcal{A}}
\newcommand{\cB}{\mathcal{B}}
\newcommand{\cE}{\mathcal{E}}

\newcommand{\cH}{\mathcal{H}}
\newcommand{\cM}{\mathcal{M}}

\newcommand{\cX}{\mathcal{X}}
\newcommand{\cZ}{\mathcal{Z}}

\newcommand{\bu}{\bar{u}}

\newcommand{\reg}{\mathrm{reg}}
\newcommand{\sing}{\mathrm{sing}}
\newcommand{\ureg}{u_\mathrm{reg}}

\newcommand{\busing}{\bu_{\sing}}

\newcommand{\thr}{{\tau^{\alpha}}}
\newcommand{\bfi}{\overline{f_i}}
\newcommand{\bfzero}{\overline{f_0}}
\newcommand{\bfun}{\overline{f_1}}
\newcommand{\bfj}{\overline{f_j}}

\newcommand{\ufs}{\fu_{\mathrm{P}}}
\newcommand{\vfs}{\fv_{\mathrm{P}}}
\newcommand{\gfs}{\gamma_{\mathrm{P}}}
\newcommand{\Yfs}{{\mathbbm{Y}}_{\mathrm{P}}}
\newcommand{\tY}{\widetilde{Y}}
\newcommand{\tz}{\zeta}
\newcommand{\qone}{{Q^1}}
\newcommand{\supp}{\operatorname{supp}}
\newcommand{\cLin}{\cH} 
\newcommand{\hBPerp}{\cX^\perp_{B,\operatorname{sg}}}
\newcommand{\keta}{\widetilde{Y}}
\newcommand{\BCP}{\Upsilon_{\mathrm{P}}}

\newcounter{stepcount}
\newcommand{\step}[1]{\medskip
	\noindent\refstepcounter{stepcount}\emph{\textbf{Step~\arabic{stepcount}.} #1}}
\AtBeginEnvironment{proof}{\setcounter{stepcount}{0}}

\title{Linear and nonlinear parabolic forward-backward problems}

\author{Anne-Laure Dalibard\texorpdfstring{\thanks{Sorbonne Université, Université Paris-Diderot SPC, CNRS, Laboratoire Jacques-Louis Lions, LJLL, Paris}}{}~\texorpdfstring{\thanks{École Normale Supérieure, Université PSL, Département de Mathématiques et applications, Paris}}{}, 
Frédéric Marbach\texorpdfstring{\thanks{Univ Rennes, CNRS, IRMAR - UMR 6625, Rennes}}{}, 
Jean Rax\texorpdfstring{\footnotemark[1]}{}}

\begin{document}
	
\maketitle
	
\begin{abstract}
    The purpose of this paper is to investigate the well-posedness of several linear and nonlinear equations with a parabolic forward-backward structure, and to highlight the similarities and differences between them.
    The epitomal linear example will be the stationary Kolmogorov equation  $y\p_x u -\p_{yy} u=f$, which we investigate in a rectangle $(x_0,x_1)\times(-1,1)$, supplemented with boundary conditions on the ``parabolic boundary'' of the domain: the top and lower boundaries $\{y=\pm 1\}$, and the lateral boundaries $\{x_0\}\times (0,1)$ and $\{x_1\}\times (-1,0)$.
    We first prove that this equation admits a finite number of singular solutions associated with regular data.
    These singular solutions, of which we provide an explicit construction, are localized in the vicinity of the points $(x_0,0)$ and $(x_1,0)$.
    Hence, the solutions to the Kolmogorov equation associated with a smooth source term $f$ are regular if and only if $f$ satisfies a finite number of orthogonality conditions. 
    This is similar to well-known phenomena in elliptic problems in polygonal domains.
    
    We then extend this theory to the Vlasov--Poisson--Fokker--Planck system $y\p_x u + E[u] \p_y u - \p_{yy} u=f$, and to two quasilinear equations: the Burgers type equation $u \p_x u - \p_{yy} u = f$ in the vicinity of the linear shear flow, and the Prandtl system in the vicinity of a recirculating solution, close to the curve where the horizontal velocity changes sign. 
    We therefore revisit part of a recent work by Iyer and Masmoudi \cite{IM2023,IM2022}.
    For the two latter quasilinear equations, we introduce a geometric change of variables which simplifies the analysis.
    In these new variables, the linear differential operator is very close to the Kolmogorov operator $y\p_x -\p_{yy}$. 
    Stepping on the linear theory, we prove existence and uniqueness of regular solutions for data within a manifold of finite codimension, corresponding to some nonlinear orthogonality conditions.
    
    Treating these three nonlinear problems in a unified way also allows us to compare their structures. 
    In particular, we show that the vorticity formulation of the Prandtl system, in an adequate set of variables, is very similar to the Burgers equation. 
    As a consequence, solutions of the Prandtl system are actually smoother than the ones of Burgers, which allows us to have a theory of weak solutions of the Prandtl system close to the recirculation zone.
\end{abstract}
    
\newpage 

\setcounter{tocdepth}{2}
\tableofcontents

\newpage
\section{Introduction}
\label{sec:intro}

This manuscript is devoted to the well-posedness of linear and nonlinear equations having a parabolic forward-backward structure. 
In the linear case, our main example will be the Kolmogorov equation $y\p_x u - \p_{yy} u=f$ 
in the rectangular domain $\Omega := (x_0, x_1) \times (-1,1)$ where $x_0 < x_1$ and $f$ is an external source term.
\nomenclature[OAZ]{$\Omega$}{Physical rectangular domain $(x_0,x_1)\times(-1,1)$, see \cref{fig:omega}}
We will also consider nonlinear perturbations of this linear setting.
The easiest nonlinear perturbation we consider is a Vlasov--Poisson--Fokker--Planck type system of the form $y\p_x u + E[u]\p_y u -\p_{yy} u=f$, where $E[u]=\p_x^{-1} \int u\dd y$.
In this case the nonlinearity does not perturb the geometry of the problem, which remains forward parabolic in the region $y>0$, and backward parabolic in the region $y<0$.

We will also  investigate the existence and uniqueness of sign-changing solutions to the Burgers type equation
\begin{equation} \label{eq:eq0-uux}
    u \p_x u - \p_{yy} u = f
\end{equation}
and to the Prandtl system
\begin{equation}
\label{Prandtl-small-domain}
\begin{aligned}
u \p_x u + v \p_y u - \p_{yy} u& = - \p_x p, \\
\p_x u + \p_y v &=0.
\end{aligned}
\end{equation}
A natural solution to \eqref{eq:eq0-uux} with a null source term $\mathbbm{f} = 0$ is the linear shear flow $\fu(x,y) := y$, which changes sign across the horizontal line $\{ y = 0 \}$.
In a similar way, semi-explicit solutions $(\ufs,\vfs)$ of the Prandtl system \eqref{Prandtl-small-domain} such that $\ufs$ changes sign have been exhibited, see the discussion in \cref{sec:motivation} below.
We are interested in strong solutions to~\eqref{eq:eq0-uux} (resp.\ \eqref{Prandtl-small-domain}) which are close in an appropriate norm to this linear shear flow $\fu$ (resp.\ to the reference solution $(\ufs, \vfs)$).
Our purpose is to construct such solutions by perturbing the lateral boundary data or the source term.
\nomenclature[OAS]{$\Sigma_0$}{Left inflow boundary $\{x_0\}\times(0,1)$, see \cref{fig:omega}}%
\nomenclature[OAS]{$\Sigma_1$}{Right inflow boundary $\{x_1\}\times(-1,0)$, see \cref{fig:omega}}

\paragraph{Forward-backward nature}
Since solutions to \eqref{eq:eq0-uux} and \eqref{Prandtl-small-domain} will change sign across a curve $\{ u = 0 \}$ lying within $\Omega$, a key feature of this work these problems must be seen as quasilinear forward-backward parabolic equations in the horizontal direction.
Thus, to ensure the existence of a solution, one must be particularly careful as to how one enforces the lateral perturbations.
More precisely, the problem is forward parabolic in the  domain above the curve $\{u=0\}$, in which $u>0$, and therefore we shall  prescribe a boundary condition on $\Sigma_0 := \{x=x_0\}\cap\{u>0\}$; and backward parabolic in the domain below the curve $\{u=0\}$, and we shall prescribe a boundary condition on $\Sigma_1 := \{x=x_1\}\cap \{u<0\}$.

\begin{figure}[ht]
    \centering
    \includegraphics{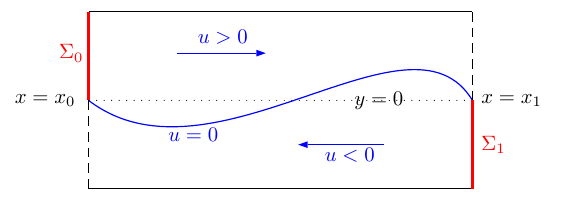}
    \caption{Fluid domain $\Omega$ and inflow boundaries $\Sigma_0 \cup \Sigma_1$}
    \label{fig:omega}
\end{figure}

We will construct solutions to these problems thanks to an abstract implicit function theorem taking into account the geometry of the problem.
More precisely, we will first straighten the free boundary $\{u=0\}$ by introducing as a new vertical variable $z=u(x,y)$.
A suitable change of unknown function will then transform \eqref{eq:eq0-uux} and \eqref{Prandtl-small-domain} into quasilinear equations with an easier-to-handle nonlinearity (see \cref{rk:semilinear}).

\paragraph{Orthogonality conditions}
Because of the nonlinearity, we need to work in a high enough regularity space in order to have a suitable control of the derivatives.
However, one key difficulty of our work lies in the fact that, even when the source term $f$ is smooth, say in $C^\infty_0(\Om)$, \emph{solutions to \eqref{eq:eq0-uux} and \eqref{Prandtl-small-domain} have  singularities in general}.
Actually, this feature is already present at the linear level, i.e.\ for the equation $y\p_x u -\p_{xx}u=f$.
We prove that if $f$ is smooth, the associated weak solution to the linear system inherits the regularity of $f$ if and only if $f$ satisfies \emph{orthogonality conditions} (i.e.\ the scalar products of $f$ with some identified profiles must vanish). We also describe the singularities that appear when these orthogonality conditions are not satisfied.
At the nonlinear level, these orthogonality conditions become a finiteness assumption on the codimension of the data manifold.

All the features described above (orthogonality conditions for linear forward-backward equations, description of the potential singularities, handling of orthogonality conditions  for quasilinear systems) appear to be new. 
We believe that the strategy we use could be extended to other nonlinear settings in which orthogonality conditions appear (elliptic equations in domains with corners, problems in which the linearized operator is Fredholm with negative index, ...)

\subsection{Statement of the main results}
\label{sec:statements}

\subsubsection{Linear theory}

Due to the forward-backward nature of the problem, we must choose the lateral perturbations and the source term in a particular product space.
We therefore introduce the vector space
\begin{equation} \label{eq:def-XB-intro}
    \begin{split}
        \cX_B := \Big\{ (f,\delta_0,\delta_1) \in H^1_x H^2_y & \times H^5(0,1) \times H^5(-1,0); \enskip f\vert_{\Sigma_0\cup\Sigma_1}=0 \\
        & \text{ and } \delta_i(0) = \delta_i((-1)^i) = \delta_i''(0) = \delta_i''((-1)^i) = 0 \Big\},
    \end{split}
\end{equation}
where $\Sigma_0=\{x_0\}\times (0,1)$ and $\Sigma_1=\{x_1\}\times (-1,0)$ are the lateral boundaries on which we prescribe boundary conditions. 
We endow $\cX_B$ with its canonical norm
\nomenclature[OAD1]{$\delta_i$}{Boundary data at the inflow boundary $\Sigma_i$}%
\nomenclature[FXB]{$\cX_B$}{Hilbert space with norm \eqref{eq:XB-norm} of data triplets $(f,\delta_0,\delta_1)$ for nonlinear Burgers}%
\begin{equation} \label{eq:XB-norm}
    \| (f,\delta_0,\delta_1) \|_{\cX_B}
    := \|  f \|_{H^1_x H^2_y} + \| \delta_0 \|_{H^5} + \| \delta_1 \|_{H^5}.
\end{equation}
We establish existence and uniqueness of solutions in the following anisotropic Sobolev space\nomenclature[FQ1]{$\qone$}{Solution space $H^{5/3}_x L^2_y \cap L^2_x H^2_y$}
\begin{equation}
    \label{eq:def-Q1-intro}
    \qone := H^{5/3}((x_0,x_1);L^2(-1,1))\cap H^1((x_0,x_1);H^2(-1,1)).
\end{equation}
In particular, for solutions with such regularity, equation \eqref{eq:eq0-uux} or its linear version $y\p_x u - \p_{yy} u = f$ hold in a strong sense, almost everywhere and the various boundary conditions hold in the usual sense of traces.
We first state a result concerning the well-posedness in $\qone$ of the stationary Kolmogorov equation (see \eqref{eq:shear} below), \emph{up to two orthogonality conditions} (see comments below). 
Although equation~\eqref{eq:shear}  has been thoroughly investigated, as we recall in \cref{sec:comments}, we could not find this statement in the existing literature. 

\begin{thm}[Orthogonality conditions for linear forward-backward parabolic equations] \label{thm:shear-Q1}
    There exists a vector subspace $\hBPerp \subset \cX_B$ of codimension 2 such that, for each $(f,\delta_0, \delta_1) \in \cX_B$, there exists a solution $u \in \qone$ to the problem
    \begin{equation}
        \label{eq:shear}
        \begin{cases}
            y \p_x u - \p_{yy} u = f,\\
            u_{\rvert \Sigma_i} = \delta_i, \\
            u_{\rvert y = \pm 1} = 0,
        \end{cases}
    \end{equation}
    if and only if $(f,\delta_0,\delta_1) \in \hBPerp$.
    Such a solution is unique and satisfies
    \begin{equation} \label{eq:u-qone}
        \|u\|_{\qone} \lesssim \| (f,\delta_0, \delta_1) \|_{\cX_B}.
    \end{equation}
\end{thm}

We emphasize that this result implies that there exist triplets $(f,\delta_0,\delta_1)$ that can be chosen arbitrarily smooth and compactly supported, and for which there are no $\qone$ solutions to \eqref{eq:shear}.
Furthermore, the vector space $\hBPerp$ can be fully characterized: classically, $\hBPerp = \ker \overline{\ell^0} \cap \ker \overline{\ell^1}$, where $\overline{\ell^0}$ and $\overline{\ell^1}$ are two linear forms on $\cX_B$ which we shall write explicitly.
If the data do not belong to $\hBPerp$, the solution has singularities, which we can describe completely.

\begin{thm}[Decomposition of solutions as a sum of singular profiles and a smooth remainder]
    \label{thm:shear-decomp}
    Let $(f,\delta_0, \delta_1) \in \cX_B$. 
    There exists a unique solution $u\in H^{2/3}_xL^2_y\cap L^2_xH^2_y$ to equation \eqref{eq:shear}. 
    Furthermore, this solution admits the following decomposition: there exists $c_0,c_1\in \R$, and $\ureg \in \qone$, such that
    \begin{equation*}
    u=c_0\busing^0 + c_1 \busing^1 + \ureg.
    \end{equation*}
    Each profile $\busing^i$ is supported in the vicinity of $(x_i,0)$ and is smooth on $\overline{\Om} \setminus \{ (x_i,0) \}$. 
    Furthermore, for $|x-x_i|\ll 1$ and $|y|\ll 1$,
    \begin{equation*}
    \busing^i(x,y)=\left(|y|^2 + |x-x_i|^{\frac 23}\right)^{\frac 14} \gnot\left((-1)^i\frac{y}{|x-x_i|^{\frac 13}}\right),
    \end{equation*}
    where $\gnot \in C^\infty(\R)$ is such that $\gnot(-\infty)=1$ and $\gnot(+\infty)=0$ (see \cref{fig:G0} page \pageref{fig:G0}).
\end{thm}

The existence of a weak solution was already known, see in particular \cite{Pagani1,Pagani2,MR0111931}. The novelty of the above theorem lies in the identification of the singular profiles $\busing^i$, and in the decomposition of any weak solution.
The function $\gnot$ is in fact the solution to an ODE, and can be characterized in terms of special functions (namely confluent hypergeometric functions of the second kind, or Tricomi's functions).

The assumptions on the data $(f,\delta_0,\delta_1)$ are not optimal and can be weakened: in particular, we merely need $H^1_x L^2_y$ regularity on the source term $f$, together with compatibility conditions in the corners. 
We will state optimal versions   of \cref{thm:shear-Q1} and \cref{thm:shear-decomp} in  \cref{sec:shear}, relying on the functional space $\cH_K$ defined in \eqref{eq:def-HK} (see respectively  \cref{thm:shear-Z1} and \cref{coro:decomposition-profile}).
In fact, in the sequel, we will need these sharper versions to prove our nonlinear results.
In this introduction, we stick with the above versions in order to avoid defining too many functional spaces.

\subsubsection{A nonlinear toy model from kinetic theory}

As a corollary to \cref{thm:shear-Q1}, we obtain a similar statement for a (nonlinear) Vlasov--Poisson--Fokker--Planck system in an interval. 
In order not to burden the introduction, we refer to \cref{sec:Fokker-Planck} for the presentation of the system and to \cref{prop:VPFP} for the full statement.
The proof of \cref{prop:VPFP} is rather straightforward since the geometry of the considered problem remains the same as for~\eqref{eq:shear}, and the nonlinearity is very weak.
We nevertheless use this example to set up our nonlinear scheme in \cref{sec:VPFP-constr-Z1} and a general abstract nonlinear existence result in \cref{sec:abstract}.

\subsubsection{The viscous Burgers system}

We then turn towards the nonlinear problem \eqref{eq:eq0-uux}.
One of the main results of this paper is the following nonlinear generalization of \cref{thm:shear-Q1} for small enough perturbations. More precisely, the norm of the perturbation must be smaller than some constant $\eta$ depending only on the size of the domain.

\begin{thm}[Existence and uniqueness of strong solutions to \eqref{eq:eq0-uux} under orthogonality conditions] \label{thm:burgers}
    There exists a Lipschitz submanifold $\cM_B$ of $\cX_B$ of codimension 2, containing $0$ and included in a ball of radius $\eta\ll 1$ in $\cX_B$, such that, for every $(f,\delta_0, \delta_1) \in \cM_B$, there exists a strong solution $u \in \qone$ to
    \begin{equation} \label{eq:yuuxuyy}
        \begin{cases}
            u \p_x u - \p_{yy} u = f,\\
            u_{\rvert \Sigma_i} = y + \delta_i(y), \\
            u_{\rvert y = \pm 1} = \pm 1.
        \end{cases}
    \end{equation}
    More precisely, $\cM_B$ is modeled on $\hBPerp$ and tangent to it at $0$.
    Such solutions are unique in a small neighborhood of $\fu(x,y) = y$ in $\qone$ and satisfy the estimate
    \begin{equation*}
        \| u - \fu \|_{\qone} \lesssim \| (f,\delta_0,\delta_1) \|_{\cX_B}.
    \end{equation*}
\end{thm}

In the statement above, the condition that the data $(f,\delta_0,\delta_1)$ belong to the manifold $\cM_B$ is the nonlinear equivalent of the orthogonality conditions from \cref{thm:shear-Q1}. 
We emphasize that this is by no means a technical restriction which could be lifted, but actually a necessary condition to solve the equation with smooth solutions, as we state in \cref{p:necessity-ortho} below.
A key difficulty lies in the fact that these orthogonality conditions \emph{depend on the solution itself}.

\begin{prop}[Necessity of the orthogonality conditions]
    \label{p:necessity-ortho}
    There exists $\eta > 0$ such that the following result holds.
    Let $(f,\delta_0,\delta_1) \in \cX_B$ with $\|(f,\delta_0,\delta_1)\|_{\cX_B} \leq \eta$.
    Let $u \in \qone$ be a solution to~\eqref{eq:yuuxuyy} such that $\|u-\fu\|_{\qone} \leq \eta$.
    Then $(f,\delta_0, \delta_1) \in \cM_B$.
\end{prop}

\begin{rmk}
    By commodity, the above results are stated using the full triplet $(f,\delta_0,\delta_1)$, and so is the remainder of this paper.
    Nevertheless, it is possible to obtain similar results either by fixing $\delta_0 = \delta_1 = 0$ and constructing a submanifold of source terms $f$ yielding regular solutions, or by fixing $f = 0$ and constructing a submanifold of boundary data $(\delta_0,\delta_1)$.
    This stems from the independence of the orthogonality conditions, which can be obtained either by \cref{free:f} or by \cref{cor:free_ellj-sansf}.
\end{rmk}

\subsubsection{The Prandtl system}
\label{sec:intro-main-Prandtl}

We also prove analogous results for the Prandtl system, revisiting the work of Iyer and Masmoudi in \cite{IM2022,IM2023} (we will comment more thoroughly on the differences between our results in the next sections). Let us now present our mathematical setting; we will provide physical motivation and background for this system in \cref{sec:motivation}.
\nomenclature[OLuFS]{$\ufs$}{Reference recirculating flow for the Prandtl system}
We consider a reference flow $(\ufs,\vfs)\in C^k([x_0,x_1]\times (0, +\infty))$ for some sufficiently large $k$ (say $k=4$), satisfying the Prandtl system
\begin{equation*}
    \begin{aligned}
    \ufs \p_x \ufs + \vfs \p_y \vfs - \p_{yy}\vfs & = - \p_x p,\\
    \p_x \ufs + \p_y \vfs &=0
    \end{aligned}
\end{equation*}
in the whole domain $(x_0, x_1)\times (0, + \infty)$, where $p$ is the trace of the pressure of some outer Euler flow on the boundary $\{y=0\}$.
We assume that there exists a curve $\overline{\Gamma}:=\{y=\gfs (x)\}$, with $\gfs $ smooth and such that $\inf_{[x_0, x_1]}\gfs >0$,  such that 
$\ufs$ changes sign on the curve $\overline{\Gamma}$:
$\ufs(x,\gfs(x))=0$, and $\ufs(x,y)<0$ (resp.\ $\ufs(x,y)>0$) for $y<\gfs(x)$ (resp.\ $y>\gfs(x)$). 
Our purpose is to construct a solution to the Prandtl system close to $(\ufs,\vfs)$ and in the vicinity of the curve $\overline{\Gamma}$, by perturbing either the inflow/outflow on the lateral boundaries or the source term.

\nomenclature[OAB2]{$\overline{\gamma_b}, \overline{\gamma_t}$}{Level sets of the function $\ufs$}
To that end, we consider $z_b<0<z_t$ such that there exist smooth functions $\overline{\gamma_b}, \overline{\gamma_t}$, with $0<\overline{\gamma_b}(x)<\gfs(x) <\overline{\gamma_t}(x) $ for all $x\in [x_0,x_1]$, and such that $\ufs(x,\overline{\gamma_j}(x))=z_j$ for $j\in \{b,t\}$.
We set $\overline{\Gamma_j}= \{y=\overline{\gamma_j}(x)\}$ for $j\in \{b,t\}$.
We consider the Prandtl system \eqref{Prandtl-small-domain} in a domain $\Omega_P$, which is defined by
 \nomenclature[OAB1]{${\gamma_b}, {\gamma_t}$}{Level sets of the unknown solution $u$ of the Prandtl system}
 \nomenclature[OAZP]{$\Om_P$}{Physical domain for the resolution of the Prandtl system}
\begin{equation*}
\Om_P := \{(x,y)\in (x_0,x_1)\times \R_+;\ \gamma_b(x)<y<\gamma_t(x)\},
\end{equation*}
where $\gamma_b, \gamma_t$ are smooth functions, which will actually be free boundaries, corresponding to the level sets $z_b$ and $z_t$ of the function $u$. 
We expect these functions (which are unknowns of the problem) to be smooth functions located in the vicinity of $\overline{\gamma_b}, \overline{\gamma_t}$.
We endow system \eqref{Prandtl-small-domain} with the following boundary conditions, which we will discuss and comment in \cref{sec:comments}.
See \cref{fig:omega-prandtl} (page \pageref{fig:omega-prandtl}) for a sketch of the geometry of the domain.
\nomenclature[OAD2]{$\delta_b,\delta_t$}{Boundary data for the vorticity on the bottom and top boundaries of $\Om_P$}
\nomenclature[OLv]{$v_b$}{Boundary datum for the vertical velocity on the bottom boundary of $\Om_P$}
\begin{enumerate}
    \item \emph{Boundary conditions on the top and bottom free boundaries:}
    On the bottom boundary $\Gamma_b=\{y=\gamma_b(x)\}$, we enforce
    \begin{equation}\label{CL-Prandtl-bottom}
    \begin{aligned} 
    u\vert_{\Gamma_b}&= \ufs\vert_{\overline{\Gamma_b}}=z_b,\\
    \p_y u\vert_{\Gamma_b}&= \p_y \ufs\vert_{\overline{\Gamma_b}} + \delta_b,\\
    v\vert_{\Gamma_b}&=\vfs\vert_{\overline{\Gamma_b}} + v_b,
    \end{aligned} 
    \end{equation}
    where $\delta_b$, $v_b$ are small, smooth perturbations. 
    
    Similarly, on the  top boundary  $\Gamma_t=\{y=\gamma_t(x)\}$, we enforce
    \begin{equation}\label{CL-Prandtl-top}
    \begin{aligned} 
    u\vert_{\Gamma_t}&= z_t,\\
    \p_y u\vert_{\Gamma_t}&= \p_y \ufs\vert_{\overline{\Gamma_t} }+ \delta_t,
    \end{aligned} 
    \end{equation}
    where $\delta_t$ is, again, a small smooth perturbation.
    
    \begin{rmk}
        Note that on the top and bottom boundary, we prescribe the trace of $u$ and of its normal derivative. Of course it would be impossible to prescribe simultaneously these two boundary conditions if the boundaries $\Gamma_b,\Gamma_t$ were fixed.
        This is only made possible by the fact that these two boundaries are free.
    \end{rmk}
    
    \item  \emph{Lateral boundary conditions:} 
    As in the case of \eqref{eq:shear} and \eqref{eq:yuuxuyy}, we enforce lateral boundary conditions on $u$, namely
   \nomenclature[OASP]{$\Sigma^P_i$}{Lateral inflow boundaries for the Prandtl system}
    \begin{equation}\label{CL-Prandtl-lateral}
    u_{|\Sigma^P_i}= \ufs + \delta_i,\qquad i\in \{0,1\},
    \end{equation}
    where $\Sigma^P_0 = \{ x_0 \} \times (\gfs(x_0), \overline{\gamma_t}(x_0))$ and $\Sigma^P_1 = \{ x_1 \} \times (\overline{\gamma_b}(x_1), \gfs(x_1))$ (see also \cref{fig:omega-prandtl} page \pageref{fig:omega-prandtl}).
    In particular $\ufs >0$ on $\Sigma^P_0$ and $\ufs < 0$ on $\Sigma^P_1$.
    For simplicity, we assume that $\delta_0(\gfs(x_0))=\delta_1(\gfs(x_1))=\delta_0(\overline{\gamma_t}(x_0))=\delta_1(\overline{\gamma_b}(x_1))=0$, and we recall that $\delta_0, \delta_1$ are assumed to be small in some sufficiently strong Sobolev norm. 
    Therefore, provided that $\p_y \ufs\vert_{\overline{\Gamma}}>0$, the signs of $\ufs (x_i,\cdot ) + \delta_i$ and of $\ufs(x_i, \cdot)$ are identical on $\Sigma^P_i$.
    
\end{enumerate}

We will in fact state two different results: one in ``low regularity'', under merely one orthogonality condition, and another one in higher regularity, under three orthogonality conditions. 
For the sake of readability, we have stated them under the same regularity and compatibility assumptions on the data, although the assumptions are not optimal in the low regularity case, and the compatibility conditions could be generalized in both cases.
We will state a more general result in \cref{sec:Prandtl} (see \cref{prop-Prandtl-Y}).
Therefore, we take our  data in the functional space 
\begin{equation*}\begin{aligned}
\cX_P:=\Big\{ (\delta_0, \delta_1, \delta_t, \delta_b, v_b) \in &\; H^6(\Sigma^P_0)\times H^6(\Sigma^P_1)\times H^2_0(x_0, x_1) \times H^2_0(x_0,x_1)\times H^1(x_0, x_1),\\& \delta_i \in H^4_0(\Sigma_i^P)\text{ for }i\in \{0,1\}\Big\},
\end{aligned}
\end{equation*}
which we endow with its natural norm.
\begin{thm}
    \label{thm:prandtl}
    There exist numbers $\eta>0$, $z_0>0$, depending only on the underlying flow $(\ufs, \vfs)$, such that if $|z_b|, z_t\leq z_0$, the following results hold.
    
    \begin{enumerate}
        \item There exists a manifold $\cM_0$ of codimension 1 in $\cX_P$ and included in a ball of radius $\eta$ in $\cX_P$, such that for all $(\delta_0, \delta_1, \delta_t, \delta_b, v_b) \in \cM_0$, equation \eqref{Prandtl-small-domain} endowed with the boundary conditions \eqref{CL-Prandtl-bottom}-\eqref{CL-Prandtl-top}-\eqref{CL-Prandtl-lateral} has a unique continuous solution $u$ such that $u\in L^2_x H^3_y (\Omega_P)$, $(x-x_0)(x-x_1) u\in H^1_x H^3_y (\Omega_P)$, $\p_y u \in L^\infty(\Omega_P)$, $u\p_x \p_y u\in L^2(\Omega_P)$, and
        \begin{equation*}
            \begin{split}
                \| u-\ufs\|_{L^2_x H^3_y} + \| \p_y (u-\ufs) \|_{L^\infty} & + \| (x-x_0)(x-x_1) (u-\ufs)\|_{ H^1_x H^3_y}\\
       & \lesssim \| (\delta_0, \delta_1, \delta_t, \delta_b, v_b)\|_{\cX_P}.
            \end{split}
        \end{equation*}
        
        \item There exists a manifold $\cM_1$ of codimension 3 in $\cX_P$ and included in a ball of radius $\eta$ in $\cX_P$, such that for all $(\delta_0, \delta_1, \delta_t, \delta_b, v_b) \in \cM_1$, equation \eqref{Prandtl-small-domain} endowed with the boundary conditions \eqref{CL-Prandtl-bottom}-\eqref{CL-Prandtl-top}-\eqref{CL-Prandtl-lateral} has a unique solution $u$ such that $u\in H^{5/3}_x H^1_y\cap H^1_x H^3_y(\Omega^P)$, and
        \begin{equation*}
        \| u - \ufs\|_{H^{5/3}_x H^1_y} + \| u - \ufs\|_{H^1_x H^3_y}\lesssim  \| (\delta_0, \delta_1, \delta_t, \delta_b, v_b)\|_{\cX_P}.
        \end{equation*}
    \end{enumerate}
\end{thm}
	
\subsection{Comments and previous results}
\label{sec:comments}

We start with a few comments on our main results and recall related known results.

\bigskip

Problem \eqref{eq:shear}, involving the operator $y \p_x - \p_{yy}$, can be seen as a particular case of the class of ``degenerate second-order elliptic-parabolic linear equations'', also referred to as ``second-order equations with nonnegative characteristic form'' (as opposed to positive definite ones), ``forward-backward'' or ``mixed type'' problems.
They date back at least to Gevrey~\cite{zbMATH02616722}.

Problem \eqref{eq:shear} itself, as well as these wide classes of equations, has received a lot of attention and has been investigated under different aspects: with variable coefficients or other geometries \cite{MR0111931,zbMATH07000095,Pagani2}, higher-order operators \cite[Ch.\ 3, 2.6]{lions1969quelques}, abstract operators \cite{beals1979abstract,paronetto2004existence}, explicit representation formulas \cite{fleming1963problem,gor1975formula} or with a focus on numerical analysis \cite{aziz1999origins}.

\paragraph{On weak solutions for the linear problem.}
It is well-known since the work of Fichera \cite{MR0111931} that weak solutions to~\eqref{eq:shear} with $L^2_x H^1_y$ regularity exist.
For general boundary-value problems for elliptic-parabolic second-order equations, one owes to Fichera the systematic separation of the boundary of the domain into three parts: a ``noncharacteristic'' part, where one sets either Dirichlet or Neumann boundary conditions (here $y = \pm 1$), an ``inflow'' part, where one sets a Dirichlet boundary condition (here $\Sigma_0\cup\Sigma_1$) and an ``outflow'' part, where one cannot set a boundary condition (here, the two sets $\{ x_0 \} \times (-1,0)$ and $\{ x_1 \} \times (0,1)$).

Baouendi and Grisvard \cite{BG} proved the uniqueness of weak solutions to \eqref{eq:shear} with $L^2_x H^1_y$ regularity, by means of a trace theorem and a Green identity (see \cref{sec:proof-uniqueness}).

\paragraph{On strong solutions for the linear problem.}
There is an extensive literature on the regularity of solutions to degenerate elliptic-parabolic linear equations, and whether weak solutions are strong.
We refer the reader in particular to the book \cite{MR0457908} by Ole\u{\i}nik and Radkevi\v{c}.
Generally speaking, depending on the exact setting considered, it is quite often possible to prove that the solutions to such equations are regular far from the boundaries of the domain and/or from the regions where the characteristic form is not positive definite.
A nice example is Kohn and Nirenberg's work \cite{kohn1967degenerate}, which proves a very general regularity result.
A key assumption of their work is that the ``outflow'' part of the boundary does not meet the ``noncharacteristic'' and ``inflow'' parts (i.e.\ they are in disjoint connected components of $\partial \Omega$).
Hence, it does not apply to~\eqref{eq:shear}, and hints towards a difficulty near the points $(x_0,0)$ and $(x_1,0)$.

In a series of papers \cite{zbMATH03431167,Pagani1,Pagani2}, Pagani proved the existence of strong solutions to \eqref{eq:shear} (and related equations).
More precisely, Pagani proved the existence of solutions such that $y \p_x u$ and $\p_{yy} u$ belong to $L^2(\Omega)$.
Moreover, he determined the exact regularity of the various traces of such solutions (trace of $u$ at $x = x_i$, at $y = \pm 1$ or $y = 0$, and trace of $\p_y u$ at $y = 0$).
These maximal regularity results play a key role in our analysis and motivate the functional spaces we introduce in \cref{sec:spaces}.

\paragraph{On orthogonality conditions for higher regularity.}
As noted by Pyatkov in \cite{pyatkov2019some}, for such forward-backward problems: ``\emph{as a rule, there is no existence theorems for smooth solutions without some additional orthogonality-type conditions on the problem data}''.
Even for the linear problem~\eqref{eq:shear}, there have been very few works concerning higher regularity (than the one given by Pagani's framework) in the whole domain.
Most of the works focused on higher regularity (such as \cite{pyatkov2019some}) involve weighted estimates which entail regularity within the domain but not near the critical points $(x_i,0)$.
An attempt for global regularity is Goldstein and Mazumdar's work \cite[Theorem 4.2]{AG}; however the proof seems incomplete (see \cref{p:shear-WP-Z1} below and its proof for more details).

A misleading aspect is that it is quite easy, assuming the existence of a smooth solution, to prove \emph{a priori} estimates at any order.
Such phenomena are usual in the theory of elliptic problems in domains with corners or mixed Dirichlet-Neumann boundary conditions (see for instance \cite{MR775683}).
Let us give an illustration of such a phenomenon in a close context.
For a source term $f \in C^\infty_c(\Omega)$, consider the elliptic problem
\begin{equation} \label{eq:elliptic-u}
    \begin{cases}
        - \Delta u = f & \text{in } \Omega, \\
        u(x_i, y) = 0 & \text{for } (-1)^i y > 0, \\
        \p_x u(x_i,y) = 0 & \text{for } (-1)^i y < 0, \\
        u(x, \pm 1) = 0 & \text{for } x \in (x_0,x_1).
    \end{cases}
\end{equation}
It is classical that such a system has a unique weak solution $u \in H^1(\Omega)$.
Moreover, assuming that $u$ is smooth enough, $v := \p_x u$ satisfies
\begin{equation} \label{eq:elliptic-v}
    \begin{cases}
        - \Delta v = \p_x f & \text{in } \Omega, \\
        \p_x v(x_i, y) = 0 & \text{for } (-1)^i y > 0, \\
        v(x_i,y) = 0 & \text{for } (-1)^i y < 0, \\
        v(x, \pm 1) = 0 & \text{for } x \in (x_0,x_1).
    \end{cases}
\end{equation}
For such systems, one has $\| v \|_{H^1} \lesssim \| \p_x f \|_{L^2}$.
Hence $\| \p_{xx} u \| \lesssim \| \p_x f \|_{L^2}$, and, using the equation, $\|u\|_{H^2} \lesssim \| f \|_{H^1}$.
So one has an \emph{a priori} estimate.
However, it is known that there exist source terms for which the unique weak solution $u \in H^1$ does not enjoy $H^2$ regularity (see \cite[Chapter~4]{MR775683} and \cref{sec:radial}).
The key point is that, when reconstructing $u$ from the solution $v$ to \eqref{eq:elliptic-v}, say by setting $u(x,y) := \int_{x_0}^x v(x',y) \dd x'$ for $y > 0$ and $u(x,y) := \int_{x_1}^x v(x',y) \dd x'$ for $y < 0$, there might be a discontinuity of $u$ or $\p_y u$ across the line $y = 0$.
Such discontinuities prevent $u$ from solving \eqref{eq:elliptic-u}.
Preventing these discontinuities requires that the source term satisfies appropriate orthogonality conditions.

Let us also emphasize that if one wishes to construct solutions of \eqref{eq:eq0-uux} with even stronger regularity, say $u\in H^k_xH^1_y$ with $k\geq 1$, then generically, one needs to ensure that $2k $ orthogonality conditions are satisfied by the source terms (see \cref{lem:H2k-regularity}).
This situation occurs (at a nonlinear level) in \cite{IM2023}.

\paragraph{On orthogonality conditions for nonlinear problems.} 
Of course, such orthogonality conditions make it very difficult to obtain results at a nonlinear level. 
Generally, one tries to avoid such difficulties when considering nonlinear problems.
For instance, for elliptic problems in polygonal domains, the classical textbook \cite[Section 8.1]{MR775683} focuses on a nonlinear case where there is no orthogonality condition at the linear level.

Nevertheless, some results are known in the \emph{semilinear} case.
For example, for semilinear Fredholm operators with negative index, a theoretical toolbox is known (see e.g.\ \cite[Chapter 11, Section 4.2]{Volpert-Vol1}) and has been implemented for some reaction-diffusion semilinear systems (see e.g.\ \cite[Chapter 7, Section 2.2]{Volpert-Vol2}, based on \cite{ducrot2008reaction}).

Outside of the semilinear setting, we are not aware of nonlinear results obtained despite the presence of orthogonality conditions at the linear level prior to our present work (we discuss the  recent preprint \cite{IM2023} by Sameer Iyer and Nader Masmoudi in \cref{sec:motivation}).

Problem \eqref{eq:eq0-uux} is only \emph{quasilinear}, and this makes the analysis harder. 
In an earlier version of this paper, we introduced  a  nonlinear scheme in which the orthogonality conditions  changed at every step. 
Tracking the  evolution of these orthogonality conditions was then a major difficulty. 

Following a very helpful remark by several colleagues\footnote{Felix Otto, Yann Brenier, and an anonymous referee, whom we warmly thank.}, we have changed our strategy.
We first perform a change of unknown which allows us  to keep the same linear operator throughout the scheme, and to treat the nonlinearity perturbatively.
This greatly simplifies the proof. In turn, this change of variables allows us to revisit the work of Iyer and Masmoudi \cite{IM2022,IM2023} on the analysis of the Prandtl system in the vicinity of a recirculating flow, since the equation for the vorticity after the change of variables has a very simple structure, see \cref{sec:intro-sketch}.
Let us also recall that at the nonlinear level, the orthogonality conditions are translated in \cref{thm:burgers} (resp.\ \cref{thm:prandtl}) as the fact that the data must lie within the manifold $\cM$ (resp.\ $\cM_1$), which can be pictured as a perturbation of the linear subspace $\hBPerp$ of data satisfying the orthogonality conditions for the linear problem.

The proof of both of our main theorems (on Burgers and Prandtl) relies on the same abstract result (see \cref{thm:abstract} in \cref{sec:abstract}) concerning quasilinear equations in a perturbative regime.

\paragraph{On entropy solutions.} 
An entirely different approach to solve \eqref{eq:eq0-uux} is to look directly for weak solutions to the nonlinear problem, for example using an entropy formulation.
The regularity for such solutions is $u \in L^\infty_{x,y} \cap L^2_x H^1_y$ and they are typically obtained as limits of solutions $u^\varepsilon$ to regularized versions of \eqref{eq:eq0-uux}, e.g.\ $u^\varepsilon \p_x u^\varepsilon - \p_{yy} u^\varepsilon - \varepsilon \p_{xx} u^\varepsilon = 0$.
Such solutions satisfy both the equation and the lateral boundary conditions only in the weak sense of appropriate inequalities linked with ``entropy pairs''.
Given $\delta_0, \delta_1 \in L^\infty(-1,1)$, the existence of an entropy solution to 
\begin{equation} \label{eq:entropy}
    \begin{cases}
        u \p_x u - \p_{yy} u = 0, \\
        u_{\rvert x = x_i} = \delta_i, \\
        u_{\rvert y = \pm 1} = 0
    \end{cases}
\end{equation}
was first proved in \cite{bocharov1978first}.
More recently, Kuznetsov proved in \cite{kuznetsov2005entropy} the uniqueness of the entropy solution to \eqref{eq:entropy}, determined in which sense the lateral boundary conditions were satisfied and proved a stability estimate of the form
\begin{equation*}
    \| u - \tilde{u} \|_{L^1(\Omega)}
    \lesssim \| \delta_0 - \tilde{\delta}_0 \|_{L^1(-1,1)} + \| \delta_1 - \tilde{\delta}_1 \|_{L^1(-1,1)}.
\end{equation*}
In particular, this stability estimate guarantees that one can construct sign-changing solutions in the vicinity of the linear shear flow.

However, an important drawback of the entropy formulation is that the boundary conditions are only satisfied in a very weak sense. 
Although functions in $L^\infty_{x,y} \cap L^2_x H^1_y$ do not have classical traces at $x = x_i$, one can give a weak sense to the traces using the equation (see \cite{kuznetsov2015traces} for more details).
Unfortunately, it is expected that these weak traces do not coincide with the supplied boundary data on sets of positive measure. 

In contrast, since the solutions we construct in this work have (at least) $H^1_x L^2_y$ regularity, they have usual traces $u_{\rvert \Sigma_i} \in L^2(\Sigma_i)$ and the equalities $u_{\rvert\Sigma_i} = \delta_i$ hold in $L^2(\Sigma_i)$, so almost everywhere.

\paragraph{On the choice of the linear shear flow for equation \eqref{eq:eq0-uux}.}
We choose to study the well-posedness of \eqref{eq:eq0-uux} in the vicinity of the linear shear flow to lighten the computations.
Nonetheless, we expect that our results and proofs can be extended to study the well-posedness of \eqref{eq:eq0-uux} in the vicinity of any sufficiently regular reference flow $\fu$ changing sign across a single curve $\{ \fu = 0 \}$, satisfying $\fu_y \geq c_0 > 0$ in~$\Omega$ (so that \eqref{eq:shear} is the correct toy model) and with either $\| \fu_x \|_\infty$ small enough, or with a restriction on the size of the domain (to ensure \emph{a priori} estimates). 
In fact, this is precisely what we do when we study the Prandtl system around a recirculating flow: the linearized equation for the vorticity then becomes a forward-backward equation with variable coefficients, see \eqref{vorticity-model} and \cref{lem:ortho-Prandtl}.

Moreover, taking a step further in the modeling of recirculation problems in fluid mechanics (see \cref{sec:motivation}), we also expect that our approach could be extended to an unbounded domain of the form $(x_0,x_1) \times (0,+\infty)$, with a reference flow such that $\fu_{\rvert y = 0} = 0$, $\fu < 0$ below some critical curve and then $\fu > 0$ above, with $\fu$ having some appropriate asymptotic behavior as $y \to +\infty$.
In such a setting, the Poincaré inequalities in the vertical direction that we use here should probably be replaced with well-suited Hardy inequalities.
As mentioned above, one of the issues is then to obtain \emph{a priori} estimates on the linearized system. 
We comment further on this point at the end of \cref{sec:strategy-Prandtl-bande-infinie}.

\paragraph{On the conditions $\delta_0(0) = \delta_1(0) = 0$ for fixed end-points.}
It is an important feature of our work that we are able to enforce precisely the exact endpoints of the curve  $\{ u = 0 \}$ at $x = x_0$ and $x = x_1$.
\cref{thm:burgers} and \cref{thm:prandtl} are stated for perturbations which satisfy $\delta_i(0) = 0$ (see \eqref{eq:def-XB-intro}), so that the full boundary data $\fu(x_i, y) + \delta_i(y)$ changes sign exactly at $y = 0$, where $\fu= y$ in \cref{thm:burgers} and $\fu=\ufs$ in \cref{thm:prandtl}.
This choice simplifies the definition of the submanifolds $\cM$, $\cM_0$ and $\cM_1$ of boundary data for which we are able to solve the problem.
Nevertheless, given $y_0, y_1$ sufficiently close to $0$ and $\delta_0, \delta_1$ such that $y + \delta_i(y)$ changes sign at $y = y_i$, we expect that similar existence results hold, provided that the perturbations are chosen in an appropriate modification of $\cM_j$, with suitable modifications to the functional spaces and where, in \eqref{eq:yuuxuyy}, the definitions of $\Sigma_i$ are generalized by setting $\Sigma_i := \{ (x_i,y) ; \enskip (-1)^i (\fu(x_i, y)+\delta_i(y)) > 0 \}$.

\paragraph{On the boundary conditions \eqref{CL-Prandtl-bottom} and \eqref{CL-Prandtl-top} for the Prandtl system in the recirculation zone.}
The boundary conditions we choose for the Prandtl system are mostly meant to simplify the present analysis as much as possible.
As stated earlier, they are slightly unconventional since we prescribe both the trace and the normal derivative of $u$, but we let the boundary remain free.
Other choices of boundary conditions are of course possible, and may lead to additional technical difficulties.
These boundary conditions are designed in order to have a nice formulation after we have performed a change of variables in order to straighten the curve $\{u=0\}$. 

Note also that we only consider here the Prandtl system in the vicinity of the curve $\{u=0\}$, and not in the whole infinite strip $(x_0, x_1)\times (0, + \infty)$.
The coupling with the outside regions $y<\gamma_b(x)$ and $y>\gamma_t(x)$ leads to additional difficulties, which have been treated by Iyer and Masmoudi in \cite{IM2022}, albeit with a different method, since conditions \eqref{CL-Prandtl-bottom}-\eqref{CL-Prandtl-top} are not considered in \cite{IM2022}.
We propose in \cref{sec:strategy-Prandtl-bande-infinie} a possible strategy to solve the Prandtl equation in an infinite strip.

\paragraph{On the compatibility conditions $\delta_i((-1)^i) = 0$ and $\delta_i''(0) = \delta_i''((-1)^i) = 0$ in \cref{thm:burgers}.}
These conditions are classical compatibility conditions for solutions to elliptic-parabolic equations.
For example, the condition $\delta_0(1) = 0$ in \cref{thm:burgers} is intended to match the condition $u_{\rvert y = 1} = 0$, and is necessary to have $L^2_xH^2_y$ regularity.
The condition $\delta_0''(0) = 0$ comes from the equation. 
Indeed, if $u$ is a sufficiently regular solution of \eqref{eq:shear} with $f(x_0,0) = 0$, the equality $\p_{yy} u = y \p_x u$ at $(x_0,0)$ enforces $\p_{yy} u(x_0,0) = 0$, so $\delta_0''(0) = 0$.
The condition $\delta_0''(1) = 0$ stems similarly from the equation and the fact that $\p_x u_{\rvert y = 1} = 0$.
It corresponds to a classical parabolic regularity compatibility condition in order to have $L^2_x H^4_y$ regularity.
We have imposed similar conditions for the Prandtl system by requiring $\delta_i \in H^4_0(\Sigma^P_i)$.
Note that in the Prandtl case, we actually require extra cancellation assumptions.
It is possible that the latter are technical, and could be removed. 
	
\paragraph{On the number of orthogonality conditions for the Prandtl system.}
Note that the number of orthogonality conditions in \cref{thm:burgers} and in \cref{thm:prandtl} is different. The reason for this is twofold.

Firstly, as previously noted by Iyer and Masmoudi in \cite{IM2022}, the ``good unknown'' for the equation is the vorticity $\p_y u$, which satisfies an equation which is very similar to \eqref{eq:shear} in a suitable set of variables (see \eqref{eq:Y-Burgers} and \eqref{eq:W-Prandtl} below). Therefore, in a sense, \eqref{Prandtl-small-domain} is smoother than \eqref{eq:eq0-uux}: indeed, without assuming any orthogonality condition, one can expect the vorticity $\p_y u$ to belong to the functional space $H^{2/3}_x L^2_y \cap L^2_x H^2_y$ (see \cref{thm:shear-decomp}), and therefore the solution of the Prandtl system belongs to  $H^{2/3}_x H^1_y \cap L^2_x H^3_y$, in which we have gained one vertical derivative. On the contrary, without any orthogonality condition, one cannot expect the solution $u$ of \eqref{eq:eq0-uux} to have better regularity than $H^{2/3}_x L^2_y \cap L^2_x H^2_y$, which is insufficient for a fixed-point argument for \eqref{eq:eq0-uux}.
This gain of vertical regularity allows us to have a theory for weaker solutions of the Prandtl system, and therefore to get rid of two of the orthogonality conditions.

Secondly,  reconstructing the velocity from the vorticity in the Prandtl system gives rise to one additional orthogonality condition, as we will explain in \cref{sec:Prandtl}. Hence the number of orthogonality conditions in \cref{thm:prandtl} is odd, while it is even in \cref{thm:burgers}.

\subsection{Motivation from recirculation problems in fluid mechanics}
\label{sec:motivation}

Our original motivation stems from fluid mechanics. Indeed, the stationary Prandtl equation \eqref{Prandtl-small-domain} describes  the behavior of a fluid with small viscosity in the vicinity of a wall.
The pressure $p(x)$ is the trace of the pressure of an outer Euler flow. This equation is usually set in a 2d domain of the form $I\times (0, + \infty)$, where $I\subset \R$ is an interval, and $y=0$ is the solid wall.
The equation is endowed with the boundary conditions $u=v=0$ on $y=0$, and $\lim_{y\to \infty } u(x,y)= u_E(x)$, where
 $u_E(x)$  is the trace of the outer Euler flow on the wall, and satisfies $u_E \p_x u_E = - \p_x p$.

As long as $u$ remains positive, \eqref{Prandtl-small-domain} can be seen as a nonlocal, nonlinear diffusion type equation, the variable $x$ being the evolution variable. Using this point of view, Oleinik (see e.g.\ \cite[Theorem 2.1.1]{MR1697762}) proved the local well-posedness of a solution to \eqref{Prandtl-small-domain} when the equation \eqref{Prandtl-small-domain} is supplemented with a boundary data $u_{|x=0}=u_0$, such that $u_0(y)>0$ for $y>0$ and $u_0'(0)>0$. Let us mention that such positive solutions exist globally when $ \p_x p\leq 0$, but are only local when $\p_x p>0$. More precisely, when $\p_x p=1$ for instance, for a large class of boundary data $u_0$, there exists $x^*>0$ such that $\lim_{x\to x^*} u_y(x,0)=0$. Furthermore, the solution may develop a singularity at $x=x^*$, known as Goldstein singularity. The point $x^*$ is called the separation point: intuitively, if the solution to Prandtl exists beyond $x^*$, then it must have a negative sign close to the boundary (and therefore change sign). We refer to the seminal works of Goldstein \cite{Goldstein} and Stewartson \cite{Stewartson} for formal computations on this problem. A first mathematical statement describing separation was given by Weinan E in \cite{E-2000} in a joint work with Luis Cafarelli, but the complete proof was never published. The first author and Nader Masmoudi then gave a complete description of the formation of the Goldstein singularity \cite{DM}. The  work \cite{Shen-Wang-Zhang} indicates that this singularity holds for a large class of initial data.

Because of this singularity, it is actually unclear that the Prandtl system is a relevant physical model in the vicinity of the separation point $x^*$, because the normal velocity $v$ becomes unbounded at $x=x^*$. Consequently, more refined models, such as the triple deck system (see \cite{Lagree} for a presentation of this model, and \cite{IyerVicol,GVDietert} for a recent mathematical analysis of its time-dependent version), were designed specifically to replace the Prandtl system with a more intricate boundary layer model in the vicinity of the separation point. However, beyond the separation point, i.e.\ for $x>x^*$, it is expected that the Prandtl system becomes valid again, but with a changing sign solution.

The well-posedness of the Prandtl system \eqref{Prandtl-small-domain} when the solution $u$ is allowed to change sign has only recently been investigated.
Such solutions are called ``recirculating solutions'', and the zone where $u<0$ is called a recirculation bubble, the usual convention being that $u_E(x)>0$, so that the flow is going forward far from the boundary. 
In the recent preprint \cite{IM2022} by Sameer Iyer and Nader Masmoudi, the authors prove \emph{a priori} estimates in high regularity norms for smooth solutions to the Prandtl equation \eqref{Prandtl-small-domain} in a domain of the form $I\times (0,+\infty)$, with restrictions on the length of the interval $I$, in the vicinity of explicit self-similar recirculating flows, called Falkner--Skan profiles. The latter are given by
\begin{align}
    u(x,y) & = x^m f'(\zeta), \\
    v(x,y) & = - y^{-1} \zeta f(\zeta) - \frac{m-1}{m+1} y^{-1} \zeta^2 f'(\zeta),
\end{align}
where $\zeta := (\frac{m+1}{2})^{\frac 12} y x^{\frac{m-1}{2}}$ is the self-similarity variable, $m$ is a real parameter and $f$ is the solution to the Falkner-Skan equation
\begin{equation*}
    f''' + f f'' + \beta (1- (f')^2) = 0,
\end{equation*}
where $\beta = \frac{2m}{m+1}$, subject to the boundary conditions $f(0) = f'(0) = 0$ and $f'(+\infty) = 1$. Such flows correspond to an outer Euler velocity field $u_E(x) = x^m$. 
For some particular values of $m$ (or, equivalently, $\beta$), these formulas provide physical solutions to \eqref{Prandtl-small-domain} which exhibit recirculation (see~\cite{MR0204011}). 
Obtaining high regularity \emph{a priori} estimates for recirculating solutions to the Prandtl system \eqref{Prandtl-small-domain} on the whole infinite strip is a difficult task.
This important step was achieved by Sameer Iyer and Nader Masmoudi in \cite{IM2022}.

In the present paper, we have chosen to focus on a different type of difficulty, and to 
consider first the toy-model \eqref{eq:eq0-uux}, which differs from \eqref{Prandtl-small-domain} through the lack of the nonlinear transport term $v \p_y u$ and its associated difficulties (nonlocality, loss of derivative) and the exclusion of the zones close to the wall and far from the wall.
For the model \eqref{eq:eq0-uux}, \emph{a priori} estimates are easy to derive, see \cite[Chapter 4]{theseRax}.
The difficulty lies elsewhere, as explained previously. 
Indeed, in order to construct a sequence of approximate solutions satisfying the \emph{a priori} estimates, we need to ensure that the orthogonality conditions are satisfied all along the sequence.
For the Prandtl system \eqref{Prandtl-small-domain}, this difficulty has  recently been tackled by Sameer Iyer and Nader Masmoudi in \cite{IM2023}, building upon their \emph{a priori} estimates of \cite{IM2022} and the ideas developed in the first version of our present work.
We revisit in \cref{sec:Prandtl} part of their work. Performing the different steps of the analysis (straightening the boundary, linearizing, differentiating with respect to the horizontal or vertical variable), we found a way to substantially simplify the analysis of the system in the vicinity of the recirculating line, and the results that we obtain are slightly different from the ones of \cite{IM2022,IM2023}. 
First, we prove a result in a rather \emph{low regularity setting}, in which $u_x$ is not even an $L^2$ function in the whole domain $\Om_P$. This result holds under merely one orthogonality condition, whose role is rather different from the ones of \cref{thm:burgers}, for instance. Indeed, the role of this additional orthonogality condition is not to ensure a certain regularity, but rather to allow for a reconstruction of the velocity from the vorticity. 
Moreover,  we use solely one change of variables, which we present in the next subsection and  which is identical to the one for the Burgers equation. 
Once the adequate change of variables is identified, we retrieve the fact that the vorticity (in these new variables) is a good unknown.
In fact, in the appropriate set of variables, the equation for the vorticity becomes remarkably simple (it is a closed, quasilinear equation).
We however prescribe lateral boundary conditions on the velocity (rather than the vorticity). 
Eventually, we do not require any condition on the \emph{horizontal size of the domain} (i.e.\ on the length $x_1-x_0$). Our understanding is that such conditions may arise when the Prandtl system in the whole infinite strip is considered. They are linked to the well-posedness of a linearized system in the whole strip. We refer to \cref{sec:strategy-Prandtl-bande-infinie} for more comments regarding this point. 

\subsection{Scheme of proof of nonlinear theorems and plan of the paper}
\label{sec:intro-sketch}

The uniqueness of solutions is fairly easy to prove.
For the linear problem \eqref{eq:shear}, uniqueness already holds at the level of weak solutions (see \cref{p:shear-X0} and \cref{sec:proof-uniqueness}).
For the nonlinear problems, uniqueness is straightforward since we are considering strong solutions.
Therefore, the main subject of this paper is the proof of the existence of solutions for the nonlinear problems \eqref{eq:yuuxuyy} and \eqref{Prandtl-small-domain} endowed with the boundary conditions \eqref{CL-Prandtl-bottom}-\eqref{CL-Prandtl-top}-\eqref{CL-Prandtl-lateral}.

A first natural idea would be to prove existence thanks to a nonlinear scheme relying on the linear problem \eqref{eq:shear}.
For example, concerning  equation  \eqref{eq:eq0-uux}, one could wish to construct a sequence of approximate solutions $(u_n)_{n\in \N}$ by setting $u_0 := 0$ (or any other initial guess) and solving
\begin{equation*}
    \begin{cases}
        y \p_x u_{n+1} - \p_{yy} u_{n+1} = f - (u_n-y) \p_x u_n, \\
        (u_{n+1})_{\rvert \Sigma_i} = \delta_i, \\
        (u_{n+1})_{\rvert y = \pm 1} = 0.
    \end{cases}
\end{equation*}
However, this strategy fails.
The key point is that the right-hand side contains a full tangential derivative of $u_n$, whereas the operator $y \p_x - \p_{yy}$ only yields a gain of $2/3$ of a derivative in this direction (more precisely, see \cref{p:z0-l2-h23}, \cref{rmk:hypo} and \cref{p:pagani-shear}).
Hence, this nonlinear scheme would exhibit a ``loss of derivative'', preventing us from proving a uniform bound on the sequence $(u_n)_{n\in \N}$.

Another drawback of this scheme is that it would not translate well to a setting where one does not assume $\delta_i(0) = 0$.
Indeed, in such a case, the inflow boundaries of the problem with the perturbed data $y + \delta_i(y)$ would not match the inflow boundaries of the linear problem \eqref{eq:shear}.

Hence, we will rather construct solutions to \eqref{eq:eq0-uux} through another iterative scheme.
As suggested to us by an anonymous referee and by other colleagues, we first straighten the curve $\{u=0\}$ by setting as  a new vertical variable  $z=u(x,y)$. 
Our new  unknown, both for  the Burgers equation \eqref{eq:eq0-uux}
and for the Prandtl system \eqref{Prandtl-small-domain}, is the inverse function of $u$, i.e.\ the function $Y$ such that $u(x,Y(x,z))=z$. 
In this new set of variables, the equation for $Y$ becomes, in the case of the Burgers equation \eqref{eq:eq0-uux},
\begin{equation}\label{eq:Y-Burgers}
z\p_x Y - (\p_z Y)^{-2} \p_z^2 Y = -\p_z Y f(x,Y),
\end{equation}
and in the case of the Prandtl system \eqref{Prandtl-small-domain}
\begin{equation}\label{eq:Y-Prandtl}
z\p_x Y -\int_{z_b}^z\p_x Y  - (\p_z Y)^{-2} \p_z^2 Y = -\p_z Y \p_x p + \vfs\vert_{\overline{\Gamma_b}} + v_b,
\end{equation}
in which the nonlocal integral term in the left-hand side stems from the transport term $v\p_y u$ in the original equation.
Differentiating \eqref{eq:Y-Prandtl} with respect to $z$, we find that  in the case of the Prandtl system, the vorticity $W=\p_z Y$ satisfies
\begin{equation}\label{eq:W-Prandtl}
z\p_x W + \p_x p \p_z W + \p_z^2 \left(\frac{1}{W}\right)=0.
\end{equation}
We immediately see that the linearized operator associated with \eqref{eq:Y-Burgers} around $Y(x,z)=z$ is equation \eqref{eq:shear}. In a similar fashion, the linearized operator of equation \eqref{eq:W-Prandtl} around the flow $\Yfs$ associated with $\ufs$ is a forward-backward operator with variable coefficients, of the form $z\p_x + \beta \p_z - \p_z^2(\alpha \cdot)$, where $\alpha= (\p_z \Yfs)^{-2}$ and $\beta = \p_x p$. Such forward-backward operators bear strong similarities with the canonical one $z\p_x -\p_{zz}$, and therefore we will rely on our linear analysis to study \eqref{eq:W-Prandtl} (see \cref{lem:WP-Z0-vorticity} and \cref{lem:ortho-Prandtl}).

We then construct solutions of \eqref{eq:Y-Burgers} and \eqref{eq:Y-Prandtl} thanks to an iterative scheme\footnote{In fact, we will state and use an abstract theorem, whose proof follows a similar scheme.}, which we now explicit in the Burgers case, the Prandtl one being similar.
We define a sequence $(\widetilde{Y}_n)_{n\in \N}$ such that
\begin{equation*}
    z \p_x \widetilde{Y}_{n+1} - \p_{zz} \widetilde{Y}_{n+1} = \frac{\p_z \tY_n (2-\p_z \tY_n)}{(1-\p_z \tY_n)^2} \p_{zz} \widetilde{Y}_{n} + (1 - \p_z \widetilde{Y}_{n}) f(x, z - \widetilde{Y}_{n}) + g^{n+1},
\end{equation*}
where the additional term $g^{n+1}$ ensures that the orthogonality conditions are satisfied at every step. We then prove that $(\widetilde{Y}_n)_{n\in \N}$ is a Cauchy sequence in the space $\qone$. Passing to the limit, we obtain a solution $Y=z-\lim_{n\to \infty} \tY_n$ to \eqref{eq:Y-Burgers} with an additional source term $g$. The manifold $\cM$ is then defined by requiring that the limit term $g$ is zero.

\begin{rmk}
    In a first version of this paper \cite{DMR-3}, we had chosen a strategy which seemed only slightly different, but which led to substantial technical difficulties.
    However, we believe that this strategy is rather natural, and could be of use in other problems. Therefore we describe it here.
    
    Let  $(u_n)_{n\in \N}$ be a sequence solving the following iterative scheme
    \begin{equation}
        \label{eq:iterative-bis}
        \begin{cases}
            u_n \p_x u_{n+1} - \p_{yy} u_{n+1} = f^{n+1}, \\
            (u_{n+1})_{\rvert \Sigma_i} = y + \delta_i^{n+1}, \\
            (u_{n+1})_{\rvert y = \pm 1} = \pm 1.
        \end{cases}
    \end{equation}
    For this scheme, it is possible to prove a uniform bound for $u_n$ in the space $H^{5/3}_x L^2_y \cap L^2_x H^5_y$ and the convergence of the sequence in an interpolation space $L^2_x H^{7/2}_y \cap H^{7/6}_x L^2_y$.
    This scheme is similar to the one used to construct solutions to quasilinear symmetric hyperbolic systems, see for instance \cite[Section 4.3]{BCD}.
    
    In \eqref{eq:iterative-bis}, the triplet $(f^{n+1},\delta_0^{n+1},\delta_1^{n+1})$ is an appropriate perturbation of the data $(f,\delta_0,\delta_1)$ tailored to satisfy the orthogonality conditions associated with the linear operator $u_n\p_x -\p_{yy}$.
    In order to define these orthogonality conditions, it is necessary to straighten the curve $\{u_n(x,y)=0\}$: hence this straightening step is still necessary, but performed \emph{after} the ``linearization'' of the equation, rather than \emph{before}.
    
    The issue lies in the fact that the orthogonality conditions change at every step, which is a key difficulty.
    In particular, in order to allow the sequence $u_n$ to converge, one must prove that these perturbations also converge.
    This amounts to proving that the linear forms associated with the  operator $u_n\p_x -\p_{yy}$ depend continuously (and even in a Lipschitz manner) on~$u_n$, \emph{for the same topology as the one within which one proves the convergence of the sequence $u_n$}.
    This continuity estimate  requires identifying quite precisely what the linear forms  are.
    
    \bigskip
    
    We believe that this methodology is rather robust and could be applied to other nonlinear problems in which orthogonality conditions are present at the linearized level,  in particular in contexts where there is no nonlinear change of variables such as the one presented above allowing to treat the nonlinearity as a perturbation.
    For example, the PDE $u(1+ \p_y u) \p_x u - \p_{yy}u = f$ could be an example where our previous methodology applies, but not the one exposed in the present paper.
\end{rmk}
	
\begin{rmk}
    Let us highlight some differences between the strategy of the present paper and the one by Iyer and Masmoudi in \cite{IM2022}. As explained  above, there are several mathematical operations which are required to complete the proof of \cref{thm:prandtl}:
    \begin{itemize}
        \item Changing variables in order to straighten the free boundary;
        
        \item Linearizing the equation around some background profile;
        
        \item Differentiating the equation with respect to the vertical variable in order to obtain an equation for the vorticity;
        
        \item Differentiating the equation with respect to the tangential variable in order to derive higher order estimates (under compatibility conditions).
    \end{itemize}
    
    These operations more or less commute at main order, and lead to the study of the equation $z\p_x u - \p_{zz} u=f$.
    However their (lower order) commutators may be a source of substantial technical difficulties. 
    Our understanding is that the authors of \cite{IM2022} perform the operations in the following order (see Section 3 of their paper):
    1. Linearize; 2. Differentiate with respect to the horizontal variable; 3. Straighten the free boundary; 4. Differentiate with respect to the vertical variable.
    
    We believe that the computations are much simpler, and the structure is better understood, when the straightening change of variables is performed first. This also allows us to have a more accurate comparison between the Burgers type equation and the Prandtl one.
\end{rmk}
	
\bigskip
	
The plan of this work is as follows. 
As a preliminary, we introduce in \cref{sec:spaces} the functional spaces we will use.
First, we study the linear problem \eqref{eq:shear} in \cref{sec:shear}, leading to \cref{thm:shear-Q1}, and prove that the two orthogonality conditions we expose are indeed nonvoid. We also construct the singular profiles $\busing^i$ and prove \cref{thm:shear-decomp}.
In order to introduce our nonlinear scheme, we extend these linear results to the Vlasov--Poisson--Fokker--Planck system as an example in \cref{sec:Fokker-Planck}, where we also set up our general nonlinear methodology.
We then turn towards the proof of \cref{thm:burgers} in \cref{sec:Burgers}, and the one of \cref{thm:prandtl} in \cref{sec:Prandtl}.
In order to prove the existence of weak solutions of the Prandtl system (i.e.\ the first point of \cref{thm:prandtl}), we will need an interpolation result: this rather technical step is performed in \cref{sec:interpolation}.
Eventually, in \cref{sec:proof-uniqueness}, we prove the uniqueness of weak solutions to various linear problems, by adapting an argument due to Baouendi and Grisvard \cite{BG}.
In \cref{sec:proofs-spaces}, we prove various technical results of functional analysis that we use throughout the paper.
\cref{sec:proof-lem:reg-away-sing-pts} contain the postponed proof of a lemma of \cref{sec:Prandtl}.
	
As the paper is quite long, a list of notations is provided starting page \pageref{sec:notations}.

\subsection{Functional spaces and interpolation results}
\label{sec:spaces}

\subsubsection{Notations}
\label{ssec:notations}
Throughout this work, an assumption of the form ``$A \ll 1$'' will mean that there exists a constant $c > 0$, depending only on $\Omega$ such that, if $A \leq c$, the result holds.
Similarly, a conclusion of the form ``$A \lesssim B$'' will mean that there exists a constant $C > 0$, depending only on $\Omega$ and on the underlying flow (namely $\fu(x,y)=y$ or $\ufs$), such that the estimate $A \leq C B$ holds.
For ease of reading, we will not keep track of the value of these constants, mostly linked with embeddings of functional spaces.
Note in particular that the sizes of these constants will depend on the length $x_1-x_0$ (see e.g.\ \cref{p:pagani-shear}).

\nomenclature[OAZ]{$\Omega_\pm$}{Upper and lower halves of the domain $\Omega$}
We will often use the notations $\Om_\pm := \Om \cap \{\pm z>0\}$.

\subsubsection{Trace spaces for the lateral boundaries}

For the traces of the solutions to \eqref{eq:shear} or \eqref{eq:yuuxuyy} at $x = x_0$ and $x = x_1$, we will need the following spaces, due to \cite{Pagani1,Pagani2}.
We define $\mathscr{L}^2_z(-1,1)$ as the completion of $L^2(-1,1)$ with respect to the following norm:
\begin{equation} \label{eq:L2z}
    \| \psi \|_{\mathscr{L}^2_z} := \left(\int_{-1}^1 |z| \psi^2(z) \dd z\right)^{\frac 12}
\end{equation}
and $\mathscr{H}^1_z(-1,1)$ as the completion of $H^1_0(-1,1)$ with respect to the following norm:
\nomenclature[FL]{$\mathscr{L}^2_z$}{Weighted $L^2$ space for boundary data with norm \eqref{eq:L2z}}
\nomenclature[FH00]{$\mathscr{H}^1_z$}{Weighted $H^1$ space for boundary data with norm \eqref{eq:H-pagani-k}}
\begin{equation} \label{eq:H-pagani-k}
    \| \psi \|_{\mathscr{H}^1_z} := \| \psi \|_{\mathscr{L}^2_z} + \| \p_z \psi \|_{\mathscr{L}^2_z}.
\end{equation}

\subsubsection{Pagani's weighted Sobolev spaces}

Let $\mathcal{O}$ be an open subset of $\R^2$, and let $\Omega:=(x_0,x_1)\times (-1,1)$.
In the works \cite{Pagani1,Pagani2} (albeit with swapped variables with respect to our setting), Pagani introduced the space $Z(\mathcal{O})$ of scalar functions $\phi$ on $\mathcal{O}$ such that $\phi$, $\p_z \phi$, $\p_{zz} \phi$ and $z \p_x \phi$ belong to $L^2(\mathcal{O})$ (in the sense of distributions).
In this work, we will refer to this space with the notation $Z^0(\mathcal{O})$.
It is a Banach space for the following norm
\nomenclature[FZ0]{$Z^0$}{Pagani solution space such that $u$, $z \p_x u$ and $\p_{zz} u$ are $L^2$, with norm \eqref{eq:Z0}}
\begin{equation}
    \label{eq:Z0}
    \| \phi \|_{Z^0}
    := 
    \| z \p_x \phi \|_{L^2} +
    \| \p_{zz} \phi \|_{L^2} + 
    \| \p_z \phi \|_{L^2} + 
    \| \phi \|_{L^2}.
\end{equation}
We will also need the space $Z^1(\mathcal{O})$, which we define as the space of scalar functions $\phi$ on $\mathcal{O}$ such that $\phi$ and $\p_x \phi$ belong to $Z^0(\mathcal{O})$, associated with the following norm
\nomenclature[FZ1]{$Z^1$}{Solution space such that $u, \p_x u \in Z^0$, with norm \eqref{eq:Z1}}
\begin{equation} \label{eq:Z1}
    \| \phi \|_{Z^1}
    := 
    \| \phi \|_{Z^0} + \| \p_x \phi \|_{Z^0}.
\end{equation}
The omitted proofs of the results of this section are postponed to \cref{sec:proofs-spaces}.
We start with a straightforward extension result, which allows transferring results on $Z^0(\R^2)$ to $Z^0(\Omega)$.

\begin{lem} \label{p:extension}
    There exists a continuous extension operator from $Z^0(\Omega)$ to $Z^0(\R^2)$.
\end{lem}

The next embedding is the most important result concerning the spaces $Z^0$.
Since solutions to $(z \p_x - \p_{zz}) u = f$ for $f \in L^2(\Omega)$ belong to $Z^0(\Omega)$ (see \cref{p:pagani-shear}), the following embedding entails that such solutions belong to $H^{2/3}(\Omega)$.

\begin{prop} \label{p:z0-l2-h23}
    $Z^0(\R^2)$ is continuously embedded in $H^{2/3}_x L^2_z$.
\end{prop}

\begin{rmk}
    \label{rmk:hypo}
    \cref{p:z0-l2-h23} can be seen as an hypoellipticity result for the operator $L = \p_{zz} - z \p_x$ in the full space $\R^2$, which is of the form $X_1^2+X_0$, where $X_1 = \p_z$, $X_0 = - z\p_x$ and $[X_0, X_1] = \p_x$, so the Lie brackets generate the full space and $L$ satisfies Hörmander's sufficient condition of \cite{hormander1967hypoelliptic} for hypoellipticity.
    In fact, in the full space $\R^2$, the $H^{2/3}_x L^2_z \cap L^2_x H^2_z$ regularity of solutions to $L u = f$ for $f \in L^2$ can be derived from the general theory of quadratic operators, which makes a link between the anisotropic gain of regularity and the number of brackets one has to take in order to generate a direction.
    For instance, this regularity follows from \cite[Theorem 2.10]{alphonse2019polar} and more precisely Example 2.11 therein applied with
    \begin{equation*}
        R = 0 \quad \text{and} \quad Q = \begin{pmatrix} 0 & 0 \\ 0 & 1 \end{pmatrix} \quad \text{and} \quad B = \begin{pmatrix} 0 & 1 \\ 0 & 0 \end{pmatrix}.
    \end{equation*}
\end{rmk}

\begin{lem} \label{lem:z0-linf-h12}
    $Z^0(\R^2)$ is continuously embedded in $C^0_z(H^{1/2}_x)$.
\end{lem}

\begin{proof}
    By definition, $Z^0(\R^2) \hookrightarrow H^2_z(L^2_x)$.
    By \cref{p:z0-l2-h23}, $Z^0(\R^2) \hookrightarrow L^2_z(H^{2/3}_x)$.
    By the ``fractional trace theorem'' \cite[Equation~(4.7), Chapter~1]{LM68}, $Z^0(\R^2) \hookrightarrow C^0_z(H^{1/2}_x)$.
\end{proof}

\begin{lem} 
    \label{lem:Z0-trace-H1z}
    $Z^0(\Omega)$ is continuously embedded in $C^0([x_0,x_1];\mathscr{H}^1_z(-1,1))$.
\end{lem}

\begin{proof}
    This is contained in the trace result \cite[Theorem~2.1]{Pagani2}.
\end{proof}

\begin{lem}
    \label{lem:Z0-trace-pz-top}
    For $\phi \in Z^0(\Om)$, $\p_z \phi \vert_{z = \pm 1} \in H^{1/4}(x_0,x_1)$.
\end{lem}

\begin{proof}
    Let $\chi_+ \in C^\infty([-1,1])$ such that $\chi \equiv 1$ in a neighborhood of $\{ z = +1 \}$ and $\chi(z) = 0$ for $z < 1/2$.
    For $\phi \in Z^0(\Om)$, $\chi_+(z) \phi(x,z) \in H^1_x L^2_z \cap L^2_x H^2_z$.
    By the ``fractional trace theorem'' \cite[Equation~(4.7), Chapter~1]{LM68}, $(\p_z (\chi_+ \phi)) \vert_{z = +1} \in H^{1/4}(x_0,x_1)$.
    The result follows since $\chi_+ \equiv 1$ near $\{ z = +1 \}$.
    The same argument applies for the trace at $z = -1$.
\end{proof}

\begin{rmk} \label{rmk:Z0-Linf}
    Although it is ``almost'' the case, there does not hold $Z^0(\R^2) \hookrightarrow C^0(\R^2)$.
    \begin{itemize}
        \item Pagani \cite[Theorem 2.1]{Pagani1} proves that the operator $\phi \mapsto \phi(\cdot, 0)$ is onto from $Z^0(\R^2)$ to $H^{\frac 12}(\R)$.
        But $H^{\frac 12}(\R)$ contains unbounded functions of $x$.
        
        \item Pagani \cite[Theorem 2.3]{Pagani1} proves that the operator $\phi \mapsto \phi(0, \cdot)$ is onto from $Z^0(\R^2)$ to the space $\mathscr{H}^1_z(\R)$.
        But this space contains unbounded functions, for example $\psi(z) := (-\ln |z|/2)^s \chi(z)$ for $s < \frac 12$ and $\chi \in C^\infty_c(\R)$ with $\chi \equiv 1$ in a neighborhood of $z=0$.
    \end{itemize}
\end{rmk}

\subsubsection{Baouendi and Grisvard's weak space}
    
In \cite{BG}, Baouendi and Grisvard introduce the space
\nomenclature[FB]{$\cB$}{Baouendi--Grisvard solution space of \eqref{def:cB}, used in \cref{sec:proof-uniqueness}}
\begin{equation} 
    \label{def:cB}
    \cB:=\{\phi \in L^2((x_0, x_1), H^1_0(-1,1)); \enskip z\p_x \phi \in L^2_x(H^{-1}_z)\}.
\end{equation}
Baouendi and Grisvard proved the uniqueness of solutions to \eqref{eq:shear} in $\cB$. 
They also proved that functions in $\cB$ have traces on $\{x=x_i\}$ in $\mathscr{L}^2_z(-1,1)$. 
These results are recalled in \cref{sec:proof-uniqueness}, and will be used abundantly throughout the paper.

The following embedding is proved in \cref{sec:proofs-spaces} and used in \cref{sec:WP-Prandtl}.

\begin{lem}
    \label{lem:embed-cB-H13x}
    $\cB$ is continuously embedded in $H^{1/3}_x L^2_z$.
\end{lem}

\begin{lem}
    \label{lem:embed-cB-C0z-L3x}
    $\cB$ is continuously embedded in $C^0_z([-1,1]; H^{1/6}_x)$.
\end{lem}

\begin{proof}
    By definition, $\cB \hookrightarrow H^1_z (L^2_x)$.
    By \cref{lem:embed-cB-H13x}, $\cB \hookrightarrow L^2_z (H^{1/3}_x)$.
    Hence, the result follows from the ``fractional trace theorem'' \cite[Equation~(4.7), Chapter~1]{LM68}.
\end{proof}

\subsubsection{Anisotropic Sobolev spaces}

We will construct solutions to \eqref{eq:shear}, \eqref{eq:yuuxuyy} and \eqref{Prandtl-small-domain} in various anisotropic Sobolev spaces such as $\qone$ of \eqref{eq:def-Q1-intro}.
Within these spaces, one has heuristically the correspondence $\p_x \approx \p_z^3$, which corresponds to the appropriate scaling due to the degeneracy of $z \p_x$ at $z = 0$.

Indeed, if $u$ is a solution to  $z\p_x u-\p_{zz}u=0$ say on the whole plane $\R^2$, then the rescaled functions $u_\lambda(x,z) := u(\lambda^3 x, \lambda z)$ are also solutions. 
This is also consistent with the shape of the singular profiles $\busing^i$ from \cref{thm:shear-decomp}, and leads to the rule of thumb ``one derivative in $x$ equals three derivatives in $z$'' (which is different from the usual parabolic scaling, because of the cancellation on the line $z=0$).

In particular, we will use abundantly the following embeddings from the interpolated Pagani spaces to anisotropic Sobolev ones.

\begin{lem}
    \label{lem:Zsigma-Qsigma}
    Let $\sigma \in [0,1]$ and $Z^\sigma(\Omega) := [Z^0(\Om),Z^1(\Om)]_\sigma$.
    Then $Z^\sigma \hookrightarrow H^{2/3+\sigma}_x L^2_z \cap H^\sigma_x L^2_z$.
    In particular $Z^1 \hookrightarrow \qone$ defined in \eqref{eq:def-Q1-intro}.
\end{lem}

\begin{proof}
    By \cref{p:z0-l2-h23}, $Z^0(\Omega) \hookrightarrow H^{2/3}_x L^2_z$ and $Z^1(\Omega) \hookrightarrow H^{5/3}_x L^2_z$.
    Hence
    \begin{equation*}
        Z^\sigma(\Omega) \hookrightarrow [H^{2/3}_x L^2_z,H^{5/3}_x L^2_z]_\sigma = H^{2/3+\sigma}_x L^2_z
    \end{equation*}
    using e.g.\ \cite[Equation (13.4), Chapter 1]{LM68}.
    
    Moreover, by definition, $Z^0(\Omega) \hookrightarrow L^2_x H^2_z$ and $Z^1(\Omega) \hookrightarrow H^1_x H^2_z$, so $Z^\sigma(\Omega) \hookrightarrow H^\sigma_x H^2_z$.
\end{proof}

\begin{rmk}
    The definition \eqref{eq:def-Q1-intro} of $\qone$ does not contain the ``full vertical'' regularity $L^2_x H^5_z$, since we do not need it to close our nonlinear estimates.
    However, assuming sufficient regularity on the source terms, one can build solutions to \eqref{eq:shear} and \eqref{eq:yuuxuyy} in $\qone \cap L^2_x H^5_x$, and this was in fact what we did in the earlier version \cite{DMR-3} of this work.
\end{rmk}

\newpage
\section{The case of the linear shear flow}
\label{sec:shear}

This section concerns the well-posedness of the linear system \eqref{eq:shear} which we restate here for convenience and by using $z$ as a vertical variable rather than $y$ to prepare for the next sections.
We thus consider, in $\Omega = (x_0,x_1) \times (-1,1)$, the system
\begin{equation}
    \label{eq:shear-z}
    \begin{cases}
        z \p_x u - \p_{zz} u = f,\\
        u_{\rvert \Sigma_i} = \delta_i, \\
        u_{\rvert z = \pm 1} = 0,
    \end{cases}
\end{equation}
where $\Sigma_0 = \{ x_0 \} \times (0,1)$ and $\Sigma_1 = \{ x_1 \} \times (-1, 0)$.

First, in \cref{sec:shear-weak}, we recall the theory of weak solutions, due to Fichera for the existence, and to Baouendi and Grisvard for the uniqueness.
Then, in \cref{sec:shear-strong}, we recall the theory of strong solutions with maximal regularity, due to Pagani.
Our contributions regarding this problem are contained in the following subsections.
In \cref{sec:shear-ortho}, we derive orthogonality conditions which are necessary to obtain higher tangential regularity and prove the existence result of \cref{thm:shear-Q1}.
In \cref{sec:radial}, we construct explicit singular solutions and prove the decomposition result of \cref{thm:shear-decomp}.
Eventually, in \cref{sec:shear-frac}, we state a result concerning the well-posedness of \eqref{eq:shear-z} with fractional tangential regularity, which will be used in \cref{sec:Prandtl} and proved in \cref{sec:interpolation}.
 
\subsection{Existence and uniqueness of weak solutions}
\label{sec:shear-weak}

\begin{defi}[Weak solution] \label{def:weak-shear}
    Let $f \in L^2((x_0,x_1);H^{-1}(-1,1))$ and $\delta_0, \delta_1 \in \mathscr{L}^2_z(-1,1)$. 
    We say that $u \in L^2((x_0,x_1);H^1_0(-1,1))$ is a  \emph{weak solution} to~\eqref{eq:shear-z} when, for all $v \in H^1(\Omega)$ vanishing on $\partial\Om \setminus (\Sigma_0\cup\Sigma_1)$, the following weak formulation holds
    \begin{equation*}
        - \int_\Omega z u \p_x v + \int_\Om \p_z u \p_z v
        = \int_\Omega f v  + \int_{\Sigma_0} z \delta_0 v - \int_{\Sigma_1} z \delta_1 v.
    \end{equation*}
\end{defi}

Weak solutions in the above sense are known to exist since the work Fichera \cite[Theorem XX]{MR0111931} (which concerns generalized versions of \eqref{eq:shear-z}, albeit with vanishing boundary data).
Uniqueness dates back to \cite[Proposition 2]{BG} by Baouendi and Grisvard. 

\begin{prop} \label{p:shear-X0}
    Let $f \in L^2((x_0,x_1);H^{-1}(-1,1))$ and $\delta_0, \delta_1 \in \mathscr{L}^2_z(-1,1)$. 
    There exists a unique weak solution $u \in L^2((x_0,x_1);H^1_0(-1,1))$ to \eqref{eq:shear-z}.
    Moreover,
    \begin{equation} \label{shear-estimate-X0}
        \| u \|_{L^2_x H^1_z} \lesssim \| f \|_{L^2_x(H^{-1}_z)} + \| \delta_0 \|_{\mathscr{L}^2_z} + \| \delta_1 \|_{\mathscr{L}^2_z}.
    \end{equation}
\end{prop}

\begin{proof}
    The proof of uniqueness is postponed to \cref{sec:proof-uniqueness} where we adapt Baouendi and Grisvard's arguments to prove the uniqueness of weak solutions to all the linear problems we encounter in this paper in \cref{lem:uniqueness-BG}.
    It relies on the proof of a trace theorem and a Green identity for the space $\cB$ defined in \eqref{def:cB}.
    
    Let us prove the existence.
    We introduce two Hilbert spaces $\mathscr{V} \hookrightarrow \mathscr{U} \hookrightarrow L^2((x_0,x_1);H^1_0(0,1))$ as follows.
    Let $\mathscr{V} := \{ v \in H^1(\Omega) ; \enskip v = 0 \text{ on } \Om \setminus (\Sigma_0 \cup \Sigma_1) \}$.
    Let $\mathscr{U}$ be the completion of $H^1(\Omega) \cap L^2((x_0,x_1);H^1_0(-1,1))$ with respect to the scalar product
    \begin{equation} \label{eq:scalar-u}
        \langle u, v \rangle_{\mathscr{U}} := \int_\Omega \p_z u \p_z v + \int_{\Sigma_0} z u v - \int_{\Sigma_1} z u v.
    \end{equation}
    For $u, v \in \mathscr{U} \times \mathscr{V}$, let
    \begin{align}
        a(u,v) & := - \int_\Omega z u \p_x v + \int_\Omega \p_z u \p_z v, \\
        \label{def:bv}
        b(v) & := \int_\Omega f v + \int_{\Sigma_0} z \delta_0 v - \int_{\Sigma_1} z \delta_1 v.
    \end{align}
    In particular, for every $v \in \mathscr{V}$, integration by parts leads to $a(v,v) = \|v\|_\mathscr{U}^2$ and
    \begin{equation} \label{eq:bv-X0}
        |b(v)| \leq \left( \| f \|_{L^2_x(H_z^{-1})} + \| \delta_0 \|_{\mathscr{L}^2_z} + \| \delta_1 \|_{\mathscr{L}^2_z} \right) \| v \|_{\mathscr{U}}.
    \end{equation}
    Hence, $b \in \mathcal{L}(\mathscr{V})$ can be extended as a linear form over $\mathscr{U}$ and existence follows from the Lax-Milgram type existence principle \cref{lax-weak} in \cref{sec:proofs-spaces}, which also yields the energy estimate \eqref{shear-estimate-X0} thanks to \eqref{eq:bv-X0} and Poincaré's inequality.
\end{proof}

\begin{rmk}
    Functions in $\mathscr{U}$ \emph{a priori} do not have traces on $\Sigma_i$ so one could wonder how definition \eqref{def:bv} makes sense when $v \in \mathscr{U}$.
    The integrals $\int_{\Sigma_i} z \delta_i v$ make sense precisely because $\mathscr{U}$ is defined as a completion with respect to \eqref{eq:scalar-u}.
    In fact, weak solutions do have traces in a strong sense, as proved in \cref{lem:tracesBG}, thanks to the extra regularity in $x$ provided by the equation.
\end{rmk}

\begin{rmk}
    Instead of using the weak Lax-Milgram existence principle \cref{lax-weak}, an alternate proof would be to regularize equation \eqref{eq:shear-z} by vanishing viscosity, and to obtain uniform $L^2_x H^1_z$ estimates on the approximation.
    This approach will be used in \cref{lem:reg-away-sing-pts} proved in the Appendix, in which we prove $H^1_x L^2_z$ regularity of the weak solutions far from the lateral boundaries.
\end{rmk}

\subsection{Strong solutions with maximal regularity}
\label{sec:shear-strong}

We now turn to strong solutions, i.e.\ solutions for which \eqref{eq:shear-z} holds almost everywhere.
The main result on this topic is due to Pagani. 

\begin{prop} \label{p:pagani-shear}
    Let $f \in L^2(\Omega)$ and $\delta_0, \delta_1 \in \mathscr{H}^1_z(-1,1)$ such that $\delta_0(1) = \delta_1(-1) = 0$.
    The unique weak solution $u$ to \eqref{eq:shear-z} belongs to $Z^0(\Omega)$ and satisfies
    \begin{equation} \label{estimate-pagani-Z0}
        \| u \|_{Z^0} \lesssim \| f \|_{L^2} + \| \delta_0 \|_{\mathscr{H}^1_z} + \| \delta_1 \|_{\mathscr{H}^1_z}.
    \end{equation}
    The boundary conditions $u_{\rvert \Sigma_i} = \delta_i$ hold in the sense of traces in $\mathscr{H}^1_z(\Sigma_i)$ (see \cref{lem:Z0-trace-H1z}).
\end{prop}

\begin{proof}
    This is a particular case of \cite[Theorem~5.2]{Pagani2}.
    Pagani's proof proceeds by localization.
    Far from the critical points $(x_0,0)$ and $(x_1,0)$, the regularity is rather straightforward.
    Near these critical points, the regularity stems from the regularity obtained for a similar problem set in a half-space $(0,+\infty)\times \R$ or $\R \times (0,+\infty)$.
    Pagani studies such half-space problems in \cite{Pagani1} where he derives explicit representation formulas for the solutions, using the Mellin transform and the Wiener-Hopf method.
    We do not reproduce these arguments here for brevity.

    Note that the implicit constant, say $C_P$, in \eqref{estimate-pagani-Z0} may depend on $\Omega$.
    By scaling arguments, one can prove that $C_P \leq C ( 1 + |x_1-x_0|^{-1})$ for some universal $C > 0$.
\end{proof}

\subsection{Orthogonality conditions for higher tangential regularity}
\label{sec:shear-ortho}

We now investigate whether solutions to \eqref{eq:shear-z} enjoy higher regularity in the horizontal direction.
As mentioned in \cref{sec:comments}, it is quite easy to obtain \emph{a priori} estimates in the space $Z^1(\Omega)$ (see \cref{p:shear-apriori-Z1}).
However, we prove in \cref{p:shear-WP-Z1} that the weak solution enjoys such a regularity \emph{if only if the data satisfies appropriate orthogonality conditions}.
Eventually, we give statements highlighting the fact that these conditions are non-empty.

\begin{prop} \label{p:shear-apriori-Z1}
    Let $f \in H^1((x_0,x_1);H^{-1}(-1,1))$ and $\delta_0, \delta_1 \in \mathscr{H}^1_z(-1,1)$ such that $\delta_0(1) = \delta_1(-1) = 0$ and such that $\Delta_0, \Delta_1 \in \mathscr{L}^2_z(-1,1)$, where
    \nomenclature[OAD3]{$\Delta_i$}{Boundary data for $\p_x u$, given by $\Delta_i = (f+\delta_i'')/z$, see \eqref{eq:def-Di}}
    \begin{equation} \label{eq:def-Di}
        \Delta_i(z) := \frac{f(x_i,z) + \p_z^2 \delta_i(z)}{z}.
    \end{equation}
    If the unique weak solution $u$ to \eqref{eq:shear-z} belongs to $H^1((x_0,x_1);H^1_0(-1,1))$, then one has the following weak solution estimate for $\p_x u$:
    \begin{equation} \label{eq:apriori-ux}
        \| \p_x u \|_{L^2_x H^1_z} 
        \lesssim 
        \| \p_x f \|_{L^2_x(H^{-1}_z)} +
        \| \Delta_0 \|_{\mathscr{L}^2_z(\Sigma_0)} + 
        \| \Delta_1 \|_{\mathscr{L}^2_z(\Sigma_1)}.
    \end{equation}
    If, moreover, $f \in H^1((x_0,x_1);L^2(-1,1))$, $\Delta_0, \Delta_1 \in \mathscr{H}^1_z(-1,1)$ and $\Delta_0(1) = \Delta_1(-1) = 0$, then $u \in Z^1(\Omega)$ and one has the following strong solution estimate for $\p_x u$:
    \begin{equation} \label{eq:apriori-uz1}
        \| \p_x u \|_{Z^0} 
        \lesssim \| \p_x f \|_{L^2} + 
        \| \Delta_0 \|_{\mathscr{H}^1_z} + 
        \| \Delta_1 \|_{\mathscr{H}^1_z}.
    \end{equation}
\end{prop}

\begin{proof}
    The key point is the following argument:
    if $\p_x u$ enjoys $L^2_x H^1_z$ regularity, then  $\p_x u$ is the unique weak solution to
    \begin{equation} \label{eq:ux}
        \begin{cases}
            z \p_x w - \p_{zz} w = \p_x f,\\
            w_{\rvert \Sigma_i} = \Delta_i, \\
            w_{\rvert z = \pm 1} = 0.
        \end{cases}
    \end{equation}
    Then estimate \eqref{eq:apriori-ux} follows from \eqref{shear-estimate-X0} and estimate \eqref{eq:apriori-uz1} follows from \eqref{estimate-pagani-Z0}.
    
    Hence, let us prove that, if $\p_x u \in L^2_x H^1_z$, then $\p_x u$ is a weak solution to \eqref{eq:ux}.
    Let
    \begin{equation*}
        \begin{split}
            \mathscr{V} := \big\{ v \in C^\infty(\overline{\Omega}) ; \quad & 
            v = 0 \text{ on } \partial \Om \setminus (\Sigma_0 \cup \Sigma_1), \\ 
            & \p_x v = 0 \text{ on } \{ x_0 \} \times (-1,0) \text{ and } \{ x_1 \} \times (0,1) \}.
        \end{split}
    \end{equation*}
    Let $v \in \mathscr{V}$. 
    Then $\p_x v$ is an admissible test function for \cref{def:weak-shear}.
    Hence, since $u$ is the weak solution to \eqref{eq:shear-z}, one has
    \begin{equation*}
        - \int_\Omega z u \p_x (\p_x v) + \int_\Om \p_z u \p_z (\p_x v)
        = \int_\Omega f (\p_x v)  + \int_{\Sigma_0} z \delta_0 (\p_x v) - \int_{\Sigma_1} z \delta_1 (\p_x v).
    \end{equation*}
    The $H^1_x H^1_z$ regularity of $u$ legitimates integrations by parts in $x$ in the left-hand side.
    Thus
    \begin{equation*}
        \begin{split}
            \left[ - \int_{-1}^{1} z u \p_x v \right]^{x_1}_{x_0}
            + \int_\Omega z (\p_x u) \p_x v 
            & + \left[ \int_{-1}^1 \p_z u \p_z v \right]^{x_1}_{x_0}
            - \int_\Omega \p_z (\p_x u) \p_z v \\
            & = \left[ \int_{-1}^1 f v \right]^{x_1}_{x_0} - \int_\Omega f_x v
            + \int_{\Sigma_0} z \delta_0 (\p_x v) - \int_{\Sigma_1} z \delta_1 (\p_x v),
        \end{split}
    \end{equation*}
    which, after taking the boundary conditions into account, integrating by parts in $z$ in the boundary terms $\int_{-1}^1 \p_z u \p_z v$ and recalling \eqref{eq:def-Di} yields
    \begin{equation*}
        - \int_\Omega z (\p_x u) \p_x v + \int_\Om \p_z (\p_x u) \p_z v
        = \int_\Omega f_x v  + \int_{\Sigma_0} z \Delta_0 v - \int_{\Sigma_1} z \Delta_1 v.
    \end{equation*}
    Since $\mathscr{V}$ is dense in the set of test functions for \cref{def:weak-shear}, this proves that $\p_x u$ is the weak solution to \eqref{eq:ux}.
\end{proof}

We start by defining ``dual profiles'' which are necessary to state our orthogonality conditions.

\begin{lem}[Dual profiles] \label{lem:def-dual-shear}
    We define $\overline{\Phi^0}$, $\overline{\Phi^1} \in Z^0(\Om_\pm)$ as the unique solutions to
    \nomenclature[OAW]{$\overline{\Phi^j}$}{Dual profiles of \cref{lem:def-dual-shear} involved in orthogonality conditions for the shear flow}
    \begin{equation} \label{eq:Phij-shear}
        \begin{cases}
            -z \p_x \overline{\Phi^j} -\p_{zz}\overline{\Phi^j} =0 & \text{in } \Om_\pm,\\
            \left[\overline{\Phi^j} \right]_{|z=0}=\mathbf 1_{j=1}, \\ \left[\p_z\overline{\Phi^j} \right]_{|z=0}=-\mathbf 1_{j=0},\\
            \overline{\Phi^j}_{\rvert \p \Om \setminus(\Sigma_0\cup \Sigma_1)}=0.
        \end{cases}
    \end{equation}
\end{lem}

\begin{proof}
    Uniqueness is straightforward.
    Given $j \in \{ 0, 1 \}$ and two solutions to \eqref{eq:Phij-shear}, let $\phi$ denote their difference.
    Then $\phi \in Z^0(\Om_\pm)$ and both $\phi$ and $\p_z \phi$ are continuous across the line $\{ z = 0 \}$. 
    Hence, $\phi \in Z^0(\Om)$ and $\phi$ is the solution to a problem of the form \eqref{eq:shear-z} (with reversed tangential direction).
    So $\phi = 0$ since weak solutions to such problems are unique in $Z^0$.
    
    We prove the existence of $\overline{\Phi^0}$. 
    We define $\overline{\Phi^0}(x,z) := - z \mathbf{1}_{z>0} \zeta(z) + \Psi^0(x,z)$, where we choose $\zeta\in C^\infty_c(\R)$ such that $\zeta\equiv 1$ in a neighborhood of $z=0$ and $\supp \zeta\subset (-1/2, 1/2)$, and where $\Psi^0 \in L^2((x_0,x_1);H^1_0(-1,1))$ is the unique weak solution to
    \begin{equation*}
    \begin{cases}
        - z \p_x \Psi^0 - \p_{zz} \Psi^0 = -2 \mathbf 1_{z>0} \zeta'(z) - z \mathbf{1}_{z>0} \zeta''(z) & \text{in } \Om, \\
        \Psi^0(x_0,z) = 0 & \text{for } z \in (-1,0), \\
        \Psi^0(x_1,z) = z \zeta(z) & \text{for } z \in (0,1), \\
        \Psi^0_{|z=\pm 1}=0.
    \end{cases}
    \end{equation*}
    By \cref{p:pagani-shear}, $\Psi^0 \in Z^0(\Omega)$.
    Hence $\p_{zz} \overline{\Phi^0} \in L^2(\Om_\pm)$ and $z\p_x \overline{\Phi^0} \in L^2(\Om_\pm)$.
    
    The construction of the profile $\overline{\Phi^1} $ is similar and is left to the reader.
    For example, one can decompose $\overline{\Phi^1}$ as $\overline{\Phi^1}(x,z) = \mathbf{1}_{z>0} \zeta(z) + \Psi^1(x,z)$, where, similarly, $\Psi^1 \in Z^0(\Omega)$.
\end{proof}

\begin{rmk} \label{rmk:Phij-saut}
    The jump conditions in \eqref{eq:Phij-shear} prevent the dual profiles from enjoying vertical regularity across the line $\{ z = 0 \}$.
    More subtly, even inside each half-domain, neither the $\overline{\Phi^j}$ nor their lifted version the $\Psi^j$ enjoy tangential regularity.
    Indeed, formally, $\p_x \overline{\Phi^j}$ and $\p_x \Psi^j$ satisfy systems of the form \eqref{eq:shear-z} (with reversed tangential direction) with zero source term and zero boundary data.
    Hence, if they were sufficiently regular, they would be zero by the uniqueness results of \cref{sec:proof-uniqueness}, and so would $\overline{\Phi^j}$ and $\Psi^j$ by integration, contradicting \eqref{eq:Phij-shear}.
    We will see in \cref{cor:decomp-Phij} that these dual profiles indeed do contain an explicit singular part localized near the endpoints $(x_i,0)$.
\end{rmk}
	
We now turn to the main result of this section, which gives a necessary and sufficient condition for the solutions to enjoy the mentioned tangential regularity.
Strangely, we could not find a proof of \cref{p:shear-WP-Z1} in the literature, although some works mention orthogonality conditions (see \cite[Equation (4.2)]{MR0111931} or \cite{pyatkov2019some}).
Hence, we provide here a full proof. 
This strategy will be extended in \cref{sec:vorticity} to equations with smooth variable coefficients (see \cref{lem:ortho-Prandtl}).
We prove further that these orthogonality conditions are not empty.

We will work with the following space\nomenclature[FHK]{$\cLin_K$}{Hilbert space with norm \eqref{eq:norm-HK} of data triplets $(f,\delta_0,\delta_1)$ for the shear flow problem} of data triplets:
\begin{equation}
    \label{eq:def-HK}
    \begin{split}
        \cLin_K := \Big\{ 
            (f,\delta_0,\delta_1) & \in H^1_x L^2_z \times \mathscr{H}^1_z(\Sigma_0) \times \mathscr{H}^1_z(\Sigma_1) ; \quad  (\Delta_0,\Delta_1) \in \mathscr{H}^1_z(\Sigma_0) \times \mathscr{H}^1_z(\Sigma_1)  \\
            & \text{ and }
            \delta_0(1) = \delta_1(-1) = \Delta_0(1) = \Delta_1(-1) = 0 
        \Big\},
    \end{split}
\end{equation}
where $\Delta_i$ is defined in \eqref{eq:def-Di}, with the associated norm
\begin{equation}
    \label{eq:norm-HK}
    \| (f,\delta_0,\delta_1) \|_{\cLin_K} := 
    \| f \|_{H^1_x L^2_z} + \sum_{i \in \{0,1\}} \|\delta_i\|_{\mathscr{H}^1_z} + \| \Delta_i \|_{\mathscr{H}^1_z}.
\end{equation}

\begin{lem} \label{lem:delta_i=H2}
    For $(f,\delta_0,\delta_1) \in \cLin_K$, one has $\delta_i \in H^2(\Sigma_i)$ with $\|\delta_i\|_{H^2} \lesssim \| (f,\delta_0,\delta_1) \|_{\cLin_K}$.
\end{lem}

\begin{proof}
    For $i \in \{0,1\}$, recalling \eqref{eq:def-Di}, one has
    \begin{equation*}
        \| \delta_i'' \|_{L^2(\Sigma_i)} \leq \| \delta_i'' + f(x_i,\cdot) \|_{L^2(\Sigma_i)} + \| f(x_i, \cdot) \|_{L^2(\Sigma_i)} \\
        \lesssim \| \Delta_i \|_{\mathscr{L}^2_z} + \| f \|_{H^1_x L^2_z}.
    \end{equation*}
    Moreover, one checks that $\|\delta_i\|_{L^2} \lesssim \|\delta_i\|_{\mathscr{L}^2_z} + \|\delta_i''\|_{L^2}$ (proceeding e.g.\ as in \cref{lem:psi-zpsi-psi''}).
\end{proof}
	
\begin{prop} \label{p:shear-WP-Z1}
    Let $(f,\delta_0,\delta_1) \in \cLin_K$.
    The unique weak solution $u$ to~\eqref{eq:shear-z} belongs to $H^1_x H^1_z$ if and only if, for $j = 0$ and $j = 1$,
    \begin{equation} \label{eq:compat-shear}
        \int_\Om \p_x f \overline{\Phi^j} + \int_{\Sigma_0} z \Delta_0 \overline{\Phi^j}  - \int_{\Sigma_1} z \Delta_1 \overline{\Phi^j}  = \p_z^j\delta_1(0)-\p_z^j \delta_0(0),
    \end{equation}
    where $\overline{\Phi^0} $ and $\overline{\Phi^1} $ are defined in \cref{lem:def-dual-shear}.

    Furthermore, in this case, $u$ actually belongs to $Z^1(\Om)$ and the following estimate holds:
    \begin{equation} 
        \label{eq:u-Z1}
        \| u\|_{Z^1} \lesssim \| (f,\delta_0,\delta_1) \|_{\cLin_K}.
    \end{equation}
\end{prop}
	
\begin{proof}
    \step{We exhibit possible discontinuities.}
    Let us consider the unique solution $u\in Z^0(\Omega)$ to \eqref{eq:shear-z}.
    Following the strategy sketched by Goldstein and Mazumdar \cite[Theorem~4.2]{AG} (see \cref{rmk:goldstein} for further comments), we introduce the unique strong  solution $w \in Z^0(\Omega)$ to \eqref{eq:ux}, so that $w$ is a good candidate for $\p_x u$.
    The idea is then to introduce the function $u_1$ defined by
    \begin{equation}\label{def:v_1}
    u_1(x,z) :=
    \begin{cases}
        \delta_0 (z) + \int_{x_0}^x w(x',z)\dd x' & \text{in } \Omega_+, \\
        \delta_1(z)-\int^{x_1}_x w(x',z)\dd x' & \text{in } \Omega_-
    \end{cases}
    \end{equation}
    so that $\p_x u_1=w$ almost everywhere. 
    Furthermore, it can be easily proved that, in $\mathcal{D}'(\Om_\pm)$,
    \begin{equation*}
        z \p_x u_1 - \p_{zz} u_1 = f.
    \end{equation*}
    However, this does not entail that $u_1$ is a solution to this equation in the whole domain. 
    Indeed, $u_1$ and $\p_z u_1$ may have discontinuities across the line $\{z=0\}$.
    One checks that $u_1$ and $\p_z u_1$ are continuous across $z=0$ if and only if
    \begin{equation}\label{pas-de-saut}
    \begin{aligned}
    \int_{x_0}^{x_1} w(x,0)\dd x&=\delta_1(0)-\delta_0(0),\\
    \int_{x_0}^{x_1} w_z(x,0)\dd x&= \p_z \delta_1 (0)- \p_z \delta_0(0).
    \end{aligned}
    \end{equation}
    The two integrals are well-defined since $w_z$ and $w_{zz}$ belong to $L^2(\Om)$.
    
    \step{We compute the horizontal mean value of $w$ and $w_z$ using the dual profiles.}
    Let $\phi \in Z^0(\Om_\pm)$ such that $\phi_{\rvert \p \Om \setminus (\Sigma_0\cap \Sigma_1)} = 0$.
    Since $w \in Z^0(\Omega)$, it satisfies \eqref{eq:ux} almost everywhere, so that we can multiply the equation by $\phi$ and integrate over $\Om_+$.
    Hence,
    \begin{equation*}
        \int_{\Om_+} f_x \phi = \int_{\Om_+} (z \p_x w - \p_{zz} w) \phi,
    \end{equation*}
    where, on the one hand,
    \begin{equation*}
        \int_{\Om_+} z (\p_x w) \phi = \int_{\Sigma_1} z \Delta_1 \phi - \int_{\Om_+} z w \p_x \phi
    \end{equation*}
    and on the other hand,
    \begin{equation*}
        - \int_{\Om_+} (\p_{zz} w) \phi 
        = \int_{x_0}^{x_1} (\p_z w \phi - w \p_z \phi)(x,0^+) \dd x - \int_{\Om_+} w \p_{zz}\phi.
    \end{equation*}
    Thus, performing the same computation on $\Om_-$ and summing both contributions yields
    \begin{equation*}
        \begin{split}
            \int_{x_0}^{x_1} (\p_z w [\phi]_{\rvert z = 0} - w [\p_z \phi]_{\rvert z = 0})(x,0) \dd x = \int_{\Om} f_x \phi & + \int_{\Sigma_0} z \Delta_0 \phi - \int_{\Sigma_1} z \Delta_1 \phi \\
            & +\sum_\pm \int_{\Om_{\pm}} w (z \p_x \phi + \p_{zz} \phi).
        \end{split}
    \end{equation*}
    Hence, for $j \in \{0,1\}$,
    \begin{equation*}
        \int_{x_0}^{x_1} \p_z^j w(x,0) \dd x =  \int_{\Om} f_x \overline{\Phi^j} + \int_{\Sigma_0} z \Delta_0 \overline{\Phi^j} - \int_{\Sigma_1} z \Delta_1 \overline{\Phi^j},
    \end{equation*}
    where the dual profiles $\overline{\Phi^0}$ and $\overline{\Phi^1}$ are defined in \cref{lem:def-dual-shear}.
    
    \step{Conclusion.}
    Assume that the orthogonality conditions \eqref{eq:compat-shear} are satisfied for $j = 0$ and $j = 1$.
    Then~\eqref{pas-de-saut} holds, and a consequence, $[u_1]_{|z=0}=[\p_z u_1]_{|z=0}=0$. Thus $u_1 \in L^2((x_0,x_1);H^1_0(-1,1))$ is a weak solution to \eqref{eq:shear-z}.
    We infer from the uniqueness of weak solutions that $u=u_1$, and therefore $\p_x u=w\in Z^0$.
    Hence $u \in H^1((x_0,x_1);H^1_0(-1,1))$. 
    Estimate \eqref{eq:u-Z1} follows from \eqref{estimate-pagani-Z0} and \eqref{eq:apriori-uz1}.
    
    Conversely, if $u$ is a solution to \eqref{eq:shear-z} with $H^1((x_0,x_1);H^1_0(-1,1))$ regularity, then $\p_x u$ is a weak solution to \eqref{eq:ux} (see the proof of \cref{p:shear-apriori-Z1}) and $u$ is given in terms of $\p_x u$ by \eqref{def:v_1} almost everywhere. 
    Thus $[u_1]_{|z=0}=[\p_z u_1]_{|z=0}=0$.
    Hence $\int_{x_0}^{x_1} u_x(x,0)\dd x= \delta_1(0) - \delta_0(0)$ and $\int_{x_0}^{x_1} u_{xz}(x,0)\dd x=\p_z \delta_1(0) - \p_z \delta_0(0)$, and thus the orthogonality conditions \eqref{eq:compat-shear} are satisfied.
\end{proof}

\begin{rmk}
    \label{rmk:goldstein}
    Oddly, in \cite[Theorem~4.2]{AG}, Goldstein and Mazumdar do not mention the orthogonality conditions \eqref{eq:compat-shear}. 
    They merely state that, ``since $\p_{zz} u_1 = z \p_x u_1 - f$ in $\mathcal{D}'(\Om_\pm)$ and since $z u_1, f \in C^0([x_0,x_1];L^2(-1,1))$, consequently $z \p_x u_1 - \p_{zz} u_1 = f$ in $L^2(\Omega)$''.
    However, these orthogonality conditions are non-empty, as we show below (see \cref{free:f}).
\end{rmk}

\begin{defi} \label{def:ell-shear}
    In the sequel, we denote by $\overline{\ell^j}$ the linear forms associated with the orthogonality conditions \eqref{eq:compat-shear} for the linear shear flow problem, i.e., for $(f,\delta_0,\delta_1) \in \cLin_K$, we set
    \nomenclature[OLlj]{$\overline{\ell^j}$}{Linear forms on $\cLin_K$ giving the orthogonality conditions for the shear flow}
    \begin{equation*} 
    \overline{\ell^j}(f,\delta_0,\delta_1)  := 
    \p_z^j \delta_0(0)-\p_z^j\delta_1(0)
    + \int_\Om \p_x f \overline{\Phi^j} 
    + \int_{\Sigma_0} z \Delta_0\overline{\Phi^j}  
    - \int_{\Sigma_1} z \Delta_1 \overline{\Phi^j}.
    \end{equation*}
\end{defi}

\begin{lem} \label{lem:continuity-ell-easy}
    The linear forms $\overline{\ell^j}$ for $j \in \{0,1\}$ are continuous over $\cLin_K$.
\end{lem}

\begin{proof}
    First, by \cref{lem:delta_i=H2}, for $(f,\delta_0,\delta_1) \in \cLin_K$, $\delta_i \in H^2(\Sigma_i)$ so that $\delta_i(0)$ and $\delta_i'(0)$ depend continuously on $(f,\delta_0,\delta_1) \in \cLin_K$.
    Second, by \cref{lem:def-dual-shear}, $\overline{\Phi^j} \in Z^0(\Om_\pm)$ so, in particular $\overline{\Phi^j} \in L^2(\Om)$.
    Hence $f \mapsto \int_{\Om} \p_x f \overline{\Phi^j}$ is continuous on $H^1_x L^2_z$.
    Eventually, by \cref{lem:Z0-trace-H1z}, $\overline{\Phi^j}(x_i,\cdot) \in \mathscr{H}^1_z(\Sigma_i) \hookrightarrow \mathscr{L}^2_z(\Sigma_i)$, so $(f,\delta_0,\delta_1) \mapsto \int_{\Sigma_i} z \Delta_i(z) \overline{\Phi^j}(x_i,z) \dd z$ is continuous on $\cLin_K$. 
\end{proof}

\begin{rmk}
    Although this continuity result will be sufficient for most of our purpose, the linear forms $\overline{\ell^j}$ are in fact continuous for weaker topologies than the one of $\cLin_K$.
    In particular, one does not need $f \in H^1_x L^2_z$ (see \cref{rmk:16-shear}).
\end{rmk}

We now prove that the orthogonality conditions \eqref{eq:compat-shear} are non-empty and independent.

\begin{prop}[Independence of the orthogonality conditions]
    \label{free:f}
    The linear forms $\overline{\ell^0}$ and $\overline{\ell^1}$ are linearly independent over $C^\infty_c(\Om) \times \{ 0 \} \times \{ 0 \} \subset \cLin_K$.
\end{prop}

\begin{proof}
    Proceeding by contradiction, let $(c_0,c_1) \in \R^2$ such that, for every $f \in C^\infty_c(\Om)$, there holds $c_0 \overline{\ell^0}(f,0,0) + c_1 \overline{\ell^1}(f,0,0) = 0$.
    Then $\Phi^c := c_0 \overline{\Phi^0}  + c_1 \overline{\Phi^1} $ satisfies $\int_\Om \p_x f \Phi^c  = 0$ for every $f \in C^\infty_c(\Om)$, so $\p_x \Phi^c = 0$ in $\mathcal{D}'(\Om_+)$.
    Since $\Phi^c(x_1,z) = 0$ for $z \in (0,1)$ and $\Phi^c \in Z^0(\Om_+)$, this implies that $\Phi^c = 0$ in $\Om_+$ (since $Z^0$ functions have traces in the usual sense, see \cref{lem:Z0-trace-H1z}).
    The same holds in $\Om_-$.
    Hence $[\Phi^c]_{\rvert z = 0} = [\p_z \Phi^c]_{\rvert z = 0} = 0$, which implies $c_0 = c_1 = 0$.
\end{proof}

\begin{rmk}
    \cref{free:f} of course implies that $\overline{\ell^0}$ and $\overline{\ell^1}$ are linearly independent on $\cLin_K$.
    Although \cref{free:f} gives a prominent role to the source term $f$, we will actually also prove that $\overline{\ell^0}$ and $\overline{\ell^1}$ are linearly independent on $\{0\} \times C^\infty_c(\Sigma_0) \times C^\infty_c(\Sigma_1) \subset \cH$.
    This property relies on the structure of the dual profiles $\overline{\Phi^j}$ near the points $(x_i,0)$, and will be proved at the end of this section (see \cref{cor:free_ellj-sansf}).
\end{rmk}

Eventually, gathering all of the above results, definitions and notations, we have proved the following well-posedness result for the linear problem.

\begin{prop} \label{thm:shear-Z1}
Let $\cH_K$ be the vector space defined in \eqref{eq:def-HK}.
    There exists a vector subspace $\cH^\perp_{K,\operatorname{sg}} = (\ker \overline{\ell^0}) \cap (\ker \overline{\ell^1})$ of codimension 2 in $\cH_K$ such that, for each $(f,\delta_0, \delta_1) \in \cH_K$, there exists a solution $u \in Z^1(\Omega)$ to~\eqref{eq:shear-z} if and only if $(f,\delta_0,\delta_1) \in \cH^\perp_{K,\operatorname{sg}}$.
    Such a solution is unique and satisfies estimate \eqref{eq:u-Z1}.
\end{prop}

\cref{thm:shear-Q1} of the introduction is a rough restatement of the above \cref{thm:shear-Z1} allowing us to avoid introducing more notations and functional spaces at this early stage.

\begin{proof}[Proof of \cref{thm:shear-Q1}]
    One easily checks from \eqref{eq:def-XB-intro} and \eqref{eq:def-HK} that $\cX_B \hookrightarrow \cLin_K$.
    Moreover, setting $\hBPerp := \cX_B \cap (\ker \overline{\ell^0}) \cap (\ker \overline{\ell^1})$, one obtains that $\hBPerp$ is of codimension 2 in $\cX_B$ from \cref{free:f} or \cref{cor:free_ellj-sansf}.
    Hence, given $(f,\delta_0,\delta_1) \in \hBPerp$, \cref{thm:shear-Z1} gives a solution $u \in Z^1$.
    By \cref{lem:Zsigma-Qsigma}, $Z^1\hookrightarrow \qone$ (defined in \eqref{eq:def-Q1-intro}), so $u \in \qone$.
    Reciprocally, given a solution $u \in \qone$ corresponding to some data triple $(f,\delta_0,\delta_1) \in \cX_B$, one has $u \in H^1_x H^2_z$ (see \eqref{eq:def-Q1-intro}), so \cref{p:shear-WP-Z1} applies and one has $(f,\delta_0,\delta_1) \in \hBPerp$.
\end{proof}

Similarly, going further, it can be easily checked that the control of $k$ derivatives in $x$ requires the cancellation of $2k$ independent conditions.
Although controlling a single $x$-derivative will be sufficient in the sequel to obtain our nonlinear result, we establish here this short higher-regularity statement as an illustration.
More precisely, we have the following result.

\begin{lem} \label{lem:H2k-regularity}
    Let $k\geq 1$.
    Let $f \in C^\infty(\overline{\Om})$, $\delta_i \in C^\infty(\overline{\Sigma_i})$.
    Define recursively $\Delta_i^n$ for $0\leq n\leq k$ and $z \in \Sigma_i$ by
    \begin{align}
        \Delta_i^0(z) & := \delta_i(z), \\
        \Delta_i^n(z) & := \frac{1}{z}\left(\p_x^{n-1} f(x_i,z) + \p_{zz}\Delta_i^{n-1}(z)\right).
    \end{align}
    Assume that the following compatibility conditions are satisfied:
    \begin{equation*}
        \forall n\in\{0,\cdots, k\},\quad \Delta_0^n(1)=\Delta_1^n(-1)=0.
    \end{equation*}
    Assume furthermore that for all $n\in\{0,\cdots, k\} $, $\Delta_i^n\in \mathscr{H}^1_z(\Sigma_i)$.
    
    Let $u$ be the unique solution to \eqref{eq:shear}. 
    Then $u\in H^k_xH^2_z$ if and only if the following orthogonality conditions are satisfied
    \begin{equation*}
    \overline{\ell^j}(\p_x^n f, \Delta_0^n,\Delta_1^n)=0, \quad \forall n\in \{0,\cdots, k-1\}, 
    \quad j \in \{0,1\}. 
    \end{equation*} 
    Furthermore, these $2k$ orthogonality conditions are linearly independent.
\end{lem}

\begin{proof}
    First, notice that $\p_x^n u$ satisfies formally
    \begin{equation*}
    \begin{cases}
        (z\p_x-\p_{zz})\p_x^n u=\p_x^n f\quad \text{in }\Om,\\
        \p_x^n u_{|z=\pm 1}=0,\\
        \p_x^n u_{|\Sigma_i}=\Delta_i^n.
    \end{cases}
    \end{equation*}
    The first part of the statement follows easily from \cref{p:shear-WP-Z1} and \cref{p:shear-apriori-Z1} and from an induction argument.
    
    Let us now check the independence of the orthogonality conditions.
    We extend the methodology used in the proof of \cref{free:f}.
    Assume that there exist $c_n^j\in \R$, $0\leq n\leq k-1$, $j=0,1$ such that for all $(f,\delta_0,\delta_1)$ satisfying the assumptions of the lemma,
    \begin{equation*}
    \sum_{j=0,1} \sum_{n=0}^{k-1} c_n^j  \overline{\ell^j}(\p_x^n f, \Delta_0^n,\Delta_1^n)=0.
    \end{equation*}
    In particular, for any $f\in C^\infty_c(\Om)$,
    \begin{equation*}
    \sum_{j=0,1}\sum_{n=0}^{k-1} c_n^j  \overline{\ell^j}(\p_x^n f, 0,0)=0,
    \end{equation*}
    i.e.
    \begin{equation*}
    \sum_{n=0}^{k-1} \int_\Om \p_x^n f \left( \sum_{j=0,1} c_n^j \overline{\Phi^j}\right)=0.
    \end{equation*}
    This means that
    \begin{equation*}
    \sum_{j=0,1}\sum_{n=0}^{k-1} (-1)^n c_n^j \p_x^n \overline{\Phi^j}=0
    \end{equation*}
    in the sense of distributions. Since $[\p_x^n \overline{\Phi^j}]_{|z=0}=[\p_x^n\p_z \overline{\Phi^j}]_{|z=0}=0$ for $n\geq 1$, we infer that 
    \begin{equation*}
    \left[ c_0^0 \overline{\Phi^0} + c_0^1 \overline{\Phi^1}\right]_{|z=0}= \left[ \p_z(c_0^0 \overline{\Phi^0} + c_0^1 \overline{\Phi^1})\right]_{|z=0}=0. 
    \end{equation*}
    Once again, using the jump conditions on $\overline{\Phi^j}$, we deduce that $c_0^j=0$, and thus
    \begin{equation*}
    \p_x\left( \sum_{j=0,1}\sum_{n=1}^{k-1} (-1)^n c_n^j \p_x^{n-1} \overline{\Phi^j}\right)=0.
    \end{equation*}
    It follows that
    \begin{equation*}
    \sum_{j=0,1}\sum_{n=1}^{k-1} (-1)^n c_n^j \p_x^{n-1} \overline{\Phi^j}=p(z)
    \end{equation*}
    for some function $p$. Note that by parabolic regularity, the profiles $\overline{\Phi^j}$ (and therefore the function $p$) are smooth away from the line $\{z=0\}$. Taking the trace of the above identity on $\{x_0\}\times (-1,0)\cup \{x_1\}\times (0,1)$, we find that $p=0$. Arguing by induction, we infer eventually that $c_n^j=0$ for all $0\leq n\leq k-1$, $j=0,1$.
\end{proof}

\begin{coro}[Biorthogonal basis] \label{lem:biorth}
    There exist $\Xi^k = (f^k, \delta_0^k, \delta_1^k) \in \cLin_K$ for $k \in \{ 0, 1 \}$ such that, for every $j,k \in \{ 0,1 \}$,
    \begin{equation*}
        \overline{\ell^j}(\Xi^k) = \overline{\ell^j}(f^k, \delta_0^k, \delta_1^k) = \mathbf{1}_{j = k}
    \end{equation*}
    and such that, within $\cLin_K$,
    \begin{equation} \label{eq:hperp-ker}
        \left(\R \Xi^0 + \R \Xi^1\right)^\perp 
        = \ker \overline{\ell^0} \cap \ker \overline{\ell^1}
    \end{equation}
    is a vector subspace of codimension 2.
\end{coro}

\begin{proof}
    Since $\overline{\ell^0}$ and $\overline{\ell^1}$ are continuous linear forms on $\cLin_K$, by the Riesz representation theorem, they can be written as scalar products with two given triplets, say $\overline{\Xi^0}, \overline{\Xi^1} \in \cLin_K$ which are linearly independent thanks to \cref{free:f}.
    Then one looks for $\Xi^k = (f^k,\delta_0^k,\delta_1^k)$ as $a_k \overline{\Xi^0} + b_k \overline{\Xi^1}$ where $a_k, b_k \in \R^2$ are such that $a_k \langle \overline{\Xi^j} ; \overline{\Xi^0} \rangle + b_k \langle \overline{\Xi^j} ; \overline{\Xi^1} \rangle = \mathbf{1}_{j=k}$.
    These systems can be solved since $\overline{\Xi^0}$ and $\overline{\Xi^1}$ are free.
    This proves the equality \eqref{eq:hperp-ker}. 
    The independence of the linear forms guarantees that \eqref{eq:hperp-ker} is of codimension 2 in $\cLin_K$.
\end{proof}
	
\subsection{Singular radial solutions in the half-plane and profile decomposition}
\label{sec:radial}

In this subsection, we give a full description of the singularities that appear when the orthogonality conditions are not satisfied.
We start by constructing singular solutions to the homogeneous equation set in the half-plane, using separation of variables in polar-like coordinates.
We then localize these solutions near the critical points $(x_i,0)$ to obtain the decomposition result of \cref{thm:shear-decomp}.

Our approach is similar to the one developed by Grisvard in \cite[Section 4.4]{MR775683} for elliptic problems in polygonal domains (see in particular the singular profiles of equation (4.4.3.7) and the decomposition result of Theorem 4.4.3.7 therein). 
The main difference is that we cannot use usual polar coordinates and that the construction of the elementary singular profiles is much more technical than, for instance, the classical solution of the form $r^{\frac 12} \sin (\theta/2)$ which is involved in the resolution of Dirichlet-Neumann junctions as in the elliptic problem \eqref{eq:elliptic-u} mentioned in the introduction.

\subsubsection{Construction of singular solutions in the half-plane}

In this paragraph, we look for elementary singular radial solutions to the following problem without source-term in the half-plane:
\begin{equation} \label{eq:half-plane}
    \begin{cases}
        z \p_x u - \p_{zz} u = 0 & x \geq 0, z \in \R, \\
        u(0,z) = 0 & z > 0.
    \end{cases}
\end{equation}

\begin{rmk}
    In \cite{fleming1963problem}, Fleming considered the related problem of finding a ``fairly explicit formula'' for solutions to $z \p_x u - \p_{zz} u = 0$ in a strip $(0,1) \times \R$, with prescribed boundary data at $x=0$, $z>0$ and $x=1$, $z<0$.
    His proof involves Whittaker functions, which are related to the confluent hypergeometric functions we use below.
    
    In \cite{gor1975formula}, Gor'kov computes a representation formula for solutions to \eqref{eq:half-plane} with a non-zero source term and boundary data, and proves uniqueness of such solutions, under a growth assumption of the form $|u(0,z)| \lesssim |z|^\sigma$ for $0 \leq \sigma < \frac 12$ on the line $\{x = 0\}$, for which he claims that uniqueness holds.
    The threshold $\sigma = \frac12$ is precisely the scaling (at which uniqueness indeed breaks) of the first fundamental singular solution $v_0$ which we construct below.
    
    Our setting is a little different from the works mentioned above, as we look for (non-zero) solutions to the homogeneous equation.
     Similar computations were also performed in \cite{hwang2014fokker,hwang2019structure}, albeit with different boundary conditions, and therefore with a different exponent for $r$, and a different asymptotic behavior for the profile $\Lambda$ in \cref{prop:rGk}.
    However we were not able to find the specific expression of the profiles from \cref{prop:rGk} in previous works. 
\end{rmk}

Near the point $(0,0)$ which is expected to be singular, balancing the terms $z \p_x$ and $\p_{zz}$ leads to the natural scaling $z \sim x^{\frac13}$.
Thus, we introduce the following polar-like coordinates $(r,t) \in [0,+\infty) \times \R$:
\nomenclature[OLR]{$r$}{Radial-like variable given by $r = (z^2+x^{\frac23})^{\frac12}$}
\nomenclature[OLT]{$t$}{Angular-like variable given by $t = z x^{-\frac13}$}
\begin{equation} \label{eq:r-t}
    r := (z^2 + x^{\frac 23})^{\frac 12}
    \quad \text{and} \quad
    t := z x^{-\frac 13}
\end{equation}
The reverse change of coordinates is given by
\begin{equation} \label{eq:x-z}
x=\frac{r^3}{(1+t^2)^{\frac 32}}
\quad \text{and} \quad
z=\frac{rt}{(1+t^2)^{\frac 12}}.
\end{equation}
Since it will be convenient to switch from cartesian coordinates $(x,z)$ to the polar-like coordinates $(r,t)$, we compute the Jacobian
\begin{equation}\label{jacobian}
J(r,t)=
\begin{pmatrix}
    \displaystyle\frac{\p r}{\p x} & \displaystyle\frac{\p r}{\p z}
    \\~\\
    \displaystyle\frac{\p t}{\p x} & 
    \displaystyle\frac{\p t}{\p z}
\end{pmatrix}
=
\begin{pmatrix}
    \displaystyle\frac{1}{3x^{\frac 13} r} & \displaystyle\frac{z}{r} \\~\\
    -\displaystyle\frac{t}{3x} & 
    \displaystyle\frac{1}{x^{\frac 13}}
\end{pmatrix}
= 
\begin{pmatrix}
    \displaystyle\frac{(1+t^2)^{\frac 12}}{3r^2} & \displaystyle\frac{t}{(1+t^2)^{\frac 12}} \\~\\
    -\displaystyle\frac{t(1+t^2)^{\frac 32}}{3r^3} & \displaystyle\frac{(1+t^2)^{\frac 12}}{r}
\end{pmatrix}
\end{equation}
where we have used the equalities \eqref{eq:x-z}.
In particular,
\begin{equation}
    \label{eq:det-J}
    \det J(r,t) = \frac{(1+t^2)^2}{3r^3},
\end{equation}
which we will use to compute integrals using the $(r,t)$ variables.

By \eqref{jacobian}, for any $C^1$ function $\varphi$,
\begin{align}
    \label{eq:px-pr-pt}
    \p_x \varphi & = \frac{(1+t^2)^{\frac12}}{3r^2}\p_r\varphi -\frac{t(1+t^2)^{\frac32}}{3r^3}\p_t \varphi,\\
    \label{eq:pz-pr-pt}
    \p_z \varphi & = \frac{t}{(1+t^2)^{\frac12}} \p_r \varphi + \frac{(1+t^2)^{\frac12}}{r}\p_t \varphi.
\end{align}
In particular, if $u(r,t) = r^\lambda \gl(t)$,
\begin{align}
    \label{eq:zdx}
    z\p_x u & = \frac{r^{\lambda-2}}{3}\left[ \lambda t \gl(t) -t^2(1+t^2) \p_t\gl(t)\right],\\
    \label{eq:dzz}
    \p_{zz} u &= \left[\frac{t}{(1+t^2)^{\frac12}} \p_r  + \displaystyle\frac{(1+t^2)^{\frac12}}{r}\p_t \right] \left( r^{\lambda-1} \left(\frac{\lambda t}{(1+t^2)^{\frac12}} \gl(t) + (1+t^2)^{\frac12} \p_t\gl(t) \right)\right)\\
    \nonumber
    &= r^{\lambda-2}\left[ (\lambda-1) \left(\frac{\lambda t^2}{1+t^2} \gl(t) + t \p_t\gl(t)\right) + (1+t^2)^{\frac12} \p_t \left(\frac{\lambda t}{(1+t^2)^{\frac12}} \gl(t) + (1+t^2)^{\frac12} \p_t\gl(t)\right)\right].
\end{align}
We are now ready to construct solutions to \eqref{eq:half-plane} using these coordinates.

\begin{figure}
    \centering
    \includegraphics[width=14cm]{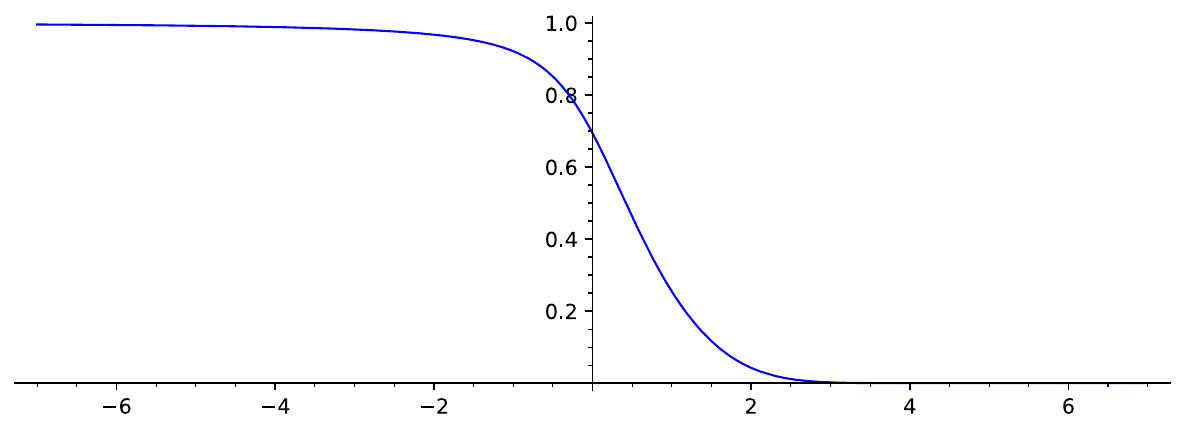}
    \caption{Plot of $t \mapsto \gnot(t)$ for $t \in (-7,7)$, highlighting the main properties: $\gnot$ is a smooth, monotone decreasing function on $\R$, such that $\gnot(-\infty)=1$ and $\gnot(+\infty) = 0$}
    \label{fig:G0}
\end{figure}

\begin{prop} 
    \label{prop:rGk}
    For every $k \in \Z$, equation \eqref{eq:half-plane} has a solution of the form 
    \nomenclature[OLV]{$v_k$}{$k$-th explicit singular solution in the half plane, $v_k = r^{\frac12+3k}\gk{k}(t)$}
    \begin{equation*}
        v_k := r^{\frac 12 + 3k} \gk{k}(t)
    \end{equation*}
    with the variables $(r,t)$ of \eqref{eq:r-t} and $\gk{k} \in C^\infty(\R;\R)$ is a smooth bounded function satisfying $\gk{k}(-\infty) = 1$ and $\gk{k}(+\infty) = 0$.
    \nomenclature[OAL]{$\gk{k}$}{Angular profile of the $k$-th explicit singular solution in the half-plane}
    The profile $\gnot$ is presented in \cref{fig:G0}.
\end{prop}

\begin{proof}
    By separation of variables, we look for a solution to \eqref{eq:half-plane} under the form $u := r^\lambda \gl(t)$ where $\lambda \in \R$ and $\gl : \R \to \R$ is a smooth function.
    The boundary condition $u(0,z) = 0$ for $z > 0$ translates to $\gl(+\infty) = 0$.
    From \eqref{eq:zdx} and \eqref{eq:dzz} above, one checks that such a $u$ satisfies $z \p_x u - \p_{zz} u = 0$ if and only if
    \begin{equation} \label{eq:EDO-G}
        \p_t^2 \gl(t) + \left( \frac{t^2}{3} + \frac{2 \lambda t}{1+t^2} \right) \p_t \gl(t) + \lambda \left( - \frac{1}{3} \frac{t}{1+t^2} + \frac{1+(\lambda-1) t^2}{(1+t^2)^2} \right) \gl(t) = 0.
    \end{equation}
    To absorb the $(1+t^2)$ factors, we perform the change of unknown $\gl(t) =: (1+t^2)^{-\frac \lambda 2} H(t)$.
    Then, $\gl$ satisfies \eqref{eq:EDO-G} if and only if $H$ is a solution to
    \begin{equation} \label{eq:EDO-H}
        \p_t^2 H(t) + \frac{t^2}{3} \p_t H(t) - \frac{\lambda t}{3} H(t) = 0.
    \end{equation}
    Moreover, for $t \neq 0$, using the change of variable $\zeta := - t^3 / 9$, and looking for $H(t) =: W(-t^3/9)$, we obtain that $H$ solves \eqref{eq:EDO-H} on $\R \setminus \{ 0 \}$ if and only if $W$ is a solution to
    \begin{equation} \label{eq:EDO-W}
        \zeta \p_\zeta^2 W(\zeta) + \left(\frac 23 - \zeta\right) \p_\zeta W(\zeta) - \left(-\frac \lambda 3\right) W(\zeta) = 0
    \end{equation}
    which corresponds to Kummer's equation, with $a = - \frac \lambda 3$ and $b = \frac 2 3$.
    It is known (see \cite[Section 13.2]{NIST}) that \eqref{eq:EDO-W} has a unique solution behaving like $\zeta^{-a}$ as $\zeta \to \infty$.
    This (complex valued) solution is usually denoted by $U(a,b,\zeta)$ and called \emph{confluent hypergeometric function of the second kind}, or \emph{Tricomi's function}.
    In general, $U$ has a branch point at $\zeta = 0$.
    More precisely, the asymptotic $\zeta^{-a}$ holds in the region $|\arg \zeta| < \frac {3\pi} 2$ and the principal branch of $U(a,b,\zeta)$ corresponds to the principal value of $\zeta^{-a}$.
    Moreover, when $b$ is not an integer, which is our case, one has (see \cite[Equation 13.2.42]{NIST}),
    \begin{equation} \label{eq:U-M}
        U(a,b,\zeta) = \frac{\Gamma(1-b)}{\Gamma(a-b+1)} M(a,b,\zeta) 
        + \frac{\Gamma(b-1)}{\Gamma(a)} \zeta^{1-b} M(a-b+1,2-b,\zeta),
    \end{equation}
    where $M$ is the \emph{confluent hypergeometric function of the first kind} or \emph{Kummer's function},
    \begin{equation*}
        M(a,b,\zeta) := \sum_{n \in \N} \frac{(a)_n}{(b)_n} \frac{\zeta^n}{n!},
    \end{equation*}
    where $(a)_n$ and $(b)_n$ denote the rising factorial.
    In particular, $M$ is an entire function of $\zeta$.
    From \eqref{eq:U-M}, we see that the singularity in Tricomi's function $U$ stems from the fractional power $\zeta^{1-b} = \zeta^{\frac 13}$.
    When $\zeta = - \rho$ (for $\rho > 0$), $\zeta^{\frac 1 3} = e^{\frac{i\pi}{3}} \rho^{\frac 13}$.
    
    We therefore set
    \begin{equation}
        \label{eq:def-W-zeta}
        W(\zeta) := \Re \left\{ e^{\frac{i\pi}{3}}  U\left(-\frac\lambda 3,\frac 2 3, \zeta\right) \right\} .
    \end{equation}
    By linearity, $W$ is still a solution to \eqref{eq:EDO-W}.
    Moreover, by \cite[Equation 13.7.3]{NIST}, as $\zeta \to \infty$,
    \begin{equation}
        \label{eq:W-zeta-asympt}
        W(\zeta) = \Re \left\{ e^{\frac{i\pi}{3}} \zeta^{-a} \left( 1 + O\left( \frac{1}{|\zeta|}\right) \right) \right\}.
    \end{equation}
    In particular, when $\lambda = \frac{1}{2} + 3k$ for $k \in \Z$ (and only in this situation), as $\rho \to +\infty$,
    \begin{equation*}
        W(-\rho) = O(\rho^{-a-1}),
    \end{equation*}
    because $\Re \{ e^{i\pi/3} e^{-ia\pi} \rho^{-a} \} = \Re \{ (-1)^k e^{i\pi/3} e^{i\pi/6} \rho^{-a}\} = (-1)^k \rho^{-a} \Re \{i\} = 0$.
    Defining $H(t) := W(-t^3/9)$ for $W$ as in \eqref{eq:def-W-zeta} and recalling that $\gl(t) = (1+t^2)^{-\lambda/2} H(t)$ implies that $\gl(+\infty) = 0$.
    Indeed, as $t \to +\infty$,
    \begin{equation} \label{eq:G+inf}
        \gl(t) = (1+t^2)^{-\frac \lambda 2} O \left( t^{3(\frac\lambda 3-1)} \right) = O(t^{-3}).
    \end{equation}
    Moreover, from \eqref{eq:W-zeta-asympt}, we obtain that $\gl$ is bounded as $t \to -\infty$.
    Indeed, as $t \to -\infty$,
    \begin{equation} \label{eq:G-inf}
        \begin{split}
            \gl(t) & = (1+t^2)^{-\frac \lambda 2} 
            \Re \left\{ e^{\frac{i\pi}{3}} \left(-\frac{t^3}{9}\right)^{-a} \left( 1 + O\left( \frac{1}{|t|^3}\right) \right) \right\}
            \\ & = \frac{1}{2} 9^{-\frac 16 - k} (1+t^2)^{-\frac \lambda 2} |t|^{-3a} \left(1 + O(|t|^{-3})\right)
            = \frac{1}{2} 9^{-\frac 16 - k} + O(|t|^{-2}).
        \end{split}
    \end{equation}
    Eventually, let us check that $H$ is an entire function of $t$, which will entail that $\gl$ is smooth.
First, note that $M(-\lambda/3, 2/3, -t^3/9)$ and $M(-(\lambda-1)/3, 4/3, -t^3/9)$ are real valued and entire functions of $t$. 
Additionally,
\begin{equation*}\begin{aligned}
\Re\left\{ e^{i\pi/3} \left(-\frac{t^3}{9}\right)^{1/3}\right\}&=
 9^{-1/3}\times \begin{cases}
|t| /2 & \text{ if } t<0,\\
t  \Re(e^{2i\pi/3}) & \text{ if } t<0
\end{cases}\\
&= -\frac{1}{2}\frac{t}{ 9^{1/3}}.
\end{aligned}
\end{equation*}
   Using \eqref{eq:def-W-zeta} and \eqref{eq:U-M}, we obtain
    \begin{equation} \label{eq:H-MM}
        H(t) = \frac{1}{2} \frac{\Gamma(1-b)}{\Gamma(a-b+1)} M\left(a,b,-\frac{t^3}{9}\right) 
        - \frac{1}{2} \frac{t}{9^{1/3}} \frac{\Gamma(b-1)}{\Gamma(a)}  M\left(a-b+1,2-b,-\frac{t^3}{9}\right),
    \end{equation}
    so that $H$ is entire because $M$ is.
    This also entails that $H$ solves \eqref{eq:EDO-H} even across $t = 0$.
    Moreover, \eqref{eq:G+inf} and \eqref{eq:G-inf} imply that $\gl$ is bounded on $\R$.
    Eventually, using \eqref{eq:G-inf}, we can define $\gk{k}$ as $2 \cdot 9^{\frac 16 + k} \gl$, which ensures that $\gk{k}(-\infty) = 1$.
    For this normalization, one deduces from \eqref{eq:H-MM} that
    \begin{equation} \label{eq:Gk-0}
        \gk{k}(0) = 9^{\frac 16 + k} \frac{\Gamma(1/3)}{\Gamma(1/6-k)},
    \end{equation}
    which will be used below.
\end{proof}

If $u$ is a solution to \eqref{eq:half-plane}, then, formally, $\p_x u$ too (the operator $z\p_x-\p_{zz}$ commutes with $\p_x$, and the boundary condition at $x = 0$ and $z > 0$ is satisfied thanks to the equation).
This property entails that the solutions $v_k = r^{\frac12+3k} \gk{k}(t)$ are related by a recurrence relation on the profiles $\gk{k}$.

\begin{lem}[Recurrence relations]
    \label{lem:rec}
    Let $k \in \Z$ and $c_k := \frac14 - 9k^2$.
    One has
    \begin{equation} \label{eq:pxvk}
        \p_x v_k = c_k v_{k-1}.
    \end{equation}
    Moreover, for every $t \in \R$,
    \begin{equation}
        \label{eq:G-rec}
        c_k \gk{k-1}(t) = \frac{(1+t^2)^{\frac12}}{3} \left(\left(\frac 12 + 3k\right) \gk{k}(t) - t(1+t^2) \gk{k}'(t) \right),
    \end{equation}
    or, equivalently,
    \begin{equation} \label{eq:gk'}
        \gk{k}'(t) = \frac{1}{t(1+t^2)} \left(  \left(\frac 12 + 3k\right) \gk{k}(t)- \frac{3 c_k \gk{k-1}(t)}{(1+t^2)^{\frac 12}} \right).
    \end{equation}
\end{lem}

\begin{proof}
    By \eqref{eq:px-pr-pt}, one has $\p_x v_k = r^{\frac12+3(k-1)} H_k(t)$, where $H_k(t)$ is the right-hand side of \eqref{eq:G-rec}.
    Thus $\p_x v_k$ is a solution to \eqref{eq:half-plane} of the form studied in \cref{prop:rGk}.
    Since the proof of \cref{prop:rGk} proceeds by equivalence, $v_{k-1}$ is the only solution of the form $r^{\frac12+3(k-1)}$. 
    This entails that $H_k(t)$ is proportional to $\gk{k-1}(t)$ and the constant can be identified by comparing the values at $0$ using \eqref{eq:Gk-0}, yielding \eqref{eq:G-rec}, \eqref{eq:pxvk} and \eqref{eq:gk'} (which are all equivalent) with $c_k = \frac14-9k^2$.
    
    Actually, these identities are linked with recurrence relations on Tricomi's function $U$.
    Let us give another proof of \eqref{eq:gk'} using this approach.
    By the proof of \cref{prop:rGk},
    \begin{equation*}
        \gk{k}(t) = 2 \cdot 9^{\frac16+k} (1+t^2)^{-\frac14-\frac32 k} \cdot \Re \left\{
        e^{\frac{i\pi}{3}} U\left(-\frac16-k,\frac23,-\frac{t^3}{9}\right) \right\}.
    \end{equation*}
    First, using the relation $\p_\zeta U(a-1,b,\zeta) = (1-a) U(a,b+1,\zeta)$ (see \cite[Equation (13.3.22)]{NIST}),
    \begin{equation*}
        \begin{split}
            \gk{k}'(t) & = -\left(\frac12+3k\right) \frac{t}{1+t^2} \gk{k}(t) \\
            & + 
            2 \cdot 9^{\frac16+k} (1+t^2)^{-\frac14-\frac32 k} \cdot \frac{3}{t} \left(k+\frac16\right) \Re \left\{
            e^{\frac{i\pi}{3}} \left(-\frac{t^3}{9}\right) U\left(-\frac16-k+1,\frac53,-\frac{t^3}{9}\right) \right\}.
        \end{split}
    \end{equation*}
    Eventually, \eqref{eq:gk'} follows from the relation $(b-a)U(a,b,\zeta)+U(a-1,b,\zeta)-\zeta U(a,b+1,\zeta) = 0$ (see \cite[Equation (13.3.10)]{NIST}).
\end{proof}

\begin{rmk}
    We will see below that $v_0$ is linked with two solutions to \eqref{eq:shear-z} which have $Z^0$ regularity, but do not belong to $H^1_x H^1_z$.
    Similarly, for each $k \geq 0$, $v_k$ is linked with solutions $u$ such that $\p^k_x u \in Z^0(\Omega)$ but $u \notin H^{k+1}_x H^1_z$.
    Conversely, for $k = -1$, one could expect to be able to construct a very weak solution $u$ based on $v_{-1}$ which would entail that uniqueness fails for solutions with less than $L^2_xH^1_z$ regularity.
\end{rmk}

\cref{lem:rec} entails the following decay estimates, which will be useful in the sequel.

\begin{lem} \label{lem:G-decay}
    For every $k \in \Z$, there exists $C_k > 0$ such that, for every $t \in \R$,
    \begin{equation*}
        |\gk{k}(t)| + |t^3 \gk{k}'(t)| + |t^4 \gk{k}''(t)| + |t^5 \gk{k}'''(t)|\leq C_k.
    \end{equation*}
\end{lem}

\begin{proof}
    For all $k \in \Z$, the bound $|\gk{k}(t)| \leq C_k$ is already contained in \cref{prop:rGk} which claims that $\gk{k}$ is bounded.
    Since $\gk{k-1}$ and $\gk{k}$ are uniformly bounded over $\R$, we deduce from  \eqref{eq:gk'} that $t^3 \gk{k}'(t)$ is also bounded on $\R$.
    Eventually, differentiating \eqref{eq:gk'} with respect to $t$ leads to a uniform bound for $|t^4 \gk{k}''(t)|$ and $|t^5 \gk{k}'''(t)|$ over $\R$.
\end{proof}

Moreover, the recurrence relations of \cref{lem:rec} also imply that the solutions $v_k$ to \eqref{eq:half-plane} are smooth, up to the boundary $\{ x = 0 \}$, except at the origin $(0,0)$.

\begin{lem} \label{lem:vk-smooth}
    For every $k \in \Z$, $v_k \in C^\infty(P_*)$, where $P_* := ([0,+\infty) \times \R) \setminus \{ (0,0) \}$.
\end{lem}

\begin{proof}
    The smoothness inside the half-plane $\{ x > 0 \}$ follows directly from \cref{prop:rGk} since $\gk{k} \in C^\infty(\R)$ and the function $r \mapsto r^{\frac12+3k}$ as well as the change of coordinates of \eqref{eq:r-t} are smooth inside this domain.
    
    By \cref{prop:rGk}, since $\gk{k}$ is continuous on $\R$ and has limits at $t = \pm \infty$, we obtain that $v_k = r^{\frac12+3k} \gk{k}(t)$ is continuous up to the boundary $\{ x = 0 \}$, except at the origin: $v_k \in C^0(P_*)$.
    
    We now turn to the continuity of derivatives. Using \eqref{eq:pz-pr-pt}, 
    \begin{equation*}
        \p_z v_k = r^{-\frac 12+3k} \left[ \left(\frac12+3k\right) \frac{t}{(1+t^2)^{\frac12}} \gk{k}(t) + (1+t^2)^{\frac12} \gk{k}'(t)\right].
    \end{equation*}
    Since $\gk{k}$ has limits at $t = \pm \infty$ and since, by \cref{lem:G-decay}, $t^3 \gk{k}'(t) = O(1)$, we obtain that $\p_z v_k$ has limits at $t = \pm \infty$.
    Hence $\p_z v_k \in C^0(P_*)$.
    
    Eventually, the $C^\infty(P_*)$ regularity follows from an induction argument.
    Indeed, by \eqref{eq:pxvk}, $\p_x v_k = c_k v_{k-1}$, so $\p_x v_k \in C^0(P_*)$ because $v_{k-1} \in C^0(P_*)$.
    And, similarly, in the vertical direction, using \eqref{eq:half-plane}, $\p_{zz} v_k = z \p_x v_k = z c_k v_{k-1}$ so $\p_{zz} v_k \in C^0(P_*)$.
    Iterating the argument concludes the proof.
\end{proof}

\subsubsection{Localization and decomposition}

We now introduce singular profiles $\busing^i$, for $i=0,1$, localized in the vicinity of $(x_i,0)$ and based on the singular profiles of the previous paragraph.
Let $\chi_i \in C^\infty(\overline{\Om})$ be a cut-off function such that $\chi_i \equiv 1$ in a neighborhood of $(x_i, 0)$, and $\supp \chi \subset B((x_i, 0), \bar{R})$ for some $\bar{R} < \min (1, x_1-x_0)/2$. 
These localized profiles are the ones involved in the main decomposition result of \cref{thm:shear-decomp}. 

\begin{defi}
    \label{def:sing-profiles}
    For $i \in \{0,1\}$, let
    \nomenclature[OLU]{$\busing^i$}{Reference singular solution localized near $(x_i,0)$}
    \nomenclature[OAX]{$\chi_i$}{Cut-off function localized near $(x_i,0)$}
    \begin{equation*}
    \busing^i(x,z) := r_i^{\frac 12} \gnot(t_i) \chi_i(x,z),
    \end{equation*}
    where $\gnot$ is constructed in \cref{prop:rGk} and
    \nomenclature[OLR]{$r_i$}{Radial-like variable near $(x_i,0)$ given by $r_i=(z^2+\lvert x-x_i \rvert^{\frac23})^{\frac12}$}
    \nomenclature[OLT]{$t_i$}{Angular-like variable near $(x_i,0)$, given by $t_i = (-1)^i z \lvert x-x_i \rvert^{-\frac13}$}
    \begin{equation*}
    r_i := \left(z^2 + |x-x_i|^{\frac 23}\right)^{\frac 12}
    \quad \text{and} \quad
    t_i := (-1)^i z |x-x_i|^{-\frac 13}.
    \end{equation*}
\end{defi}

\begin{lem} \label{lem:busing}
    \nomenclature[OLF]{$\bfi$}{Smooth source term associated with the singular solution $\busing^i$}
    For $i \in \{ 0, 1\}$, there exists $\bfi \in C^\infty(\overline{\Om})$, with $\bfi \equiv 0$ in neighborhoods of $(x_i,0)$ and $\{ z = \pm 1 \}$, such that $\busing^i$ is the unique solution with $Z^0(\Omega)$ regularity to
    \begin{equation}
        \label{eq:busing}
        \begin{cases}
            z \p_x \busing^i - \p_{zz} \busing^i=\bfi,\\
            \busing^i {}_{|\Sigma_0 \cup \Sigma_1}=0,\\
            \busing^i {}_{|z=\pm 1}=0.
        \end{cases}
    \end{equation}
    Moreover, $\busing^i \in C^\infty(\overline{\Omega} \setminus \{ (x_i,0) \})$ but $\busing^i \notin H^1_x H^1_z$.
\end{lem}

\begin{proof}
    By symmetry, we only prove the statement for $\bfzero$ and $\busing^0$.
    In order to alleviate the notation, we drop the index $0$ in $r_0$, $t_0$ and $\chi_0$.
    We introduce positive numbers $0 < r_- < r_+$ such that $\chi \equiv 1$ for $r \leq r_-$ and $\chi \equiv 0$ for $r \geq r_+$. 
    In particular, all derivatives of $\chi$ are smooth, bounded, and supported in $\mathbf{1}_{r_-<r<r_+}$.
    
    Straightforward computations lead to \eqref{eq:busing}, provided that one defines
    \begin{equation} \label{eq:f0}
        \bfzero := r^{\frac 12}\gnot(t)\left(z\p_x \chi - \p_{zz} \chi\right)- 2 \p_z (r^{\frac 12}\gnot(t)) \p_z \chi
        = v_0 \left(z\p_x \chi - \p_{zz} \chi\right) - 2 \p_z v_0 \p_z \chi.
    \end{equation}
    Since the derivatives of $\chi$ are supported away from the point $(x_0,0)$, the $C^\infty(\overline{\Om})$ regularity of $\bfzero$ follows directly from the smoothness of $v_0$ away from the origin proved in \cref{lem:vk-smooth}.
    Since $\busing^0(x,z) = v_0(x,z) \chi(x,z)$, the $C^\infty(\overline{\Om} \setminus \{ (x_i,0) \})$ regularity of $\busing^0$ follows from \cref{lem:vk-smooth}.
    
    Therefore, to prove the lemma, there remains to prove that $\busing^0$, $\p_{zz} \busing^0$ and $z \p_x \busing^0$ are in $L^2(\Omega)$ but $\p_x \p_z \busing^0 \notin L^2(\Omega)$.
    We will use the change of coordinates from cartesian to polar-like ones of Jacobian given by \eqref{eq:det-J}, so that, for $\varphi : \Omega \to \R$, 
    \begin{equation*}
        \| \varphi \|_{L^2(\Om)}^2 = 
        \int_0^\infty \int_{\R} \frac{3 r^3}{(1+t^2)^2} \varphi(r,t)^2 \dd t \dd r. 
    \end{equation*}
    In particular, we have the following integrability criterion. 
    Assume that $\varphi$ is of the form $r^{\mu} H(t) \psi$ where $H(t) = O_{t \to \pm \infty}(|t|)$ and $\supp \psi \subset \mathbf{1}_{r < r_+}$.
    If $\mu > -2$ or $\supp \psi \subset \mathbf{1}_{r_- < r}$, then $\varphi \in L^2(\Omega)$.
    
    \step{Preliminary estimates.}
    Let $\psi$ such that $\supp \psi \subset \mathbf{1}_{r < r_+}$.
    By the previous integrability criterion, since $\gnot(t) = O(1)$, $r^{\frac 12} \gnot(t) \psi \in L^2(\Omega)$.
    Using \eqref{eq:pz-pr-pt},
    \begin{equation*} 
    \p_z \left(r^{\frac 12}\gnot(t)\right)= r^{-\frac 12} \left[ \frac{t}{2(1+t^2)^{\frac12}} \gnot(t) + (1+t^2)^{\frac12} \gnot'(t)\right].
    \end{equation*}
    By \cref{lem:G-decay}, $|t| \gnot'(t) = O(|t|^{-2})$.
    Thus, $\p_z (r^{\frac 12} \gnot(t)) \psi \in L^2(\Omega)$.
    Using \eqref{eq:pz-pr-pt} again,
    \begin{equation*} 
    \begin{split}
        \p_{zz} \left(r^{\frac 12}\gnot(t)\right) = - \frac 12 r^{-\frac 32} \frac{t}{(1+t^2)^{\frac 12}} & \left[ \frac{t}{2(1+t^2)^{\frac12}} \gnot(t) + (1+t^2)^{\frac12} \gnot'(t)\right] \\
        & + r^{-\frac 32} (1+t^2)^{\frac12} \p_t  \left[ \frac{t}{2(1+t^2)^{\frac12}} \gnot(t) + (1+t^2)^{\frac12} \gnot'(t)\right].
    \end{split}
    \end{equation*}
    Using \eqref{eq:px-pr-pt},
    \begin{equation*}
    \p_x \left( r^{\frac 12}\gnot(t)\right)= r^{-\frac 52} \left[ \frac{(1+t^2)^{\frac12}}{6} \gnot(t) - \frac{t(1+t^2)^{\frac 32}}{3} \gnot'(t)\right]. \end{equation*} 
    By \cref{lem:G-decay}, $\gnot'(t) = O(|t|^{-3})$.
    Hence, $|t| \gnot(t) = O(|t|)$ and $|t|^4 \gnot'(t) = O(|t|)$ so, assuming additionally that $\supp \psi \subset \mathbf{1}_{r_- < r < r_+}$, one concludes that $\p_x (r^{\frac 12} \gnot(t)) \psi \in L^2(\Omega)$.
    
    Eventually, using \eqref{eq:pz-pr-pt},
    \begin{equation} \label{eq:dxz-r12Gt} \p_{xz} (r^{\frac12}\gnot(t)) = r^{-\frac72} \left( - \frac t 4 \gnot(t) - \frac 1 6 \gnot'(t) (1+t^2)(1+3t^2) - \frac t 3 (1+t^2)^2 \gnot''(t) \right).
    \end{equation}
    
    
    \step{$Z^0$ estimates on $\busing^0$.}
    By Step 1, $\busing^0$ and $\p_{zz} \busing^0$ belong to $L^2(\Omega)$.
    Since $z\p_x \busing^0 = \bfzero + \p_{zz} \busing^0$ and $\bfzero \in L^2(\Omega)$, we infer that $z \p_x \busing^0 \in L^2(\Omega)$. 
    Hence $\busing^0 \in Z^0(\Om)$.
    
    \step{Lack of $H^1_x H^1_z$ estimate for $\busing^0$.}
    Recalling \eqref{eq:dxz-r12Gt},
    \begin{equation}\label{dxdzusing}
        \p_x \p_z \busing^0 = 
        r^{-\frac72} h(t) \chi +  \p_x (r^{\frac12} \gnot(t)) \p_z \chi + \p_z (r^{\frac12} \gnot(t) \p_x \chi),
    \end{equation}
    where, by \eqref{eq:dxz-r12Gt}, the function $h$ is given by
    \begin{equation*}
    h(t)=- \frac t 4 \gnot(t) - \frac 1 6 \gnot'(t) (1+t^2)(1+3t^2) - \frac t 3 (1+t^2)^2 \gnot''(t) .
    \end{equation*}
Using \eqref{eq:gk'} together with the relation $\gl_k(-\infty)= 9^{-\frac{1}{6}-k}/2$, we find that as $t\to -\infty$,
\begin{equation*}
\begin{aligned} 
\gnot(t)= &a + \frac{b}{t^2} + \frac{c}{t^3} + O(t^{-4}),\\
\gnot'(t)=& - \frac{2b}{t^3} - \frac{3c}{t^4} + O(t^{-5}),\\
\gnot''(t)=&\frac{6b}{t^4} + \frac{12c}{t^5}+ O(t^{-6}),
\end{aligned} 
\end{equation*}
and the coefficients $a,b,c$ are defined by $a= \gnot(-\infty)$, $-2b= a/2$, and $-3c=3 c_0 \gl_{-1}(-\infty)$. We infer that as $t\to -\infty$,
\begin{equation*}
h(t)= \frac{3c}{2} - \frac{12 c}{3} + O(t^{-1})\sim \frac{5}{2}c_0 \gl_{-1}(-\infty)\neq 0.
\end{equation*}
Hence $h\neq 0$.

    The last two terms in the right-hand side of \eqref{dxdzusing} belong to $L^2(\Omega)$ according to the previous computations. Since $h \neq 0$, the $L^2$ norm of the first term is bounded from below by
    \begin{equation*} 
    c \int_0^{r_-} r^{-7} r^3 \dd r = + \infty.
    \end{equation*}
    and thus $\p_x\p_z \busing^0\notin L^2(\Om)$.
\end{proof}

Actually, we have the following regularity on the profiles $\busing^i$, which is slightly better than $Z^0$.

\begin{lem} \label{lem:busing-h56}
    For all $\sigma<\frac12$,  $\busing^i\in H^{\frac{2+\sigma}{3}}_x L^2_z \cap L^2_x H^{2+\sigma}_z \hookrightarrow H^{\frac{2+\sigma}{3}}_x L^2_z \cap H^{\frac{\sigma}{3}}_x H^2_z$ and this is optimal.
    More precisely, $\busing^i \notin H^{\frac 56}_x L^2_z \cap H^{\frac 16}_x H^2_z$.
\end{lem}

\begin{proof}
    The proof follows from an easy scaling argument. 
    We start with the $z$ derivative and focus on $\busing^0$.
    Dropping the index 0 in $r_0$ and $t_0$ as in the previous proof, we have, using \eqref{eq:r-t} and \cref{def:sing-profiles}, and setting $\chi(x,z) := \chi_0(x_0+x,z)$,
    \begin{equation*}
    \busing^0(x_0+x,z)= x^{\frac16} \varphi\left(\frac{z}{x^{\frac13}}\right) \chi(x,z),
    \end{equation*}
    where $\varphi(t)= (1+t^2)^{\frac14} \gnot(t)$.
    Therefore,
    \begin{equation*}
        \p_z^2 \busing^0(x_0+x,z) = x^{-\frac12} \varphi''\left(\frac{z}{x^{\frac13}}\right) \chi + 2 x^{-\frac16} \varphi'\left(\frac{z}{x^{\frac13}}\right) \chi_z + x^{\frac16} \varphi \left(\frac{z}{x^{\frac13}}\right) \chi_{zz}.
    \end{equation*}
    We focus on the regularity of the first term, which is the most singular. 
    We have, for any $\sigma>0$,
    \begin{equation*}
    \left\|  x^{-\frac12} \varphi''\left(\frac{z}{x^{\frac13}}\right) \chi(x,z)\right\|_{L^2_x H^\sigma_z}^2\leq \int \frac{1}{x}\frac{1}{|z-z'|^{1+2\sigma} }\left( \varphi''\left( \frac{z}{x^{\frac13}}\right)  -  \varphi''\left( \frac{z'}{x^{\frac13}}\right)\right)^2\dd x \dd z \dd z'.
    \end{equation*}
    Changing variables in the above integral, we get
    \begin{equation*}
    \left\|  x^{-\frac12} \varphi''\left(\frac{z}{x^{\frac13}}\right) \chi(x,z)\right\|_{L^2_x H^\sigma_z}^2\lesssim \| \varphi''\|_{H^\sigma(\R)}^2 \int_0^{x_1-x_0} |x|^{-\frac{2}{3} - \frac{2\sigma}{3}}\dd x.
    \end{equation*}
    The integral in the right-hand side is finite if and only if $\sigma<\frac12$.
    Moreover, $\|\varphi''\|_{H^\sigma(\R)}^2 \leq \|\varphi''\|_{H^1(\R)}^2$.
    From the definition of $\varphi$ and the decay bounds of \cref{lem:G-decay}, we infer that $\varphi'' \in H^1(\R)$.
    This shows that $\busing^i\in L^2_x H^{2+\sigma}_z$ for $\sigma<\frac12$.
    
    The bound in $H^{\frac{2+\sigma}{3}}_x L^2_z $ is obtained similarly and left to the reader.

    Conversely, if one had $\busing^i \in H^{\frac 56}_x L^2_z \cap H^{\frac 16}_x H^2_z$, by the fractional trace theorem \cite[Equation (4.7), Chapter 1]{lions1969quelques}, one would have $\busing^i \in C^0_z(H^{2/3}_x)$.
    In particular, $\busing^i(\cdot,0) \in H^{2/3}(x_0,x_1)$.
    But, in a neighborhood of $x=0$, $\busing^i(x_0+x,0) = \gnot(0) x^{\frac 16}$ with $\gnot(0) \neq 0$.
    One checks that $x \mapsto x^{\frac 1 6} \in H^s(0,1)$ if and only if $s < 2/3$, which completes the proof.
\end{proof}

Eventually, we introduce the following $2 \times 2$ nonsingular matrix which translates the fact that $\busing^0$ and $\busing^1$ are indeed independent elementary solutions related with the non-satisfaction of the orthogonality constraints associated with $\overline{\ell^0}$ and $\overline{\ell^1}$.
We will use this reference matrix multiple times in the sequel for perturbations of this shear flow situation.

\begin{lem} \label{lem:barM}
    Let $\bfzero$, $\bfun$ as in \cref{lem:busing} and $\overline{\Phi^0}$, $\overline{\Phi^1}$ as in \cref{lem:def-dual-shear}.
    The matrix 
    \nomenclature[OLM]{$\overline{M}$}{Invertible matrix relating the singular solutions $\busing^i$ with the dual profiles $\overline{\Phi^j}$}
    \begin{equation*} 
    \overline{M}:=\left(\int_\Om \p_x \bfj\;  \overline{\Phi^i}\right)_{0\leq i,j\leq 1}\in \cM_2(\R)
    \end{equation*}
    is invertible.
\end{lem}

\begin{proof}
    Let $c \in \R^2$ such that $\overline{M} c=0$.
    Then, for $j = 0,1$,
    \begin{equation*}
    \int_\Om \p_x (c_0 \bfzero + c_1 \bfun) \overline{\Phi^j}=0.
    \end{equation*}
    Thus, the source term for the function $c_0 \busing^0 + c_1 \busing^1$ satisfies the orthogonality conditions \eqref{eq:compat-shear} (note that in this case, the boundary data are null). 
    It then follows from \cref{p:shear-WP-Z1} that $c_0 \busing^0 + c_1 \busing^1\in H^1_xH^1_z$.
    Localizing in the vicinity of $(x_i, 0)$, we infer that $c_i \busing^i \in H^1_xH^1_z$, which, since $\busing^i \notin H^1_x H^1_z$ (by \cref{lem:busing}), implies that $c_i= 0$.
    Therefore, $c = 0$ and $\overline{M}$ is invertible.
\end{proof}

\begin{coro}[Decomposition into singular profiles]
    \label{coro:decomposition-profile}
    Let $(f,\delta_0,\delta_1) \in \cLin_K$ and $u\in Z^0(\Omega)$ be the unique solution to \eqref{eq:shear-z}. 
    Then there exists two real constants $c_0, c_1$ and a function $\ureg\in Z^1(\Om)$, as defined in \eqref{eq:Z1}, such that 
    \begin{equation*} 
    u= c_0 \busing^0 + c_1 \busing^1 + \ureg.
    \end{equation*}
\end{coro}

\begin{proof}
    We recall the definition of the matrix $\overline{M}$ from \cref{lem:barM}. Since $\overline{M}$ is invertible, we may define $c=(c_0, c_1)$ such that
    \begin{equation} \label{def:c}
    \overline{M} c = \begin{pmatrix}
        \overline{\ell^0}(f,\delta_0,\delta_1) \\
        \overline{\ell^1}(f,\delta_0,\delta_1)
    \end{pmatrix}.
    \end{equation}
    Let $\bfzero$ and $\bfun$ as in \cref{lem:busing}.
    By  construction, the triplet $(f-c_0 \bfzero-c_1 \bfun,\delta_0, \delta_1)$ satisfies the orthogonality conditions from \cref{p:shear-WP-Z1}. 
    It follows that the solution 
    $\ureg$ to
    \begin{equation*}
        \begin{cases}
            z\p_x \ureg- \p_{zz}\ureg = f-c_0 \bfzero-c_1 \bfun & \text{in }\Om,\\
            u_{\reg|\Sigma_i}=\delta_i,\\
            u_{\reg|z=\pm 1}=0
        \end{cases}
    \end{equation*}
    satisfies $\ureg\in H^1_x(H^1_z)$. 
    Thus, estimate \eqref{eq:apriori-uz1} of \cref{p:shear-apriori-Z1} ensures that $\p_x \ureg\in Z^0(\Omega)$, i.e.\ $\ureg\in Z^1$.
    Now, $u$ and $ \ureg + c_0 \busing^0 + c_1 \busing^1$ both belong to $Z^0(\Omega)$ and satisfy system \eqref{eq:shear-z}. 
    By the uniqueness result of \cref{p:shear-X0}, the result follows.
\end{proof}

\cref{thm:shear-decomp} follows easily from \cref{coro:decomposition-profile}. 
Indeed, one easily checks from \eqref{eq:def-XB-intro} and \eqref{eq:def-HK} that $\cX_B \hookrightarrow \cLin_K$.
Moreover, by \cref{p:z0-l2-h23}, $Z^0 \hookrightarrow H^{2/3}_x L^2_z \cap L^2_x H^2_z$ and, by \cref{lem:Zsigma-Qsigma}, $Z^1 \hookrightarrow \qone$ (defined in \eqref{eq:def-Q1-intro}).
The rest of the conclusions on $\busing^i$ are derived in \cref{lem:busing}.

\begin{rmk}
    The constants $c_0,c_1$ from \cref{coro:decomposition-profile} depend (linearly) on $u$, but do not depend on the choice of the truncation functions $\chi_i$.
    Indeed, if $\chi_0',\chi_1'$ is another truncation, associated with constants $c_0',c_1'$, then applying \cref{coro:decomposition-profile} twice yields
    \begin{equation*}
    c_0 \busing^0 + c_1 \busing^1 - c_0' (\busing^0)' - c_1' (\busing^1)'\in Z^1.
    \end{equation*}
    Therefore, in a small neighborhood $V_i=\chi_i^{-1}(\{1\})\cap (\chi_i')^{-1}(\{1\})$ of $(x_i,0)$, we obtain
    \begin{equation*}
    (c_i-c_i')r_i^{\frac12} \gnot(t_i) \in H^1_xH^1_z(V_i), 
    \end{equation*}
    and therefore $c_i=c_i'$.
\end{rmk}

As already claimed in \cref{rmk:Phij-saut}, we can also prove a related decomposition result for the dual profiles $\overline{\Phi^j}$ defined in \cref{lem:def-dual-shear}.
Here, the decomposition \emph{always} involves a singular part.

\begin{coro}
    \label{cor:decomp-Phij}
    Let $(c_0,c_1) \in \R^2 \setminus \{0\}$.
    There exists $(d_0,d_1) \in \R^2 \setminus \{0\}$ and $\Phi_\reg \in Z^1$, as defined in \eqref{eq:Z1}, such that
    \begin{equation} \label{eq:Phic-shear-dec}
        c_0 \overline{\Phi^0} + c_1 \overline{\Phi^1} 
        = 
        (-c_0 z + c_1) \mathbf{1}_{z > 0} \zeta(z) 
        + d_0 \busing^0(x,-z) + d_1 \busing^1(x,-z) + \Phi_\reg,
    \end{equation}
    where $\zeta$ is a smooth cut-off function, equal to $1$ near $z = 0$ and compactly supported in $(-1,1)$.
\end{coro}

\begin{proof}
    Using the same decomposition as in \cref{lem:def-dual-shear}, set
    \begin{equation*}
    \Psi^c := c_0 \overline{\Phi^0} + c_1 \overline{\Phi^1} - (-c_0 z + c_1) \mathbf{1}_{z > 0} \zeta(z).
    \end{equation*}
    Then $\widetilde{\Psi^c}(x,z) := \Psi^c(x,-z)$ is the solution to
    \begin{equation*}
    \begin{cases}
        z \p_x \widetilde{\Psi^c} - \p_{zz} \widetilde{\Psi^c} = g_c & \text{in } \Om, \\
        \Psi^c(x_0,z) = 0 & \text{for } z \in (0,1), \\
        \Psi^c(x_1,z) = (-c_0 z - c_1) \zeta(-z) & \text{for } z \in (-1,0), \\
        \Psi^c_{|z=\pm 1}=0,
    \end{cases}
    \end{equation*}
    where $g_c = (c_1 \zeta''(-z) - 2 c_0 \zeta'(-z) + c_0 z \zeta''(-z)) \mathbf 1_{z<0}$.
    Thus, \eqref{eq:Phic-shear-dec} follows from \cref{coro:decomposition-profile}, applied with $f = g_c \in C^\infty(\overline{\Om})$, $\delta_0 = \Delta_0 = 0$ and $\delta_1(z) = (-c_0 z - c_1) \zeta(-z)$ and $\Delta_1 = 0$.
    
    It remains to prove that $(d_0,d_1) \neq (0,0)$.
    By \cref{p:shear-WP-Z1}, $\widetilde{\Psi^c} \in H^1_x H^1_z$ if and only if $\overline{\ell^j}(g_c,0,\delta_1) = 0$ for $j = 0,1$.
    By \cref{def:ell-shear}, since $\p_x g_c = 0$ and $\Delta_0 = \Delta_1 = 0$,
    \begin{equation*}
        \overline{\ell^j}(g_c,0,\delta_1) = 0
        \Longleftrightarrow \p_z^j\delta_0(0) - \p_z^j\delta_1(0) = 0.
    \end{equation*}
    Since $\delta_0 = 0$ and $\delta_1(0) = -c_1$ and $\delta_1'(0) = - c_0$, $(d_0,d_1) = (0,0)$ if and only if $\widetilde{\Psi^c} \in H^1_x H^1_z$, if and only if $(c_0,c_1) = 0$.
\end{proof}

\begin{rmk} \label{rmk:16-shear}
    Using \cref{cor:decomp-Phij} and the regularity result \cref{lem:busing-h56} on $\busing^i$, we see that $f \mapsto \overline{\ell^j}(f,0,0) = \int_\Omega \p_x f \overline{\Phi^j}$ is not only continuous on $H^1_x L^2_z$ but also on $H^\sigma_x L^2_z$ for every $\sigma > \frac 16$.
    We will encounter a related threshold of tangential regularity in \cref{lem:shear-Hs}.
\end{rmk}

Using the decomposition of the dual profiles, we can show that the orthogonality conditions are also independent when considering only variations of the inflow boundary data.

\begin{prop}
    \label{cor:free_ellj-sansf}
    The linear forms $\overline{\ell^0}$ and $\overline{\ell^1}$ are independent on $\{0\} \times C^\infty_c(\Sigma_0) \times C^\infty_c(\Sigma_1)$.
\end{prop}

\begin{proof}
    By contradiction, let $(c_0,c_1) \in \R^2 \setminus \{ 0 \}$ such that, for every $\delta_0 \in C^\infty_c(\Sigma_0)$ and $\delta_1 \in C^\infty_c(\Sigma_1)$,
    \begin{equation*}
        c_0 \overline{\ell^0}(0,\delta_0,\delta_1) + c_1 \overline{\ell^1}(0,\delta_0,\delta_1) = 0.
    \end{equation*}
    Let $(d_0,d_1) \in \R^2 \setminus \{0\}$, $\zeta$ and $\Phi_\reg \in Z^1$ be given by \cref{cor:decomp-Phij}.
    By symmetry, assume that $d_0 \neq 0$.
    Then, by \cref{def:ell-shear}, for every $\delta_0 \in C^\infty_c(\Sigma_0)$, defining $\Delta_0(z):=\delta_0''(z)/z$
    \begin{equation*}
        \begin{split}
            0 & = c_0 \overline{\ell^0}(0,\delta_0,0) + c_1 \overline{\ell^1}(0,\delta_0,0) \\
            & = \int_{\Sigma_0} z \Delta_0 \left[ (-c_0 z + c_1) \zeta(z) + d_0 \busing^0(x_0,-z) + \Phi_\reg (x_0, z)\right]\\
            &=\int_{\Sigma_0} \delta_0'' \left[ (-c_0 z + c_1) \zeta(z) + d_0 \busing^0(x_0,-z) + \Phi_\reg  (x_0, z)\right].
        \end{split}
    \end{equation*}
    Let $\bar{z} > 0$ small enough, one can ensure that $\zeta \equiv 1$ on $(0,\bar{z})$ and $\busing^0(x_0,-z) = z^{\frac12} \gnot(-\infty) = z^{\frac12}$ (see \cref{def:sing-profiles} and \cref{prop:rGk}).
    Since $Z^1 \hookrightarrow H^1_x H^2_z$, $\Phi_{\reg\rvert\Sigma_0} \in H^2(\Sigma_0)$.
    If $\supp \delta_0 \subset (0,\bar{z})$ for $\bar{z} > 0$, integrating by parts yields
    \begin{equation*}
        0 = d_0 \int_0^1 \left[ - \frac14 z^{-\frac32} + \varphi(z) \right] \delta_0(z), 
    \end{equation*}
    where $\varphi(z) := \p_{zz} \Phi_\reg(x_0,z) \in L^2(\Sigma_0)$.
    Since $z \mapsto z^{-\frac32}$ does not belong to $L^2(0,\bar{z})$ but $\varphi$ does, one easily deduces that there exists $\delta_0 \in C^\infty_c((0,\bar{z}))$ such that the right-hand side is non-zero, reaching a contradiction. 
\end{proof}

Let us conclude this section with an easy consequence of the decomposition result from \cref{coro:decomposition-profile}, which will be used in \cref{sec:vorticity}.

\begin{coro}[Single orthogonality condition for localized solutions]
    \label{coro:ortho-localized}
    There exists a couple $(a_0, a_1) \in \R^2\setminus\{(0,0)\}$ such that the following result holds.
    
    Let $(f,\delta_0,\delta_1) \in \cLin_K$ and let $u\in Z^0(\Omega)$ be the unique solution to \eqref{eq:shear-z}.
    Assume that there exists $0<r<\min (x_1-x_0, 1)$ such that $\supp u \subset B((x_1,0), r)^c$. 
    Then $u\in Z^1(\Om)$ if and only if
    \begin{equation*}
    \left(a_0 \overline{\ell^0} + a_1 \overline{\ell^1}\right) (f,\delta_0,\delta_1)=0.
    \end{equation*}
\end{coro}

\begin{proof}
    Let us choose the cut-off function $\chi_1$ from \cref{def:sing-profiles} such that $\supp u \cap \supp \chi_1 = \emptyset$.
    According to \cref{coro:decomposition-profile}, there exists $(c_0,c_1)\in \R^2$ and $\ureg\in Z^1(\Om)$ such that $u=c_0 \busing^0 + c_1 \busing^1 + \ureg$. 
    Multiplying this identity by $\chi_1$, we infer that $c_1 \busing^1 \chi_1 =- \ureg \chi_1 \in Z^1(\Om)$.  
    \cref{lem:busing} then entails that $c_1=0$. 
    
    Therefore $u \in Z^1(\Om)$ if and only if $c_0=0$. We then recall \eqref{def:c}, and we denote by $(a_0, a_1)$ the two coefficients in the first line of $\overline{M}^{-1}$. 
    The result follows.
\end{proof}
	
\subsection{Interpolation and fractional regularity}
\label{sec:shear-frac}

For further purposes, we will also need some fractional regularity results. Their proof relies on interpolation arguments, and therefore on the explicit expressions of the singular profiles. 
Due to a subtle technical difficulty, the proof of these results are postponed to \cref{sec:interpolation}.

\begin{prop} 
    \label{lem:shear-Hs}
    Let $\sigma \in (0,1) \setminus \{ 1/6, 1/2 \}$. 
    Let $f \in H^\sigma_x L^2_z$, $\delta_0\in H^2(\Sigma_0)$, $\delta_1\in H^2(\Sigma_1)$ such that $\delta_0(1)=\delta_1(-1)=0$.
    \begin{itemize}
        \item If $\sigma > 1/6$, assume that $\overline{\ell^0}(f,\delta_0,\delta_1) = \overline{\ell^1}(f,\delta_0,\delta_1) = 0$.
        \item If $\sigma > 1/2$, assume also that $\Delta_i\in \mathscr H^1_z(\Sigma_i)$ and $\Delta_1(-1)=\Delta_0(1)=0$ (recall \eqref{eq:def-Di}).
    \end{itemize}
    The unique strong solution $u \in Z^0(\Omega)$ to \eqref{eq:shear-z}  satisfies $u \in Z^\sigma(\Omega) := [Z^0(\Omega), Z^1(\Omega)]_\sigma$, \nomenclature[FZ1s]{$Z^\sigma$}{Interpolation space $[Z^0,Z^1]$ for fractional Pagani regularity}%
    with
    \begin{equation} \label{eq:estimate-zsigma-shear}
        \| u \|_{Z^\sigma} \lesssim \| f \|_{H^\sigma_x L^2_z} + \|\delta_0\|_{H^2} + \|\delta_1\|_{H^2} + \mathbf 1_{\sigma>1/2} \left( \| \Delta_0\|_{\mathscr H^1_z} + \| \Delta_1\|_{\mathscr{H}^1_z}\right).
    \end{equation}
\end{prop}

\begin{rmk}
    The case $\sigma = 1/6$ is not covered in the above result. 
    This critical level of regularity corresponds to the maximal continuity of the orthogonality conditions.
    Such critical levels are excluded from the abstract interpolation results on which we rely (see \cref{lem:lofstrom}).
    In this case, one would expect a similar result to hold, but with a supplementary norm on the data, in the spirit of \cite{MR2387839,MR1964285}.
    The case $\sigma = 1/2$ is also excluded, but it would be possible to include it provided one introduces an appropriate additional norm.
\end{rmk}

\begin{rmk}
    The regularity assumptions on the $\delta_i$'s are not optimal and could be weakened.
\end{rmk}

We also obtain the following analogue of \cref{coro:ortho-localized} in fractional regularity. 

\begin{coro}[Single orthogonality condition for localized solutions in fractional regularity]
    \label{coro:ortho-localized-frac}
    Let $(a_0, a_1) \in \R^2\setminus\{(0,0)\}$ be the couple from \cref{coro:ortho-localized}.
    Let $\sigma \in (1/6,1) \setminus\{1/2\}$. 
    Let $f \in H^\sigma_x L^2_z$, $\delta_0\in H^2(\Sigma_0)$ such that $\delta_0(1)=0$. For $\sigma>1/2$, assume also that $\Delta_0\in \mathscr H^1_z(\Sigma_0)$ and $\Delta_0(1)=0$.  Let $u\in Z^0(\Omega)$ be the unique solution to \eqref{eq:shear-z} associated with $(f,\delta_0, 0)$.

    Assume that there exists $0<r<\min (x_1-x_0, 1)$ such that $\supp u \subset B((x_1,0), r)^c$. 
    Then $u\in Z^\sigma(\Om)$ if and only if
    \begin{equation*}
    \left(a_0 \overline{\ell^0} + a_1 \overline{\ell^1}\right) (f,\delta_0,0)=0,
    \end{equation*}
    and in this case
    \begin{equation*}
    \|u\|_{Z^\sigma}\lesssim \|f\|_{H^\sigma_x L^2_z} + \|\delta_0\|_{H^2(\Sigma_0)}+ \mathbf 1_{\sigma>1/2}\| \Delta_0\|_{\mathscr H^1_z(\Sigma_0)}.
    \end{equation*}
\end{coro}

\begin{proof}
    The proof follows the same structure as the $Z^1$ case. 

    Using \cref{lem:shear-Hs}, we first prove an analogue of the decomposition result \cref{coro:decomposition-profile} for source terms $f \in H^\sigma_x L^2_z$ with $\sigma \in (1/6,1)$, where the conclusion is that $u_\reg \in Z^\sigma$.

    The conclusion then stems from the fact that $\busing^0 \notin Z^\sigma$.
    Indeed, by \cref{lem:Zsigma-Qsigma}, for $\sigma \geq 1/6$, $Z^\sigma \hookrightarrow H^{\frac 5 6}_x L^2_z \cap H^{\frac 1 6}_x H^2_z$.
    But, from \cref{lem:busing-h56}, $\busing^0 \notin H^{\frac 5 6}_x L^2_z \cap H^{\frac 1 6}_x H^2_z$.
\end{proof}

\newpage
\section{A first nonlinear example in kinetic theory}
\label{sec:Fokker-Planck}

In this section, we explain how the linear theory of \cref{sec:shear} can be used in a simple nonlinear context.
Before moving on to nonlinear examples from fluid mechanics in \cref{sec:Burgers,sec:Prandtl} (which involve additional difficulties), we encourage the reader to start by reading this section where we set up the basics of our method to construct perturbative solutions to semilinear or quasilinear problems despite orthogonality conditions. 
In particular, we formulate a black-box abstract result in \cref{sec:abstract} which we will use in the sequel.

\subsection{Description of the model and main result}

As an example, we will show how one can build regular solutions to a stationary nonlinear system of Vlasov--Poisson--Fokker--Planck type, set on a bounded interval. 
For the sake of readability, we will focus on the following system:
\begin{equation}\label{eq:VPFP}
\begin{cases}
    z \p_x u + E[u] \p_z u - \p_{zz} u = f, \\
    u_{\rvert \Sigma_i} = \delta_i, \\
    u_{\rvert z=\pm 1}=0,
\end{cases}
\end{equation}
where $E[u]$ is an electric force deriving from a potential $V[u]$ satisfying a Poisson equation:
\begin{equation}\label{eq:potential}
    E = \p_x V \quad \text{where} \quad
    \begin{cases}
    \p_{xx} V(x) = \int_{-1}^1 u(x,z)\dd z\qquad\text{for } x \in (x_0,x_1),\\
    \p_x V_{\rvert x=x_0} = 0.
    \end{cases}
\end{equation}
In this toy model, the term $E[u] \p_z u$ corresponds to a semilinear contribution, which is easily estimated since explicit integration of \eqref{eq:potential} and the Cauchy--Schwarz inequality yield
\begin{equation} \label{eq:bound-E}
    \| E[u] \|_{L^\infty(x_0,x_1)} \lesssim \| E[u] \|_{H^1(x_0,x_1)} \lesssim \| u \|_{L^2(\Om)}.
\end{equation}

\begin{rmk}
    Our toy kinetic model \eqref{eq:VPFP}-\eqref{eq:potential} departs from classical kinetic models such as the one studied in \cite{hwang2019vlasov} in the following ways:
    \begin{itemize}
        \item As mentioned before, the variable $z$ is more commonly denoted by~$v$ and represents the velocity of the particles.
        We keep the notation $z$ by consistency with the remainder of the paper.
        \item Usually, even if the position variable $x$ lives in a bounded domain, the velocity variable $z$ lives in $\R$ so that particles can take arbitrary speeds.
        Since our motivation is to understand what happens near the critical line $\{ z = 0 \}$ we focus here on the region $z \in [-1,1]$.
        We expect that our techniques can be applied to the unbounded case to obtain similar results, provided that one works in the appropriate functional spaces to encode decay as $|z| \to \infty$.
        \item One could also enforce a non-zero Neumann boundary condition for the potential $V$ at the left endpoint $x=x_0$, namely $\p_x V_{\rvert x = x_0} = g_0 \in \R$ as in \cite{hwang2019vlasov}.
        This is a straightforward adaptation of the results presented below.
    \end{itemize}
\end{rmk}

The goal of the next paragraphs is to prove the following counterparts of \cref{p:pagani-shear} and \cref{thm:shear-Z1} concerning the linear model \eqref{eq:shear-z} for our nonlinear toy model.
We will work with the following spaces of data triplets:
\nomenclature[FHFP]{$ \cH_{FP}$}{Hilbert space with norm \eqref{def:norm-FP}  of data triplets $ (f,\delta_0,\delta_1)$ for the kinetic toy model}
\begin{equation}
    \label{eq:def-H_FP}
    \cH_{FP} := \left\{ (f,\delta_0,\delta_1) \in \cLin_K ; \quad 
    \delta_i'(z)/z \in \mathscr{H}^1_z(\Sigma_i) \text{ and } 
    \delta_i'((-1)^i) = 0 \text{ for } i \in \{ 0,1 \} \right\}
\end{equation}
with the norm
\begin{equation}\label{def:norm-FP}
    \| (f,\delta_0,\delta_1) \|_{\cH_{FP}} := \| (f,\delta_0,\delta_1) \|_{\cLin_K} + \| \delta_0'(z)/z \|_{\mathscr{H}^1_z} + \| \delta_1'(z)/z \|_{\mathscr{H}^1_z},
\end{equation}
where we recall that the space $\mathscr{H}^1_z$ is defined in \eqref{eq:H-pagani-k} and the space $\cLin_K$ in \eqref{eq:def-HK}. We also define
\begin{equation*}
    \cH_{FP,\operatorname{sg}}^\perp := \{ (f,\delta_0,\delta_1) \in \cH_{FP} ; \quad 
    \overline{\ell^0}(f,\delta_0,\delta_1) = \overline{\ell^1}(f,\delta_0,\delta_1) = 0 \}.
\end{equation*}

\begin{thm}
    \label{prop:VPFP}
    There exists a constant $\eta > 0$, and a Lipschitz submanifold $\cM_{FP}$ of $\cH_{FP}$ of codimension~2, containing $0$ and included in the ball of radius $\eta$ in $\cH_{FP}$, modeled on $\cH_{FP,\operatorname{sg}}^\perp$ and tangent to it at $0$ (see \cref{rmk:VPFP-tangent}), such that the following statements hold:
    \begin{enumerate}
        \item \label{item:prop:VPFP-1} 
        For all $(f,\delta_0,\delta_1) \in L^2(\Omega) \times \mathscr{H}^1_z(\Sigma_0) \times \mathscr{H}^1_z(\Sigma_1)$ with $\delta_0(1) = \delta_1(-1) = 0$ such that
        \begin{equation} \label{eq:fL2+d0+d1<eta}
            \| f \|_{L^2} + \| \delta_0 \|_{\mathscr{H}^1_z} + \| \delta_1 \|_{\mathscr{H}^1_z} \leq \eta,
        \end{equation}
        system \eqref{eq:VPFP}-\eqref{eq:potential} has a solution $u\in Z^0(\Omega)$ satisfying
        \begin{equation} \label{eq:VPFP-Z0}
            \| u \|_{Z^0} \lesssim \| f \|_{L^2} + \| \delta_0 \|_{\mathscr{H}^1_z} + \| \delta_1 \|_{\mathscr{H}^1_z}
        \end{equation}
        and which is unique in a neighborhood of $0$ in $Z^0(\Omega)$.
        
        \item \label{item:prop:VPFP-2} 
        For all $(f, \delta_0, \delta_1) \in \cH_{FP}$ such that 
        \begin{equation*}
            \|(f,\delta_0,\delta_1)\|_{\cH_{FP}} \leq \eta,
        \end{equation*}
        the locally unique solution $u\in Z^0(\Omega)$ to \eqref{eq:VPFP}-\eqref{eq:potential} enjoys $Z^1(\Omega)$ regularity if and only if $(f, \delta_0,\delta_1) \in \cM_{FP}$, which corresponds to two nonlinear orthogonality conditions.
        
        For such data, one has
        \begin{equation} \label{eq:VPFP-Z1}
            \| u \|_{Z^1} \lesssim \| (f,\delta_0,\delta_1) \|_{\cH_{FP}}.
        \end{equation}
    \end{enumerate}
\end{thm}

\begin{rmk}
The nonlinearity of the Vlasov--Poisson--Fokker--Planck system \eqref{eq:VPFP} is sufficiently mild to allow for a theory of weak solutions, leading to the first statement of \cref{prop:VPFP}.
The Prandtl system in the vicinity of the recirculation zone enjoys the same feature, accounting for the first part of \cref{thm:prandtl}.
However, the nonlinearity in the Burgers system \eqref{eq:eq0-uux} is stronger, and prevents us from proving the analogue of the first statement of the above theorem.

\end{rmk}

\subsection{Well-posedness theory with low regularity}

We prove in this subsection \cref{item:prop:VPFP-1} of \cref{prop:VPFP}, which corresponds to the well-posedness theory at regularity $Z^0$, and is therefore a nonlinear counterpart of \cref{p:pagani-shear}.

\begin{lem}[Existence of $Z^0$ solutions of \eqref{eq:VPFP}-\eqref{eq:potential}]
    There exists $\eta > 0$ such that, for any $(f,\delta_0,\delta_1) \in L^2(\Omega) \times \mathscr{H}^1_z(\Sigma_0) \times \mathscr{H}^1_z(\Sigma_1)$ with $\delta_0(1) = \delta_1(-1) = 0$ satisfying \eqref{eq:fL2+d0+d1<eta}, there exists a solution $u \in Z^0(\Om)$ to \eqref{eq:VPFP}-\eqref{eq:potential} with \eqref{eq:VPFP-Z0}.
\label{lem:VPFP-Z0}
\end{lem}

\begin{proof}
    Let $(f,\delta_0,\delta_1) \in L^2(\Omega) \times \mathscr{H}^1_z(\Sigma_0) \times \mathscr{H}^1_z(\Sigma_1)$ with $\delta_0(1) = \delta_1(-1) = 0$ satisfying \eqref{eq:fL2+d0+d1<eta}
    for some $\eta > 0$ small enough to be chosen later.

    \begin{itemize}
        \item \textbf{Definition of the sequence.}
        We construct a sequence by setting $u_0 := 0$ and, for all $n\in \N$, we define $u_{n+1} \in Z^0(\Omega)$ by induction as the solution to
        \begin{equation*} 
        \begin{cases}
            z \p_x u_{n+1}- \p_{zz} u_{n+1}= f - E_n \p_z u_n, \\
            (u_{n+1})_{\rvert \Sigma_i} = \delta_i, \\
            (u_{n+1})_{\rvert z=\pm 1}=0,
        \end{cases}
        \end{equation*}
        where $E_n := E[u_n]$.
        At each step, by \eqref{eq:bound-E}, $E_n \in L^\infty(x_0,x_1)$.
        Hence, since $u_n \in Z^0(\Om)$, $f - E_n \p_z u_n \in L^2(\Om)$, so the existence of $u_{n+1} \in Z^0(\Om)$ follows from \cref{p:pagani-shear}.

        \item \textbf{Uniform bound in $Z^0$.}
        Let us prove by induction that $\|u_n\|_{Z^0} \leq 2 C_P \eta$ for all $n\in \N$, where $C_P$ is the constant in Pagani's estimate \eqref{estimate-pagani-Z0}, provided that $\eta$ is small enough.
        The statement is true for $n=0$. 
        For $n\geq 0$, by \eqref{eq:bound-E}, $\| E_n \|_{L^\infty} \lesssim \| u_n \|_{L^2} \lesssim \eta$.
        As a consequence, it follows from \cref{p:pagani-shear} that
        \begin{equation*}
            \| u_{n+1}\|_{Z^0}
            \leq C_P(\| f \|_{L^2} + \| \delta_0 \|_{\mathscr{H}^1_z} + \| \delta_1 \|_{\mathscr{H}^1_z} + \|E_n\|_\infty \|\p_z u_n\|_{L^2} )\leq C_P \eta + C \eta^2,
        \end{equation*}
        for some $C$ depending only on $\Om$. 
        Therefore, if $C\eta<C_P$, the bound propagates by induction.

        \item \textbf{Convergence.}
        Now, let $w_n:= u_{n+1}-u_n$. 
        Then, for $n \geq 1$, $w_n$ is a solution to
        \begin{equation*}
            \begin{cases}
                z \p_x w_n  - \p_{zz} w_n = - (E_n - E_{n-1}) \p_z u_{n-1} - E_n \p_z (u_n-u_{n-1}), \\
                (w_n)_{\rvert \Sigma_i} = 0, \\
                (w_n)_{\rvert z = \pm 1} = 0.
            \end{cases}
        \end{equation*}
        By \eqref{eq:bound-E}, $\| E_n - E_{n-1} \|_{L^\infty} \lesssim \| u_n - u_{n-1} \|_{L^2}$ and $\| E_n \|_{L^\infty} \lesssim \|u_n\|_{L^2}$. 
        Hence, by \cref{p:pagani-shear}, 
        \begin{equation*}
            \|w_n\|_{Z^0} 
            \lesssim
            \| (E_n - E_{n-1}) \p_z u_{n-1} \|_{L^2} + \| E_n \p_z (u_n-u_{n-1}) \|_{L^2}
            \lesssim \eta \|w_{n-1}\|_{Z^0}
        \end{equation*}
        and thus $(u_n)_{n\in \N}$ is a Cauchy sequence in $Z^0(\Om)$ provided that $\eta$ is small enough.
        Passing to the limit as $n\to \infty$, we obtain a strong solution $u \in Z^0$ with $\|u\|_{Z^0} \leq 2 C_P \eta$ to \eqref{eq:VPFP}-\eqref{eq:potential}.

        Eventually, the uniform bound propagated on the sequence also passes to the limit and implies \eqref{eq:VPFP-Z0}.
        \qedhere
    \end{itemize}
\end{proof}

\begin{lem}[Uniqueness of $Z^0$ solutions of \eqref{eq:VPFP}-\eqref{eq:potential}]
    There exists $\eta > 0$ such that, for any $(f,\delta_0,\delta_1) \in L^2(\Omega) \times \mathscr{H}^1_z(\Sigma_0) \times \mathscr{H}^1_z(\Sigma_1)$, \eqref{eq:VPFP}-\eqref{eq:potential} has at most one solution $u \in Z^0(\Om)$ such that $\|u\|_{Z^0} \leq \eta$.
\end{lem}

\begin{proof}
    Let $(f,\delta_0,\delta_1) \in L^2(\Omega) \times \mathscr{H}^1_z(\Sigma_0) \times \mathscr{H}^1_z(\Sigma_1)$ and $u, u' \in Z^0(\Omega)$ be two solutions to \eqref{eq:VPFP}-\eqref{eq:potential}.
    Then $w := u - u' \in Z^0(\Omega)$ is a solution to
    \begin{equation} \label{eq:VFPP-w}
        \begin{cases}
            z \p_x w - \p_{zz} w = (E[u']-E[u]) \p_z u' - E[u] \p_z w, \\
            w_{\rvert \Sigma_i} = 0, \\
            w_{\rvert z=\pm 1}=0.
        \end{cases}
    \end{equation}
    Multiplying \eqref{eq:VFPP-w} by $w$, integrating by parts and using the boundary conditions and $\p_z E[u] = 0$, we obtain
    \begin{equation*}
        \int_\Omega (\p_z w)^2 \leq \int_\Omega \left|(E[u']-E[u]) \p_z u' w\right|.
    \end{equation*}
    By \eqref{eq:bound-E}, $\|E[u']-E[u]\|_{L^\infty} \lesssim \|w\|_{L^2}$.
    Hence, since $w_{\rvert z = \pm 1} = 0$, Poincaré's inequality entails that
    \begin{equation*}
        \| w \|_{L^2}^2 
        \lesssim \| \p_z w \|_{L^2}^2 
        \lesssim \| E[u'] - E[u] \|_{L^\infty} \| \p_z u' \|_{L^2} \| w \|_{L^2}
        \lesssim \| \p_z u' \|_{L^2} \| w \|_{L^2}^2.
    \end{equation*}
    Hence, there exists $C_1 > 0$ (depending only on $\Omega$) such that
    \begin{equation} \label{eq:VFPP-w-leq-w}
        \| w \|_{L^2}^2 \leq C_1 \| \p_z u' \|_{L^2} \| w \|_{L^2}^2.
    \end{equation}
    If $C_1 \| \p_z u' \|_{L^2} < 1$, \eqref{eq:VFPP-w-leq-w} implies $w = 0$, so uniqueness holds in the ball of radius $1/C_1$ of $Z^0(\Omega)$.
\end{proof}

\subsection{Nonlinear orthogonality conditions for higher regularity}
\label{sec:VPFP-constr-Z1}

We now prove \cref{item:prop:VPFP-2} of \cref{prop:VPFP}, which corresponds to the well-posedness theory at regularity~$Z^1$, under orthogonality conditions, and is therefore a nonlinear counterpart of \cref{thm:shear-Z1}.

\begin{lem} \label{lem:VPFP-biorth}
    There exists $(f^k,\delta_0^k,\delta_1^k) \in \cH_{FP}$ such that, for $j,k\in\{0,1\}$, $\overline{\ell^j}(f^k,\delta_0^k,\delta_1^k) = \mathbf{1}_{j=k}$.
\end{lem}

\begin{proof}
    As \cref{lem:biorth}, this follows from \cref{free:f}. 
\end{proof}

\begin{prop} \label{p:VPFP-constr-Z1}
    There exist $\eta > 0$ and maps $U_{FP} : B_\eta \to Z^1(\Om)$ and $(\nu_{FP}^0,\nu_{FP}^1) : B_\eta \to \R^2$, where $B_\eta$ is the ball of radius $\eta$ in $\cH_{FP}$ such that, for any $(f,\delta_0,\delta_1) \in B_\eta$, $u := U_{FP}(f,\delta_0,\delta_1) \in Z^1(\Om)$ and $\nu^j := \nu_{FP}^j(f,\delta_0,\delta_1)$ obey the equation
    \begin{equation} \label{FP-nu}
        \begin{cases}
            z \p_x u + E[u] \p_z u -\p_{zz} u = f + \nu^0 f^0 + \nu^1 f^1, \\
            u_{\rvert \Sigma_i} = \delta_i + \nu^0 \delta_i^0 + \nu^1 \delta_i^1, \\
            u_{\rvert z = \pm 1} = 0
        \end{cases}
    \end{equation}
    where the triplets $(f^k,\delta_0^k,\delta_1^k)$ for $k \in \{0,1\}$ are defined in \cref{lem:VPFP-biorth}.
    Furthermore,  $u$ and $\nu$ satisfy the estimate
    \begin{equation} \label{eq:VPFP-unu-leq-Xi}
        \| u \|_{Z^1} + |\nu^0| + |\nu^1| \lesssim \|(f,\delta_0,\delta_1) \|_{\cH_{FP}}
    \end{equation}
    and the orthogonality conditions 
    \begin{equation} \label{eq:VPFP-nu-limit}
        \nu^j = - \overline{\ell^j}(f- E[u] \p_z u, \delta_0,\delta_1)\quad \text{for }j \in \{0,1\}.
    \end{equation}
\end{prop}

\begin{proof}
    Let $(f,\delta_0,\delta_1) \in \cH_{FP}$ with $\|(f,\delta_0,\delta_1)\|_{\cH_{FP}} \leq \eta$ small enough to be chosen later on.
    We modify our iterative scheme to construct $Z^1$ solutions using \cref{thm:shear-Z1} and accommodate for the two orthogonality conditions at each step.
    
    \begin{itemize}
        \item \textbf{Definition of the sequence.}
        More precisely, we take $u_0 := 0$ and, for $n \in \N$, given $u_n \in Z^1(\Omega)$ such that $u_n\vert_{\Sigma_i}\in \mathscr{H}^1_z(\Sigma_i)$, we define $u_{n+1} \in Z^1(\Om)$ as the solution to
        \begin{equation} \label{eq:VPFP-un-Z1}
            \begin{cases}
                z \p_x u_{n+1}- \p_{zz} u_{n+1}= f - E_n \p_z u_n + \nu_{n+1}^0 f^0 + \nu_{n+1}^1 f^1, \\
                (u_{n+1})_{\rvert \Sigma_i} = \delta_i + \nu_{n+1}^0 \delta^0_i + \nu_{n+1}^1 \delta^1_i, \\
                (u_{n+1})_{\rvert z=\pm 1}=0,
            \end{cases}
        \end{equation}
        where $E_n := E[u_n]$, the triplets $(f^k,\delta_0^k,\delta_1^k)$ for $k \in \{0,1\}$ are defined in \cref{lem:VPFP-biorth} and
        \begin{equation} \label{eq:VPFP-nu}
            \nu_{n+1}^j := - \overline{\ell^j} (f-E_n\p_z u_n, \delta_0,\delta_1).
        \end{equation}
        This choice ensures that the two orthogonality conditions
        \begin{equation*}
            \overline{\ell^j} \big(
            f - E_n \p_z u_n + \nu_{n+1}^0 f^0 + \nu_{n+1}^1 f^1,
            \delta_0 + \nu_{n+1}^0 \delta^0_0 + \nu_{n+1}^1 \delta^1_0,
            \delta_1 + \nu_{n+1}^0 \delta^0_1 + \nu_{n+1}^1 \delta^1_1
            \big) = 0
        \end{equation*}
        are satisfied.
         
        We now verify that the data of \eqref{eq:VPFP-un-Z1} satisfy the assumptions of \cref{thm:shear-Z1}. This mostly follows from the inclusion $\cH_{FP} \subset \cLin_K$.
        It only remains to check that $(- E_n \p_z u_n, 0,0) \in \cLin_K$, i.e.\ that $- E_n \p_z u_n \in H^1_x L^2_z$, $(- E_n \p_z u_n / z)_{\rvert \Sigma_i} \in \mathscr{H}^1_z(\Sigma_i)$ and $(- E_n \p_z u_n / z)(x_i,(-1)^i) = 0$.
         The condition $\p_z u_n(x_i,(-1)^i) = 0$ is guaranteed by the constraint $\delta_i'((-1)^i) = 0$ contained in definition \eqref{eq:def-H_FP} which also entails $(\delta_i^k)'((-1)^i) = 0$.
        We now estimate the norms of $ E_n \p_z u_n$.
First, from the lateral boundary conditions, we derive that
        \begin{equation}\label{est:pzu_Sigmai}
            \begin{split}
            \| \p_z u_{n+1}(x_i,z) / z \|_{\mathscr{H}^1_z(\Sigma_i)}
            & \leq \| \delta_i'(z) / z \|_{\mathscr{H}^1_z(\Sigma_i)}
            + \sum_{k \in \{0,1\}} |\nu^k_{n+1} | \cdot \| (\delta_i^k)'(z) / z \|_{\mathscr{H}^1_z(\Sigma_i)} \\
            & \lesssim \| (f,\delta_0,\delta_1) \|_{\cH_{FP}} + |\nu^0_{n+1}| + |\nu^1_{n+1}|.
            \end{split}
        \end{equation}
Since  $E_n$ does not depend on $z$, we obtain, using  \eqref{eq:bound-E},
        \begin{equation}\label{est:Epzu_sigmai}
            \| E_n(x_i) \p_z u_n(x_i,z) / z \|_{\mathscr{H}^1_z(\Sigma_i)} 
            \leq \| E_n \|_{L^\infty} \| \p_z u_n(x_i,z) / z \|_{\mathscr{H}^1_z(\Sigma_i)},
        \end{equation}
and therefore $(- E_n \p_z u_n / z)_{\rvert \Sigma_i} \in \mathscr{H}^1_z(\Sigma_i)$.
        Moreover,  using  again \eqref{eq:bound-E},
        \begin{equation}\label{est:Epzu-H1}
            \| E_n \p_z u_n \|_{H^1_x L^2_z}
            \lesssim \| E_n \|_{H^1_x} \| \p_z u_n \|_{L^\infty_x L^2_z}
            + \| E_n \|_{L^\infty_x} \| \p_z u_n \|_{H^1_x L^2_z}
            \lesssim \| u_n \|_{L^2} \| u_n \|_{Z^1}.
        \end{equation}
        By \cref{thm:shear-Z1}, we conclude that $u_{n+1} \in Z^1(\Om)$.
      
        \item \textbf{Uniform bound in $Z^1$.}
        Let us prove by induction that there exists a constant $C_1 > 0$ such that, if $\eta$ is small enough, then, for all $n \in \N$,
        \begin{equation*}
            U_n := \| u_n \|_{Z^1} + \sum_{i \in \{0,1\}} \| \p_z u_n(x_i,z) / z \|_{\mathscr{H}^1_z(\Sigma_i)} \leq 2 C_1 \eta.
        \end{equation*}
        This holds for $n = 0$. 
        For $n \in \N$, it follows from \eqref{eq:u-Z1} that
        \begin{equation*}
            \| u_{n+1} \|_{Z^1} \lesssim
            \| (f,\delta_0,\delta_1) \|_{\cLin_K} + |\nu_{n+1}| + \| E_n \p_z u_n \|_{H^1_x L^2_z} + \sum_{i \in \{0,1\}} \| (E_n \p_z u_n)\vert_{\Sigma_i} / z \|_{\mathscr{H}^1_z(\Sigma_i)}.
        \end{equation*}     
        We obtain from \eqref{eq:VPFP-nu} and \cref{lem:continuity-ell-easy} that
        \begin{equation*}
            |\nu^j_{n+1}| \lesssim \| (f,\delta_0,\delta_1) \|_{\cLin_K} + \| E_n \p_z u_n \|_{H^1_x L^2_z} + \sum_{i \in \{0,1\}} \| (E_n \p_z u_n)\vert_{\Sigma_i} / z \|_{\mathscr{H}^1_z(\Sigma_i)}.
        \end{equation*}
Using \eqref{est:pzu_Sigmai}, \eqref{est:Epzu_sigmai}, \eqref{est:Epzu-H1}, we infer that
there exists $C_1 > 0$ (depending only on $\Omega$) such that
        \begin{equation*}
            U_{n+1} \leq C_1 \left( \| (f,\delta_0,\delta_1) \|_{\cH_{FP}} + U_n^2 \right).
        \end{equation*}
        Thus, if $\eta \leq 1 / (4 C_1^2)$, then the bound $U_n \leq 2 C_1 \eta$ propagates by induction.

        \item \textbf{Convergence.}
        As in the low regularity case, we let $w_n := u_{n+1} - u_n$.
        Then, for $n \geq 1$, $w_n$ is now a solution to:
        \begin{equation*}
            \begin{cases}
                \begin{aligned}
                    z \p_x w_n  - \p_{zz} w_n = & - (E_n - E_{n-1}) \p_z u_{n-1} - E_n \p_z (u_n-u_{n-1}) \\
                    & + (\nu^0_{n+1}-\nu^0_n) f^0 + (\nu^1_{n+1}-\nu^1_n) f^1,
                \end{aligned} \\
                (w_n)_{\rvert \Sigma_i} = (\nu^0_{n+1}-\nu^0_n) \delta_i^0 + (\nu^1_{n+1}-\nu^1_n) \delta_i^1, \\
                (w_n)_{\rvert z = \pm 1} = 0.
            \end{cases}
        \end{equation*}
        Using the same type of proof as above, we derive from \eqref{eq:VPFP-nu} and \cref{def:ell-shear} that
        \begin{equation*}
        | \nu^j_{n+1}-\nu^j_n| \lesssim \eta \| w_{n-1} \|_{Z^1}.
        \end{equation*}
        Therefore, estimate \eqref{eq:u-Z1} of \cref{thm:shear-Z1} entails that
        \begin{equation*}
        \| w_n \|_{Z^1} \lesssim \eta \| w_{n-1} \|_{Z^1}.
        \end{equation*}
        Thus $(u_n)_{n\in \N} $ and $(\nu^j_n)_{n\in \N}$ are Cauchy sequences. 
        Passing to the limit, we deduce that there exist $u\in Z^1(\Om)$ and $(\nu^0, \nu^1) \in \R^2$ satisfying \eqref{FP-nu}, \eqref{eq:VPFP-unu-leq-Xi} and \eqref{eq:VPFP-nu-limit}. \qedhere
    \end{itemize}
\end{proof}

\begin{defi}
    For $\eta > 0$ small enough, we define $\cM_{FP}$ as
    \begin{equation} \label{eq:MFP}
    \cM_{FP}:=\{ (f, \delta_0,\delta_1)\in \cH_{FP}; \quad \|(f, \delta_0,\delta_1)\|_{\cH_{FP}} \leq \eta \text{ and } \mathcal \nu_{FP}(f, \delta_0,\delta_1)=(0,0)\}.
    \end{equation}
    By definition, for any $(f, \delta_0,\delta_1) \in \cM_{FP}$, there exists a solution $u\in Z^1(\Om)$ to \eqref{eq:VPFP} (since \eqref{FP-nu} is satisfied with $\nu^0=\nu^1=0$), which satisfies \eqref{eq:VPFP-Z1} thanks to \eqref{eq:VPFP-unu-leq-Xi}.
\end{defi}

\begin{prop} \label{p:VPFP-MFP-necessary}
    There exists $\eta > 0$ such that, for any $(f,\delta_0,\delta_1) \in \cH_{FP}$ and $u \in Z^1(\Omega)$ solution to \eqref{eq:VPFP}-\eqref{eq:potential} satisfying $\|(f,\delta_0,\delta_1)\|_{\cH_{FP}} \leq \eta$ and $\|u\|_{Z^1} \leq \eta$, one has $(f,\delta_0,\delta_1) \in \cM_{FP}$.
\end{prop}

\begin{proof}
    Let $(f,\delta_0,\delta_1) \in \cH_{FP}$ and $u \in Z^1(\Omega)$ be a solution to \eqref{eq:VPFP}-\eqref{eq:potential}. 
    Assume that $\|u\|_{Z^1} \leq \eta$ and $\|(f,\delta_0,\delta_1)\|_{\cH_{FP}} \leq \eta$ for some $\eta > 0$ small enough to be chosen later.

    Since $u \in Z^1(\Omega)$, one has $- E[u] \p_z u \in H^1_x L^2_z$.
    Thus, viewing \eqref{eq:VPFP} as a linear equation with source term $f - E[u] \p_z u$, \cref{thm:shear-Z1} implies that
    \begin{equation*}
        \overline{\ell^j} (f - E[u] \p_z u, \delta_0, \delta_1) = 0.
    \end{equation*}
    Now, let $(\tilde{u}, \nu^0, \nu^1) \in Z^1(\Om) \times \R^2$ be the solution to \eqref{FP-nu} constructed from $(f,\delta_0,\delta_1)$ in \cref{p:VPFP-constr-Z1}.
    By \eqref{eq:VPFP-nu-limit},
    \begin{equation*}
        \nu^j = - \overline{\ell^j} (f - E[\tilde{u}] \p_z \tilde{u}, \delta_0, \delta_1).
    \end{equation*}
    Combining both equalities leads to
    \begin{equation*}
        | \nu^0 | + | \nu^1 | \lesssim (\|u\|_{Z^1} + \| \tilde u\|_{Z^1})\| u - \tilde{u} \|_{Z^1}.
    \end{equation*}
    Therefore, writing the system satisfied by $w := u - \tilde{u}$ and applying estimate \eqref{eq:u-Z1} of \cref{p:shear-WP-Z1} leads to
    \begin{equation*}
        \| w \|_{Z^1} \lesssim \eta \| w \|_{Z^1}.
    \end{equation*}
    If $\eta > 0$ is small enough, this implies that $w = 0$, so $\nu^0 = \nu^1 = 0$, and $(f,\delta_0,\delta_1) \in \cM_{FP}$.
\end{proof}

\begin{rmk}[An alternative approach]
    Another potential proof of \cref{item:prop:VPFP-2} of \cref{prop:VPFP} could be the following. 
    Consider the map
    \begin{equation*}
        \nu^j_{FP} : (f,\delta_0,\delta_1)\in \cH_{FP}\mapsto -\overline{\ell^j}(f,\delta_0,\delta_1)+\overline{\ell^j}(E[u]\p_z u,0,0)\in \R^2,
    \end{equation*}
    where $u\in Z^0$ is the unique solution to system \eqref{eq:VPFP}-\eqref{eq:potential}, provided by \cref{lem:VPFP-Z0}.
    
    Since $u\in Z^0$, $E[u] \p_z u \in H^{1/3}_x L^2_z$, and therefore $ \nu^j_{FP}$ is well-defined thanks to \cref{rmk:16-shear}. We then set $\cM_{FP}:=\{(f,\delta_0,\delta_1)\in \cH_{FP} ; \enskip  \nu^j_{FP}(f,\delta_0,\delta_1)=0\}$. 
    Then for all $(f,\delta_0,\delta_1) \in \cH_{FP}$, $u$ is a solution to an equation of the type $z\p_x -\p_{zz} u =g$, where the right-hand side $g$ belongs to $H^{1/3}_x L^2_z$ and satisfies orthogonality conditions. 
    It follows from the interpolation result \cref{lem:shear-Hs} and \cref{lem:Zsigma-Qsigma} that $u\in Z^{1/3} \hookrightarrow H^1_x L^2_z \cap H^{1/3}_x H^2_z$, and therefore $E[u] \p_z u\in H^{2/3}_x L^2_z$. 
    Bootstrapping twice the same argument, we eventually infer that $u\in Z^1$.
    
    However this argument is based on the existence of $Z^0$ solutions of the nonlinear problem without any orthogonality condition. 
    For the Burgers equation, the nonlinearity is too strong for such a theory of weak solutions to be available. 
    Therefore, in order to unify the presentation, we have chosen to present a different proof, based on a modification of the iterative scheme.
\end{rmk}

\subsection{Regularity and tangent space of the manifold}
\label{sec:manifold-FP}

We now give another description of the set $\cM_{FP}$ defined in \eqref{eq:MFP}, which we use to prove that it is indeed a Lipschitz submanifold of $\cH_{FP}$ of codimension 2, modeled on $\cH_{FP,\operatorname{sg}}^\perp$, and we describe its tangent space at the origin.
\nomenclature[OAN]{$\Xi$}{Shorthand for a data triplet $\Xi = (f,\delta_0,\delta_1).$}
Throughout this paragraph, we denote by $\Xi=(f,\delta_0,\delta_1)$ an element of $\cH_{FP}$. 

We recall that there exist $\Xi^0,\Xi^1\in \cH_{FP}$ such that $\overline{\ell^j}(\Xi^k)=\mathbf 1_{j=k}$ (see \cref{lem:VPFP-biorth}), and such that $\cH_{FP,\operatorname{sg}}^\perp := \ker \overline{\ell^0} \cap \ker \overline{\ell^1} \cap \cH_{FP} = (\R \Xi^0 + \R \Xi^1)^{\bot}$.
For every $\Xi  \in \cH_{FP}$, one has the decomposition
\begin{equation} \label{eq:Xi-decomposition}
    \Xi = \Xi^\perp + \langle \Xi ; \Xi^0 \rangle_{\cH_{FP}} \Xi^0 + \langle \Xi ; \Xi^1 \rangle_{\cH_{FP}} \Xi^1,
\end{equation}
where $\Xi^\perp \in \cH_{FP,\operatorname{sg}}^\perp$ and the linear maps $\Xi \mapsto \Xi^\perp$ and $\Xi \mapsto \langle \Xi^k ; \Xi \rangle_{\cH_{FP}}$ are continuous.
 
\begin{lem}
    For $\eta > 0$ small enough, the set $\cM_{FP}$ defined in \eqref{eq:MFP} is equal to
    \begin{equation} \label{eq:MFP'}
        \widetilde{M}_{FP} := \left\{ \widetilde{\Xi} \in \cH_{FP}; \quad \|\widetilde{\Xi}\|_{\cH_{FP}} \leq \eta \text{ and } \langle \widetilde{\Xi} ; \Xi^j \rangle = \nu_{FP}^j(\widetilde{\Xi}^\perp) \text{ for } j \in \{0,1\} \right\}.
    \end{equation}
\end{lem}

\begin{proof}
    We proceed by double inclusion.

   $\bullet$  Let $\widetilde{\Xi} \in \widetilde{M}_{FP}$.
    Consider the solution $(u,\nu^0,\nu^1) \in Z^1(\Om) \times \R^2$ constructed for the data $\widetilde{\Xi}^\perp$ in \cref{p:VPFP-constr-Z1}.
    Then $u \in Z^1(\Om)$ is a solution to \eqref{eq:VPFP}-\eqref{eq:potential} with data $\widetilde{\Xi}^\perp + \nu^0_{FP}(\widetilde{\Xi}^\perp) \Xi^0 + \nu^1_{FP}(\widetilde{\Xi}^\perp) \Xi^1$.
    Since $\widetilde{\Xi} \in \widetilde{M}_{FP}$, we infer from \eqref{eq:Xi-decomposition} that $u \in Z^1(\Om)$ is actually a solution with data $\widetilde{\Xi}$.
    Thus, \cref{p:VPFP-MFP-necessary} implies that $\widetilde{\Xi} \in \cM_{FP}$.

   $\bullet$  Let $\Xi \in \cM_{FP}$.
    We introduce
    \begin{equation*}
        \widetilde{\Xi} := \Xi^\perp + \nu^0_{FP}(\Xi^\perp) \Xi^0 + \nu^1_{FP}(\Xi^\perp) \Xi^1,
    \end{equation*}
    which can be thought of as a good projection of $\Xi$ on $\widetilde{\cM}_{FP}$ since $\widetilde{\Xi}^\perp = \Xi^\perp$ and $\widetilde{\Xi} \in \widetilde{\cM}_{FP}$.
    Let $u, \widetilde{u} \in Z^1(\Om)$ denote the solutions constructed in \cref{p:VPFP-constr-Z1} from $\Xi$ and $\Xi^\perp$.
    For $k \in \{0,1\}$, we also introduce the coefficients $\mu^k := \nu^k_{FP}(\Xi^\perp) - \langle \Xi ; \Xi^k \rangle_{\cH_{FP}}$, which characterize how far $\Xi$ is from $\widetilde{\cM}_{FP}$.
    Then $w := \widetilde{u} - u$ belongs to $Z^1(\Om)$ with
    \begin{equation*}
        \| w \|_{Z^1} \lesssim \eta
    \end{equation*}
    and is a solution to 
    \begin{equation} \label{eq:VPFP-w-diff-tilde}
        \begin{cases}
            z \p_x w - \p_{zz} w = E[u] \p_z u - E[\tilde{u}] \p_z \tilde{u} + \mu^0 f^0 + \mu^1 f^1, \\
            w_{\rvert \Sigma_i} = \mu^0 \delta^0_i + \mu^1 \delta^1_i, \\
            w_{\rvert z = \pm 1} = 0.
        \end{cases}
    \end{equation}
    By \cref{thm:shear-Z1}, since $w\in Z^1(\Om)$, the following orthogonality conditions are satisfied for $j=0,1$:
    \begin{equation} \label{eq:VPFP-orth-necessary}
        \begin{split}
            0 & = \overline{\ell^j} (E[u] \p_z u - E[\tilde{u}] \p_z \tilde{u} + \mu^0 f^0 + \mu^1 f^1, \mu^0 \delta_0^0 + \mu^1 \delta_0^1,
            \mu^0 \delta_1^0 + \mu^1 \delta_1^1  )
            \\
            & = \overline{\ell^j} (E[u] \p_z u - E[\tilde{u}] \p_z \tilde{u}, 0,0)
            + \mu^j.
        \end{split}
    \end{equation}
    Moreover, since
    \begin{equation*}
        \| E[u] \p_z u - E[\tilde{u}] \p_z \tilde{u} \|_{H^1_x L^2_z} \lesssim \eta \| w \|_{Z^1}
    \end{equation*}
    we infer from \eqref{eq:VPFP-orth-necessary} that
    \begin{equation*}
        | \mu^j | \lesssim \eta \| w \|_{Z^1}.
    \end{equation*}
    Applying estimate \eqref{eq:u-qone}  to \eqref{eq:VPFP-w-diff-tilde}, we obtain
    \begin{equation*}
        \| w \|_{Z^1} \lesssim \eta \| w \|_{Z^1} + |\mu^0| + |\mu^1|\lesssim \eta \|w\|_{Z^1}.
    \end{equation*}
    For $\eta > 0$ small enough, this entails that $w = 0$ and $\mu^0 = \mu^1 = 0$, so that $\Xi \in \widetilde{\cM}_{FP}$.
\end{proof}

\begin{lem}
    The maps $U_{FP}$ and $\nu_{FP}$ of \cref{p:VPFP-constr-Z1} are Lipschitz-continuous.
\end{lem}

\begin{proof}
    Taking two triplets $\Xi, \Xi' \in \cH_{FP}$, one can consider the constructed sequences $u_n, u_n' \in Z^1(\Om)$ and $\nu_n, \nu'_n \in \R^2$ from \eqref{eq:VPFP-un-Z1}.
    Then, for $n \geq 1$, $w_n := u_n - u_n'$ is the solution to
    \begin{equation*}
        \begin{cases}
            \begin{aligned}
                z \p_x w_n  - \p_{zz} w_n = (f-f') & - E[w_{n-1}] \p_z u_{n-1} - E[u'_{n-1}] \p_z w_{n-1} \\
                & + (\nu^0_n-\nu'^0_n) f^0 + (\nu^1_n-\nu'^1_n) f^1,
            \end{aligned} \\
            (w_n)_{\rvert \Sigma_i} = (\delta_i - \delta_i') + (\nu^0_n-\nu'^0_n) \delta_i^0 + (\nu^1_n-\nu'^1_n) \delta_i^1, \\
            (w_n)_{\rvert z = \pm 1} = 0,
        \end{cases}
    \end{equation*}
    where, from \eqref{eq:VPFP-nu} and \cref{def:ell-shear},
    \begin{equation*}
        | \nu_n-\nu'_n| \lesssim \| \Xi - \Xi' \|_{\cH_{FP}} + \eta \| w_{n-1} \|_{Z^1}.
    \end{equation*}
    Thus, we obtain from \cref{thm:shear-Z1} that
    \begin{equation*}
        \| w_n \|_{Z^1} \lesssim \| \Xi - \Xi' \|_{\cH_{FP}} + \eta \| w_{n-1} \|_{Z^1}.
    \end{equation*}
    For $\eta$ small enough, we obtain at the limit that
    \begin{equation*}
        \| u - u' \|_{Z^1} + |\nu-\nu'| \lesssim \| \Xi - \Xi' \|_{\cH_{FP}},
    \end{equation*}
    which concludes the proof.
\end{proof}

\begin{rmk}
    \label{rmk:VPFP-tangent}
    Since we only proved Lipschitz regularity for the map $\nu_{FP}$, \eqref{eq:MFP} (and equivalently \eqref{eq:MFP'}) \emph{a priori} only defines a Lipschitz manifold.
	Hence, it is difficult to define tangent spaces to $\cM_{FP}$.
	Nevertheless, one can say that $\cH_{FP,\operatorname{sg}}^\perp$ is tangent to $\cM_{FP}$ at $0$ in the following weak senses:
    \begin{itemize}
        \item For $\Xi \in \cM_{FP}$, $\operatorname{d} ( \Xi, \cH_{FP,\operatorname{sg}}^\perp ) \lesssim \| \Xi \|^2_{\cH_{FP}}$.
        \item For every $\Xi^\bot \in \cH_{FP,\operatorname{sg}}^\perp$, for $\eps \in \R$ small enough, $\operatorname{d} (\eps \Xi^\bot, \cM_{FP}) \lesssim \eps^2$.
    \end{itemize}
    Both facts are straightforward consequences of the equivalent definitions \eqref{eq:MFP} and \eqref{eq:MFP'} and of the estimate 
    \begin{equation*}
        | \nu^j(\Xi) + \overline{\ell^j}(\Xi) | \lesssim \| \Xi \|_{\cH_{FP}}^2
    \end{equation*}
    which follows from \eqref{eq:VPFP-nu-limit} and \eqref{eq:VPFP-unu-leq-Xi}. 
\end{rmk}

\begin{rmk}
    It is likely that similar techniques can be used to prove that $\cM_{FP}$ has in fact more regularity (say $C^1$ for example) and characterize its tangent spaces in a neighborhood of the origin by computing the orthogonality conditions associated with the linearized problems around small enough solutions $u \in Z^1(\Om)$, but this is not our focus here. 
\end{rmk}

\subsection{A general formalization}
\label{sec:abstract}

The construction deployed in \cref{sec:VPFP-constr-Z1} can be seen as a particular case (see \cref{rmk:abs-VPFP}) of a more general approach to construct solutions to semilinear or quasilinear equations in the presence of orthogonality conditions, in a perturbative regime.
We give here a statement in an abstract framework which we will use in the following sections for the Burgers and Prandtl systems.

Our abstract result is related with general results for semilinear problems associated with Fredholm operators with negative index, such as the ones of \cite[Chapter 11, Section 4.2.3]{Volpert-Vol1}. 
However, the approach in this reference consists in modifying parameters in the nonlinearity to ensure the orthogonality conditions, while we focus on constructing a submanifold of data for which the nonlinear problem has a regular solution.

We intend to construct solutions to problems of the form $L u = N(\Xi, u)$, where $u \in \cZ$ (the space of solutions), $\Xi \in \cX$ (the space of data for the nonlinear problem), $N : \cX \times \cZ \to \cLin$ is the nonlinearity, with values in $\cLin$ (the space of source terms $\Theta \in \cLin$ for the linear problem $L u = \Theta$).

To avoid investigating the $C^1$ dependency of the solutions to our nonlinear systems on the data, we use a version of the implicit function theorem for functions which are not $C^1$ but only ``strongly Fréchet-differentiable at a point''.
We refer the reader to \cite[Chapter 25]{zbMATH01059780}.

\begin{defi}
    \label{def:strong-Frechet}
    Let $E, F$ be Banach spaces, $f : E \to F$ and $x^* \in E$.
    We say that $f$ is \emph{strongly Fréchet-differentiable} at $x^*$ when there exists a continuous linear map $L : E \to F$ such that
    \begin{equation*}
        \| f(x_1) - f(x_2) - L (x_1 - x_2) \|_F = \underset{x_1, x_2 \to x^*}{o} \left( \| x_1 - x_2 \|_{E} \right).
    \end{equation*}
\end{defi}

The following implicit function theorem is proved in \cite[Paragraph 25.13]{zbMATH01059780}.

\begin{lem}
    \label{lem:TFI-Strong-Frechet}
    Let $E_1, E_2, F$ be Banach spaces and $f : E_1 \times E_2 \to F$ such that $f(0,0) = 0$.
    Assume that $f$ is strongly Fréchet-differentiable at $(0,0)$ and that $\p_2 f(0,0) : E_2 \to F$ is a linear isomorphism.
    Then there exists a Lipschitz-continous map $g:E_1\to E_2$ defined in a  neighborhood of $0 \in E_1$ such that, for every $(x,y)$ in a neighborhood of $(0,0) \in E_1 \times E_2$, $f(x,y) = 0$ if and only if $y = g(x)$.
    Moreover, $g$ is strongly Fréchet-differentiable at $0$ and $Dg(0) = - (\p_2 f(0,0))^{-1} \p_1 f(0,0)$.
\end{lem}

Our general result is the following:

\begin{thm} \label{thm:abstract}
    Let $\cLin, \cX, \cZ$ be Banach spaces and $d \in \N$. 
    Let $\ell : \cLin \to \R^d$ and $L : \cZ \to \cLin$ be continuous linear maps.
    Let $N$ be a (nonlinear) map from $\cX \times \cZ$ to $\cLin$ such that $N(0,0) = 0$.
    Assume that
    \begin{enumerate}[i)] 
        \item \label{item:abs-1}
        for all $\Theta \in \cLin$, the equation $L u = \Theta$ has a unique solution $u \in \cZ$ if and only if $\Theta \in \ker \ell$, which moreover satisfies $\| u \|_{\cZ} \lesssim \| \Theta \|_{\cLin}$;
        \item \label{item:abs-2} 
        $N$ is strongly Fréchet-differentiable at $(0,0)$ and $\p_u N(0,0) = 0$, i.e.\ there exists a continuous linear map $\p_\Xi N(0,0) : \cX \to \cLin$ such that, as $\Xi,\Xi' \in \cX$ and $u, u' \in \cZ$ go to $0$,
        \begin{equation}
            \label{eq:N-magic}
            \| N(\Xi,u) - N(\Xi',u') - (\p_\Xi N(0,0)) (\Xi-\Xi') \|_{\cLin} = o \big( 
            \|\Xi-\Xi'\|_{\cX} + \| u - u' \|_{\cZ} \big);
        \end{equation}
        \item \label{item:abs-3} 
        $\ell_N := \ell \circ \p_\Xi N(0,0)$ is onto from $\cX$ to $\R^d$.
    \end{enumerate}
    Then there exists a local Lipschitz submanifold $\cM$ of $\cX$, modeled on $\ker \ell_N$ (of codimension $d$) and tangent to it at $0$, such that, for any $\Xi \in \cX$ small enough, the equation $L u = N(\Xi,u)$ has a solution $u \in \cZ$ if and only if $\Xi \in \cM$.
    Such a solution satisfies $\|u\|_{\cZ} \lesssim \|\Xi\|_{\cX}$ and is unique.
\end{thm}

\begin{proof}
    Using \cref{item:abs-3}), we fix $\Xi^1, \dotsc, \Xi^d \in \cX$ such that 
    $\ell^j_N(\Xi^k) = \mathbf{1}_{j=k}$, and we set
     $\Theta^k := \p_\Xi N(0,0) \Xi^k$.
    We could then mimic the iterative scheme of \cref{sec:VPFP-constr-Z1} by defining sequences $u_n \in \cZ$ and $\nu_n \in \R^d$ such that
\begin{equation*}
L u_{n+1}= N(\Xi + \sum_k \nu^k_{n+1} \Xi^k, u_n).
\end{equation*}
        Instead, we provide a shorter proof directly relying on the bundled result \cref{lem:TFI-Strong-Frechet}.

Let $\Xi \mapsto \Xi^\perp$ be the linear continuous projection from $\cX$ to $\ker \ell_N$ parallel to the space $\operatorname{span} \; (\Xi^1, \dotsc, \Xi^d) $, i.e.\ $\Xi^\perp= \Xi -\sum_{j=1}^d \ell^j_N(\Xi) \Xi^j$.
        Let $f : \ker \ell_N \times (\cZ \times \R^d) \to \cLin$ defined by
    \begin{equation*}
        f(\Xi^\perp, (u, a)) := L u - N (\Xi^\perp + a_1 \Xi^1 + \dotsb + a_d \Xi^d, u).
    \end{equation*}
    By \cref{item:abs-2}) and continuity of $L$ on $\cZ$, $f$ is strongly Fréchet-differentiable at $(0,0)$.
    Moreover, $\p_2 f(0,0) : (u, a) \mapsto L u - a_1 \Theta^1 - \dotsb - a_d \Theta^d$ is a linear isomorphism from $\cZ \times \R^d$ to $\cLin$ by \cref{item:abs-1}) and continuity of $\ell$ on $\cLin$.
    Indeed, given $h \in \cLin$, setting $a^h := - \ell(h)$ and $u^h \in \cZ$ the solution to $L u^h = h + a_1^h \Theta^1 + \dotsb + a_d^h \Theta^d$, one has $\p_2 f(0,0) (u^h, a^h) = h$ and $\|u_h\|_{\cZ} \lesssim \|h\|_{\cLin}$, $|a^h| \lesssim \|h\|_{\cLin}$.
    
    Hence, the implicit function theorem of \cref{lem:TFI-Strong-Frechet} yields the existence of Lipschitz-continuous functions $(U, \mu) : \ker \ell_N \to \cZ \times \R^d$ such that, for every $\Xi^\bot  \in \ker \ell_N$, $u \in \cZ$ and $a\in \R^d$ small enough, 
    \begin{equation*}
    L u = N (\Xi^\perp + a_1 \Xi^1 + \dotsb + a_d \Xi^d, u)
    \end{equation*}
    if and only if $a=\mu(\Xi^\perp)$ and  $u = U(\Xi^\perp)$. From there, we infer that for all $\Xi\in \cX$ and $u\in \cZ$ small enough,
        $Lu = N(\Xi,u)$ if and only if $\ell_N(\Xi) = \mu(\Xi^\perp)$ and $u = U(\Xi^\perp)$.
    Thus the conclusions of the theorem hold provided that we set
    \begin{equation} \label{eq:MN}
        \cM := \left\{ \Xi \in \cX ; \quad 
        \|\Xi\|_{\cX} \leq \eta \text{ and } \ell_N(\Xi) = \mu(\Xi^\perp) \right\}.
    \end{equation}
    Indeed, \eqref{eq:MN} corresponds to the graph characterization of a local Lipschitz submanifold of $\cX$ containing~$0$ and modeled on $\ker \ell_N$; therefore of codimension $d$ by \cref{item:abs-3}).
    Eventually, $\mu$ is strongly Fréchet-differentiable at $0$ and, since $Dg(0) = - (\p_2 f(0,0))^{-1} \p_1 f(0,0)$ with the notation of \cref{lem:TFI-Strong-Frechet}, we obtain that $D\mu(0) = -\ell \circ \p_\Xi N(0,0)= -\ell_N$ so $D\mu(0) = 0$ on $\ker \ell_N$, which justifies the claim that $\cM$ is tangent to $\ker \ell_N$ at $0$.
\end{proof}


\begin{rmk} \label{rmk:abs-VPFP}
    \cref{item:prop:VPFP-2} of \cref{prop:VPFP} can be recovered as a particular case of \cref{thm:abstract} with the following setting:
    \begin{itemize}
        \item $\cX = \cLin = \cH_{FP}$ defined in \eqref{eq:def-H_FP};
        
        \item the solution space
        \begin{equation*}
            \cZ := \Big\{ u \in Z^1(\Om) ; 
            \quad u_{\rvert z = \pm 1} = 0, 
            \quad (\p_z u(x_i,z)) / z \in \mathscr{H}^1_z(\Sigma_i), 
            \quad \p_z u(x_i,(-1)^i) = 0 \Big\},
        \end{equation*}
        with
        \begin{equation*}
            \| u \|_{\cZ} := \| u \|_{Z^1} + \sum_{i \in \{0,1\}} \| (\p_z u(x_i,z)) / z \|_{\mathscr{H}^1_z};
        \end{equation*}
        
        \item $L : \cZ \to \cX$ defined by $L u := (z \p_x u - \p_{zz} u, u_{\rvert \Sigma_0}, u_{\rvert \Sigma_1})$, for which one easily checks that the assumption \cref{item:abs-1}) of \cref{thm:abstract} is satisfied thanks to \cref{thm:shear-Z1};
        
        \item $N : \cX \times \cZ \to \cLin$ defined by $N(\Xi,u) := (f, \delta_0, \delta_1) - (E[u] \p_z u, 0, 0)$.
        In particular, one has $\p_\Xi N(0,0) = \operatorname{Id}$.
        To check that $N$ takes values in $\cLin = \cH_{FP} \subset \cLin_K$, we must check that $E[u](x_i) \p_z u (x_i,(-1)^i) = 0$, which follows from the fact that, for $u \in \cZ$, $\p_z u(x_i,(-1)^i) = 0$. 
        We now check that $N$ satisfies \cref{item:abs-2}) of \cref{thm:abstract}.
        
        First, for $u, u' \in Z^1(\Om)$, by \eqref{eq:bound-E},
        \begin{equation*}
            \begin{split}
            \| E[u] \p_z u - E[u'] \p_z u' \|_{H^1_x L^2_z}
            & \leq \| E[u-u'] \p_z u \|_{H^1_x L^2_z} + \| E[u'] \p_z (u-u') \|_{H^1_x L^2_z} \\
            & \lesssim \| u - u' \|_{L^2} \| \p_z u \|_{H^1_x L^2_z} + \|u'\|_{L^2} \| \p_z(u-u') \|_{H^1_x L^2_z} \\
            & \lesssim (\|u\|_{Z^1}+\|u'\|_{Z^1}) \|u-u'\|_{Z^1}.
            \end{split}
        \end{equation*}
        Second, one similarly checks that
        \begin{equation*}
            \| (E[u] \p_z u - E[u'] \p_z u')(x_i,z) / z \|_{\mathscr{H}^1_z(\Sigma_i)} \lesssim (\|u\|_{\cZ}+\|u'\|_{\cZ}) \| u - u' \|_{\cZ}.
        \end{equation*}
        Hence, we conclude that
        \begin{equation*}
            \begin{split}
            \| (E[u] \p_z u - E[u'] \p_z u', 0, 0) \|_{\cH_{FP}}
            & = 
            \| (E[u] \p_z u - E[u'] \p_z u', 0, 0) \|_{\cLin_K} \\
            & \lesssim (\|u\|_{\cZ}+\|u'\|_{\cZ}) \| u - u' \|_{\cZ}
            \end{split}
        \end{equation*}
        so that estimate \eqref{eq:N-magic} is satisfied.
        
        \item $d = 2$, $\ell := (\overline{\ell^0}, \overline{\ell^1})_{\rvert \cH_{FP}}$ defined in \cref{def:ell-shear}, continuous on $\cH_{FP}$ by \cref{lem:continuity-ell-easy} and $\cH_{FP} \hookrightarrow \cLin_K$, satisfying $\ell_N(\cX) = \ell(\cX) = \R^2$ by \cref{lem:VPFP-biorth} and $\p_\Xi N(0,0) = \operatorname{Id}$.
    \end{itemize}
\end{rmk}

\newpage
\section{A viscous Burgers equation}
\label{sec:Burgers}
    
In this section, we consider the following nonlinear parabolic forward-backward system, which can be envisioned as a kind of stationary Burgers equation with transverse viscosity: 
\begin{equation} \label{eq:yuuxuyy-bis}
    \begin{cases}
        u \p_x u - \p_{yy} u = f,\\
        u_{\rvert \Sigma_i} = \fu_{\rvert \Sigma_i} + \delta_i, \\
        u_{\rvert y = \pm 1} = \fu_{\rvert y = \pm 1}.
    \end{cases}
\end{equation}
As detailed in the introduction, the perturbation $(f,\delta_0,\delta_1)$ is small and we look for solutions $u$ which are close to the shear flow $\fu(x,y) := y$, which corresponds to $(f,\delta_0,\delta_1) = (0,0,0)$.
Thanks to the nonlinear change of variables described in the Introduction and detailed in \cref{sec:Burgers-change}, the local well-posedness of \eqref{eq:yuuxuyy-bis} can be proved using the formalism of \cref{sec:abstract} (see \cref{sec:Burgers-WP-new-var,sec:Burgers-main-proof}).

\subsection{A nonlinear change of variables}
\label{sec:Burgers-change}

As is classical for problems with free boundaries, we perform a change of variables which straightens the critical curve $\{ u = 0 \}$.
Heuristically, we swap the roles of the vertical coordinate~$y$ and the unknown~$u$, the latter becoming the vertical coordinate, and the former the unknown of the new PDE.
Keeping in mind that we are looking for perturbative solutions with $u$ close enough to $\fu$ (in particular $\|u_y - 1 \|_{L^\infty} \ll 1$), we  change the vertical coordinate $y$ into $z$, defined as
\begin{equation*}
    z(x,y) := u(x,y). 
\end{equation*}
The new unknown $Y(x,z)$ is defined by the implicit relation
\begin{equation} \label{eq:u-to-Y}
    u(x,Y(x,z)) = z.
\end{equation}
In particular, thanks to the boundary conditions $u_{\rvert y = \pm 1} = \fu_{\rvert y = \pm 1} = \pm 1$, one checks that the domain $(x,y) \in \Omega = [x_0,x_1] \times [-1,1]$ is indeed mapped to $(x,z) \in \Omega$, and one still has $Y_{\rvert z = \pm 1} = \pm 1$.
Similarly, if $\delta_i(0) = 0$ and $\delta_i((-1)^i) = 0$, the inflow boundary regions $\Sigma_i$ are also left invariant by this change of variable.

\begin{rmk}
    More rigorously, given $u$ defined on $\Omega$ and close enough to $\fu$ (for example in $H^1_x H^2_z$ topology), for each $x \in [x_0,x_1]$, the map $y \mapsto u(x,y)$ is a $C^1$ monotone increasing bijection from $[-1,1]$ to itself, and the implicit definition \eqref{eq:u-to-Y} is equivalent to setting
    \begin{equation} \label{eq:Y=inverse}
        Y(x,z) := (u(x,\cdot))^{-1}(z).
    \end{equation}
\end{rmk}

From \eqref{eq:u-to-Y}, we successively derive the relations
\begin{equation} \label{eq:CDV-u-Y}
    \begin{aligned}
        \p_y u (x, Y(x,z))&= \frac{1}{\p_z Y(x,z)},\\
        \p_x u (x,Y(x,z))&= - \p_x Y(x,z) \p_y u(x,Y(x,z))= - \frac{\p_x Y(x,z)}{\p_z Y(x,z)},\\
        \p_{yy} u(x,Y(x,z))&= - \frac{\p_{zz} Y(x,z)}{(\p_z Y(x,z))^3}.
    \end{aligned}
\end{equation}
These identities lead to the following PDE for $Y$:
\begin{equation} \label{eq:PDE-Burgers-Y}
    z \p_x Y - (\p_z Y)^{-2} \p_{zz} Y = - \p_z Y f(x,Y).
\end{equation}
Moreover, by \eqref{eq:Y=inverse}, denoting by $(\cdot + \delta_i(\cdot))^{-1}$ the functional inverse of the function $z\mapsto z + \delta_i(z)$ and letting 
\nomenclature[OAVGB]{$\Upsilon[\delta_i]$}{Lateral boundary data for the Burgers system after the change of variables}
\begin{equation}\label{eq:def-tilde-delta_i}
    \Upsilon[\delta_i] (z) := z - (\cdot + \delta_i(\cdot))^{-1}(z),
\end{equation}
we have $Y(x,z) = z - \Upsilon[\delta_i](z)$ for $(x,z) \in \Sigma_i$, where we used 
$\delta_i(0) = 0$ and $\delta_i((-1)^i) = 0$.

Therefore, we obtain the system
\begin{equation} \label{eq:Y-B-nonlinear}
    \begin{cases}
        z \p_x Y - (\p_z Y)^{-2} \p_{zz} Y = - \p_z Y f(x,Y), \\
        Y_{\rvert \Sigma_i} = z - \Upsilon[\delta_i](z), \\
        Y_{\rvert z = \pm 1} = \pm 1.
    \end{cases}
\end{equation}
Eventually, to make the perturbative nature of this system explicit, we write $Y(x,z) = z - \keta(x,z)$, which leads to the system
\begin{equation} \label{eq:Burgers-keta}
    \begin{cases}
        z \p_x \keta - \p_{zz} \keta = N_B(f,\keta), \\
        \keta_{\rvert \Sigma_i} = \Upsilon[\delta_i], \\
        \keta_{\rvert z = \pm 1} = 0
    \end{cases}
\end{equation}
where the nonlinearity is given by
\nomenclature[OLNB]{$N_B$}{Nonlinearity associated with the Burgers-type system}
\begin{equation} \label{eq:NB}
    N_B(f,\keta) := \frac{\p_z \keta (2 - \p_z \keta)}{(1-\p_z\keta)^2} \p_{zz} \keta + (1 - \p_z \keta) f(x,z-\keta).
\end{equation}
We prove the well-posedness of \eqref{eq:Burgers-keta} in \cref{sec:Burgers-WP-new-var} and use it to prove \cref{thm:burgers} in \cref{sec:Burgers-main-proof}.

\begin{rmk} \label{rk:semilinear}
    The initial PDE $u \p_x u - \p_{yy} u = f$ is quasilinear.
    After the change of variables described in this paragraph, we obtain system \eqref{eq:Y-B-nonlinear}, which is still a quasilinear one (since the viscosity in front of $\p_{zz} \keta$ depends on $\keta$).
    However, we know from \cref{sec:shear} that, for the linear problem $z \p_x u - \p_{zz} u = f$,  there is no loss of derivative in the vertical direction.
    This key point allows us to apprehend \eqref{eq:Y-B-nonlinear} under the form \eqref{eq:Burgers-keta}, treating this nonlinearity perturbatively as the first term of $N_B$ in \eqref{eq:NB}.
    The fact that there is no loss of vertical derivative explains why we will be able to prove in the following paragraph that the nonlinearity $N_B$ satisfies the mild estimates of \cref{thm:abstract}.
    This would not have been possible in the initial form $u \p_x u - \p_{yy} u = f$, since the linear theory involves a loss of $\frac 13$ derivative in the horizontal direction. 
\end{rmk}

\subsection{Well-posedness in the new variables}
\label{sec:Burgers-WP-new-var}

We now prove the following well-posedness result with $Z^1(\Om)$ regularity under two orthogonality conditions for system \eqref{eq:Burgers-keta}.
Let
\nomenclature[FHB]{$\cLin_B$}{Hilbert space of data triplets for the solvability of the linearized Burgers system}
\begin{align}
    \label{eq:def-HB}
    \cLin_B &:= \left\{ (f,\delta_0,\delta_1) \in \cLin_K ; \quad \delta_i''(z)/z \in \mathscr{H}^1_z(\Sigma_i), \quad \delta_i''((-1)^i) = 0 \right\}, \\
    \label{eq:def-XB}
    \cX_B &:= \left\{ (f,\delta_0,\delta_1) \in \cLin_B ; \quad f \in H^1_x H^2_z, \quad f_{\rvert \Sigma_i} = 0, \quad \delta_i \in H^5(\Sigma_i), \quad \delta_i(0) = \delta_i''(0) = 0 \right\},
\end{align}
where we recall that the spaces $\cLin_K$ and $ \mathscr{H}^1_z$ are defined in \eqref{eq:def-HK} and \eqref{eq:H-pagani-k} respectively, 
with the norms
\begin{align}
    \| (f,\delta_0,\delta_1) \|_{\cLin_B} & : =\| (f,\delta_0,\delta_1) \|_{\cLin_K} + \| \delta_0''(z)/z \|_{\mathscr{H}^1_z} + \| \delta_1''(z)/z \|_{\mathscr{H}^1_z}, \\
    \| (f,\delta_0,\delta_1) \|_{\cX_B} & : =\| (f,\delta_0,\delta_1) \|_{\cLin_B} + \| f \|_{H^1_x H^2_z} + \| \delta_0 \|_{H^5} + \|\delta_1 \|_{H^5}.
\end{align}
The restriction that $f_{\rvert \Sigma_i} = 0$ lightens the exposition but could be partially relaxed.
The space $\cX_B$ of \eqref{eq:def-XB} is the same as the one defined in \eqref{eq:def-XB-intro} in the introduction.

Our result on system \eqref{eq:Burgers-keta} is the following:
\begin{prop}
    \label{prop:burgers}
    There exists $\eta > 0$ and a local Lipschitz submanifold $\cM_B$ of $\cX_B$ included in the ball of radius $\eta$, modeled on $\cX_B \cap \ker (\overline{\ell^0}, \overline{\ell^1})$ (of codimension 2) and tangent to it at $0$ such that, for every $(f,\delta_0,\delta_1) \in \cX_B$ such that $\|(f,\delta_0,\delta_1))\|_{\cX_B} \leq \eta$, \eqref{eq:Burgers-keta} has a solution $\keta \in Z^1(\Om)$ if and only if $(f,\delta_0,\delta_1) \in \cM_B$.
    Such solutions are unique and satisfy $\|\keta\|_{Z^1} \lesssim \|(f,\delta_0,\delta_1)\|_{\cX_B}$.
\end{prop}

\begin{proof}
    Our strategy is to apply the same nonlinear argument as for our kinetic theory toy model (see \cref{sec:Fokker-Planck}).
    Before moving on to the formal proof using the abstract \cref{thm:abstract}, let us give an heuristic overview of the corresponding concrete nonlinear scheme.

    \bigskip
    
    \textbf{Heuristic overview of the nonlinear scheme.}
    We follow the scheme described in \cref{sec:VPFP-constr-Z1}.
    Let $(f,\delta_0,\delta_1) \in \cX_B$ with $\|(f,\delta_0,\delta_1)\|_{\cX_B} \leq \eta$ small enough to be chosen later on.
    We construct a sequence $\keta_n$ of $Z^1(\Om)$ functions using \cref{thm:shear-Z1},  accommodating for the two orthogonality conditions at each step.
    We take $\keta_0 := 0$ and, for $n \in \N$, given $\keta_n \in Z^1(\Omega)$, we define $\keta_{n+1} \in Z^1(\Om)$ as the solution to
    \begin{equation*}
        \begin{cases}
            z \p_x \keta_{n+1}- \p_{zz} \keta_{n+1} = N_B(f,\keta_n) + \nu_{n+1}^0 f^0 + \nu_{n+1}^1 f^1, \\
            (\keta_{n+1})_{\rvert \Sigma_i} = \Upsilon[\delta_i] + \nu_{n+1}^0 \delta^0_i + \nu_{n+1}^1 \delta^1_i, \\
            (\keta_{n+1})_{\rvert z=\pm 1}=0,
        \end{cases}
    \end{equation*}
    where the triplets $(f^k,\delta_0^k,\delta_1^k) \in \cX_B$ for $k \in \{0,1\}$ are such that $\overline{\ell^j}(f^k,\delta_0^k,\delta_1^k) = \mathbf{1}_{j=k}$ and are constructed as in \cref{lem:biorth} and
    \begin{equation*}
        \nu_{n+1}^j := - \overline{\ell^j} (N_B(f,\keta_n),\Upsilon[\delta_0],\Upsilon[\delta_1]).
    \end{equation*}
    This choice ensures that the two orthogonality conditions
    \begin{equation*}
        \begin{split}
        \overline{\ell^j} \big(
        & N_B(f,\keta_n) + \nu_{n+1}^0 f^0 + \nu_{n+1}^1 f^1, \\
        & \Upsilon[\delta_0]+ \nu_{n+1}^0 \delta^0_0 + \nu_{n+1}^1 \delta^1_0, \\
        & \Upsilon[\delta_1] + \nu_{n+1}^0 \delta^0_1 + \nu_{n+1}^1 \delta^1_1
        \big) = 0
        \end{split}
    \end{equation*}
    are satisfied.
    One checks that  \cref{thm:shear-Z1} can be applied, yielding $\keta_{n+1} \in Z^1(\Om)$. One can then prove that $(\keta_n)_{n\in \N}$ is  uniformly bounded by $C\eta$ and is a Cauchy sequence in $Z^1(\Om)$.

    \bigskip

    \textbf{Proof using our abstract toolbox.}
    More precisely, this result follows from \cref{thm:abstract}, applied with the following setting: $\cLin_B$ defined in \eqref{eq:def-HB} and $\cX_B$ defined in \eqref{eq:def-XB},
    \begin{itemize}
        \item the solution space 
        \nomenclature[FZb]{$\cZ_B$}{Solution space for the Burgers system}
        \begin{equation*}
            \begin{split}
                \cZ_B := \Big\{ u \in Z^1(\Om) ; \enskip  u_{\rvert z = \pm 1} = 0, 
                \enskip \p_{zz}u(x_i,z)/z \in \mathscr{H}^1_z(\Sigma_i),
                \enskip \p_{zz} u(x_i,(-1)^i) = 0 \Big\},
            \end{split}
        \end{equation*}
        with
        \begin{equation*}
            \| u \|_{\cZ_B} := \| u \|_{Z^1} + \sum_{i \in \{0,1\}} \| \p_{zz} u(x_i,z) / z \|_{\mathscr{H}^1_z(\Sigma_i)};
        \end{equation*}

\item $d = 2$, $\ell := (\overline{\ell^0}, \overline{\ell^1})_{\rvert \cLin_B}$ defined in \cref{def:ell-shear}, continuous on $\cLin_B$ by \cref{lem:continuity-ell-easy} and $\cLin_B \hookrightarrow \cLin_K$, satisfying $\ell(\cX_B) = \R^2$ by \cref{free:f} since $C^\infty_c(\Om) \times \{0\} \times \{0\} \subset \cX_B$;

        \item $L : \cZ_B \to \cLin_B$ defined by $L u := (z \p_x u - \p_{zz} u, u_{\rvert \Sigma_0}, u_{\rvert \Sigma_1})$, for which one easily checks that the assumption \cref{item:abs-1}) of \cref{thm:abstract} is satisfied thanks to \cref{thm:shear-Z1};
        
        \item $N : \cX_B \times \cZ_B \to \cLin_B$ defined by $N(\Xi,\keta) := (N_B(f,\keta), \Upsilon[\delta_0], \Upsilon[\delta_1])$.
        To prove that $N$ takes values in $\cLin_B \subset \cLin_K$, we must check that:
        \begin{enumerate}[a)]
            \item \label{it:Burgers-a} 
            $N_B(f,\keta) \in H^1_x L^2_z$: this follows from \cref{lem:fx-z-phi-LIP} and \cref{lem:Burgers-NB-Lip-H1L2-(uz*uzz)} below;
            \item $\Upsilon[\delta_i]\in \mathscr{H}^1_z(\Sigma_i)$ and $\p_{zz} \Upsilon[\delta_i]/z\in \mathscr{H}^1_z(\Sigma_i)$: this follows essentially from the chain rule, and is proved in \cref{lem:Upsilon-strong-diff} below;
            \item \label{it:Burgers-b}
            $(N_B(f,\keta)/z) \vert_{\Sigma_i} \in \mathscr{H}^1_z(\Sigma_i)$: this follows from \cref{coro:Burgers-NB-boundary} below;
            \item \label{it:Burgers-c} 
            $N_B(f,\keta)(x_i, (-1)^i) = 0$: this property follows from the fact that, for $\Xi \in \cX_B$, $f_{\rvert \Sigma_i} = 0$ and, for $\keta \in \cZ_B$, $\p_{zz} \keta(x_i,(-1)^i) = 0$;
            \item $\Upsilon[\delta_i]((-1)^i)= \p_{zz} \Upsilon[\delta_i]((-1)^i)=0$: this follows from the properties $\delta_i((-1)^i)= \delta_i'' ((-1)^i)$, see also the proof of \cref{lem:Upsilon-strong-diff} below.
        \end{enumerate}
        Eventually, we claim that $N$ is strongly Fréchet-differentiable at $(0,0)$ in the sense of \cref{def:strong-Frechet} with $\partial_u N(0,0) = 0$ and $\partial_\Xi N(0,0) = \operatorname{Id}$, which corresponds to the following estimate, as $\Xi, \Xi' \in \cX_B$ and $\keta, \keta' \in \cZ_B$ go to $0$, 
        \begin{equation}\label{est:N_B-Frechet}
            \begin{split}
            \| (N_B(f,\keta),\Upsilon[\delta_0],\Upsilon[\delta_1]) & - (N_B(f',\keta'),\Upsilon[\delta_0'],\Upsilon[\delta_1']) - (\Xi-\Xi') \|_{\cLin_B} \\ & = o \left( \|\Xi-\Xi'\|_{\cX_B} + \|\keta-\keta'\|_{\cZ_B} \right).
            \end{split}
        \end{equation}
        This follows from \cref{coro:Burgers-Lip-source} for the part involving $N_B$, and from \cref{coro:Burgers-NB-boundary} and \cref{coro:Burgers-Upsilon} for the estimate of the boundary terms.
        \qedhere
    \end{itemize}
\end{proof}
Note also that \cref{item:abs-3} from \cref{thm:abstract} follows from a variant of \cref{lem:biorth} for the space $\cLin_B$, stepping on \cref{free:f} and recalling that $\ell_N=\ell$.

The next subsections are dedicated to the proof of \cref{it:Burgers-a}) and \cref{it:Burgers-b}) and of estimate \eqref{est:N_B-Frechet} above. 
We will repeatedly use the following classical result:

\begin{lem}
    \label{lem:pointwise-H1H2-H1L2}
    The pointwise product is (bilinearly) continuous from $H^1_x H^1_z \times H^1_x L^2_z$ to $H^1_x L^2_z$.
\end{lem}

\subsubsection{Forcing term}

We first derive estimates for the main forcing term $(1-\p_z \keta) f(x,z-\keta(x,z))$ from \eqref{eq:NB}.
We start with an easy one-dimensional lemma:
\begin{lem}
    For $\phi, \psi \in (H^2 \cap H^1_0)(-1,1)$ small enough (so that the changes of variables $z \mapsto z - \phi(z)$ and $z \mapsto z - \psi(z)$ are well-defined on $[-1,1]$) and $f \in H^1(-1,1)$, one has
    \begin{align}
        \label{eq:f-o-phi-L2}
        \| f(z - \phi(z)) \|_{L^2} & \lesssim \| f \|_{L^2}, \\
        \label{eq:f-o-phi-psi-L2}
        \| f(z - \phi(z)) - f(z - \psi(z)) \|_{L^2} & \lesssim \| \p_z f \|_{L^2} \| \phi - \psi \|_{L^\infty}.
    \end{align}
\end{lem}

\begin{proof}
    First, \eqref{eq:f-o-phi-L2} is straight-forward since the Jacobian of the change of variables $z \mapsto z - \phi(z)$ is bounded from below and from above for $\phi$ small enough in $(H^2 \cap H^1_0)(-1,1)$.

    Second, for $z \in [-1,1]$, we write
    \begin{equation*}
        f(z-\phi(z)) - f(z-\psi(z)) = (\psi(z) - \phi(z)) \int_0^1 \p_z f(z - s \phi(z) - (1-s) \psi(z)) \dd s.
    \end{equation*}
    Hence, by Cauchy--Schwarz,
    \begin{equation*}
        \begin{split}
            \| f(z - \phi(z)) - f(z - \psi(z)) \|_{L^2}^2 
            \lesssim \| \phi - \psi \|_{L^\infty}^2 \int_0^1 \| \p_z f ( z - (s \phi(z) + (1-s) \psi(z))) \|_{L^2}^2 \dd s
        \end{split}
    \end{equation*}
    so that \eqref{eq:f-o-phi-psi-L2} follows from \eqref{eq:f-o-phi-L2} applied to $\p_z f$ and $s \phi + (1-s) \psi$.
\end{proof}
The next Lemma states that the main forcing term belongs to $H^1_xL^2_z$:
\begin{lem}
    \label{lem:fx-z-phi-LIP}
    For $\phi, \psi \in H^1_x (H^2_z \cap H^1_0)$ small enough and $f \in H^1_x H^2_z$, one has
    \begin{equation*}
    \begin{aligned}
    \|(1-\p_z \phi) f(x,z-\phi)\|_{H^1_x L^2_z} &\lesssim \| f \|_{H^1_x H^2_z},\\
        \| f(x,z-\phi(x,z)) - f(x,z-\psi(x,z)) \|_{H^1_x L^2_z} &\lesssim \| f \|_{H^1_x H^2_z} \| \phi - \psi \|_{H^1_x H^2_z}.
        \end{aligned}
    \end{equation*}
\end{lem}

\begin{proof}
First, we observe that $\p_z\phi, \p_z\psi \in L^\infty$, and $\p_z f\in L^\infty$. Since
\begin{equation*}
\p_x (f(x,z-\phi))=\p_x f(x,z-\phi) - \p_x \phi \p_z f(x,z-\phi)
\end{equation*}
we infer that $f(x,z-\phi)\in H^1_x L^2_z$. From there, we easily deduce the first estimate.

 We then turn towards the second estimate.   By the chain rule and the triangular inequality, one has
    \begin{equation*}
        \begin{split}
            \| f(x,z-\phi) - f(x,z-\psi) \|_{H^1_x L^2_z} & \lesssim \| f(x,z-\phi) - f(x,z-\psi) \|_{L^2_x L^2_z} \\
            & \quad + \| \p_x f(x,z-\phi) - \p_x f (x,z-\psi) \|_{L^2_x L^2_z} \\
            & \quad + \| (\p_z f(x,z-\phi) - \p_z f(x,z-\psi)) \phi_x  \|_{L^2_x L^2_z} \\
            & \quad + \| \p_z f(x,z-\psi) (\phi_x - \psi_x)  \|_{L^2_x L^2_z}. 
        \end{split}
    \end{equation*}
    By \eqref{eq:f-o-phi-psi-L2}, the first two terms are bounded by $\|\p_z f\|_{L^2} \|\phi-\psi\|_{L^\infty}$ and $\|\p_{xz} f \|_{L^2} \|\phi-\psi\|_{L^\infty}$.
    
    For the third term, using \eqref{eq:f-o-phi-psi-L2},
    \begin{equation*}
        \begin{split}
        \| (\p_z f(x,z-\phi) & - \p_z f(x,z-\psi)) \phi_x  \|_{L^2_x L^2_z} \\
        & \leq \| \p_z f(x,z-\phi) - \p_z f(x,z-\psi) \|_{L^\infty_x L^2_z} \| \phi_x \|_{L^2_x L^\infty_z} \\
        & \lesssim \| \p_{zz} f \|_{L^\infty_x L^2_z} \| \phi-\psi\|_{L^\infty}\| \phi_x \|_{L^2_x L^\infty_z}.
        \end{split}
    \end{equation*}
    For the fourth term, using \eqref{eq:f-o-phi-L2},
    \begin{equation*}
        \begin{split}
        \| \p_z f(x,z-\psi) (\phi_x - \psi_x)  \|_{L^2_x L^2_z} 
        & \leq \| \p_z f(x,z-\psi) \|_{L^\infty_x L^2_z} \| \phi_x - \psi_x \|_{L^2_x L^\infty_z} \\
        & \lesssim \| \p_z f \|_{L^\infty_x L^2_z} \| \phi_x - \psi_x \|_{L^2_x L^\infty_z}.
        \end{split}
    \end{equation*}
    Gathering these inequalities concludes the proof using usual Sobolev embeddings.
\end{proof}

We then prove the strong Fréchet differentiability at $(0,0)$ of the main forcing term:
\begin{lem}
    \label{lem:Burgers-NB-Lip-H1L2-(uz*f(u))}
    For $\phi_1, \phi_2 \in H^1_x (H^2_z \cap H^1_0)$ small enough and $f_1, f_2 \in H^1_x H^2_z$, 
    \begin{equation*}
        \begin{split}
           & \| (1-\p_z \phi_1) f_1(x,z-\phi_1) - (1-\p_z \phi_2)  f_2(x,z-\phi_2) - (f_1 - f_2) \|_{H^1_x L^2_z} \\ 
            & \lesssim\left(\|f_1\|_{H^1_x H^2_z} + \|f_2\|_{H^1_x H^2_z} + \|\phi_1\|_{H^1_x H^2_z}+ \|\phi_2\|_{H^1_x H^2_z}\right) \big( \| \phi_1 - \phi_2 \|_{H^1_x H^2_z} + \| f_1 - f_2 \|_{H^1_x H^2_z} \big).
        \end{split}
    \end{equation*}
\end{lem}

\begin{proof}
    First, we write
    \begin{equation*}
        \begin{split}
            f_1(x,z-\phi_1) -  f_2(x,z-\phi_2) - (f_1 - f_2) 
            & = (f_1 - f_2)(x,z-\phi_1) - (f_1 - f_2)(x,z-0) \\
            & \quad + f_2(x,z-\phi_1) - f_2(x,z-\phi_2).
        \end{split}
    \end{equation*}
    Applying \cref{lem:fx-z-phi-LIP} to both lines, we have
    \begin{equation*}
         \begin{split}
        \| f_1(x,z-\phi_1) & -  f_2(x,z-\phi_2) - (f_1 - f_2) \|_{H^1_x L^2_z} \\
        & \lesssim \| f_1 - f_2 \|_{H^1_x H^2_z} \| \phi_1 - 0 \|_{H^1_x H^2_z} + \| f_2 \|_{H^1_x H^2_z} \| \phi_1 - \phi_2 \|_{H^1_x H^2_z}, 
         \end{split}
    \end{equation*}
    which allows to conclude the proof thanks to \cref{lem:pointwise-H1H2-H1L2}.
\end{proof}

\subsubsection{Nonlinear viscous term}

We derive estimates for the main nonlinear viscous term $\frac{\p_z \keta (2 - \p_z \keta)}{(1- \p_z \keta)^2} \p_{zz} \keta$ from \eqref{eq:NB}.

\begin{lem} \label{lem:g(phi_z)}
    Let $g : \R \to \R$ be a $C^3$ function in a neighborhood of $0$ with $g(0) = 0$.
    Then the map $G : H^1_x H^2_z \to H^1_x H^1_z$ given by $G(\phi) := g(\phi_z)$ is well-defined and Lipschitz-continuous in a neighborhood of $0$, and satisfies $G(0) = 0$.
\end{lem}

\begin{proof}
    First, for $\phi, \psi \in H^1_x H^2_z$ small enough,
    \begin{equation*}
        \| g(\phi_z) - g(\psi_z) \|_{L^\infty} 
        \leq \| g' \|_{L^\infty} \| \phi_z - \psi_z \|_{L^\infty} 
        \lesssim \| \phi_z - \psi_z \|_{H^1_x H^1_z}. 
    \end{equation*}
    Second, for $\phi \in H^1_x H^2_z$ small enough,
    \begin{equation*}
        \p_{xz} \left( g(\phi_z) \right) = g'(\phi_z) \phi_{xzz} + g''(\phi_z) \phi_{xz}\phi_{zz}.
    \end{equation*}
    Hence, for $\phi,\psi \in H^1_x H^2_z$ small enough, using that $g \in C^3$ and decomposing the difference, one obtains
    \begin{equation*}
        \| \p_{xz} \left( g(\phi_z) - g(\psi_z) \right) \|_{L^2} \lesssim \| \phi-\psi\|_{H^1_x H^2_z},
    \end{equation*}
    which concludes the proof.
\end{proof}

The next lemma proves that the nonlinear viscous term belongs to $H^1_x L^2_z$:
\begin{lem}
    \label{lem:Burgers-NB-Lip-H1L2-(uz*uzz)}
    For $\phi, \psi \in H^1_x H^2_z$ small enough, one has
    \begin{equation*}
        \left\| \frac{\p_z \phi (2 - \p_z\phi)}{(1 - \p_z\phi)^2} \p_{zz} \phi - \frac{\p_z \psi (2 - \p_z\psi)}{(1 - \p_z\psi)^2} \p_{zz} \psi \right \|_{H^1_x L^2_z} \lesssim \left(\|\phi\|_{H^1_x H^2_z} + \|\psi\|_{H^1_x H^2_z}\right) \| \phi - \psi \|_{H^1_x H^2_z}.
    \end{equation*}
\end{lem}

\begin{proof}
    Since $\p_{zz}$ is Lipschitz-continuous from $H^1_x H^2_z$ to $H^1_x L^2_z$, by \cref{lem:pointwise-H1H2-H1L2}, the result follows from the Lipschitz continuity of $\phi \mapsto \p_z\phi (2 - \p_z\phi) (1 - \p_z\phi)^{-2}$ from $H^1_x H^2_z$ to $H^1_x H^1_z$, which is a consequence of \cref{lem:g(phi_z)} with $g(s) := s (2-s)(1-s)^{-2}$.
\end{proof}

Gathering \cref{lem:Burgers-NB-Lip-H1L2-(uz*f(u))} (for the part involving $f$) and \cref{lem:Burgers-NB-Lip-H1L2-(uz*uzz)} (for the quadratic part involving $\tY$ only), we obtain the Fréchet differentiability of $N_B$ at $(0,0)$:

\begin{coro}
    \label{coro:Burgers-Lip-source}
    For $\Xi=(f,\delta_0,\delta_1)$, $\Xi'=(f',\delta_0',\delta_1')\in \cX_B$ and $\tY, \tY'\in \cZ_B$ small enough,
    \begin{equation*}
        \| N_B(f, \tY) - N_B(f',\tY') - (f-f')\|_{H^1_x L^2_z} 
        = o\left( \| \Xi-\Xi'\|_{\cX_B} + \|\tY-\tY'\|_{\cZ_B}\right).
    \end{equation*}
\end{coro}

\subsubsection{Boundary contribution of the nonlinearity}

We now derive estimates concerning the $\mathscr{H}^1_z(\Sigma_i)$ contribution of $N_B(f,\keta)$.

\begin{lem}
    \label{lem:H1z=>zpsi-Linf}
    For $\psi \in \mathscr{H}^1_z(0,1)$, one has $z \psi \in L^\infty(0,1)$ with $\| z \psi \|_{L^\infty} \lesssim \| \psi \|_{\mathscr{H}^1_z}$.
\end{lem}

\begin{proof}
    Let $\psi \in \mathscr{H}^1_z(0,1)$.
    First, $\psi \in H^1(1/2,1)$ and one has $|\psi(1)| \lesssim \| \psi \|_{\mathscr{H}^1_z}$.
    Thus, for $z_0 \in (0,1)$,
    \begin{equation*}
        |\psi(z_0)| \leq |\psi(1)| + \int_{z_0}^1 |\psi_z| \leq |\psi(1)| + |\ln z_0|^{\frac 12} \| z^{1/2} \psi_z \|_{L^2(0,1)}
        \lesssim |\ln z_0|^{\frac 12} \| \psi \|_{\mathscr{H}^1_z},
    \end{equation*}
    which proves that $|\ln z|^{-\frac12} \psi \in L^\infty$, so that, in particular, $z \psi \in L^\infty$.
\end{proof}

\begin{lem}
    \label{lem:Burgers-NB-Lip-Bord}
    Let $g : \R \to \R$ be a $C^2$ function in a neighborhood of $0$ with $g(0) = 0$.
    Let $\cE := \{ \psi \in L^2(0,1) ; \enskip \psi_{zz} / z \in \mathscr{H}^1_z \}$ with the associated canonical norm.
    There exists $\eta > 0$ small enough such that, for $\phi, \psi \in \cE$ with $\|\phi\|_{\cE} \leq \eta$ and $\|\psi\|_{\cE} \leq \eta$,
    \begin{equation*}
        \left\| \frac{g(\phi_z) \phi_{zz}}{z} - \frac{g(\psi_z) \psi_{zz}}{z} \right\|_{\mathscr{H}^1_z}
        \lesssim (\|\phi\|_{\cE} + \|\psi\|_{\cE}) \| \phi-\psi \|_{\cE}.
    \end{equation*}
\end{lem}

\begin{proof}
    First, $\cE \hookrightarrow H^2 \hookrightarrow W^{1,\infty}$ and, thanks to \cref{lem:H1z=>zpsi-Linf}, $\cE \hookrightarrow W^{2,\infty}$.
    We write
    \begin{equation*}
        \frac{g(\phi_z) \phi_{zz}}{z} - \frac{g(\psi_z) \psi_{zz}}{z}
        =
        (g(\phi_z) - g(\psi_z)) \frac{\phi_{zz}}{z}
        + g(\psi_z) \frac{\phi_{zz}-\psi_{zz}}{z}.
    \end{equation*}
    For the first term, we have
    \begin{equation*}
        \left\| (g(\phi_z) - g(\psi_z)) \frac{\phi_{zz}}{z} \right\|_{\mathscr{H}^1_z}
        \lesssim \| \phi_z - \psi_z \|_{L^\infty} \| \phi_{zz} / z \|_{\mathscr{H}^1_z}
        + \| \phi_{zz}/z \|_{\mathscr{L}^2_z} \| \p_z (g(\phi_z)-g(\psi_z)) \|_{L^\infty}
    \end{equation*}
    where
    \begin{equation*}
        \begin{split}
        \| \p_z (g(\phi_z)-g(\psi_z)) \|_{L^\infty}
        & \leq \| (g'(\phi_z) - g'(\psi_z)) \phi_{zz} \|_{L^\infty}
        + \| g'(\phi_z) (\phi_{zz} - \psi_{zz}) \|_{L^\infty} \\
        & \lesssim \|  \phi_z - \psi_z \|_{L^\infty} \| \phi_{zz} \|_{L^\infty} +  \|\phi_{zz}-\psi_{zz}\|_{L^\infty}.
        \end{split}
    \end{equation*}
    For the second term,
    \begin{equation*}
        \left\| g(\psi_z) \frac{\phi_{zz}-\psi_{zz}}{z} \right\|_{\mathscr{H}^1_z}
        \lesssim \| g(\psi_z) \|_{L^\infty} \| (\phi_{zz} - \psi_{zz})/z \|_{\mathscr{H}^1_z} + \| (\phi_{zz}-\psi_{zz})/z \|_{\mathscr{L}^2_z} \| g'(\psi_z) \psi_{zz} \|_{L^\infty}.
    \end{equation*}
    Hence, the claimed estimate follows from the embedding $\cE \hookrightarrow W^{2,\infty}$.
\end{proof}

We infer that the boundary contribution of the nonlinearity is Fréchet-differentiable at $(0,0)$:
\begin{coro}
    \label{coro:Burgers-NB-boundary}
    For $\Xi=(f,\delta_0,\delta_1)$, $\Xi'=(f',\delta_0',\delta_1')\in \cX_B$ and $\tY, \tY'\in \cZ_B$ small enough, $z^{-1} N_B(f, \tY)\vert_{\Sigma_i}\in \mathscr H^1_z(\Sigma_i)$ and $z^{-1} N_B(f', \tY')\vert_{\Sigma_i}\in \mathscr H^1_z(\Sigma_i)$ and
    \begin{equation*}
        \left\| \frac{N_B(f, \tY) - N_B(f',\tY')}{z} \right\|_{\mathscr H^1_z(\Sigma_i)} = o\left( \|\tY-\tY'\|_{\cZ_B}\right). 
    \end{equation*}
\end{coro}

\begin{proof}
    Since $\Xi \in \cX_B$, $f\vert_{\Sigma_i} = 0$.
    Thus
    \begin{equation*}
        N_B(f, \tY)\vert_{\Sigma_i} = g(\p_z \tY \vert_{\Sigma_i}) \p_{zz} \tY\vert_{\Sigma_i} \quad \text{with} \quad g(a)=\frac{a(2-a)}{(1-a)^2}.
    \end{equation*}
    The result follows from \cref{lem:Burgers-NB-Lip-Bord}, noting that, for $\tY \in \cZ_B$, $\tY\vert_{\Sigma_i} \in \cE$ of \cref{lem:Burgers-NB-Lip-Bord}.
\end{proof}

\subsubsection{Contribution of the inversion of the boundary data}

We now move on to estimates concerning the Fréchet-differentiability of the map $\Upsilon$ of \eqref{eq:def-tilde-delta_i}.

\begin{lem}
    \label{lem:H1z-from-H2}
    For $\phi \in H^2(0,1)$ such that $\phi(0) = 0$,
    \begin{equation*}
        \| \phi(z) / z \|_{H^1} \lesssim \| \phi \|_{H^2}.
    \end{equation*}
\end{lem}

\begin{proof}
    Writing a second-order Taylor expansion, one has
    \begin{equation*}
        \phi(z) = z \phi'(0) + \int_0^z (z-s) \phi''(s) \dd s.
    \end{equation*}
    Thus
    \begin{equation*}
        \frac{\dd}{\dd z} \left( \frac{\phi(z)}{z} \right) = \frac{1}{z^2} \int_0^z s \phi''(s) \dd s,
    \end{equation*}
    from which the conclusion follows by the Hardy inequality of \cref{lem:Hardy-45}.
\end{proof}

\begin{lem}
    \label{lem:Upsilon-strong-diff}
    Consider the spaces
    \begin{align}
        \cE_1 & := \{ \delta \in \mathscr{H}^1_z(0,1) ; \enskip \delta''(z)/z \in \mathscr{H}^1_z(0,1), \enskip \delta(0) = \delta(1) = \delta''(0) = \delta''(1) =0 \}, \\
        \cE_2 & := \{ \delta \in H^5(0,1) ; \enskip \delta(0) = \delta( 1) = \delta''(0) = \delta''(1) =0 \}.
    \end{align}
    Then the map $\Upsilon[\delta](z) := z - (\cdot + \delta(\cdot))^{-1}(z)$ as in \eqref{eq:def-tilde-delta_i} is well-defined for $\delta$ small enough and strongly Fréchet-differentiable at $0$ from $\cE_2$ to $\cE_1$.
    More precisely, for $\delta, \eta \in \cE_2$ small enough,
    \begin{equation}
        \label{eq:Upsilon-strong-diff}
        \| \Upsilon[\delta] - \Upsilon[\eta] - (\delta - \eta) \|_{\cE_1} \lesssim\left( \|\delta\|_{\cE_2} + \|\eta\|_{\cE_2}\right) \| \delta - \eta \|_{\cE_2}.
    \end{equation}
\end{lem}

\begin{proof}
    \step{We first check that $\Upsilon$ is well-defined.}
    Since $\cE_2 \hookrightarrow W^{1,\infty}$, $\widetilde{\delta} := \Upsilon[\delta]$ is well-defined for $\delta \in \cE_2$ small enough, and the boundary conditions $\delta(0) = \delta(1) = 0$ of $\cE_2$ entail that $\widetilde{\delta}(0) = \widetilde{\delta}(1) = 0$.
    
    Moreover, one has
    \begin{equation} \label{eq:delta-to-tilde}
        \widetilde{\delta}(z) = \delta(z - \widetilde{\delta}(z)).
    \end{equation}
    From this relation, we derive that
    \begin{equation} 
        \label{eq:tilde-delta''}
        \widetilde{\delta}'(z) = \frac{\delta'}{1+\delta'}(z - \widetilde{\delta}(z)) 
        \quad \text{and} \quad 
        \widetilde{\delta}''(z) = \frac{\delta''}{(1+\delta')^3}(z - \widetilde{\delta}(z))
    \end{equation}
    which ensures that $\widetilde{\delta}''(0) = \widetilde{\delta}''(1) = 0$ since $\delta''(0) = \delta''(1) = 0$.

    \step{We prove the strong Fréchet-differentiability at $0$.}
    To control the $\cE_1$ norm, it suffices to control the $L^2$ norm and the $\mathscr{H}^1_z$ norm of the quotient $\p_{zz} (\cdot) / z$.
    For $\delta, \eta \in \cE_2$ by \eqref{eq:delta-to-tilde},
    \begin{equation} \label{eq:E1-E2-decomp-L2}
        (\widetilde{\delta} - \widetilde{\eta})(z) = (\delta - \eta) (z - \widetilde{\delta}) + (\eta (z - \widetilde{\delta}) - \eta(z - \widetilde{\eta})).
    \end{equation}
    Hence
    \begin{equation*}
        \| \widetilde{\delta} - \widetilde{\eta} \|_{L^\infty} \leq \| \delta - \eta \|_{L^\infty} + \| \p_z \eta \|_{L^\infty} \| \widetilde{\delta} - \widetilde{\eta} \|_{L^\infty}.
    \end{equation*}
    In particular, for $\eta$ small enough in $\cE_2$,
    \begin{equation*}
        \| \widetilde{\delta} - \widetilde{\eta} \|_{L^\infty}
        \leq 2 \| \delta - \eta \|_{L^\infty}.
    \end{equation*}
    Thus, applying estimate \eqref{eq:f-o-phi-psi-L2} to \eqref{eq:E1-E2-decomp-L2}, we obtain
    \begin{equation}
        \label{eq:E1-E2-estim-L2}
        \begin{split}
            \| (\widetilde{\delta} - \widetilde{\eta}) - (\delta - \eta) \|_{L^2} 
            & \leq \| (\delta - \eta) (\cdot - \widetilde{\delta}) - (\delta - \eta)(\cdot) \|_{L^2} + \| \eta (\cdot - \widetilde{\delta}) - \eta(\cdot - \widetilde{\eta}) \|_{L^2} \\
            & \lesssim \| \p_z (\delta - \eta) \|_{L^2} \| \widetilde{\delta} \|_{L^\infty} + \| \p_z \eta \|_{L^2} \| \widetilde{\delta} - \widetilde{\eta} \|_{L^\infty} \\
            & \lesssim \| \p_z (\delta - \eta) \|_{L^2} \| \delta \|_{L^\infty} + \| \p_z \eta \|_{L^2} \| \delta - \eta \|_{L^\infty} \\
            & \lesssim \left( \| \delta\|_{H^1} +  \| \eta\|_{H^1}\right) \| \delta - \eta \|_{H^1}.
        \end{split}
    \end{equation}
    We now move to the estimate of the $\mathscr{H}^1_z$ norm of the quotient $\p_{zz}(\cdot)/z$.
    By \cref{lem:H1z-from-H2} (which even yields an $H^1$ estimate, not only $\mathscr{H}^1_z$), since all our functions have null second derivative at~$0$, it suffices to obtain an $H^4$ estimate.
    Differentiating~\eqref{eq:tilde-delta''} twice, we obtain
    \begin{equation*}
        \p_z^4\widetilde{\delta}(z) = \frac{\p_z^4 \delta}{(1+\p_z \delta)^5} (z - \widetilde{\delta}(z)) + \text{lower order terms.}
    \end{equation*}
    Decomposing the difference in a similar manner as in \eqref{eq:E1-E2-estim-L2} and applying \eqref{eq:f-o-phi-psi-L2}, one can prove
    \begin{equation*}
        \| (\widetilde{\delta} - \widetilde{\eta}) - (\delta - \eta) \|_{H^4} =  \left( \| \delta  \|_{H^5}  + \|  \eta \|_{H^5}\right)\| \delta - \eta \|_{H^5}. 
    \end{equation*}
    Together with \cref{lem:H1z-from-H2}, this concludes the proof of \eqref{eq:Upsilon-strong-diff}.
\end{proof}

\begin{coro}
    \label{coro:Burgers-Upsilon}
    For $\Xi=(f,\delta_0,\delta_1)$, $\Xi'=(f',\delta_0',\delta_1')\in \cX_B$ small enough,
    \begin{equation*}
            \| (0,\Upsilon[\delta_0],\Upsilon[\delta_1]) - (0,\Upsilon[\delta_0'],\Upsilon[\delta_1']) - (0,\delta_0-\delta_0',\delta_1-\delta_1')\|_{\cH_B}
            = o\left( \| \Xi-\Xi'\|_{\cX_B} \right).
    \end{equation*}
\end{coro}

\begin{proof}
    Recalling that, for $\Xi = (f,\delta_0,\delta_1) \in \cX_B$, $f\vert_{\Sigma_i} = 0$, this is a direct consequence of \cref{lem:Upsilon-strong-diff} and the definitions \eqref{eq:def-HB} and \eqref{eq:def-XB} of $\cH_B$ and $\cX_B$.
\end{proof}

\subsection{Reverse change of variables}
\label{sec:Burgers-main-proof}

\begin{proof}[Proofs of \cref{thm:burgers} and \cref{p:necessity-ortho}]
    It only remains to prove that the change of variables of \cref{sec:Burgers-change} is justified in both directions.

    First, given $(f,\delta_0,\delta_1) \in \cM_B$, let $\tY \in Z^1$ be the solution to \eqref{eq:Burgers-keta} given by \cref{prop:burgers} and let $Y(x,z) := z - \tY(x,z)$ the associated solution to \eqref{eq:PDE-Burgers-Y}.
    By \cref{prop:burgers}, $\|Y\|_{Z^1} \lesssim 1+\|(f,\delta_0,\delta_1)\|_{\cX_B}$.
    By \cref{lem:Zsigma-Qsigma}, $\|Y\|_{\qone} \lesssim 1+ \|(f,\delta_0,\delta_1)\|_{\cX_B}$.
    Since $Y$ is a solution to \eqref{eq:PDE-Burgers-Y}, we have
    \begin{equation*}
        \p_{zz} Y = (\p_z Y)^2 (z \p_x Y) - (\p_z Y) f(x,Y(x,z)).
    \end{equation*}
    We check that the right-hand side is $L^2_x H^1_z$, from which we deduce that $\p_z^3 Y \in L^2$. Repeating this argument, we find that the right-hand side of the above equation is in fact
     $L^2_x H^2_z$ and that
    \begin{equation*}
        \| \p_z^4 Y \|_{L^2} \lesssim \| Y \|_{Z^1} + \| f \|_{L^2_x H^2_y}.
    \end{equation*}
    Thus, $Y \in \qone \cap L^2_x H^4_z$ and
    \begin{equation*}
        \| Y(x,z) - z \|_{\qone} + \| Y(x,z) - z \|_{L^2_x H^4_z} \lesssim \| (f,\delta_0,\delta_1) \|_{\cX_B}.
    \end{equation*}
    By \cref{lem:inverse-qone}, \eqref{eq:u-to-Y} defines a $u \in \qone \cap L^2_x H^4_y$ such that
    \begin{equation*}
        \| u(x,y) - y \|_{\qone} + \| u(x,y) - y \|_{L^2_x H^4_y} \lesssim \| (f,\delta_0,\delta_1) \|_{\cX_B}.
    \end{equation*}
    In particular, since both $\p_y u$ and $\p_z Y$ are continuous functions on $\Omega$ with $\| \p_y u - 1 \|_{L^\infty} \ll 1$ and $\| \p_z Y - 1 \|_{L^\infty} \ll 1$, the computations of \cref{sec:Burgers-change} hold.
    Thus, we have constructed a $u \in \qone$ solution to \eqref{eq:yuuxuyy-bis}.
    This proves the existence claim of \cref{thm:burgers}.

    Reciprocally, to prove the claim of \cref{thm:burgers} concerning the uniqueness of the solution to~\eqref{eq:yuuxuyy-bis} and the one of \cref{p:necessity-ortho} concerning the necessity of the nonlinear orthogonality conditions $(f,\delta_0,\delta_1) \in \cM_B$, we must perform the reasoning in the other direction.
    Let $(f,\delta_0,\delta_1) \in \cX_B$ small enough, and let $u \in \qone$ be a solution to \eqref{eq:yuuxuyy} such that $\|u\|_{\qone}\ll 1$.
    Writing the PDE as
    \begin{equation*}
        \p_y^2 u = u \p_x u - f
    \end{equation*}
    we obtain that $u \in \qone \cap L^2_x H^4_y$.
    By \cref{lem:inverse-qone}, \eqref{eq:u-to-Y} defines a function $Y \in \qone \cap L^2_x H^4_z$ such that
    \begin{equation*}
        \| Y(x,z)-z \|_{\qone} + \| Y(x,z) -z \|_{L^2_x H^4_z} \lesssim \| u(x,y) - y \|_{\qone} + \| u(x,y) - y  \|_{L^2_x H^4_z} \ll 1.
    \end{equation*}
    In particular, since both $\p_y u$ and $\p_z Y$ are continuous functions on $\Omega$ with $\| \p_y u - 1 \|_{L^\infty} \ll 1$ and $\| \p_z Y - 1 \|_{L^\infty} \ll 1$, the computations of \cref{sec:Burgers-change} hold.
    Thus, $Y$ is a solution to \eqref{eq:PDE-Burgers-Y}.
    Since $Y \in \qone \cap L^2_x H^4_z$, we have $Y \in H^1_x H^2_z$.
    From the equation \eqref{eq:PDE-Burgers-Y}, we recover that $z \p_x (\p_x Y) \in L^2$.
    Thus $Y \in Z^1(\Om)$.
    Hence, the conclusions of \cref{prop:burgers} apply: $Y$ is unique and $(f,\delta_0,\delta_1) \in \cM_B$.
\end{proof}

\newpage
\section{The Prandtl system in the recirculation zone}
\label{sec:Prandtl}

Let us now  continue our analysis of nonlinear parabolic forward-backward systems by considering the Prandtl equation in the vicinity of a recirculating flow $(\ufs,\vfs)$, revisiting the results of Iyer and Masmoudi from \cite{IM2023,IM2022}.
Throughout this section, the index $t$ stands for `top' and the index~$b$ for `bottom'. 
We refer to \cref{sec:intro-main-Prandtl} of the introduction for the assumptions on $(\ufs,\vfs)$.

We consider the system
\begin{equation} 
\label{prandtl}
\begin{cases}
	u u_x + v u_y - u_{yy}= -\p_x p + f & \text{ in }\Om_P,\\
	u_x + v_y =0 & \text{ in }\Om_P,\\
\end{cases}
\end{equation}
where the pressure gradient $\p_x p$ is the one associated with $(\ufs, \vfs)$, and where we recall that the domain $\Omega_P$ is defined by
\begin{equation}
    \label{eq:def-OmP}
    \Omega_P := \{(x,y)\in (x_0, x_1)\times \R_+;\enskip \gamma_b (x) < y<\gamma_t(x)\}.
\end{equation}
This system is	endowed with the boundary conditions \eqref{CL-Prandtl-bottom}-\eqref{CL-Prandtl-top}-\eqref{CL-Prandtl-lateral}, which we now recall for the reader's convenience:
\begin{equation}\label{BC-Prandtl-compact}
\begin{cases} 
u\vert_{y=\gamma_b}= z_b,\quad \p_y u\vert_{y=\gamma_b}=\p_y \ufs\vert_{y=\overline{\gamma_b}} + \delta_b,\quad v\vert_{y=\gamma_b}= \vfs\vert_{y=\overline{\gamma_b}} + v_b & \text{ (bottom BC),}\\
u\vert_{y=\gamma_t}= z_t,\quad \p_y u\vert_{y=\gamma_t}=\p_y \ufs\vert_{y=\overline{\gamma_t}} + \delta_t, & \text{ (top BC),}\\
u\vert_{\Sigma_i^P}= \ufs\vert_{\Sigma_i^P} + \delta_i & \text{ (lateral BC).}
\end{cases}
\end{equation}
We recall that the lines $\{y=\gamma_j(x)\}$ for $j\in \{t,b\}$, which are level sets of the function $u$, are free boundaries which are expected to lie in the vicinity of the level sets $\{y=\overline{\gamma_j}(x)\}$ of the function~$\ufs$. 
We refer to the introduction  for further comments on these boundary conditions.
See \cref{fig:omega-prandtl} for a sketch of the geometry of the domain.

\begin{figure}
    \centering
    \includegraphics{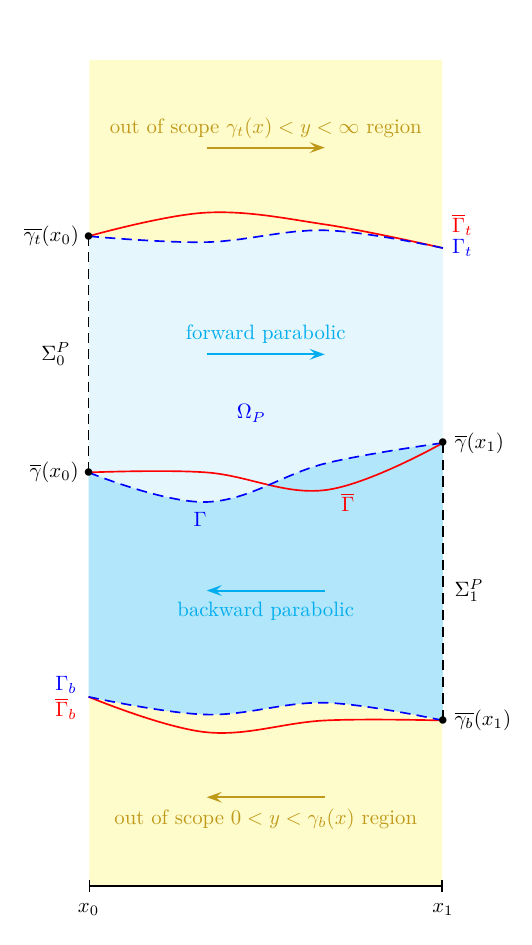}
    \caption{Fluid domain $\Omega_P$ defined in \eqref{eq:def-OmP} with free top and bottom boundaries $\Gamma_t$ and $\Gamma_b$, and fixed inflow boundaries $\Sigma^P_0$ and $\Sigma^P_1$.}
    \label{fig:omega-prandtl}
\end{figure}

The source term $f$ in \eqref{prandtl} is a small regular perturbation of the pressure term. From the physical point of view, it is relevant to consider perturbations which depend only on $x$, since the right-hand side in the Prandtl system is the trace of the pressure gradient of some outer Euler flow on the boundary. 
However, the analysis is essentially unchanged if we allow $f$ to depend on the vertical variable $y$, and therefore in the following $f$ will be a smooth function depending on both $x$ and $y$, for the sake of generality.

Our analysis in this section follows the one from \cref{sec:Burgers}.
We first perform in \cref{sec:change-variables-Prandtl} a nonlinear change of variables in order to straighten the free boundary $\{(x,y); \ u(x,y)=0\}$. 
The whole analysis then takes place in these new variables. 
One remarkable point lies in the fact that the linear problem associated with the Prandtl system is similar to, but slightly different from the one for the Burgers equation.
In fact, the linear problem associated with the \emph{vorticity} studied in \cref{sec:vorticity} has the same structure as~\eqref{eq:shear}.
Retrieving the velocity from the vorticity in \cref{sec:Prandtl-reconstruct} gives rise to an additional orthogonality condition. 
Moreover, since the vorticity plays the same role as the function $u$ from \cref{sec:Burgers}, it turns out that the Prandtl system is actually \emph{more regular} than the Burgers equation~\eqref{eq:eq0-uux}: indeed, there is a gain of one vertical derivative (corresponding to  a vertical integration of the velocity) between Burgers and Prandtl. 
This will allow us to construct solutions with a minimal requirement of regularity, and just one orthogonality condition.
We construct solutions to the nonlinear problem in the new variables in \cref{sec:Prandtl-abstract}, and  conclude the proof of \cref{thm:prandtl} in \cref{sec:WP-Prandtl}.

We recall that we focus here on the behavior of the system in the vicinity of the curve $\{u=0\}$.
When studying the system in the whole infinite strip $(x_0,x_1) \times \R_+$, special care must be taken to ``glue together'' the different zones. As explained in \cite{IM2022},  information flows from bottom to top.
The analysis of the system in the vicinity of the lower boundary and for large values of $y$ requires specific tools, which go beyond the scope of the present paper.
We refer the interested reader to \cite{IM2023,IM2022} for the study of the Prandtl system in the whole domain, and for a description of the difficulties associated with the interplay between the different zones.
We also present in \cref{sec:strategy-Prandtl-bande-infinie} a potential strategy to construct a solution to the Prandtl system in the whole infinite strip, stepping on the analysis of the present paper. 
In particular, we explain why the analysis of the system in an infinite vertical domain may call for an assumption on the horizontal size of the domain $x_1-x_0$: in \cite{IM2022}, the well-posedness of the system holds when $|x_1-x_0|$ is either small, or outside a countable set (corresponding to the zeros of an analytic function). No such assumption is required when the Prandtl system is studied in the recirculation zone only, see \cref{thm:prandtl} or \cref{prop-Prandtl-Y} below. 
Let us also recall that our purpose here is merely to present, in a unified framework, different forward-backward problems. 
Therefore we will put an emphasis on the specific features associated with the Prandtl system in the recirculation zone $\Omega_P$, and on the similarities and differences with the Burgers type system \eqref{eq:eq0-uux} studied in \cref{sec:Burgers}.

\subsection{Nonlinear change of variables}
\label{sec:change-variables-Prandtl}

At this stage, we assume that a smooth solution to \eqref{prandtl} exists in order to write the equation in a form that is more amenable to mathematical analysis.
We will come back on the justification of the computations below in \cref{sec:WP-Prandtl}.

As in \cref{sec:Burgers-change}, we change variables by setting $(x,z)=(x, u(x,y))$, where $u$ is the unknown tangential velocity. 
This maps the unknown domain $\Om_P = \{ \gamma_b(x) < y < \gamma_t(x) \}$ depending on the solution $u$ (since the lines $\gamma_b$ and $\gamma_t$ are defined by $u(x,\gamma_j(x)) = z_j$ for $j \in \{b,t\}$) to the fixed rectangular domain $(x_0,x_1) \times (z_b,z_t)$. 

We denote by $(x, Y(x,z))$ the diffeomorphism such that $u(x, Y(x,z))=z$. 
As a consequence, we have the same relations \eqref{eq:CDV-u-Y} between the derivatives of $u$ and $Y$ as for the Burgers case.
The top and bottom boundary conditions become $Y(x,z_j) = \gamma_j(x)$ for $j \in \{b,t\}$.

Furthermore, integrating the divergence-free condition and using \eqref{CL-Prandtl-bottom},
\begin{equation*}
    \begin{split}
        v(x,Y(x,z))&= v\vert_{\Gamma_b}  -\int_{\gamma_b(x)}^{Y(x,z)} \p_x u(x,y')\dd y'\\
    	&=\vfs\vert_{\overline{\Gamma_b}}+ v_b+ \int_{z_b}^{z} \frac{\p_x Y(x,z')}{\p_z Y(x,z')} \p_z Y(x,z')\dd z'\\
    	&= \vfs\vert_{\overline{\Gamma_b}} +  v_b+ \int_{z_b}^{z} \p_x Y(x,z')\dd z'.
    \end{split}
\end{equation*}
Replacing this expression and \eqref{eq:CDV-u-Y} into \eqref{prandtl} and evaluating the equation at $y=Y(x,z)$, we find that
\begin{equation*} 
-\frac{1}{\p_z Y} \left[ z \p_x Y - \int_{z_b}^z \p_x Y - \vfs\vert_{y=\overline{\gamma_b}} -  v_b \right] + \frac{1}{(\p_z Y)^3} \p_z^2 Y= - \p_x p + f(x,Y(x,z)).
\end{equation*}
Let us now denote by $\Yfs$ the function such that $\ufs(x,\Yfs(x,z))=z$. 
\nomenclature[OLYFS]{$\Yfs$}{Inverse function of the reference flow $\ufs$}
Following the same computations as above, this function satisfies
\begin{equation*}
-\frac{1}{\p_z \Yfs} \left[ z \p_x \Yfs - \int_{z_b}^z \p_x \Yfs - \vfs\vert_{\overline{\Gamma_b}} \right] + \frac{1}{(\p_z \Yfs)^3} \p_z^2 \Yfs= - \p_x p.
\end{equation*}
Let $\tY:= \Yfs - Y$. 
Then
\begin{equation*}
    \begin{split}
        \frac{1}{(\p_z \Yfs)^2} \p_z^2 \Yfs  -\frac{1}{(\p_z Y)^2} \p_z^2 Y 
        & = \p_z \left(\frac{1}{\p_z Y} - \frac{1}{\p_z \Yfs} \right)\\
        & = \p_z \left( \frac{\p_z \tY}{(\p_z \Yfs)^2}\right) 
        + \p_z \left( \frac{(\p_z \tY)^2}{(\p_z \Yfs)^2 \p_z Y} \right).
    \end{split}
\end{equation*}
We obtain eventually the following very simple  equation
\begin{equation} 
    \label{eq:Prandtl-changement-var}
    z \p_x \tY - \int_{z_b}^z \p_x \tY(x,z')\dd z'  - \p_x p \p_z \tY -  \p_z \left( \frac{\p_z \tY}{(\p_z \Yfs)^2}\right) = g(x,z)
\end{equation}
where
\begin{equation}
    \label{eq:Prandtl-g}
    g(x,z) := f(x, Y(x,z)) \p_z (\Yfs- \tY) - v_b(x) +  \p_z \left( \frac{(\p_z \tY)^2}{(\p_z \Yfs)^2 \p_z Y} \right).
\end{equation}
The top and bottom boundary conditions \eqref{CL-Prandtl-top} and \eqref{CL-Prandtl-bottom} become, for $j\in \{t,b\}$,
\nomenclature[OAVPb]{$\BCP^b[\delta_b]$}{Boundary data at the bottom for the Prandtl system in the new variables}
\nomenclature[OAVPt]{$\BCP^t[\delta_t]$}{Boundary data at the top  for the Prandtl system in the new variables}
\begin{equation}
    \label{cond:tY-top-bottom}
    \begin{split}
        \p_z \tY(x,z_j)
        & = \p_z \Yfs(x,z_j) - \p_z Y(x,z_j) \\
        & = \frac{1}{\p_y \ufs (x,\overline{\gamma_j}(x)) } - \frac{1}{\p_y \ufs (x,\overline{\gamma_j}(x)) + \delta_j(x)} 
        =: \BCP^j[\delta_j] (x).
    \end{split}
\end{equation}
The unknown function $\gamma_j$ can be retrieved from $\tY$ by 
\begin{equation*}
    \gamma_j(x) = Y(x,z_j) = \Yfs(x,z_j) - \tY(x,z_j) = \overline{\gamma_j}(x) - \tY(x,z_j).
\end{equation*}
We still denote by $\Sigma_0$ and $\Sigma_1$ the lateral boundaries, i.e.\ $\Sigma_0=\{x_0\}\times (0, z_t)$, $\Sigma_1=\{x_1\}\times (z_b,0)$.
In order to simplify the definition of the functional spaces for the lateral boundary data, we assume that $\delta_0(\gfs(x_0))=\delta_1(\gfs(x_1))=\delta_0(\overline{\gamma_t}(x_0))=\delta_1(\overline{\gamma_b}(x_1))=0$.
The lateral boundary conditions \eqref{CL-Prandtl-lateral} are then given by the implicit equation
\begin{equation*}
z= \ufs(x_i, Y(x_i,z)) + \delta_i (Y(x_i,z))\quad \text{on } \Sigma_i,
\end{equation*}
which becomes, after noticing that $\ufs (x_i, \cdot ) + \delta_i$ is strictly increasing on $\Sigma_i^P$ and therefore bijective from $\Sigma_i^P$ to $\Sigma_i$, 
\nomenclature[OAViP]{$\BCP^i[\delta_i]$}{Boundary data on $\Sigma_i$ for the Prandtl system in the new variables}
\begin{equation}\label{cond:tY-Sigma-i}
    \tY(x_i,z) = \Yfs(z) - \left(\ufs(x_i, \cdot) + \delta_i\right)^{-1} (z) =: \BCP^i[\delta_i]\quad \text{on } \Sigma_i.
\end{equation}
For further purposes, we note that the function $\BCP^j[\delta_j]$ (resp.\ $\BCP^i[\delta_i]$) has the same regularity and size as $\delta_j$ (resp.\ $\delta_i$). 

\begin{rmk}
	When $\ufs(x,y)=y$ (linear shear flow), \eqref{eq:Prandtl-changement-var} simply becomes, at main order
	\begin{equation*}
	z \p_x \tY + \widetilde{V} - \p_z^2 \tY = g,
	\end{equation*}
	where $\widetilde{V}=-\int_{z_b}^z \p_x \tY$. 
	Differentiating this equation with respect to $z$, and setting $W := \p_z \tY$ ($W$ is the vorticity in our new variables) we find
	\begin{equation*}
	z \p_x W - \p_{z}^2 W= \p_z g.
	\end{equation*}
	Therefore, when we consider the Prandtl equation in the vicinity of the linear shear flow, the equation for the vorticity in the new variables is \eqref{eq:shear}.
	We retrieve here the following fact, which was already identified by Iyer and Masmoudi in \cite{IM2022}: the Prandtl system in vorticity form is very close to \eqref{eq:shear}. This will also be central in our analysis below.
\end{rmk}

Let us now state our main result on system \eqref{eq:Prandtl-changement-var}. 
Since we will state two results within different regularity frameworks, we will work with two different functional spaces for the data.
Note that since the boundaries $\gamma_b$, $\gamma_t$ are free, we allow the function $f$ to be defined on a domain that is possibly larger, in the vertical direction, than the reference domain $\{(x,y)\in (x_0, x_1)\times (0, +\infty), \ \overline{\gamma_b}(x)< y <\overline{\gamma_t}(x)\}$.
Hence, in order to simplify the statements, we assume that $f$ is defined in the whole infinite strip $(x_0,x_1)\times (0, +\infty)$.

\begin{itemize}
	\item In the low regularity setting, we choose an index $\sigma\in (0, 1/6)$. Our functional space will be
\nomenclature[FXs]{$\cX^\sigma$}{Banach space of data $(f, \delta_0,  \delta_1, \delta_t,\delta_b, v_b)$ with low regularity for the Prandtl problem, defined in \eqref{def:X-sigma-prandtl}}	
  \begin{equation}\label{def:X-sigma-prandtl}
	\begin{aligned} 
	\cX^\sigma=&\Big\{ (f, \delta_0,  \delta_1, \delta_t,\delta_b, v_b)\in  H^\sigma_x H^2_z\times H^4(\Sigma_0^P)\times H^4(\Sigma_1^P)\times H^2(x_0,x_1)^2 \times H^\sigma(x_0,x_1),\\
    &\qquad \delta_0(\gfs(x_0))=\delta_1(\gfs(x_1))=\delta_0(\overline{\gamma_t}(x_0))=\delta_1(\overline{\gamma_b}(x_1))=0,\\
        &\qquad f\in  L^4_x H^3_z\cap H^{\frac{1}{2}+\sigma}_x H^1_z\cap L^\infty_x W^{2,\infty}_z,\  (x-x_0)(x-x_1)\p_x \p_z f\in L^2,\\
        &\qquad\BCP^t[\delta_t](x_0)=\p_z \BCP^0[\delta_0](z_t),\ \BCP^b[\delta_b](x_1)=\p_z \BCP^1[\delta_1](z_b)
        \Big\}
	\end{aligned}
	\end{equation}
	which we endow with its canonical norm. 
 
	\item  In the high regularity setting, our functional space will be
 \nomenclature[FXs1]{$\cX^1$}{Banach space of data $(f, \delta_0,  \delta_1, \delta_t,\delta_b, v_b)$ with higher regularity for the Prandtl problem, defined in \eqref{def:X-1-prandtl} }	
	\begin{equation}\label{def:X-1-prandtl}
	\begin{aligned} 
	 \cX^1:=&\Big\{(f, \delta_0,  \delta_1, \delta_t,\delta_b, v_b)\in  H^1_x H^3_z  \times H^6(\Sigma_0^P)\times H^6(\Sigma_1^P)\times H^2(x_0,x_1)^2 \times H^1(x_0,x_1),\\
&\qquad f\vert_{\Sigma^P_i}=0,\quad \delta_0(\overline{\gamma_t}(x_0))=\delta_1(\overline{\gamma_b}(x_1))=\p_z^k \delta_i(\gfs(x_i))=0\ \forall k\in \{0,\dotsc, 3\},\\
 &\qquad \BCP^t[\delta_t](x_0)=\p_z \BCP^0[\delta_0](z_t),\ \BCP^b[\delta_b](x_1)=\p_z \BCP^1[\delta_1](z_b),\\
        &\qquad \Delta_0(z_t)=\p_x  \BCP^t[\delta_t](x_0),\ \Delta_1(z_b)=\p_x \BCP^b[\delta_b](x_1)\Big\},
          \end{aligned}\end{equation}
    where
    \begin{equation*}\begin{aligned} 
    	\Delta_i := \frac{1}{z}\p_z^2\Big[
    	&\frac{( \p_z\BCP^i[\delta_i])^2}{(\p_z \Yfs(x_i,\cdot))^2 (\p_z \Yfs(x_i,\cdot ) - \p_z \BCP^i[\delta_i])}
    	+ \frac{\p_z \BCP^i[\delta_i]}{(\p_z \Yfs(x_i,\cdot))^2} + \p_x p(x_i) \BCP^i[\delta_i] \Big]\\
    	= \frac{1}{z}\p_z^2\Big[
    	&  \frac{\p_z \BCP^i[\delta_i] }{\p_z \Yfs(x_i,z)(\p_z  \Yfs(x_i,\cdot) -\p_z \BCP^i[\delta_i])} + \p_x p(x_i)  \BCP^i[\delta_i] \Big].
	\end{aligned}
	\end{equation*}
   	Once again, we endow $\cX^1$ with its canonical norm. 
    The assumptions on $f$, $\delta_0$ and $\delta_1$ could be relaxed slightly: in particular, it is not compulsory to assume that $\delta_0$ and $\delta_1$ vanish up to order three near $z=0$, or that $f$ vanishes on the lateral boundary. However this simplifies the formulation of some compatibility conditions.  
\end{itemize}

Our result is the following:
\begin{prop}\label{prop-Prandtl-Y}
	Let $(\ufs, \vfs)$ be a smooth solution to \eqref{prandtl} on $(x_0,x_1)\times (0, +\infty)$ such that $\p_y \ufs>0$ on $\{\overline{\gamma_b}(x) \leq y \leq \overline{\gamma_t}(x),\ x\in [x_0, x_1]\}$.
	Let $\sigma\in (0, 1/6)$. 
    There exists $\eta>0$ and $z_0>0$  such that if $|z_b|, z_t\leq z_0$, the following result holds.
	
	\begin{itemize}
		\item There exists a manifold $\cM_\sigma\subset \cX^\sigma$, of codimension 1 within the ball of radius $\eta$ in $\cX^\sigma$, such that \eqref{eq:Prandtl-changement-var}-\eqref{cond:tY-Sigma-i}-\eqref{cond:tY-top-bottom} has a solution  in $H^{\frac{2}{3} + \sigma}_x H^1_z\cap H^\sigma_x  H^{3}_z$ if and only if $(f,\delta_0,\delta_1, \delta_b,\delta_t, v_b)\in \cM_\sigma$.
        
		This solution, if it exists, is unique.
		
		\item There exists a manifold $\cM_1\subset \cX^1$, of codimension 3 within the ball of radius $\eta$ in $\cX^1$,
		such that \eqref{eq:Prandtl-changement-var}-\eqref{cond:tY-Sigma-i}-\eqref{cond:tY-top-bottom}  has a solution  in $H^{5/3}_x H^1_z\cap H^1_x H^{3}_z$ if and only if $(f,\delta_0,\delta_1, \delta_b,\delta_t, v_b)\in \cM_1$.
	\end{itemize}
\end{prop}

The proof of \cref{prop-Prandtl-Y} is similar to the one of \cref{thm:burgers}. 
We construct a solution to~\eqref{eq:Prandtl-changement-var} thanks to an iterative scheme (or equivalently, thanks to the abstract \cref{thm:abstract}), relying on several important observations:
\begin{itemize}
	\item First, the left-hand side of \eqref{eq:Prandtl-changement-var} depends \emph{linearly} on $\tY $, and the right-hand side depends smoothly 
 on $\tY$. 
    This nice feature stems directly from our change of variables. Note also that our choice of boundary conditions \eqref{CL-Prandtl-bottom}-\eqref{CL-Prandtl-top}, which are slightly unusual when we formulate them on the unknown function $u$, are in fact designed so that they become classical boundary conditions in the variable $\tY$. 
    Indeed, the top and bottom boundaries in the $z$ variable are now fixed (and flat), and the boundary condition for $\tY$ on these boundaries is merely a Neumann condition (so a Dirichlet condition for the vorticity $\p_z \tY$).
	
	\item Second, as mentioned above, the vorticity $\p_z \tY$ satisfies an equation with a very nice structure.
    More precisely, setting
	\begin{equation*}
	\begin{aligned} 
	\alpha(x,z)&:= \frac{1}{(\p_z \Yfs(x,z))^2}>0,\\
	\beta(x)&:=-\p_x p,
	\end{aligned}
	\end{equation*}
	and differentiating \eqref{eq:Prandtl-changement-var} with respect to $z$, we find that $W := \p_z \tY$ is a solution to
    \begin{equation}
        \label{eq:vorticity-Prandtl}
        \begin{cases}
            z \p_x W + \beta \p_z W - \p_z^2 (\alpha  W)=\p_z g & \quad \text{in } (x_0,x_1)\times (z_b,z_t), \\
            W\vert_{\Sigma_i} = \p_z \BCP^i[\delta_i] & \quad \text{for } i\in \{0,1\},\\
            W\vert_{z=z_j} = \BCP^j[\delta_j] &\quad \text{for } j\in \{t,b\}.
        \end{cases}
    \end{equation}
	The coefficients $\alpha$ and $\beta$ are smooth and depend only on the underlying flow $(\ufs,\vfs)$.
	Furthermore, $\inf \alpha>0$ in $(x_0,x_1)\times (z_b,z_t)$ by assumption.
	Hence the structure of system~\eqref{eq:vorticity-Prandtl} is very similar to the one of~\eqref{eq:shear}, albeit with variable coefficients.
	The smallness condition on $z_b$ and $z_t$ ensures that we have nice \emph{a priori} estimates for \eqref{eq:vorticity-Prandtl} (see \cref{lem:WP-Z0-vorticity} below).
	
	\item Eventually, we observe that, from \eqref{eq:Prandtl-g},
	\begin{equation*}
        \p_z g = \p_y f(x, Y)  (\p_z \Yfs - W)^2 + f(x,Y) (\p_z^2 \Yfs - \p_z W) + \p_z^2 \left(\frac{W^2}{(\p_z \Yfs)^2 (\p_z \Yfs -W)}\right).
	\end{equation*}
	In order to design a convergent iterative scheme for \eqref{eq:vorticity-Prandtl}, it is necessary to work in a functional space controlling the $L^\infty$ norm of $W$ (for example to ensure that the denominator does not vanish, or that the application $W\mapsto \p_z^2 (W^2)\in L^2$ is Lipschitz continous).
    Having $W\in Z^0$ is not sufficient as we barely miss the embedding in $L^\infty$ (see \cref{rmk:Z0-Linf}). 
    However, the functional space $W\in H^{\frac{2}{3}+\sigma}_x L^2_z\cap H^\sigma_x H^2_z$, with $\sigma$ strictly positive and small, will be suitable for our purposes. 
    This is in sharp contrast with the nonlinear scheme for the Burgers system, for which we also needed that $\p_z \tY = W \in L^\infty$ but for which the function $\tY$ (rather than $\p_z \tY = W$)  was a solution to~\eqref{eq:shear}.
    Therefore, having $W \in L^\infty$ required $\tY \in H^{\frac{2}{3}+\sigma}_x L^2_z\cap H^\sigma_x H^2_z$ for some $\sigma>1/3$. 
    Such a regularity requires two orthogonality conditions (see \cref{lem:shear-Hs}).
    This gain of one derivative in the vertical variable (corresponding to a gain of $1/3$ of derivative in the horizontal variable) allows us to get rid of two of the orthogonality conditions, leading to the first statement of \cref{prop-Prandtl-Y}.
\end{itemize}

\subsection{The linearized vorticity equation}
\label{sec:vorticity}

This section is devoted to the analysis of system~\eqref{eq:vorticity-Prandtl}, for a given source term $\p_z g \in L^2(\Om)$. 
Adapting and stepping on the analysis of \cref{sec:shear}, we prove the existence and uniqueness of solutions in $Z^0(\Om)$. 
We also exhibit necessary and sufficient conditions for higher regularity. 

For the sake of simplicity, within this section, we denote by $\Omega$ the rectangle $(x_0,x_1)\times (z_b,z_t)$, which is a slight abuse of notation since $(z_b,z_t) \neq (-1,1)$. 
We still denote by $\Sigma_0 = \{x_0\}\times (0, z_t)$ and $\Sigma_1 = \{x_1\}\times (z_b,0)$ the lateral boundaries.

\begin{lem}[Well-posedness of the linear vorticity equation]
    \label{lem:WP-Z0-vorticity}
    Let $\alpha \in C^2(\overline{\Om})$ and $\beta \in L^\infty(x_0, x_1)$. 
    Assume that there exists $\lambda>0$ such that
	\begin{equation}\label{ellip-alpha}
        \forall (x,z) \in \Om, \quad
	       \frac{1}{\lambda}\leq \alpha (x,z)\leq \lambda.
	\end{equation}
    There exists $z_0>0$, depending only on $\alpha$, such that if $|z_b|, z_t \leq z_0$, the following result holds.
	
	Let $h\in L^2(\Om)$, $w_t, w_b\in H^{3/4}(x_0, x_1)$, and $w_i \in \mathscr H^1_z(\Sigma_i)$. 
    Assume that the compatibility conditions $w_t(x_0)= w_0(z_t)$, $w_b(x_1)=w_1(z_b)$ are satisfied.
	
	Consider the system
	\begin{equation} 
	\label{vorticity-model}
	\begin{cases}
		z \p_x W + \beta \p_z W - \p_z^2(\alpha  W)= h &\text{ in }\Om,\\
		W\vert_{\Sigma_i}= w_i & \text{ for }i\in \{0,1\},\\
		W\vert_{z=z_j}= w_j & \text{ for }j\in \{t, b\}.
	\end{cases}
	\end{equation}
	Then \eqref{vorticity-model} has a unique solution $W\in Z^0(\Om)$, which moreover satisfies
	\begin{equation*}
	   \| W\|_{Z^0}\leq C \left( \|h\|_{L^2} + \|w_t\|_{H^{3/4}} + \|w_b\|_{H^{3/4}} + \|w_0\|_{\mathscr H^1_z(\Sigma_0)} + \|w_1\|_{\mathscr H^1_z(\Sigma_1)}\right),    
	\end{equation*}
	where the constant $C$ depends only on $\lambda$, $\|\beta\|_\infty$ and $\| \p_z \alpha\|_{\infty}$.
\end{lem}

\begin{proof}
	According to \cite[Theorem 2.1]{Pagani2} it is sufficient to prove the result when $w_t=w_b=w_0=w_1=0$, since one could lift these boundary conditions for the given regularity. 
    
    In this case, we note that since $\p_z \beta=0$, we have the $L^2_x H^1_z$ energy estimate
    \begin{equation*}
        \int_\Om \alpha (\p_z W)^2 \leq \| h \|_{L^2} \|W\|_{L^2} + \| \p_z \alpha \|_{L^\infty} \| W \|_{L^2} \| \p_z W \|_{L^2}.
    \end{equation*}
	If $|z_b|, z_t\leq z_0$, then $\|W\|_{L^2(\Om)} \leq z_0  \| \p_z W \|_{L^2(\Om)}$. 
    As a consequence, if $z_0\leq 1/(2\lambda\|\p_z \alpha\|_\infty)$, we obtain $\|W\|_{L^2_x H^1_z} \lesssim \|h\|_{L^2}$. 
    From there, following the same arguments as in \cref{p:shear-X0}, we infer that there exists a solution $W\in \cB$ to \eqref{vorticity-model} satisfying $\|W\|_{\cB} \lesssim \|h\|_{L^2}$.
    The uniqueness of this solution is proved in \cref{sec:proof-uniqueness}.
    Eventually, we see $W \in \cB$ as the solution to
    \begin{equation*}
         z \p_x W - \p_z (\alpha \p_z W) = h - \beta \p_z W + \p_z (\p_z \alpha W)  
    \end{equation*}
    where the right-hand side belongs to $L^2(\Om)$ since $\beta \in L^\infty$ and $\alpha \in C^2(\overline{\Om})$.
    Since $\alpha \in C^2(\overline{\Om})$, applying Pagani's result \cite[Theorem 5.1]{Pagani2} to the operator $z \p_x - \p_z(\alpha \p_z \cdot)$ which is in conservative form, we obtain that $W \in Z^0$ and $\|W\|_{Z^0} \lesssim \|h\|_{L^2} + \|W\|_{\cB}$.
\end{proof}

We now rely on the analysis of \cref{sec:shear} in order to identify two necessary and sufficient orthogonality conditions for higher regularity. 
Let us first remark that the only potential singular points are $(x_0,0)$ and $(x_1,0)$. 
Indeed, we recall that $z\p_x W\in L^2(\Om)$, and therefore $W\in H^1_xL^2_z(\{|z|\geq z_0\})$ for all $z_0>0$. 
Regularity away from the lateral boundaries is ensured by the following lemma.

\begin{lem}
	\label{lem:reg-away-sing-pts}
 
	Let $\alpha \in C^3(\overline{\Om})$ satisfying \eqref{ellip-alpha} and $\beta \in C^1([x_0,x_1])$.	
    There exists $z_0>0$, depending only on $\alpha$, such that if $|z_b|, z_t \leq z_0$, the following result holds.
	
	Let $h\in L^2(\Om)$ such that $(x-x_0)(x-x_1)\p_x h \in L^2$. Let $w_t,w_b\in H^2(x_0,x_1)$ and $w_i\in H^2(\Sigma_i)$ such that the compatibility conditions $w_t(x_0)= w_0(z_t)$, $w_b(x_1)=w_1(z_b)$ are satisfied.
	
	Let $W\in Z^0$ be the unique solution to \eqref{vorticity-model}. 
    Then $(x-x_0)(x-x_1) \p_x  W \in Z^0$.
\end{lem}

The proof is postponed to \cref{sec:proof-lem:reg-away-sing-pts}, in order not to burden this section. 
We are now ready to state our orthogonality conditions for system \eqref{vorticity-model}.
To that end, for $\alpha \in C^4(\overline{\Om})$,  $\beta \in C^1([x_0,x_1])$, $\sigma\in (0,1]$, we introduce the space\nomenclature[FHsab]{$\mathcal{H}^\sigma_{\alpha,\beta}$}{Space of data tuples \eqref{def:H-alpha-beta} for the vorticity equation \eqref{vorticity-model} at regularity $\sigma$}
\begin{equation}\label{def:H-alpha-beta}
    \begin{aligned}
        \cH_{\alpha,\beta}^\sigma:=\Big\{ &(h,w_0,w_1,w_t,w_b)\in H^{\sigma}_x L^2_z\times H^2(\Sigma_0)\times H^2(\Sigma_1)\times H^2(x_0,x_1)^2,\\
        &\qquad (x-x_0)(x-x_1)\p_x h \in L^2,\quad w_t(x_0)= w_0(z_t),\  w_b(x_1)=w_1(z_b),\\
         &\qquad \text{and}\ \Delta_i\in \mathscr{H}^1_z(\Sigma_i)\text{ if }\sigma>1/2,\\
        &\qquad\text{and } \Delta_0(z_t)= \p_x w_t(x_0),\  \Delta_1(z_b)=\p_x w_b(x_1)\text{ if }\sigma>1/2,\\
        &\qquad\text{where }\Delta_i:=\frac{1}{z} \left( h(x_i,\cdot) + \p_z^2 (\alpha(x_i, \cdot )w_i)-\beta(x_i)\p_z w_i\right) \Big\}.
    \end{aligned}
\end{equation}

We now state a proposition extending the results of \cref{sec:shear} to equations with smooth variable coefficients:
\begin{prop}
	\label{lem:ortho-Prandtl}
	Let $\alpha \in C^4(\overline{\Om})$ satisfying \eqref{ellip-alpha} and $\beta \in C^1([x_0,x_1])$.
	There exist two  linear forms $\widehat{\ell^0}$, $\widehat{\ell^1}$, continuous on $\cH_{\alpha,\beta}^\sigma$ for all $\sigma\in (1/6, 1]$, such that the following result holds.
	
	\begin{itemize}
		\item Let $\sigma\in (0,1/6)$, and let  $(h,w_0,w_1,w_t,w_b) \in \cH_{\alpha,\beta}^\sigma$.
		Let $W\in Z^0$ be the unique solution to~\eqref{vorticity-model}.
		
		Then $W\in Z^\sigma =[Z^0,Z^1]_\sigma\hookrightarrow H^{\frac{2}{3}+ \sigma}_x L^2_z \cap H^\sigma_x H^{2}_z \hookrightarrow L^\infty(\Om)$, and
\begin{equation*}
    \| W\|_{Z^\sigma}\lesssim \|h\|_{H^\sigma_x L^2_z} + \|(x-x_0)(x-x_1) \p_x h\|_{L^2}+ \sum_{j\in \{b,t\} }\|w_j\|_{H^2(x_0,x_1)} + \sum_{i\in\{0,1\}} \|w_i\|_{H^2(\Sigma_i)} .
\end{equation*}
        
		\item Let $\sigma\in (1/6,1]\setminus\{1/2\}$, and let $(h,w_0,w_1,w_t,w_b) \in \cH_{\alpha,\beta}^\sigma$.
        Let $W\in Z^0$ be the unique solution to \eqref{vorticity-model}.
        Then $ W\in Z^\sigma$ if and only if
		\begin{equation*}
		\widehat{\ell^0}(h,w_0,w_1,w_t,w_b)=\widehat{\ell^1}(h,w_t,w_b,w_0,w_1)=0,
		\end{equation*}
		and in that case
		\begin{equation*}\begin{aligned}
		      \| W\|_{Z^\sigma}\lesssim &\|h\|_{H^\sigma_x L^2_z} + \|(x-x_0)(x-x_1) \p_x h\|_{L^2}\\&+ \sum_{j\in \{b,t\} }\|w_j\|_{H^2(x_0,x_1)} + \sum_{i\in\{0,1\}} \|w_i\|_{H^2(\Sigma_i)}  + \|\Delta_i\|_{\mathscr H^1_z(\Sigma_i)}.
		\end{aligned}\end{equation*}
	\end{itemize}
\end{prop}

\begin{proof}
    We start with the first statement, and we take $\sigma \in (0,1/6)$ fixed.

	 \step{Lifting the top and bottom boundary conditions.}
In order to use the theory from \cref{sec:shear}, which is stated with homogeneous Dirichlet boundary conditions at the top and bottom, we first
		lift the latter.  We change         $W$ into $W-\rho(z-z_t) w_t - \rho(z-z_b) w_b$, where $\rho\in C^\infty_c(\R)$ is such that $\rho\equiv 1$ in a neighbourhood of zero, and $\supp \rho \subset (-r,r)$ for some $r<\min(|z_b|,z_t)/2$. This changes the source term $h$ into 
  \begin{equation*}
  h-\sum_{j\in \{t,b\}} (z\p_x w_j \rho(z-z_j) + \beta w_j \rho'(z-z_j) - w_j\p_z^2(\alpha \rho(z-z_j)),
  \end{equation*}
   which belongs to $H^\sigma_x L^2_z$, and the boundary condition $w_0$ (resp.\ $w_1$) into $w_0- w_0(z_t) \eta(z-z_t)$ (resp.\ $w_1- w_1(z_b) \eta(z-z_b)$), which belongs to $H^2(\Sigma_0)$ (resp.\ $H^2(\Sigma_1)$). With a slight abuse of notation, we still denote by $W$ the unknown function, and by $(h,w_0, w_1,0,0)$ the data. Note that this operation does not affect the compatibility conditions in the corners.

    \step{Localization in the vicinity of the singular points.}
    We then localize horizontally the solution in the vicinity of $x_0$ and $x_1$. 
    We only treat the localization in the vicinity of $x_0$ since the other boundary is identical.
    Let $\chi_0\in C^\infty_c(\R)$ be such that $\chi_0\equiv 1$ in a neighborhood of $x_0$, and $\supp \chi_0\subset B(x_0,r)$ for some small $0<r< (x_1-x_0)/2$.
	Then $W_0:=W\chi_0(x)$ is a solution to
	\begin{equation} 
        \label{eq:W0}
    	z\p_x W_0 + \beta \p_z W_0- \p_z^2(\alpha  W_0) = h \chi_0 + z W \p_x \chi_0 .
	\end{equation}
	Since $W\in Z^0$, $W \in H^{2/3}_x L^2_z$, so the right-hand side belongs to $H^\sigma_x L^2_z$. 
    We then localize the coefficient $\alpha$. 
    Let $\alpha_0(z):=\alpha(x_0,z)$. 
    Then
	\begin{equation}
	    \label{eq:W_0-localized}
        \begin{split}
	        z\p_x W_0 -\p_z^2(\alpha_0(z) W_0) = h \chi_0 & + z W \p_x \chi_0 - \beta \p_z W \chi_0  \\
	        & - \p_z^2((\alpha_0-\alpha)  W_0).
	    \end{split}
	\end{equation}
	On the support of $\chi_0$, there exists a constant $C$ such that $|\alpha_0 - \alpha| \leq C |x -x_0|$ and $(\alpha-\alpha_0)/(x-x_0)$ is a $C^3$ function of $(x,z)$.
    According to \cref{lem:reg-away-sing-pts}, $(\alpha_0-\alpha) \p_z^2 W_0 \in H^1_x L^2_z$.
    Hence, the right-hand side of \eqref{eq:W_0-localized} belongs to $H^\sigma_x L^2_z$.
    Note furthermore that $W_0$ vanishes on $\{z=z_t\}$ and $\{z=z_b\}$ thanks to the first step.

	\step{Vertical change of variables to work with constant coefficients.}
	In order to use the theory from \cref{sec:shear}, we now change the vertical coordinate so that the equation in the new variables is formulated thanks to the Kolmogorov operator.
    More precisely, we set $W_0(x,z) = \omega_0(x,\tz)$	where~$\tz$ is a function of $z$ such that $\tz(0)=0$. 
    We have
    \begin{equation*}
        \p_z^2 (\alpha_0 W_0) = \alpha_0 (\tz')^2 \p_{\tz}^2 \omega_0 + (\alpha_0 \tz'' + 2 \p_z \alpha_0 \tz') \p_{\tz} \omega_0 + (\p_z^2 \alpha_0) \omega_0.
    \end{equation*}
    We first choose the function $\tz$ so that $\tz(0)=0$ and
	\begin{equation*}
	    \frac{z}{\alpha_0(z) (\tz'(z))^2 } = \tz,
        \quad 
    	\text{i.e.}
    	\quad
    	\tz'(z) = \sqrt{\frac{z}{\alpha_0(z) \tz(z)}}.
	\end{equation*}
	Explicit resolution for $z > 0$ yields (with a similar formula for $z < 0$):
	\begin{equation}
	    \label{def:tz}
    	\tz(z) = \left( \frac{3}{2} \int_0^z \sqrt{\frac{t}{\alpha_0(t)}}\dd t\right)^{2/3}.
	\end{equation}
	It can be easily checked that the function $\tz$ thus defined has the same regularity as $\alpha_0$ on $(z_b,z_t)$ and that $C^{-1} \leq \tz' \leq C$ for some positive constant $C$.
    Moreover, $(\alpha_0(0))^{\frac 13} \zeta(z) \sim  z$ as $z \to 0$.
	
    The function $\omega_0$ then solves
	\begin{equation}
	    \label{eq:omega0}
        \tz \p_x \omega_0- \p_{\tz}^2 \omega_0 = s_0(x,\zeta)
	\end{equation}
	where $s_0 \in H^\sigma_x L^2_z$.
    Furthermore, $\omega_0$ is supported in the vicinity of $(x_0,0)$. 
    We denote by $\mu_0$ the lateral boundary condition on $\Sigma_0$ in the new vertical variable, i.e.\ $\mu_0(\tz(z))= \chi_0(x_0,z) w_0(z)$. 
    Note that $\mu_0$ and $w_0$ enjoy the same regularity, so that $\mu_0\in H^2(\Sigma_0)$.
	
	\step{Small fractional regularity.}
	We now consider \eqref{eq:omega0}, whose right-hand side belongs to $H^\sigma_x L^2_z$. 
	The equation is endowed with homogeneous data on $\{z=z_t\}\cup\{z=z_b\}\cup \Sigma_1$, and with $H^2$ data on $\Sigma_0$ satisfying a compatibility condition at $(x_0,z_t)$.
	Using \cref{lem:shear-Hs} and \cref{lem:Zsigma-Qsigma}, we infer that $\omega_0 \in Z^\sigma \hookrightarrow H^{\frac{2}{3} + \sigma}_x L^2_\tz\cap H^\sigma_x H^{2}_\tz$, and thus $W_0$ enjoys the same regularity.
	Performing a similar change of variables near $(x_1,0)$, we deduce that $W\in Z^\sigma$. 
 This completes the proof of the first statement from \cref{lem:ortho-Prandtl}.
	
	\step{Identification of the orthogonality conditions.}
	Let us now assume that $\sigma \in (1/6, 1/3]$ and $h\in H^\sigma_x L^2_z$.
	The right-hand side of \eqref{eq:omega0} now belongs to $H^\sigma_x L^2_z$. 
    Furthermore, in a neighborhood of $\tz=0$, 
    \begin{equation*}
    s_0(x_0,\zeta)=\frac{1}{\alpha_0 (z)(\zeta'(z))^2}\left[h(x_0,z) - \beta(x_0) w_0'(z) + (\alpha_0 \zeta''+ 2 \alpha_0'\zeta') (z)\p_\zeta \mu_0 (\zeta) + \alpha_0''(z) \mu_0(\zeta) \right],
    \end{equation*}
    where the primes always denote derivatives with respect to $z$.
    Using this equality together with the identity $\p_\zeta \mu_0(\zeta) = w_0'(z)/\zeta'(z)$, we find, after some tedious but straightforward computations, and for $\tz$ in a neighborhood of zero,
    \begin{equation*}
        \p_\tz^2 \mu_0(\zeta) + s_0(x_0,\zeta) = \frac{\zeta(z)}{z} \Delta_0(z)
    \end{equation*}
    Hence $(\p_{\tz}^2 \mu_0 + s_0(x_0,\cdot)) / \tz \in \mathscr{H}^1_{\tz}(\Sigma_0)$.
    Note also that the compatibility conditions in the corners are satisfied. 
	We then apply \cref{coro:ortho-localized-frac} to \eqref{eq:omega0} whose right-hand side is in $H^{\sigma}_x L^2_\zeta$. 
    We infer that if
	\begin{equation*}
	   (a_0 \overline{\ell^0} + a_1 \overline{\ell^1}) ( s_0, \mu_0, 0)=0, 
	\end{equation*}
	then $\omega_0 \in Z^\sigma \hookrightarrow H^{\frac{2}{3} + \sigma}_x L^2_\tz\cap H^\sigma_x H^{2}_\tz$ by \cref{lem:Zsigma-Qsigma}. 
    Similarly, $\omega_1 \in Z^{\sigma}$, so $W \in Z^{\sigma}$.

    For $\sigma\geq 1/3$ and $\sigma\neq 1/2$, we  use a bootstrap argument.
    Going back to \eqref{eq:W0}, we now know that the right-hand side is in $H^{\min (\sigma,2/3)}_x L^2_z$, so that we can apply \cref{coro:ortho-localized-frac} to \eqref{eq:omega0} whose right-hand side belongs to $H^{\min (\sigma,2/3)}_x L^2_\zeta$.
    This implies that $W \in Z^{\min (\sigma,2/3)}$.
    We then repeat this procedure one last time if $\sigma\geq 2/3$.

	Setting
	\begin{equation}\label{def:hat-ell-0}
	\widehat{\ell^0}(h,w_0,w_1,w_t,w_b)= (a_0 \overline{\ell^0} + a_1 \overline{\ell^1}) ( \lambda \omega_0 + s_0, \mu_0, 0),
	\end{equation}
	and defining in a similar fashion the linear form $\widehat{\ell^1}$ associated with the regularity in the vicinity of $(x_1,0)$,
	we obtain the desired result.

 Eventually, it follows from  the definition of $\widehat{\ell^0}$ in \eqref{def:hat-ell-0} and from \cref{rmk:16-shear} that the linear forms $\widehat{\ell^j}$ are continuous on $\cH^\sigma_{\alpha,\beta}$ for all $\sigma>1/6$.
\end{proof}

\begin{lem}
    \label{lem:indep-ortho-Prandtl-1}
	The two linear forms $\widehat{\ell_0},\widehat{\ell_1}: \cH_{\alpha,\beta}^1\to \R$ defined in \cref{lem:ortho-Prandtl} are independent. 
    Furthermore, there exist $g^0, g^1\in C^\infty_c (\Om)$ such that
	\begin{equation*}
	\widehat{\ell^j}(g^i, 0, 0,0,0)=\delta_{i,j}\quad \forall i,j\in \{0,1\}.
	\end{equation*}
\end{lem}

\begin{proof}
	We begin with the following remark. Following the notations of the proof of  \cref{lem:ortho-Prandtl} above, we set $\alpha_i(z):=\alpha(x_i,z)$,. With the same change of variables as in Step 3 of the proof (see \eqref{def:tz}), we define
	\begin{equation*}
	U^0(x,z)= \busing^0 (x,\tz_0).
	\end{equation*}
	Then
	\begin{equation*}
	z\p_x  U^0 + \beta \p_z  U^0 - \p_z^2(\alpha_0 U^0)= \alpha_0  (\tz_0')^2 \overline{f^0} + \lambda_0 \busing^0 (x,\tz_0) + \gamma_0 \p_{\tz_0}  \busing^0 (x,\tz_0).
	\end{equation*}
	for some smooth functions $\lambda_0$, $\gamma_0$ depending on $\beta$ and $\alpha$. The right-hand side therefore belongs to $H^{1/3}_x L^2_z\cap H^1_x L^2_z ((x-x_0)^2(x-x_1)^2)$. Furthermore $U^0$ vanishes on $\Sigma_0\cup \Sigma_1\cup \{z=z_b\} \cup \{z=z_t\}$.

	Of course we may perform the same procedure around $(x_1,0)$, and we define a function $U^1(x,z)$, localized in a neighborhood of $(x_1,0)$ and with the same regularity as $\busing^1$, such that
	\begin{equation*}
	z\p_x U^1 + \beta \p_z U^1 -\p_z^2(\alpha_1  U^1)\in H^{1/3}_x L^2_z.
	\end{equation*}
	Note that $U^0$ and $U^1$ vanish on $\Sigma_0\cup \Sigma_1 \cup \{z=z_b\}\cup \{z=z_t\}$. 
	Now, for $i=0,1$, let
	\begin{equation*}
	h^i:= z\p_x U^i + \beta \p_z U^i - \p_z^2(\alpha  U^i).
	\end{equation*}
	By construction, $h_i$ and $U_i$ are localized in the vicinity of $(x_i,0)$, and $h^i\in H^{1/3}_x L^2_z$, $(x-x_0)(x-x_1)\p_x h_i\in L^2$. Furthermore $h_i\vert_{\Sigma_0\cup \Sigma_1}=0$. As a consequence, recalling the definition of $\widehat{\ell^0}$ and $\widehat{\ell^1}$ (see \eqref{def:hat-ell-0} together with \cref{coro:ortho-localized-frac}), we infer that
	\begin{equation*}
	\widehat{\ell^0}(h_1,0,0,0,0)=\widehat{\ell^1}(h_0,0,0,0,0)=0.
	\end{equation*}
	Now, assume that $c_0\widehat{\ell^0} + c_1 \widehat{\ell^1}=0$ for some $(c_0,c_1)\in \R^2$. We deduce from the above equalities that
	\begin{equation*}
	\widehat{\ell^0}(c_0h_0 + c_1 h_1, 0,0,0,0)= c_0 \widehat{\ell^0}(h_0,0,0,0,0)= (c_0\widehat{\ell^0} + c_1\widehat{\ell^1})(h_0,0,0,0,0)=0,
	\end{equation*}
	and similarly $\widehat{\ell^1}(c_0h_0 + c_1 h_1, 0,0,0,0)=0$. 
 Using \cref{lem:ortho-Prandtl}, we infer that 
  $c_0U^0 + c_1 U^1\in  Z^{1/3} \hookrightarrow H^1_x L^2_z \cap H^{1/3}_x H^2_z$. 
	Since $U^i$ has the regularity of $\busing^i$ and is localized in the vicinity of $(x_i,0)$, it follows from \cref{lem:busing-h56} that $c_0=c_1=0$.
	
	Note that the above argument also ensures that $\widehat{\ell^i}(h^i,0,0,0,0)\neq 0$. Hence, up to a multiplication by a constant, we may always assume that $\widehat{\ell^j}(h^i,0,0,0,0)=\delta_{i,j}$.
	Let us now take, for $\eps>0$ small, $h^i_\eps\in C^\infty_c({\Om})$ such that $\| h^i_\eps-h^i\|_{H^{1/3}_x L^2_z}\leq \eps$ and $\| (x-x_0)(x-x_1)\p_x (h^i-h^i_\eps)\|_{L^2}\leq \eps$. Then, since the linear forms $\widehat{\ell^j}$ are continuous on $\cH_{\alpha,\beta}^{1/3}$, we obtain
	$|\widehat{\ell^j}( h^i_\eps, 0,0,0,0) - \delta_{i,j}| \lesssim \eps$. As a consequence, there exists $a^0_\eps, a^1_\eps, b^0_\eps, b^1_\eps$  such that
	\begin{equation*}
	\begin{aligned}
	\widehat{\ell^0}(a^0_\eps h^0_\eps + a^1_\eps h^1_\eps, 0,0,0,0)= \widehat{\ell^1}(b^0_\eps h^0_\eps + b^1_\eps h^1_\eps, 0,0,0,0)=1,\\
	\widehat{\ell^0}(b^0_\eps h^0_\eps + b^1_\eps h^1_\eps, 0,0,0,0)= \widehat{\ell^1}(a^0_\eps h^0_\eps + a^1_\eps h^1_\eps, 0,0,0,0)=0,
	\end{aligned} 
	\end{equation*}
	and $|a^0_\eps-1| , |b^1_\eps-1|\lesssim \eps$, $|a^1_\eps|, | b^0_\eps|\lesssim \eps$. The result follows, taking $g^0=a_\eps^0 h^0_\eps + a_\eps^1 h^1_\eps$ and $g^1= b^0_\eps h^0_\eps + b^1_\eps h^1_\eps$.
\end{proof}

\subsection{Reconstructing the velocity from the vorticity}
\label{sec:Prandtl-reconstruct}

Let $(g, \widetilde{\delta}_0,\widetilde{\delta}_1)\in L^2_x H^1_z\times H^3(\Sigma_0)\times H^3(\Sigma_1) $, $w_t,w_b\in H^2(x_0, x_1) $. 
Assume that $\p_z \widetilde{\delta_1}(z_b)=w_b(x_1)$, $\p_z \widetilde{\delta_0}(z_t)=w_t(x_0)$.
According to \cref{lem:WP-Z0-vorticity}, there exists a unique solution $W\in Z^0$ to \eqref{vorticity-model} with $h = \p_z g$ and $w_i=\p_z \widetilde{\delta}_i$.
The purpose of this subsection is to construct a solution to the system
\begin{equation} \label{eq:tY-model}
\begin{cases}
	z\p_x \tY - \int_{z_b}^z \p_x \tY + \beta \p_z \tY - \p_z(\alpha \p_{z}\tY)= g & \text{in }\Om,\\
	\tY\vert_{\Sigma_i}= \widetilde{\delta}_i & \text{for }i\in \{0,1\},\\
	\p_z\tY\vert_{z=z_j}= w_j& \text{for }j\in \{t,b\}.\\
\end{cases}
\end{equation}
We therefore set, for $(x,z)\in \Om$,
\begin{equation}\label{def:tY}
\tY(x,z):=\widetilde{\gamma_b}(x)  + \int_{z_b}^z W(x,z')\dd z',
\end{equation}
where the function $\widetilde{\gamma_b}$ solves the differential equation
\begin{equation} \label{def:gamma_b}
    \begin{aligned}
        & z_b \p_x \widetilde{\gamma_b}  + (\beta -\p_z \alpha(\cdot, z_b))w_b - \alpha(x,z_b) \p_z W(x,z_b) = g(x,z_b), \\
        & \widetilde{\gamma_b}(x_1) =  \widetilde{\delta_1}(z_b).
    \end{aligned}
\end{equation}
Since $W\in Z^0$, the trace $\p_z W(\cdot, z_b)$ belongs to $H^{1/4}(x_0,x_1)$ by \cref{lem:Z0-trace-pz-top}. 
Thus $\widetilde{\gamma_b} \in H^{1}(x_0,x_1)$, 
and $\tY \in H^{2/3}_x H^1_z\cap L^2_x H^3_z\subset C^0(\overline{\Om}).$
Furthermore $\tY \in H^1_x H^1_z (z_b, z_b/2)$.

By construction, we have, in the sense of distributions on $\Om$,
\begin{equation*}
    \p_z \left[ z \p_x \tY - \int_{z_b}^z \p_x \tY(x,z')\dd z'  +\beta \p_z \tY - \p_z(\alpha \p_z \tY) - g \right]=0,
\end{equation*}
and therefore there exists a function $G$ depending only on $x$ such that 
\begin{equation*}
    z \p_x \tY - \int_{z_b}^z \p_x \tY(x,z')\dd z'  +\beta \p_z \tY - \alpha \p_z^2 \tY = g (x,z)+ G(x).
\end{equation*}
The choice of the function $\widetilde{\gamma_b}$ (see \eqref{def:gamma_b}) then ensures that $G \equiv 0$.
By definition of $\tY$, we have $\p_z \tY\vert_{z=z_j}=W\vert_{z=z_j}=w_j$ for $j\in \{t,b\}$.

Let us now investigate the lateral boundary conditions. On $\Sigma_i$, we have
\begin{equation*}
\p_z \tY(x_i,z)=W(x_i, z)= \p_z \widetilde{\delta}_i(z).
\end{equation*}
Hence, in order to ensure that $\tY\vert_{\Sigma_i}= \widetilde{\delta}_i$, it suffices to check that $\tY(x_i,z_i)= \widetilde{\delta}_i(z_i)$ for some $(x_i,z_i)\in \overline{\Sigma_i}$. From there, we treat separately (and differently) the two boundaries $\Sigma_0$ and $\Sigma_1$.

\begin{itemize}
	\item On $\Sigma_1$, we note that $\tY(x_1,z_b)=\widetilde{\delta}_1(z_b)$ by definition of $\widetilde{\gamma_b}$. 
	Therefore $\tY\vert_{\Sigma_1} = \widetilde{\delta}_1$.
	\item On $\Sigma_0$, the situation is different, since 
	$\tY(x_0,0)\neq \widetilde{\delta}_0(0)$ \emph{a priori}.
	Indeed, 
    \begin{equation*}
        \begin{split}
        	\tY(x_0,0)&= \widetilde{\gamma_b}(x_0) + \int_{z_b}^0 W(x_0,z')\dd z'\\
        	&=\frac{1}{z_b}\left( - \int_{x_0}^{x_1} \left( g(x,z_b) + (\p_z \alpha(x,z_b)- \beta(x)) w_b (x)+ \alpha(x,z_b) \p_z W(x,z_b)\right)\dd x\right)
            \\& \quad  + \int_{z_b}^0 W(x_0,z')\dd z' + \widetilde{\delta}_1(z_b).
        \end{split}
    \end{equation*}
\end{itemize}

The right-hand side of the above equality is a linear form in $(g, \widetilde{\delta}_0,\widetilde{\delta}_1, w_t,w_b)$, which leads to the following definition.

\begin{defi}[Additional linear form for the solvability of the Prandtl system]
	\label{def:ell-2-Prandtl}
	Let $(g, \widetilde{\delta}_0,\widetilde{\delta}_1)\in L^2_x H^1_z\times H^3(\Sigma_0)\times H^3(\Sigma_1) $, $w_t,w_b\in H^2(x_0, x_1) $ such that
	$w_t(x_0)=\p_z \widetilde{\delta}_0(z_t)$, $w_b(x_1)=\p_z \widetilde{\delta}_1(z_b)$.
	Let $W\in Z^0$ be the unique solution to \eqref{vorticity-model} with $h=\p_z g$ and $w_i=\p_z \widetilde{\delta}_i$.
	
	The linear form $\ell^2$ is defined by\nomenclature[OLl2]{$\ell^2$}{Additional linear form of \cref{def:ell-2-Prandtl} to reconstruct the velocity from the vorticity}
	\begin{equation*}
	    \begin{split}
        	\ell^2\left(g,\widetilde{\delta_0}, \widetilde{\delta_1}, w_t,w_b\right) &:= -\frac{1}{z_b} \int_{x_0}^{x_1} \left( g(x,z_b) + (\p_z \alpha(x,z_b)- \beta(x)) w_b (x)+ \alpha(x,z_b) \p_z W(x,z_b)\right)\dd x \\& \quad + \int_{z_b}^0 W(x_0,z')\dd z'+ \widetilde{\delta}_1(z_b) - \widetilde{\delta}_0(0). 
	    \end{split}
	\end{equation*}
\end{defi}

The above computations lead to the following result.

\begin{lem}\label{lem:solvability-Prandtl}
	Let $(g, \widetilde{\delta}_0,\widetilde{\delta}_1)\in L^2_x H^1_z\times H^3(\Sigma_0)\times H^3(\Sigma_1) $, $w_t,w_b\in H^2(x_0, x_1) $ such that
	$w_t(x_0)=\p_z \widetilde{\delta}_0(z_t)$, $w_b(x_1)=\p_z \widetilde{\delta}_1(z_b)$. 
	
	Then system \eqref{eq:tY-model} has a solution $\tY \in H^{2/3}_x H^1_z\cap L^2_x H^3_z$ if and only if
	\begin{equation*}
	\ell^2\left(g,\widetilde{\delta_0}, \widetilde{\delta_1}, w_t,w_b\right)=0.
	\end{equation*}
	This solution is given by \eqref{def:tY}, and satisfies the estimate
	\begin{equation*}
	\| \tY\|_{H^{2/3}_x H^1_z} + \| \tY\|_{ L^2_x H^3_z} + \| \p_x\tY\|_{L^2_x H^1_z(\{z < z_b/2\})} \lesssim \|g\|_{L^2_x H^1_z} + \| \widetilde{\delta}_i\|_{H^3(\Sigma_i)} + \| w_j\|_{H^2(x_0,x_1)},
	\end{equation*}
	where we implicitely sum over $i\in \{0,1\}$ and $j\in \{t,b\}$ in the right-hand side.
\end{lem}

\begin{proof}
	First, assume that \eqref{eq:tY-model} has a solution $\tY \in H^{2/3}_x H^1_z\cap L^2_x H^3_z$. 
    Then, $W=\p_z \tY$ is an $L^2_x H^1_z$ solution to \eqref{vorticity-model} with $h = \p_z g$ and $w_i = \p_z \widetilde{\delta}_i$.
    By uniqueness arguments such as in \cref{lem:uniqueness-BG}, it is equal to the unique $Z^0$ solution to \eqref{vorticity-model} constructed in \cref{lem:WP-Z0-vorticity}.
	Furthermore, for $z\neq 0$,
	\begin{equation*}
	z^2 \p_z\left( \frac{\int_{z_b}^{z} \p_x \tY}{z}\right) = g + \p_z(\alpha \p_z \tY) - \beta \p_z \tY \in L^2_x H^1_z.
	\end{equation*}
	It follows that $\p_z\left( \frac{\int_{z_b}^{z} \p_x \tY}{z}\right)\in L^2_x H^1_z(\{z < z_b/2\})$, and thus $\p_x \tY\in L^2_x H^1_z(\{z < z_b/2\})$. 
    In particular, $\p_x \tY\vert_{z=z_b} \in L^2(x_0,x_1)$.

	Taking the trace of \eqref{eq:tY-model} at $z=z_b$, we infer that
	\begin{equation*}
	z_b\p_x \tY\vert_{z=z_b} + (\beta-\p_z \alpha(x,z_b)) w_b -\alpha \p_z W(x,z_b)= g(x,z_b)\quad \text{and}\quad \tY(x_1,z_b)=\widetilde{\delta_1}(z_b).
	\end{equation*}
	Therefore $\tY\vert_{z=z_b}=\widetilde{\gamma_b}$, where $\widetilde{\gamma_b}$ is defined by \eqref{def:gamma_b}.
	Since $\tY(x_0,0)=\widetilde{\delta_0}(0)$, we then deduce that
	\begin{equation*}
	\tY(x_0, z_b) + \int_{z_b}^0 W(x_0,z)\dd z=\widetilde{\delta_0}(0),
	\end{equation*}
	which is precisely the condition $\ell^2(g,\widetilde{\delta_0}, \widetilde{\delta_1}, w_t,w_b)=0$.
	
	Conversely, the above computations ensure that if $\ell^2(g,\widetilde{\delta_0}, \widetilde{\delta_1}, w_t,w_b)=0$, the function defined by \eqref{def:tY} is a solution to \eqref{eq:tY-model}.
\end{proof}

Assume that $\ell^2(g,\widetilde{\delta_0}, \widetilde{\delta_1}, w_t,w_b)=0$, and let $\tY \in H^{2/3}_x H^1_z\cap L^2_x H^3_z$ be the unique solution to \eqref{eq:tY-model}.
For further purposes,  we define the function $\widetilde{\gamma_t}$ by
\begin{equation*}
\widetilde{\gamma_t}(x):= \tY (x,z_t) .
\end{equation*}
Since $\tY\in C^0(\overline{\Om})$, we have
$\widetilde{\gamma_t}(x_0) = \widetilde{\delta_0}(z_t).$

\begin{rmk}
    As we already mentioned, the nonlocal term $v\p_y u$ in the Prandtl equation (which becomes $-\int_{z_b}^z \p_x Y$ in our new variables) creates a flow of information upwards, therefore inducing an asymmetry between $z$ and $-z$. 
	Because of the forward-backward nature of the equation, this results in an asymmetry between the lateral boundaries $\Sigma_0$ and $\Sigma_1$.
    We deal with this issue by introducing an additional orthogonality condition.

    Note that this feature is also present, in a slightly different fashion, in the work of Iyer and Masmoudi \cite{IM2022,IM2023}.
    In their work, the left extremity of the curve $\{ u = 0 \}$ is left as a free parameter, and boundary data on the vorticity are enforced.
\end{rmk}

In order to simplify the future discussion, it will be useful to modify slightly the definition of the linear forms $\ell^i$ for $i \in \{0,1\}$, so that they are defined on the same space as the linear form $\ell^2$.

\begin{defi}
    \label{def:ell-01-Prandtl}
	We denote by ${\ell^i}$ for $i \in \{0,1\}$ the linear forms defined by\nomenclature[OLl01]{$\ell^0,\ell^1$}{Linear orthogonality conditions of \cref{def:ell-01-Prandtl} for the solvability of Prandtl at high regularity}
	\begin{equation*}
	{\ell^i}(g, \widetilde{\delta_0}, \widetilde{\delta_1}, w_t,w_b) := \widehat{\ell^i}(\p_z g, \p_z \widetilde{\delta_0}, \p_z \widetilde{\delta_1}, w_t,w_b).
	\end{equation*}
\end{defi}

\begin{rmk}
	In spite of their similar appearance, the purpose of the orthogonality conditions $\ell^0(g, \widetilde{\delta_0}, \widetilde{\delta_1}, w_t,w_b)=\ell^1( g,  \widetilde{\delta_0},  \widetilde{\delta_1}, w_t,w_b)=0$ on the one hand, and $\ell^2(g, \widetilde{\delta_0}, \widetilde{\delta_1}, w_t,w_b)=0$ on the other hand is quite different. 
    The former are necessary and sufficient conditions for the existence of smooth solutions to the vorticity equation \eqref{vorticity-model}, while the latter is a necessary and sufficient condition for the solvability of system \eqref{eq:tY-model} at a lower level of regularity, corresponding to $Z^0$ solutions of the vorticity equation \eqref{vorticity-model}. 
    In other words, the condition $\ell^2 =0$ is a necessary and sufficient condition to reconstruct $\tY$ from the vorticity.
\end{rmk}

\begin{lem}
	The linear forms ${\ell^0}$, ${\ell^1}$, $\ell^2$ are linearly independent on $C^\infty(\overline{\Om})\times C^\infty_c (\Sigma_0)\times C^\infty_c(\Sigma_1) \times C^\infty_c (x_0,x_1)^2$.
    There exist $\Xi^0, \Xi^1, \Xi^2$ such that, for $i,j\in \{0,1,2\}$,
	\begin{equation*}
	{\ell^i}(\Xi_j)=\delta_{i,j},\qquad \Xi^j\in C^\infty(\overline{\Om})\times C^\infty_c (\Sigma_0)\times C^\infty_c(\Sigma_1) \times C^\infty_c (x_0,x_1)^2.
	\end{equation*}
	One may choose $\Xi^j=(f^j, 0,0,0,0)$, with $f^j\in C^\infty(\overline{\Om})$ such that $f^j\vert_{\Sigma_0\cup \Sigma_1}=0$.
	\label{lem:indep-ortho-Prandtl-2}
\end{lem}

\begin{proof}
	Assume that there exists $(c_0,c_1,c_2)\in \R^3$ such that
	\begin{equation*}
	c_0 \ell^0 + c_1 \ell^1+ c_2 \ell^2=0.
	\end{equation*}
	Let $W\in C^\infty_c(\overline{\Om})$ such that $\supp W \subset [x_0, x_0+\delta] \times [-\delta,  -\delta/2] $ for some small $\delta>0$ such that $\delta<(x_1-x_0)/2$ and $\delta<|z_b|/2$. We further assume that $\int_{z_b}^0 W(x_0,z)\dd z=1$ and $\int_{z_b}^0 z\p_x W(x_0,z)\dd z=0$.
	We set $w_t=w_b=0$, $\widetilde{\delta_0}=\widetilde{\delta_1}=0$, and
	\begin{equation*}
	f^2(x,z):=\int_{z_b}^z \left(z'\p_x W (x,z')+ \beta (x)\p_z W(x,z') - \p_z^2(\alpha(x,z') W(x,z')\right)\dd z'.
	\end{equation*}
	Then by definition, $W$ is a solution to \eqref{vorticity-model} with $h=\p_z f^2$, and with homogeneous boundary data. Note also that $f^2(x_0,0)= \int_{z_b}^0 z\p_x W(x_0,z)\dd z=0$. Therefore $f^2\vert_{\Sigma_0\cup \Sigma_1}=0$. The compatibility conditions from \cref{lem:ortho-Prandtl} are satisfied.
	Since $W$ is smooth, according to \cref{lem:ortho-Prandtl},
	\begin{equation*}
	\widehat{\ell^0}(\p_z f^2, 0, 0,0,0)=\widehat{\ell^1}(\p_z f^2, 0,0,0,0)=0.
	\end{equation*}
	Hence
	\begin{equation*}
	c_2 \ell^2(f^2,0,0,0,0)=0.
	\end{equation*}
	Now, by definition of $\ell^2$ and $f^2$, since $W$ and $f^2$ are identically zero  for $z\leq -\delta$,
	\begin{equation*}
	\ell^2(f^2, 0,0,0,0)=\int_{z_b}^0 W= 1.
	\end{equation*}
	We infer that $c_2=0$. The result then follows from \cref{lem:indep-ortho-Prandtl-1}, taking $f^i=\int_0^z g^i$ for $i=0,1$.
\end{proof}

Gathering the results of \cref{lem:ortho-Prandtl} and \cref{lem:solvability-Prandtl}, we obtain the following statement:

\begin{coro}
    \label{coro-regul-Prandtl}
    Let $\alpha \in C^4(\overline{\Om})$ satisfying \eqref{ellip-alpha} and $\beta \in C^1(x_0,x_1)$.

    \begin{itemize}
        \item Let $\sigma\in (0,1/6)$, and let $g\in H^\sigma_x H^1_z$ such that  $(\p_z g, \p_z \widetilde \delta_0, \p_z \widetilde \delta_1, w_t,w_b)\in \cH_{\alpha,\beta}^\sigma$ defined in \eqref{def:H-alpha-beta}.

        Then \eqref{eq:tY-model} has a  solution $\tY\in H^{\frac{2}{3} + \sigma}_x H^1_z \cap H^\sigma_x H^3_z$ if and only if $\ell^2(g, \widetilde{\delta_0}, \widetilde{\delta_1}, w_t,w_b)=0$, and this solution, if it exists, is unique and satisfies the estimate
        \begin{equation*}
            \begin{split}
           & \|\tY\|_{H^{\frac{2}{3} + \sigma}_x H^1_z \cap H^\sigma_x H^3_z} + \| (x-x_0)(x-x_1)\p_x \p_z^3 \tY\|_{L^2} + \|\p_x \p_z \tY\|_{L^2((x_0,x_1)\times (z_b,z_b/2))} \\
           & \quad \lesssim \|g\|_{H^\sigma_x H^1_z} + \| (x-x_0)(x-x_1)\p_x \p_z g\|_{L^2} + \| w_j \|_{H^2_x}  + \|\widetilde \delta_i\|_{H^3(\Sigma_i)}.
            \end{split}
        \end{equation*}
    
        \item Let $g\in H^1_xH^1_z$, and assume that $(\p_z g, \p_z \widetilde \delta_0, \p_z \widetilde \delta_1, w_t,w_b)\in \cH_{\alpha,\beta}^1$ defined in \eqref{def:H-alpha-beta}.
    
        Assume that $\ell^2(g, \widetilde{\delta_0}, \widetilde{\delta_1}, w_t,w_b)=0$, and let $\tY\in H^{\frac{2}{3} }_x H^1_z \cap L^2_x H^3_z$ be the unique solution to~\eqref{eq:tY-model}. 
        
        Then $\tY \in H^{5/3}_x H^1_z \cap H^1_x H^3_z$ if and only if $\ell^j(g, \widetilde{\delta_0}, \widetilde{\delta_1}, w_t,w_b)=0$ for $j\in \{0, 1\}$, and in this case $\tY$ satisfies the estimate
        \begin{equation*}
            \| \tY\|_{ H^{5/3}_x H^1_z} + \| \tY \|_{ H^1_x H^3_z}
            \lesssim \| g\|_{H^1_x H^1_z}+ \|(\p_z g, \p_z \widetilde \delta_0, \p_z \widetilde \delta_1, w_t,w_b)\|_{\cH_{\alpha,\beta}^1}.
        \end{equation*}
    \end{itemize}
\end{coro}

\begin{rmk}
    The regularity assumptions on $g$ in the first (resp.\ second) statement of the above corollary can be relaxed into $\p_z g\in H^\sigma_x L^2_z$, $(x-x_0)(x-x_1)\p_x \p_z g\in L^2$ and $g\vert_{z=z_b}\in L^2(x_0,x_1)$ (resp.\ $\p_z g\in H^1_x L^2_z$ and $g\vert_{z=z_b}\in H^{2/3}(x_0,x_1)$), but we have kept the above assumptions for the sake of simplicity.
\end{rmk}
 
\subsection{Local nonlinear well-posedness in the new variables}
\label{sec:Prandtl-abstract}
	
We are now ready to prove \cref{prop-Prandtl-Y}. 
The spirit of the proof is very similar to the one of \cref{sec:Burgers}.
In order to avoid repetition, we do not write the iterative scheme, and we rather apply \cref{thm:abstract} directly.
We will work with two different settings:
\begin{enumerate}
    \item \emph{Low regularity setting:} for $\sigma\in (0, 1/6)$ fixed, we take \nomenclature[FZps]{$\cZ^\sigma$}{Solution space for Prandtl at low regularity}\nomenclature[FHps]{$\cLin^\sigma$}{Space of data for Prandtl at low regularity}
    \begin{align*}
        \cZ^\sigma=& \big\{ Y\in H^{\frac{2}{3}+\sigma} _x H^1_z \cap H^\sigma_x H^3_z, \ (x-x_0)(x-x_1)\p_x \p_z^3 Y \in L^2(\Om),\\
        &\qquad \p_x \p_z Y\in L^2((x_0,x_1)\times (z_b,z_b/2)), \\
        & \qquad Y\vert_{\Sigma_i}\in H^3(\Sigma_i),\ \p_z Y\vert_{z=z_j}\in H^2(x_0,x_1)\big\},\\
        \cLin^\sigma=&\Big\{ (f, \widetilde{\delta}_0,  \widetilde{\delta}_1, w_t,w_b)\in H^\sigma_x H^1_z \times H^3(\Sigma_0)\times H^3(\Sigma_1) \times H^2(x_0,x_1)\times H^2(x_0,x_1),\\
        &\qquad (x-x_0)(x-x_1)\p_x \p_z f\in L^2,\\
        &\qquad 
        w_t(x_0)=\p_z \widetilde \delta_0(z_t),\ w_t(x_0)=\p_z \widetilde \delta_0(z_t) \Big\},
   \end{align*}
   and our space of data is the space $\cX^\sigma$ defined in \eqref{def:X-sigma-prandtl}.
    Furthermore, in the low regularity setting, $d=1$ and the linear form $\ell$ coincides with the linear form $\ell^2$ defined in \cref{def:ell-2-Prandtl}.
    
    \item \emph{High regularity setting:} we take\nomenclature[FZp1]{$\cZ^1$}{Solution space for Prandtl at high regularity}\nomenclature[FHp1]{$\cLin^1$}{Space of data for Prandtl at high regularity}
    \begin{align*}
        \cZ^1=& \big\{ Y\in H^{\frac{5}{3}} _x H^1_z \cap H^1_x H^3_z ,\quad
         Y\vert_{\Sigma_i}\in H^5(\Sigma_i),\ \p_z Y\vert_{z=z_j}\in H^2(x_0,x_1)\big\},\\
        \cLin^1=&\Big\{ (f, \widetilde\delta_0, \widetilde \delta_1, w_t,w_b)\in H^1_x H^1_z \times H^5(\Sigma_0)\times H^5(\Sigma_1) \times H^2(x_0,x_1)\times H^2(x_0,x_1),\\
        &\qquad \p_z^k \widetilde\delta_i(0)=0\quad \forall k\in \{0,\cdots, 3\},\quad
         z^{-1}\p_z f(x_i,z)\in \mathscr H^1_z(\Sigma_i),\\
&\qquad w_t(x_0)=\p_z \widetilde \delta_0(z_t),\ w_t(x_0)=\p_z \widetilde \delta_0(z_t), \\
&\qquad \Delta_0(z_t)=\p_x  w_t(x_0),\ \Delta_1(z_b)=\p_x w_b[\delta_b](x_1),\big\}
\end{align*}
where
\begin{equation*}
\Delta_i(z)=\frac{1}{z}\p_z \left[f(x_i,z) + \p_z(\alpha(x_i,z)\p_z \widetilde{\delta}_i )- \beta(x_i)\p_z \widetilde{\delta}_i \right].
\end{equation*}
        Note that 
        \begin{equation*}
        (f, \widetilde\delta_0, \widetilde \delta_1, w_t,w_b)\in\cLin^1\Rightarrow (\p_z f, \p_z \widetilde\delta_0, \p_z\widetilde \delta_1, w_t,w_b)\in \cH_{\alpha,\beta}^1,
        \end{equation*}
where the space  $\cH_{\alpha,\beta}^\sigma$ for $\sigma \in (0,1]$ is defined in  \eqref{def:H-alpha-beta}.
             Our space of data is the space $\cX^1$ defined in \eqref{def:X-1-prandtl}.
        In the high regularity setting, we take $d=3$ and $\ell=(\ell^0,\ell^1,\ell^2)$ defined in \cref{def:ell-01-Prandtl} and \cref{def:ell-2-Prandtl}.
\end{enumerate}

\begin{rmk}
    As in the previous sections, in $\cX^1$, we could also consider source terms $f$ which do not vanish on $\Sigma^P_i$, up to additional technical complications.
\end{rmk}

In both settings, the linear operator $L_P$ is defined as
\begin{equation*}
    L_P \tY := \left( z \p_x \tY - \int_{z_b}^z \p_x \tY + \beta \p_z \tY - \p_z(\alpha \p_z \tY) , \tY\vert_{\Sigma_0}, \tY\vert_{\Sigma_1}, \p_z \tY\vert_{z=z_t}, \p_z \tY\vert_{z=z_b}\right),
\end{equation*}
and the nonlinearity $N$ is defined as
\begin{align*}
    N(\Xi, \tY) & := (N_P(\Xi, \tY), \BCP^0[\delta_0] , \BCP^1[\delta_1],\BCP^t[\delta_t],\BCP^b[\delta_b]), \\
    N_P(\Xi,\tY) & := f(x, \Yfs - \tY) \p_z (\Yfs - \tY) - v_b + \p_z \left(\frac{(\p_z \tY)^2}{(\p_z \Yfs)^2(\p_z \Yfs - \p_z \tY)}\right),
\end{align*}
\nomenclature[OLNP]{$N_P$}{Nonlinearity associated with the Prandtl system}%
where the operators $\BCP^i$, $\BCP^j$ for $i\in \{0,1\}$, $j\in \{t,b\}$ are defined in \eqref{cond:tY-Sigma-i} and \eqref{cond:tY-top-bottom} respectively.

Let us now check that the assumptions of \cref{thm:abstract} are satisfied in the two settings. 
The continuity of $L_P$ from $\cZ^\sigma$ to $\cLin^\sigma$ for $\sigma\in (0,1/6)\cup\{1\}$ is a consequence of the definition of the spaces~$\cZ^\sigma$.
\cref{item:abs-1}) follows from \cref{coro-regul-Prandtl}. 
Furthermore,
\begin{equation*}
N(\Xi, 0)= \left( f(x, \Yfs ) \p_z \Yfs - v_b , \BCP^0[\delta_0] , \BCP^1[\delta_1],\BCP^t[\delta_t],\BCP^b[\delta_b]\right).
\end{equation*}
Hence it is easily checked that $N(\cdot, 0)$ is differentiable at $\Xi=0$, and its (partial) differential is given by
\begin{multline*}
    \p_\Xi N(0,0)(\Xi)= \Big( f(x, \Yfs ) \p_z \Yfs - v_b , \p_z \Yfs (x_0,z) \delta_0(\Yfs(x_0,z)), \p_z \Yfs (x_1,z) \delta_1(\Yfs(x_1,z)),\\ (\p_z \Yfs(x, z_t))^2 \delta_t(x),(\p_z \Yfs(x, z_b))^2 \delta_b(x) \Big).
\end{multline*}
As a consequence, $N_P(\Xi, \tY)- N_P(\Xi', \tY') - \p_\Xi N_P(0,0)(\Xi-\Xi')$ is 
\begin{equation}
    \label{NL-source}
    \begin{split}
        \p_z \Yfs & \left[ \left(f(\cdot,\Yfs - \tY) - f(\cdot, \Yfs)\right) -\left(f'(\cdot,\Yfs - \tY') - f'(\cdot, \Yfs)\right)   \right]\\
        &- \p_z \tY (f(\cdot, \Yfs - \tY)-f'(\cdot, \Yfs - \tY')) -\p_z(\tY-\tY') f'(\cdot, \Yfs - \tY')\\
        &+ \p_z \left(\frac{(\p_z \tY)^2}{(\p_z \Yfs)^2(\p_z \Yfs - \p_z \tY)}\right)- \p_z \left(\frac{(\p_z \tY')^2}{(\p_z \Yfs)^2(\p_z \Yfs - \p_z \tY')}\right).
    \end{split}
\end{equation}

We therefore turn towards the verification of \cref{item:abs-2}) and \cref{item:abs-3}) from \cref{thm:abstract}.

\paragraph{Verification of \cref{item:abs-2}) in the low regularity setting.}

For $\sigma\in (0,1/6)$, let $\Xi,\Xi'\in \cX^\sigma$, and $\tY,\tY'\in \cZ^\sigma$ small enough.
We need to estimate \eqref{NL-source} in $H^\sigma_x H^1_z\cap H^1_x H^1_z ((x-x_0)^2(x-x_1)^2)$.

In order not to burden the proof, we only estimate some of the norms above, and leave the other estimates to the reader. We focus for instance on
\begin{equation*}
\left\| \p_z \Yfs \p_z\left[ \left(f(\cdot,\Yfs - \tY) - f(\cdot, \Yfs)\right) -\left(f'(\cdot,\Yfs - \tY') - f'(\cdot, \Yfs)\right)   \right]\right\|_{H^\sigma_x L^2_z}.
\end{equation*}
Using \cref{lem:prod-hs} in the Appendix, we bound this term by
\begin{equation*}
    \begin{split}
        \| (\p_z \Yfs )^2 & \|_{L^\infty_z(H^{\frac{1}{2} + \sigma}_x)} \Big( \| \p_y (f-f')(x, \Yfs -\tY ) - \p_y (f-f') (x,\Yfs)\|_{H^\sigma_x L^2_z}\\&\qquad\qquad\qquad\qquad\qquad+ \| \p_y f'(x, \Yfs - \tY ) - \p_y f'(x, \Yfs- \tY')\|_{H^\sigma_x L^2_z} \Big)\\
        &+ \| \p_z \Yfs \p_z \tY  \|_{L^\infty_z (H^{\frac{1}{2} + \sigma}_x )} \| \p_y (f-f') (x, \Yfs - \tY) \|_{H^\sigma_x L^2_z}\\
        &+ \| \p_z \Yfs \p_z \tY  \|_{L^\infty_z (H^{\frac{1}{2} + \sigma}_x )} \| \p_y f'(x, \Yfs - \tY) -\p_y f'(x, \Yfs - \tY') \|_{H^\sigma_x L^2_z}\\
        &+ \| \p_z \Yfs (\p_z \tY - \p_z \tY')\|_{L^\infty_z (H^{\frac{1}{2} + \sigma}_x )} \|  \p_y f' (x, \Yfs - \tY')\|_{H^\sigma_x L^2_z}.
    \end{split}
\end{equation*}
Using the fractional trace theorem \cite[Equation (4.7), Chapter 1]{lions1969quelques}, $\cZ^\sigma \hookrightarrow C^1_z (H^{\frac{1}{2} + \sigma}_x)$. Furthermore, since $L^\infty_z (H^{\frac{1}{2} + \sigma}_x )$ is an algebra, 
\begin{equation*}
    \begin{aligned} 
     \| \p_z \Yfs \p_z \tY  \|_{L^\infty_z (H^{\frac{1}{2} + \sigma}_x )} & \lesssim \|\tY\|_{\cZ^\sigma},\\
     \| \p_z \Yfs (\p_z \tY - \p_z \tY')\|_{L^\infty_z (H^{\frac{1}{2} + \sigma}_x )} & \lesssim \|\tY-\tY'\|_{\cZ^\sigma}.
    \end{aligned}
\end{equation*}
There remains to estimate the norms involving $f$ and $f'$. Using \cref{lem:composition-bis} in the Appendix, we infer that
\begin{equation*}
    \begin{split}
        \| \p_y (f-f')(x, \Yfs - \tY ) & - \p_y (f-f') (x,\Yfs)\|_{H^\sigma_x L^2_z} = \left\| \tY \int_0^1 \p_y^2 (f-f') (x, \Yfs - \tau \tY) \dd\tau \right\|_{H^\sigma_x L^2_z}
        \\ & \lesssim 
         \| \tY \|_{H^{2/3}_x H^1_z} \left(\| \p_y^2 (f-f') \|_{H^\sigma_x L^2_z} + \|\p_y^3 (f-f')\|_{L^4_x L^2_y}\right) 
        \\ & \lesssim \|\tY \|_{\cZ^\sigma} \|\Xi-\Xi'\|_{\cX^\sigma}.
    \end{split}
\end{equation*}
In a similar fashion,
\begin{align*}
    \|  \p_y f' (x, \Yfs - \tY')\|_{H^\sigma_x L^2_z}\lesssim &\|\Xi'\|_{\cX^\sigma},\\
    \| \p_y f'(x, \Yfs - \tY ) - \p_y f'(x, \Yfs- \tY')\|_{H^\sigma_x L^2_z} 
    \lesssim & \|\tY - \tY'\|_{\cZ^\sigma} \|\Xi'\|_{\cX^\sigma}.
\end{align*}
The other terms are evaluated in a similar way. For instance, using again \cref{lem:prod-hs} and the embedding $\cZ^\sigma \hookrightarrow C^1_z (H^{\frac{1}{2} + \sigma}_x)$,
\begin{equation*}
    \begin{split}
        \left\| \p_z^3 (\tY-\tY') \frac{\p_z \tY}{(\p_z \Yfs)^2 \p_z (\Yfs - \tY)}\right\|_{H^\sigma_x L^2_z}
        & \lesssim \| \p_z^3 (\tY-\tY')\|_{H^\sigma_x L^2_z} \left\| \frac{\p_z \tY}{(\p_z \Yfs)^2 \p_z (\Yfs - \tY)}\right\|_{L^\infty_z H^{\frac{1}{2} + \sigma}_x}\\
        & \lesssim \|\tY - \tY'\|_{H^\sigma_x H^3_z}\| \p_z \tY\|_{L^\infty_z H^{\frac{1}{2} + \sigma}_x}\\
        & \lesssim \|\tY - \tY'\|_{\cZ^\sigma}\| \tY\|_{\cZ^\sigma},
    \end{split}
\end{equation*}
and, using once again \cref{lem:composition-bis},
\begin{equation*}
    \begin{split}
        \Big\| \p_z^2 \tY (f(x, \Yfs - \tY) & - f(x, \Yfs - \tY')) \Big\|_{H^\sigma_x L^2_z}
        \\ & \lesssim \| \p_z^2 \tY\|_{H^\sigma_x H^1_z} \left\| f(x, \Yfs - \tY) - f(x, \Yfs - \tY')\right\|_{H^{\frac{1}{2}+\sigma}_x L^2_z}
        \\ & \lesssim \| \tY\|_{H^\sigma_x H^3_z} \| \tY - \tY'\|_{H^{\frac{1}{2}+\sigma}_x H^1_z}\left(\|\p_y f\|_{H^{\frac{1}{2}+\sigma}_x L^2_z} + \|\p_y^2 f\|_{L^\infty}\right).
    \end{split}
\end{equation*}

The estimate on the $H^1_x H^1_z ((x-x_0)^2(x-x_1)^2)$  norm follows from similar arguments and is left to the reader.

We then turn towards the estimation of the boundary terms.
\begin{itemize}
    \item For $i \in \{0,1\}$ and $\delta_i, \eta_i \in H^4(\Sigma_i^P)$, we obtain that
    \begin{equation*}
        \left\| \BCP^i[\delta_i] - \BCP^i[\eta_i] - \p_z \Yfs (x_i,z) (\delta_i-\eta_i)(\Yfs(x_i,z))\right\|_{H^3(\Sigma_i)}
        = o \left( \| \delta_i - \eta_i \|_{H^4(\Sigma_i^P)} \right).
    \end{equation*}
    The proof is similar to the one of \cref{lem:Upsilon-strong-diff} for the Burgers case, although slightly less technical because we only need a standard Sobolev estimate here, and slightly more technical because the reference flow is now $\ufs$ instead of the linear shear flow. 

    \item For $j \in \{t,b\}$ and $\delta_j, \eta_j \in H^2(x_0,x_1)$, we obtain that
    \begin{equation*}
        \left\| \BCP^j[\delta_j] - \BCP^j[\eta_j] - (\p_z \Yfs (x,z_j))^2(\delta_j-\eta_j)(x)\right\|_{H^2(x_0,x_1)} = o \left(\| \delta_j - \eta_j \|_{H^2(x_0,x_1)}\right).
    \end{equation*}
    The proof is immediate because the maps $\BCP^j$ defined in \eqref{cond:tY-top-bottom} are in fact of the form $\BCP^j[\delta_j](x) = h_j(x,\delta_j(x))$ where $h_j : (x_0,x_1)\times \R \to \R$ is a smooth function with $h_j(\cdot,0) = 0$.
\end{itemize}
Eventually, we conclude that
\begin{equation*}
    \|N(\Xi,\tY)-N(\Xi',\tY') - \p_\Xi N(0,0)(\Xi-\Xi')\|_{\cLin^\sigma} = o \left( \|\Xi -\Xi'\|_{\cX^\sigma} + \|\tY-\tY'\|_{\cZ^\sigma} \right).
\end{equation*}

\paragraph{Verification of \cref{item:abs-2}) in the high regularity setting.}

The estimates in this case are similar to the low regularity setting and left to the reader. 
They are actually slightly easier since $H^1(x_0,x_1)$ is an algebra, and close to the ones performed for the Burgers system.

The only new estimate bears on the boundary term. More precisely, taking two data tuples $\Xi=(f, \delta_0, \delta_1, \delta_t, \delta_b, v_b)$ and $\Xi'=(f', \eta_0, \eta_1, \eta_t, \eta_b, v_b')$, we need to bound in $\mathscr H^1_z(\Sigma_i)$ the quantity 
\begin{equation*}
    z^{-1}\Big[\p_z (N_P(\Xi,\tY) - N_P(\Xi',\tY') -\p_\Xi N_P(0,0)(\Xi-\Xi'))\vert_{\Sigma_i} \Big].
\end{equation*}
We recall that $f\vert_{\Sigma_i^P}= f'\vert_{\Sigma_i^P}=0$, so that the terms stemming from $f$ and $f'$ in $N_P$ vanish on the boundary.
We therefore consider
\begin{equation}\label{est:bord-Prandtl-3}
    \left\| z^{-1} \p_z^2 \left[ \frac{(\p_z \BCP^i[\delta_i])^2}{(\p_z \Yfs)^2 (\p_z \Yfs - \p_z  \BCP^i[\delta_i])} -  \frac{(\p_z \BCP^i[\eta_i])^2}{(\p_z \Yfs)^2 (\p_z \Yfs - \p_z  \BCP^i[\eta_i])} \right]\right\|_{\mathscr H^1_z}.
\end{equation} 
Since $\p_z^k \delta_i(0)=0 $ for $k\in \{0,\dotsc,3\}$, we have $\p_z^k \BCP^i[\delta_i](z=0)=0 $.
According to \cref{lem:H1z-from-H2}, 
\begin{equation*}
    \begin{split}
        \eqref{est:bord-Prandtl-3} & \lesssim \left\| \frac{(\p_z \BCP^i[\delta_i])^2}{(\p_z \Yfs)^2 (\p_z \Yfs - \p_z  \BCP^i[\delta_i])} -  \frac{(\p_z \BCP^i[\eta_i])^2}{(\p_z \Yfs)^2 (\p_z \Yfs - \p_z  \BCP^i[\eta_i])} \right\|_{H^4(\Sigma_i)}\\
        & \lesssim \Big(  \| \delta_i\|_{H^6(\Sigma_i)} + \| \eta_i\|_{H^6(\Sigma_i)} \Big) 
             \| \eta_i-\delta_i\|_{H^6(\Sigma_i)}. 
    \end{split}
\end{equation*}
We obtain eventually
\begin{equation*}
    \|N(\Xi,\tY)-N(\Xi',\tY') - \p_\Xi N(0,0)(\Xi-\Xi')\|_{\cLin^1}
    = o \left( \|\Xi -\Xi'\|_{\cX^1} + \|\tY-\tY'\|_{\cZ^1} \right).
\end{equation*}

\paragraph{Verification of \cref{item:abs-3}) in the low regularity setting.}

We just need to check that the application $\ell^2\circ \p_\Xi N(0,0)$ is not identically zero. This is actually trivial: take $\Xi= (0,0,0,0,0, v_b)$. 
Then the solution to the vorticity equation is zero, and recalling \cref{def:ell-2-Prandtl},
\begin{equation*}
\ell^2\circ \p_\Xi N(0,0)= -\frac{1}{z_b}\int_{x_0}^{x_1} v_b(x)\dd x.
\end{equation*}
Therefore it suffices to choose $v_b$ such that the above integral is non-zero.

\paragraph{Verification of \cref{item:abs-3}) in the high regularity setting.}

Using \cref{lem:indep-ortho-Prandtl-2}, we take $\Theta^j=(f^j,0,0,0,0)$ such that $\ell^i(\Theta^j)=\delta_{i,j}$ for $0\leq i,j\leq 2$, with $f^i\in C^\infty(\overline{\Om})$ such that $f\vert_{\Sigma_0\cup \Sigma_1}=0$.
We then set $\Xi^j:=(g^j,0,0,0,0,0)$, where
\begin{equation*}
g^j(x,y):=\p_y \ufs(x,y) f^j(x,\ufs(x,y)).
\end{equation*}
Then, by design, $\p_{\Xi}N(0,0)(\Xi^j)=\Theta^j$, so that $\ell^j\circ\p_{\Xi}N(0,0)(\Xi^j)=\delta_{i,j}$.
Furthermore $\Xi^j\in \cX^1$. The result follows.

\paragraph{Conclusion.}

We have checked the assumptions of \cref{thm:abstract} both in the low regularity case $\sigma\in (0,1/6)$ and in the high regularity case $\sigma=1$.
\cref{prop-Prandtl-Y} is now a straightforward consequence of our abstract framework.

\subsection{Well-posedness of the Prandtl system}
\label{sec:WP-Prandtl}

We conclude this section with the proof of \cref{thm:prandtl}, which follows from \cref{prop-Prandtl-Y}.

\paragraph{High regularity case.}
The proof of \cref{thm:prandtl} in the high regularity case corresponding to $(f,\delta_0,\delta_1,\delta_t,\delta_b, v_b)\in \cM_1$ is very similar to the proof for Burgers carried out in \cref{sec:Burgers-main-proof}.
We leave it to the reader.
As in the Burgers case, one uses the equations satisfied by $u$ and $\tY$ to check that they actually also enjoy $L^2_x H^5_y$ regularity and one can prove a lemma similar to \cref{lem:inverse-qone} to prove that the formula $u(x,Y(x,z)) = z$ allows to transfer such a regularity back and forth.

\paragraph{Low regularity case.}
We focus on the case when $(f,\delta_0,\delta_1,\delta_t,\delta_b, v_b)\in \cM_\sigma$ with $\sigma\in (0,1/6)$, and we consider the unique solution $\tY\in \cZ^\sigma$ of \eqref{eq:Prandtl-changement-var}. 
Let $\Omega_P=\{(x,y)\in (x_0, x_1)\times \R, \ \gamma_b(x)< y < \gamma_t(x)\}$, where $\gamma_j(x)=Y(x, z_j)$.
For almost every  $x\in (x_0,x_1)$, $z\mapsto \Yfs(x,z) + \tY(x,z)$ is an $H^3$ diffeomorphism. We note that there exists a constant $\lambda>0$ such that $\lambda^{-1}\leq \p_z Y \leq \lambda$ in $\Om$.
Let us define the reverse change of variables $u$ such that $u(x,(\Yfs+\tY)(x,z))=z$. Classical results ensure that for a.e. $x$, $u(x,\cdot)\in H^3_y$. Furthermore, differentiating the formula  \eqref{eq:CDV-u-Y}, we obtain
\begin{equation*}
\p_y^3 u(x,Y(x,z))= -\frac{\p_z^3 Y(x,z)}{(\p_z Y(x,z))^4} + 3 \frac{(\p_z^2 Y(x,z))^2 }{(\p_z Y(x,z))^5},
\end{equation*}
which ensures that $\p_y^3 u \in L^2(\Omega_P)$. Since
\begin{equation*}
\p_y u(x,y)=\frac{1}{\p_z Y (x,u(x,y))},
\end{equation*}
we also infer that $\p_y u \in L^\infty$ and $\lambda^{-1}\leq \p_y u\leq \lambda$ for some $\lambda >0$.

Additionally, since $(x-x_0)(x-x_1)Y\in H^1_x H^3_z$, we also infer that $(x-x_0)(x-x_1)\p_y^k u (x, Y(x,z))\in H^1_x L^2_y(\Om_P)$ for $0\leq k \leq 3$. From there, we deduce that $(x-x_0)(x-x_1)\p_x\p_y^k u\in L^2(\Om_P)$ for   $0\leq k \leq 3$.
Furthermore, since $z\p_x \p_z Y\in L^2$, we also deduce that $u\p_x \p_y u \in L^2$.
Tracing back the computations at the beginning of \cref{sec:change-variables-Prandtl}, and noticing that $u\in H^1_x H^3_y (\omega)$ for all $\omega\Subset \Om_P$ as well as in the vicinity of $\overline{\Gamma_b}$, we infer that $u$ is a weak solution to the Prandtl system \eqref{prandtl}. This proves the existence of a solution to \eqref{prandtl}-\eqref{BC-Prandtl-compact}.
In order to prove the continuity of $u$, we  observe that for all $(x,y), (x',y')\in \Om_P$, setting $z=u(x,y)$,
\begin{equation*}
    \begin{split}
    |u(x,y)-u(x',y')| & \leq |u(x,y)-u(x',y)| + \|\p_y u\|_{\infty}|y-y'|\\
    &\leq | z - u(x', Y(x,z))| +  \|\p_y u\|_{\infty}|y-y'|\\
    &\leq |u(x', Y(x',z)) - u(x', Y(x,z))| +   \|\p_y u\|_{\infty}|y-y'| \\&\leq  \|\p_y u\|_{\infty}\left( |Y(x',z) - Y(x,z)| + |y-y'|\right).
    \end{split}
\end{equation*}
Since $Y\in H^{\frac{2}{3}}_x H^1_z\hookrightarrow C^\alpha$ for some $\alpha>0$, we infer that $u$ is Hölder continuous.

Let us now prove the uniqueness of this solution within the regularity class
\begin{equation*}
\begin{aligned} 
u\in L^2_x H^3_y(\Omega_P),\quad \p_y u \in L^\infty,\quad
(x-x_0)(x-x_1) u \in H^1_x H^3_y(\Omega_P),\quad
u\p_x \p_y u \in L^2(\Omega_P),
\end{aligned}
\end{equation*}
and assuming that $u$ is close to $\ufs$ in the associated norm. Note that this implies in particular that $\p_y u$ is bounded pointwise from above and below. The associated function $Y$ is such that $Y\in L^2_x H^3_z$, $\p_z Y \in L^\infty$, $(x-x_0)(x-x_1) Y \in H^1_x H^3_y$, and $z\p_x \p_z Y\in L^2$. 
In particular, $\p_z Y \in Z^0$. This regularity is sufficient to justify the computations of \cref{sec:change-variables-Prandtl}, and thus $\tY=Y-\Yfs$ is a solution to \eqref{eq:Prandtl-changement-var} in the sense of distributions.
It follows that $\p_z \tY$ is a solution to \eqref{eq:vorticity-Prandtl}, and $\p_zY$ is bounded pointwise from above and below by positive constants. From there, we deduce that $\p_z Y \in \cZ^\sigma$. Applying the first statement of \cref{prop-Prandtl-Y}, we deduce that $(f, \delta_0,\delta_1,\delta_b,\delta_t, v_b)\in \mathcal M_\sigma$.

Now, let $u_1,u_2$ be two solutions of \eqref{prandtl} within the above regularity class, corresponding to solutions $\tY_1,\tY_2$ of \eqref{eq:Prandtl-changement-var}. Let
\begin{equation*}
g_i:=f(x,Y_i)\p_z Y_i - v_b + \p_z \left(\frac{(\p_z \tY_i)^2}{(\p_z \Yfs)^2 \p_z Y_i}\right).
\end{equation*}
Then $W:=\p_z (\tY_1-\tY_2)\in Z^0$ is a solution to \eqref{vorticity-model} with homogeneous boundary data and with a source term $h=\p_z g_1-\p_z g_2$.
Therefore, multiplying the equation by $W$ and integrating by parts, we obtain
\begin{equation*}
\int \alpha |\p_z W|^2 \leq C (\|g_1-g_2\|_{L^2}^2 + \| W\|_{L^2}^2),
\end{equation*}
where the constant $C$ depends only on the underlying flow $\ufs$. As in the proof of \cref{lem:WP-Z0-vorticity}, for $|z_b|, z_t\leq z_0$, we infer that
\begin{equation*}
\| W\|_{L^2_x H^1_z}\lesssim \|g_1-g_2\|_{L^2}.
\end{equation*}
From there, using equation \eqref{vorticity-model}, we obtain
\begin{equation*}
\|W\|_{\cB} \lesssim \|g_1-g_2\|_{L^2}.
\end{equation*}
Using the formula for $g_i$ above, we deduce that
\begin{align*}
	\|g_1-g_2\|_{L^2} \lesssim & \|\p_y f \|_\infty \|Y_1-Y_2\|_{L^2} \|\p_z Y_1\|_\infty  + \|f\|_\infty \|\p_z(Y_1-Y_2)\|_{L^2}\\
	&+ \|\p_z \tY_1\|_\infty \|\p_z^2(Y_1-Y_2)\|_{L^2} + \|\p_z(Y_1-Y_2)\|_{L^\infty_z(L^3_x)}\|\p_z^2 \tY_2\|_{L^2_z(L^6_x)}
\end{align*}
Setting
\begin{equation*}
\eta:= \| \p_z \tY_1\|_{\infty} + \| \p_z \tY_2\|_{Z^0} + \|f\|_{L^\infty_x W^{1,\infty}_y},
\end{equation*}
and using the embeddings $Z^0 \hookrightarrow L^2_zH^{1/3}_x \hookrightarrow L^2_z(L^6_x)$, $\cB \hookrightarrow C^0_z([z_b,z_t]; H^{1/6}_x) \hookrightarrow L^\infty_z(L^3_x)$ (see \cref{lem:embed-cB-C0z-L3x}), we infer
\begin{equation*}
\|g_1-g_2\|_{L^2}\lesssim \eta \| W\|_{\cB}.
\end{equation*}
Hence we obtain $\|W\|_{\cB} \lesssim \eta \|W\|_{\cB}$, and provided $\eta$ is small enough, $W=0$.

\begin{rmk}
	Note that in the case $\sigma\in (0, 1/6)$, we are not able to transfer completely the fractional horizontal regularity from $Y$ to $u$. 
	Indeed, one can easily check from the formulas in \eqref{eq:CDV-u-Y} that $\p_y^k u (x, Y(x,z))\in H^{\frac{3-k}{3}+ \sigma}_x L^2_z\cap H^\sigma_x H^{3-k}_z$ for $k\in \{1,\cdots , 3\}$. Then, one may try to get some regularity on $u$ by computing
	\begin{equation*}
	\| \p_y u \|_{H^{\frac{2}{3} + \sigma}_x L^2_y}^2 = \| \p_y u \|_{L^2}^2 + \int_{x_0}^{x_1}\int_{x_0}^{x_1} \int_{\R} \mathbf 1_{(x,y)\in \Om_P}\mathbf 1_{(x',y)\in \Om_P} \frac{|\p_yu(x,y)-\p_y u(x',y)|^2}{|x-x'|^{\frac{7}{3} + 2 \sigma}}\:\dd x\:\dd x'\: \dd y.
	\end{equation*}
	It is quite natural to change variables in the second integral in the right-hand side by setting $y=Y(x,z)$, the associated jacobian being bounded from above and below, and to split the resulting integral into
	\begin{equation*}\begin{aligned}
	&\int_{x_0}^{x_1}\int_{x_0}^{x_1} \int_{z_b}^{z_t} \frac{|\p_yu(x,Y(x,z))-\p_yu(x',Y(x',z))|^2}{|x-x'|^{\frac{7}{3} + 2 \sigma}}\:\dd x\dd x'\dd z \\&+ \int_{x_0}^{x_1}\int_{x_0}^{x_1} \int_{z_b}^{z_t}\frac{|\p_yu(x',Y(x',z))-\p_yu(x',Y(x,z))|^2}{|x-x'|^{\frac{7}{3} + 2 \sigma}}\dd x\dd x'\dd z.\end{aligned}
	\end{equation*}
	The first integral above is bounded by $\|\p_y u (x, Y(x,z))\|_{H^{2/3+ \sigma}_x L^2_z}^2$. As for the second integral, if $\p_y u$ were Lipschitz continuous with respect to $y$ (or even Hölder continuous with some suitable exponent), we would bound this integral by $\|Y\|_{H^{2/3+ \sigma}_x L^2_z}^2$.
	But unfortunately, this Lipschitz regularity does not hold in general.
	However, thanks to the regularity result far from the lateral boundaries from \cref{lem:reg-away-sing-pts}, we have sufficient regularity on $u$ to ensure uniqueness.
\end{rmk}

\subsection{Potential strategy in a whole infinite strip}
\label{sec:strategy-Prandtl-bande-infinie}

In this paragraph, we sketch a potential strategy to solve the Prandtl equation \eqref{prandtl} in the whole infinite strip $(x_0,x_1)\times (0, +\infty)$, based on the previous analysis.
To that end, we first propose a scheme to solve a system with a modified source term (and without any orthogonality condition). Once the solvability of this modified system is understood, the solvability of the original system follows for data within a finite codimensional manifold.

We start from a smooth solution $(\ufs, \vfs)$ to \eqref{prandtl} such that $\ufs(x, \gfs(x))=0$ for some smooth function $\gfs$, and $\ufs(x,y)<0$ (resp.\ $\ufs(x,y)>0$) for $y\in (0, \gfs(x))$ (resp.\ for $y> \gfs(x)$). 
We also have the boundary conditions $\ufs\vert_{y=0}=\vfs\vert_{y=0}=0$, and $\ufs(x,y)\to u_\infty(x)$ as $y\to \infty$, where $u_\infty u_\infty' = -\p_x p$.
As before, we fix two small  numbers $z_b<0<z_t$, such that there exist smooth lines $\{y=\overline{\gamma_j}(x)\}$ with $\ufs(x, \overline{\gamma_j}(x))=z_j$.
We consider perturbations $\Theta:=(\delta_0, \delta_1, f)\in H^k(0, + \infty)\times H^k(0, +\infty) \times H^k((x_0, x_1)\times (0, +\infty))$ for some sufficiently large~$k$, and for simplicity, we also assume that $\delta_i$ vanishes at $\gfs(x_i)$.
We then define an application $\mathcal A: (\Theta; {\gamma_b}, \delta_t)\mapsto (\gamma_b', \delta_t')$ in the following way:
\begin{enumerate}
	\item We solve the Prandtl system in the domain $\{(x,y)\in (x_0, x_1)\times (0, +\infty), \ y<{\gamma_b}(x)\}$ in the vicinity of the flow $(\ufs, \vfs)$, with source term $-\p_x p + f$ and boundary data
	\begin{equation*}
	\begin{aligned} 
	u\vert_{x=x_1}= &\ufs\vert_{x=x_1} + \delta_1,\\
	u\vert_{y=0}=v\vert_{y=0}=&0,\\
	u\vert_{y={\gamma_b}(x)}=& z_b.
	\end{aligned}
	\end{equation*}
	In (the interior of) this domain, $\ufs<0$, and therefore the system is backward parabolic. Hence we expect that it is solvable (see \cite{MR1697762}). A possible way to solve it could be to introduce the ``von Mise type good unknown'' from \cite{IM2022}.
	
	Assuming that the above system is solvable, we set $v_b:=v\vert_{y={\gamma_b}} - \vfs\vert_{y={\gamma_b}}$, and $\delta_b:=\p_y u\vert_{y={\gamma_b}} - \p_y \ufs\vert_{y={\gamma_b}}$. 
	Note that there are typically compatibility conditions which are necessary to ensure the existence of smooth solutions of this system. We leave this issue aside in the present discussion. 
    The compatibility conditions are automatically ensured if $f$ is supported in $(x_0+\delta_x, x_1-\delta_x)$ for $|\delta_x|\ll 1$, and  if $\delta_0$ (resp.\ $\delta_1$) is compactly supported in $(\gamma_P(x_0), \overline{\gamma_t}(x_0))$ (resp.\ $(\overline{\gamma_b}(x_1), \gamma_P(x_1)$).
	
	\item We then consider the Prandtl system in the recirculating zone. More precisely, using the analysis of the previous subsections, we construct a solution to 
	\begin{equation*}
	\begin{aligned} 
	u u_x + v u_y - \p_{yy} u &= - \p_x p + f -\left( \nu^0 f^0 + \nu^1 f^1 + \nu^2 f^2\right)(x, u(x,y))\p_y u(x,y) \\
	u_x + v_y &=0,\\
	\end{aligned} 
	\end{equation*}
	together with the boundary conditions \eqref{CL-Prandtl-bottom}, \eqref{CL-Prandtl-top}, \eqref{CL-Prandtl-lateral}, in which the bottom data $v_b, \delta_b$ are provided by the first step.
	Note that the new free boundary $\{y=\gamma_b'(x):=Y(x,z_b)\}$, with the notations of the previous sections, is different from the boundary $\{y= {\gamma_b}(x)\}$ \emph{a priori}. The coefficients $(\nu^0,\nu^1, \nu^2)$ are Lipschitz functions of the data $(f, \delta_0,\delta_1)$ and ensure that the associated solution $u$ belongs to $H^{5/3}_x H^1_y \cap H^1_x H^3_y$. 
	Note that the structure of the right-hand side is designed so that the equation in the variables $(x,z)$ is
	\begin{equation*}
        \begin{split}
        	z\p_x Y - \int_{z_b}^z \p_x Y - \frac{1}{(\p_z Y)^2}\p_z^2 Y & = \left(\p_x p - f(x,Y)\right) \p_z Y + \vfs\vert_{y=\overline{\gamma_b}} \\
            & \quad + v_b + \nu^0 f^0 + \nu^1 f^1 + \nu^2 f^2.
        \end{split}
	\end{equation*}
	Let $\mathcal V(f, \delta_0, \delta_1; {\gamma_b}, \delta_t)$ denote the quantity $v\vert_{y=\gamma_t(x)}$, where $\gamma_t(x)=Y(x,z_t)$. 
    The boundary $\{y=\gamma_t(x)\}$ will be the lower boundary of the upper domain considered in the next step, but is not a variable of the implicit function argument.
	
	\item Eventually, we solve the Prandtl system in $\{(x,y)\in (x_0, x_1)\times (0, +\infty), \ y>{\gamma_t}(x)\}$ in the vicinity of the flow $(\ufs, \vfs)$, with source term $-\p_x p + f$ and boundary data
	\begin{equation*}
	\begin{aligned} 
	u\vert_{x=x_0}= &\ufs\vert_{x=x_0} + \delta_0,\\
	u\vert_{y=\gamma_t(x)}=&z_t,\\
	v\vert_{y=\gamma_t(x)}= &\mathcal V(f, \delta_0, \delta_1; {\gamma_b}, \delta_t),\\
	\lim_{y\to \infty } u(x,y)=&u_\infty(x).
	\end{aligned}
	\end{equation*}
	This system is now forward parabolic. It can be solved with the tools of \cite{MR1697762}. Note that $\inf \ufs>0$ in the upper domain, so that the system is in fact non degenerate after a suitable change of variables.
	We then define $\delta_t'=\p_y u\vert_{y=\gamma_t(x)}$.

\end{enumerate}
Eventually, we set $\mathcal A(\Theta; {\gamma_b}, \delta_t)= (\gamma_b', \delta_t')$.
The first question which needs to be solved is the following:
\begin{center}
	\emph{  For every $\Theta\in H^k(0, + \infty)\times H^k(0, +\infty) \times H^k((x_0, x_1)\times (0, +\infty))$ such that $\|\Theta\|\leq \delta$, \\Find $({\gamma_b}, \delta_t)$ such that $\mathcal A(\Theta; {\gamma_b}, \delta_t)=({\gamma_b}, \delta_t)$.}
\end{center}

For $\Theta=0$, by definition of the application $\mathcal A$, one has $\mathcal A(0; \overline{\gamma_b}, \p_y \ufs\vert_{y=\overline{\gamma_t}})= (\overline{\gamma_b}, \p_y \ufs\vert_{y=\overline{\gamma_t}})$. 
Hence a possible strategy could be to apply an implicit function theorem, in the spirit of \cite{dalibard2017nonlinear} or \cref{lem:TFI-Strong-Frechet}.
This requires to prove the invertibility of the function $d_{({\gamma_b}, \delta_t)} \mathcal A(0; \overline{\gamma_b}, \p_y \ufs\vert_{y=\overline{\gamma_t}}) - \mathrm{Id}$. In turn, this requires to prove the well-posedness of a linearized type Prandtl system 
(or of three coupled linearized Prandtl systems) in the infinite strip $(x_0, x_1)\times (0, +\infty)$.
Such a result may typically involve restrictions on the size of the domain, as the following toy example demonstrates.  Let $a\in L^\infty((x_0, x_1)\times \R)$. Consider the forward-backward system
\begin{equation*}
\begin{cases}
	z\p_x u - \p_{zz} u - au=0 &\text{ in }(x_0,x_1)\times \R,\\
	u(x_0,z)=0 &\text{ for } z>0,\\
	u(x_1,z)=0 &\text{ for } z<0.
\end{cases}
\end{equation*}
Let us assume that there exists a solution with high enough decay for $|z|\gg 1$; our purpose is to prove that such a  solution is identically zero. To that end, we multiply the above system by $u \rho$ where $\rho(x,z) := \exp(-(x-x_0)z/(x_1-x_0))$, and perform integrations by parts. We obtain
\begin{equation*}
    \begin{split}
        \frac{1}{2(x_1-x_0)}\int_{x_0}^{x_1}\int_\R & z^2 u^2 \rho \dd x\dd z+ \int_{x_0}^{x_1}\int_\R(\p_z u)^2 \rho \dd x\dd z
        \\&= - \int_{x_0}^{x_1}\int_\R a u^2 \rho \dd x\dd z+ \frac{1}{2}\int_{x_0}^{x_1}\int_\R \left( \frac{x-x_0}{x_1-x_0}\right)^2 u^2  \rho \dd x\dd z.
    \end{split}
\end{equation*}
For $|z|\geq 2(\|a\|_\infty +1)^{1/2} (x_1-x_0)^{1/2}  $, the two terms in the right-hand side can be absorbed in the left-hand side. On the other hand, for $|z| \leq 2(\|a\|_\infty +1)^{1/2} (x_1-x_0)^{1/2} $ and $|x_1-x_0|\leq 1$, the weight $ \exp\left(- \frac{x-x_0}{x_1-x_0} z\right)$ is bounded from above and below. We then use the inequality
\begin{equation*}
\| \phi\|_{L^2_z}\lesssim \|z\phi\|_{L^2_z} + \|\p_z \phi\|_{L^2_z}
\end{equation*}
for any $\phi\in H^1(\R)$ such that $z\phi\in L^2(\R)$. The proof of the inequality follows from arguments similar to the ones of \cref{lem:psi-zpsi-psi''} and is left to the reader. We infer that for $x_1-x_0$ small enough, the only decaying solution to the above system is $\phi\equiv 0$. For $x_1-x_0$ large, the situation is not so clear. These considerations could be seen as a toy example of why  Iyer and Masmoudi in \cite{IM2022} need to exclude a ``resonant set'' of lengths $x_1-x_0$ for which non trivial solutions of a system similar to the one above may exist.

\newpage
\section{Interpolation estimate for the linear shear flow problem}
\label{sec:interpolation}

\newcommand{\iX}{\mathcal{Y}}
\newcommand{\iXL}{\mathcal{Y}_1^{\overline{\ell}}}
\newcommand{\iN}{I}

This section is devoted to the proof of \cref{lem:shear-Hs}, which is used in particular in the construction of weak solutions for the Prandtl system (see \cref{lem:ortho-Prandtl}). 
The idea is to interpolate between the $Z^0$ estimate from \cref{p:pagani-shear}, and the $Z^1$ estimate from \cref{p:shear-WP-Z1}.
However, because of the orthogonality conditions, justifying that the interpolate space for the source terms is the expected one turns out to be quite complicated.

We introduce the following spaces for the source terms
\nomenclature[FY]{$\iX_0$}{Notation for $L^2(\Omega)$ during discussions on interpolation}
\nomenclature[FY]{$\iX_1$}{Notation for $H^1_x L^2_y$ with $f_{\rvert\Sigma_0\cup\Sigma_1} =0$ during discussions on interpolation}
\begin{align} 
	\label{eq:cX0}
	\iX_0 & := \{ f \in L^2(\Om) \}, \\
	\label{eq:cX1}
	\iX_1 & := \left\{ f \in H^1_x L^2_y ; \enskip f_{\rvert \Sigma_0 \cup \Sigma_1} = 0 \right\}
\end{align}
endowed with their usual norms and 
\begin{equation} \label{def:cal-Y}
\iXL:= \left\{ f \in \iX_1 ; \enskip \overline{\ell^0}(f,0,0)=\overline{\ell^1}(f,0,0)=0\right\},
\end{equation}
endowed with the norm of $\iX_1$, where $ \overline{\ell^0}$ and $ \overline{\ell^1}$ are defined in \cref{def:ell-shear}.
Since  $ \overline{\ell^0}$ and $ \overline{\ell^1}$ are continuous for the $H^1_x L^2_y$ norm, $\iXL$ is a closed subspace of $\iX_1$.

We wish to interpolate between $\iX_0$ and $\iXL$.
Using classical interpolation theory, one can determine $\iX_{\sigma} := \left[ \iX_0, \iX_1 \right]_{\sigma}$ quite easily (see \cref{lem:x-theta} below).
Nevertheless, there is a difficulty in the determination of the space $[\iX_0,\iXL]_{\sigma}$.
This corresponds to the well-known problem of ``subspace interpolation'', for which we give a short survey in \cref{sec:subspace-interpolation}.

The proof of \cref{lem:shear-Hs} relies on  a careful analysis of the dual profiles $\overline{\Phi^j}$, and in particular on a decomposition of the latter into an explicit singular part and a regular part. 
This decomposition allows us to have quantitative upper and lower bounds on the functions $\tau \mapsto \iN(\tau,\overline{\ell^j})$, which play a paramount role in interpolation theory (see \cite{MR1343130} and \cref{sec:lofstrom} below). 
	
The organization of this section is as follows. 
We start by introducing the theory of subspace interpolation, and associated notations in \cref{sec:subspace-interpolation}.
We then turn towards the proof of \cref{lem:shear-Hs} in \cref{sec:interpolation-shear}, illustrating how the general theory can be applied for our problem, thanks to the knowledge of the singular profiles of \cref{sec:radial}.

\subsection{A primer on subspace interpolation}
\label{sec:subspace-interpolation}

Using interpolation theory in a context where constraints are enforced on the data comes with a specific difficulty, known as ``subspace interpolation''.
In this subsection, we give a short introduction and set up notations and a lemma that will be used in the next subsections.

\subsubsection{An introduction to subspace interpolation}
\label{sec:subspace}

Let us start by a short introduction to the topic of subspace interpolation and the associated difficulty.
This difficulty is \emph{not} linked with the difference between complex and real interpolation methods.
Indeed, it occurs even in the case of ``quadratic'' interpolation between separable Hilbert spaces, for which all methods construct the same interpolation spaces (see \cite[Remark~3.6]{CWHM2015} and \cite[Section~3.3, item~(4)]{CWHM2022} based on the initial geometric argument of \cite{MR1289759}).

\paragraph{Setting of the problem}
Let $\iX_0$ and $\iX_1$ denote two Banach spaces with a dense continuous embedding $\iX_1 \hookrightarrow \iX_0$.
Let $\iX_\sigma := [\iX_0,\iX_1]_\sigma$, for $\sigma \in (0,1)$, say for the complex method to fix ideas.
Let $\ell$ be a continuous linear form on $\iX_1$, which is however unbounded on $\iX_0$, and define its kernel $\iX_1^\ell := \{ f \in \iX_1 ; \enskip \ell(f) = 0 \}$, which is a closed subspace of $\iX_1$.
The question of ``subspace interpolation'' consists in  determining the relation between $\iX_\sigma$ and $[\iX_0,\iX_1^\ell]_\sigma$.
This question of course admits a straightforward generalization to the case of a finite number of orthogonality conditions.

Generally, one checks that the closure of $[\iX_0,\iX_1^\ell]_\sigma$ in $\iX_\sigma$ is either a subspace of codimension~1, when $\ell$ is continuous on $\iX_\sigma$, or the whole of $\iX_\sigma$, when $\ell$ is unbounded on $\iX_\sigma$.
In the former case, there is no guarantee that $[\iX_0,\iX_1^\ell]_\sigma$ itself is closed in $\iX_\sigma$ (or, equivalently, that the associated norms are equivalent on $[\iX_0,\iX_1^\ell]_\sigma$).
The first systematic occurrence of this question seems to date back to \cite[Problem 18.5, Chapter 1]{LM68}, which claims that a major difficulty to use interpolation theory is that ``\emph{l'interpolé de sous-espaces fermés n'est pas nécessairement un sous-espace fermé dans l'interpolé}'' (the interpolation space between closed subspaces is not necessarily a closed subspace in the interpolation space), and asks for sufficient conditions for $[\iX_0,\iX_1^\ell]_\sigma$ to be closed in~$\iX_\sigma$.

\begin{rmk}
    When $\ell$ is continuous for the topology of $\iX_0$, there is no difficulty.
	Indeed, one checks that, for every $\sigma \in (0,1)$, $[\iX_0,\iX_1^\ell]_\sigma = \{ f \in \iX_\sigma ; \enskip \ell(f) = 0 \}$, endowed with the topology of $\iX_\sigma$, for which $\ell$ is continuous (see e.g.\ the related result \cite[Theorem 13.3, Chapter 1]{LM68}).
\end{rmk}

\paragraph{Some examples}
The best known and most simple example of such a phenomenon, introduced in \cite[Theorem 11.7, Chapter 1]{LM68} concerns the construction of the space $H^{1/2}_{00}(0,1)=[L^2(0,1),H^1_0(0,1)]_{1/2}$.
It is known that $H^{1/2}_{00}(0,1)$ is not closed in $H^{1/2}(0,1)$ and that the associated norm involves a non-equivalent ``additional term''.

In \cite{wallsten1988remarks}, using real interpolation between $L^1$ and $L^\infty$, Wallstén constructed examples illustrating that this pathological behavior is not limited to exceptional values of the interpolation parameter, since there exist constraints for which it occurs for every $\sigma \in (0,1)$.

\paragraph{Short survey of known results}
Precising earlier results of Löfström \cite{MR1343130,lofstrom},
Ivanov and Kalton proved in \cite{MR1834861} that, in the general case, there exist two thresholds $0 \leq \sigma_0 \leq \sigma_1 \leq 1$ such that:
\begin{itemize}
	\item when $0 < \sigma < \sigma_0$, $[\iX_0,\iX_1^\ell]_\sigma = \iX_\sigma$, with equivalent norms,
	\item when $\sigma_0 \leq \sigma \leq \sigma_1$, the norm on $[\iX_0,\iX_1^\ell]_\sigma$ is not equivalent to the one on $\iX_\sigma$,
	\item when $\sigma_1 < \sigma < 1$, $[\iX_0,\iX_1^\ell]_\sigma$ is a closed subspace of codimension 1 in $\iX_\sigma$.
\end{itemize}
In the first case, $\ell$ is unbounded on $\iX_\sigma$ (the constraint does not make sense).
In the second and third cases, $\ell$ admits a continuous extension to $\iX_\sigma$ and the closure of $[\iX_0,\iX_1^\ell]_\sigma$ in $\iX_\sigma$ is of codimension~1.

This classification has generalizations to the case of multiple constraints (see \cite{MR3300954}), potentially involving multiple pathological intervals, associated with each constraint.

In the difficult regime $\sigma_0 \leq \sigma \leq \sigma_1$, more precise results \cite{MR2387839,MR1964285} allow the computation of the ``additional norm'' stemming from the presence of the constraints.

The recent work \cite{MR4431930} considers a kind of dual problem, by computing interpolation spaces between $\iX_0$ and $\iX_1 \oplus \R \omega$, where $\omega$ is a singular function of $\iX_0 \setminus \iX_1$, whose singularity is expressed in polar coordinates.
In this work, $\sigma_0 = \sigma_1$.
This is also our case below, and our dual profiles also involve singular parts which are expressed in radial-like coordinates, as constructed in \cref{sec:radial}.

\subsubsection{A variant of a criterium due to Löfström}
\label{sec:lofstrom}

To prove \cref{lem:shear-Hs}, we will rely on an abstract interpolation result proved by Löfström in~\cite{MR1343130}.
Let $\iX_0$ and $\iX_1$ denote two Hilbert spaces with a dense continuous embedding $\iX_1 \hookrightarrow \iX_0$.

For $f \in \iX_1$ and $\tau \in (0,1)$, let
\begin{equation} \label{eq:norm-t}
	\| f \|_\tau^2 := \| f \|_{\iX_0}^2 + \tau^2 \| f \|_{\iX_1}^2.
\end{equation} 
This notation is actually inspired by \cite{MR1834861}, as \cite{MR1343130} uses instead the quantity $\max \left(\| f \|_{\iX_0}, \tau \| f \|_{\iX_1}\right)$.
Since $\|f\|_\tau / \sqrt{2} \leq \max \left(\| f \|_{\iX_0}, \tau \| f \|_{\iX_1}\right) \leq \|f\|_\tau$, both can be used equivalently.

Given a linear form $\ell$ on $\iX_1$, one defines, for $\tau \in (0,1)$,
\begin{equation} \label{eq:N}
	\iN(\tau,\ell) := \sup_{f \in \iX_1 \setminus \{ 0 \}} \frac{\ell(f)}{\|f\|_\tau}.
\end{equation}
As $\tau \to 0$, upper bounds on $\iN(\tau,\ell)$ are linked with the boundedness of $\ell$ on intermediate spaces between $\iX_0$ and $\iX_1$, while lower bounds on $\iN(\tau,\ell)$ are linked with the non-degeneracy of $\ell$ on these spaces.
In particular, one has the following result, which is a reformulation of \cite[Theorem~2]{MR1343130} in the particular case of two constraints having the same ``order''.

\begin{lem} \label{lem:lofstrom}
	Let $\iX_0$ and $\iX_1$ denote two Hilbert spaces with a dense continuous embedding $\iX_1 \hookrightarrow \iX_0$.
	Let $\ell^0$, $\ell^1$ be two linear forms on $\iX_1$.
	Assume that there exists $C_\pm > 0$ and $\bar{\sigma} \in (0,1)$  such that, for every $(c_0,c_1) \in \mathbb{S}^1$ and every $\tau \in (0,1)$,
	\begin{equation} \label{eq:lof-c0c1}
		C_- \tau^{-\bar{\sigma}} \leq \iN(\tau, c_0\ell^0 + c_1 \ell^1) \leq C_+ \tau^{-\bar{\sigma}}.
	\end{equation}
	As in \cref{sec:subspace}, let $\iX_1^\ell := \{ f \in \iX_1 ; \enskip \ell^0(f) = \ell^1(f) = 0 \}$ and, for $\sigma \in (0,1)$, $\iX_\sigma := [\iX_0,\iX_1]_\sigma$, for the complex interpolation method.
	Then,
	\begin{itemize}
		\item for every $\sigma \in (0,\bar{\sigma})$, $[\iX_0,\iX_1^\ell]_\sigma = \iX_\sigma$, with equivalent norms,
		
		\item for every $\sigma \in (\bar{\sigma},1)$, the linear forms $\ell^0$ and $\ell^1$ have continuous extensions to $\iX_\sigma$ and $[\iX_0,\iX_1^\ell]_\sigma = \{ f \in \iX_\sigma ; \enskip \ell^0(f) = \ell^1(f) = 0 \}$, endowed with the norm of $\iX_\sigma$.
	\end{itemize}
\end{lem}

\begin{rmk}
	\cref{lem:lofstrom} does not say anything on $[\iX_0,\iX_1^\ell]_{\sigma}$ for the critical value $\sigma = \bar{\sigma}$.
	In fact, with the notations of \cite{MR1834861} mentioned above, one has $\sigma_0 = \sigma_1 = \bar{\sigma}$, so the norm of $[\iX_0,\iX_1^\ell]_{\bar{\sigma}}$ is not equivalent to the norm of $\iX_{\bar{\sigma}}$.
\end{rmk}

\begin{rmk}
	In assumption \eqref{eq:lof-c0c1}, it is important to consider arbitrary linear combinations of the two linear forms $\ell^0$ and $\ell^1$.
	It would not be sufficient to assume \eqref{eq:lof-c0c1} with $(c_0,c_1) = (1,0)$ and $(c_0,c_1)=(0,1)$.
	Indeed, the lower bound of this condition ensures that the two linear forms remain sufficiently independent on the intermediate spaces.
	We state here a formulation giving a symmetrical role to $\ell^0$ and $\ell^1$, whereas \cite{MR1343130} uses a hierarchical formulation.
	We prove below that our formulation indeed implies Löfström's one.
\end{rmk}

\begin{proof}[Proof of \cref{lem:lofstrom}]
	This is an application of \cite[Theorem 2]{MR1343130}.
	By \eqref{eq:lof-c0c1} with $(c_0,c_1) = (1,0)$ and $(c_0,c_1)=(0,1)$, both $\ell^0$ and $\ell^1$ have ``order'' $\bar{\sigma}$ in Löfström's vocabulary.
	Therefore, there only remains to check that they form a ``strongly independent basis'', i.e.\ that there exists $C > 0$ such that, for every $\tau \in (0,1)$,
	\begin{equation} \label{eq:Nt-N0t}
		\iN(\tau,\ell^1) \leq C \iN_0(\tau,\ell^1),
	\end{equation}
	where
	\begin{equation} \label{def:N0t}
		\iN_0(\tau,\ell^1) := \sup \left\{ \frac{\ell^1(f)}{\|f\|_\tau} ; \enskip f \in \iX_1 \setminus \{0\} \text{ and } \ell^0(f) = 0 \right\}.
	\end{equation}
	Let $\tau \in (0,1)$.
	Denote by $\langle \cdot, \cdot \rangle_\tau$ the scalar product associated with the norm $\|\cdot\|_\tau$ on $\iX_1$.
	By the Riesz representation theorem, there exists $g^0_\tau,  g^1_\tau \in \iX_1$ such that $\ell^j = \langle g^j_\tau, \cdot \rangle_\tau$.
	In particular $\iN(\tau,\ell^j) = \|g^j_\tau\|_\tau$.
	Moreover, by \eqref{def:N0t}, $\iN_0(\tau,\ell^1)$ is the supremum of $\ell^1$ on the intersection of $\ker \ell^0$ with the unit ball in $\iX_1$ for the norm $\|\cdot\|_\tau$.
	Thus, a natural candidate to bound $\iN_0(\tau,\ell^1)$ from below is the orthogonal projection of $g^1_\tau / \|g^1_\tau\|_\tau$ on $\ker \ell^0$, namely,
	\begin{equation*}
		f_\tau^1 := \frac{g^1_\tau}{\|g^1_\tau\|_\tau} - R_\tau \frac{g^0_\tau}{\|g^0_\tau\|_\tau}
		\quad \text{where} \quad
		R_\tau := \left\langle \frac{g^1_\tau}{\|g^1_\tau\|_\tau} , \frac{g^0_\tau}{\|g^0_\tau\|_\tau} \right\rangle_\tau.
	\end{equation*}
	In particular $\|f^1_\tau\|_\tau = (1-R_\tau^2)^{\frac12}$ and $\ell^0(f^1_\tau) = \left\langle g^0_\tau, f^1_\tau \right\rangle = 0$.
	Thus
	\begin{equation}
		\label{eq:N0-low}
		\iN_0(\tau,\ell^1) \geq \frac{\langle f^1_\tau, g^1_\tau \rangle_\tau}{\|f^1_\tau\|_\tau}
		= (1-R_\tau^2)^{\frac12} \|g^1_\tau\|_\tau
		= (1-R_\tau^2)^{\frac12} \iN(\tau,\ell^1).
	\end{equation}
	Thus, to prove \eqref{eq:Nt-N0t}, it is sufficient to prove that the ratio $R_\tau^2$ is bounded away from $1$.
	By \eqref{eq:lof-c0c1}, for every $(c_0,c_1) \in \mathbb{S}^1$,
	\begin{equation}\label{encadrement-cg-tau}
		C_- \tau^{-\bar{\sigma}} \leq \| c_0 g^0_\tau + c_1 g^1_\tau \|_\tau \leq C_+ \tau^{-\bar{\sigma}}.
	\end{equation}
	In particular,
	\begin{equation} \label{eq:gjt}
		C_- \tau^{-\bar{\sigma}} \leq \| g^j_\tau \|_\tau \leq C_+ \tau^{-\bar{\sigma}}.
	\end{equation}
	By homogeneity, from \eqref{encadrement-cg-tau}, for every $(c_0,c_1) \in \R^2$,
	\begin{equation*}
		C_-^2 \tau^{-2\bar{\sigma}} (c_0^2+c_1^2) \leq c_0^2 \| g^0_\tau \|_\tau^2 + c_1^2 \| g^1_\tau \|_\tau^2 + 2 c_0 c_1 \langle g^0_\tau, g^1_\tau \rangle_\tau \leq C_+^2 \tau^{-2\bar{\sigma}} (c_0^2+ c_1^2).
	\end{equation*}
	Substituting $c_j \gets c_j / \|g^j_\tau\|_\tau$ and using \eqref{eq:gjt} leads to the fact that, for every $(c_0,c_1) \in \R^2$,
	\begin{equation*}
		\rho^2 (c_0^2+c_1^2) \leq c_0^2 + c_1^2 + 2 R_\tau c_0 c_1 \leq \rho^{-2} (c_0^2+c_1^2), 
	\end{equation*}
	where $\rho := C_- / C_+$.
	In particular, using $(c_0,c_1) = (1,1)$ and $(1,-1)$ yields $\rho^2 \leq 1+R_\tau$ and $\rho^2 \leq 1-R_\tau$.
	Hence, \eqref{eq:N0-low} proves that
	\begin{equation*}
		\iN_0(\tau,\ell^1) \geq \rho^2 \iN(\tau,\ell^1),
	\end{equation*}
	which implies \eqref{eq:Nt-N0t} with $C = \rho^{-2}$.
	So $\ell^0$ and $\ell^1$ form a ``strongly independent basis'' and \cref{lem:lofstrom} follows from \cite[Theorem 2]{MR1343130}.
\end{proof}
	
\subsection{Interpolated theory in the case of the linear shear flow}
\label{sec:interpolation-shear}

In this subsection, we consider the problem \eqref{eq:shear-z} at the linear shear flow, with vanishing boundary data.
We proved in \cref{sec:shear-strong} (see \cref{p:pagani-shear}) that, when $f \in L^2_x L^2_z$, the solutions to this problem have $Z^0$ regularity, and in \cref{sec:shear-ortho} (see \cref{p:shear-WP-Z1}) that they have $Z^1$ regularity when $f \in H^1_x L^2_z$ and the two orthogonality conditions \eqref{eq:compat-shear} are satisfied.
Here, we establish an interpolated theory for the problem \eqref{eq:shear-z} with source terms $f \in H^\sigma_x L^2_z$, $\sigma \in (0,1)$, see \cref{lem:shear-Hs}.
This interpolated theory involves the difficulty exposed in \cref{sec:subspace-interpolation}.
We define $\iX_0$, $\iX_1$ and $\iXL$ by \eqref{eq:cX0}, \eqref{eq:cX1} and \eqref{def:cal-Y} respectively, endowed
with their usual norms.
For $\sigma \in (0,1)$, let $\iX_\sigma := [\iX_0,\iX_1]_\sigma$. The identification of the space~$\iX_\sigma$ is classical and provided by the following lemma:

\begin{lem} \label{lem:x-theta}
	Let $\sigma \in (0,1)$.
	Let $\iX_\sigma := [\iX_0,\iX_1]_\sigma$, for the complex interpolation method.
	\begin{itemize}
		\item When $\sigma \in (0,1/2)$, $\iX_{\sigma} = H^\sigma_x L^2_y$.
		\item When $\sigma = 1/2$, recalling that $\Om_\pm=\Om\cap \{\pm y >0\}$
		\begin{equation*}
			\| f \|_{\iX_{1/2}}^2 \approx \| f \|_{H^{1/2}_x L^2_y}^2+ \int_{\Om_+} \frac{f^2(x,y)}{|x-x_0|} \dd x\dd y+ \int_{\Om_-} \frac{f^2(x,y)}{|x-x_1|} \dd x\dd y.
		\end{equation*}
		\item When $\sigma \in (1/2,1)$, $\iX_{\sigma} = \{ f \in H^\sigma_x L^2_y ; \enskip f_{\rvert \Sigma_0 \cup \Sigma_1 = 0} \}$, with the usual norm.
	\end{itemize}
\end{lem}

\begin{proof}
	This follows from classical interpolation theory for intersections (see \cite[Theorem 13.1 and Equation (13.4), Chapter 1]{LM68}), and from (one-sided versions of) the equality $H^{1/2}_{00}(x_0,x_1) = [H^1_0(x_0,x_1), L^2(x_0,x_1)]_{\frac 12}$ (see \cite[Theorem 11.7, Chapter 1]{LM68}) .
\end{proof}

In order to extend the theory of \cref{sec:shear} to fractional tangential regularity, we start by identifying the spaces $[\iX_0,\iXL]_\sigma$.
More precisely, we prove the following characterization.

\begin{lem} \label{lem:X0X1-shear}
	Let $\iX_0$, $\iX_1$ and $\iXL$ be given by \eqref{eq:cX0}, \eqref{eq:cX1} and \eqref{def:cal-Y} respectively.
	Then,
	\begin{itemize}
		\item For every $\sigma \in (0,1/6)$, $[\iX_0,\iXL]_\sigma = \iX_\sigma$ with equivalent norms.
		\item For every $\sigma \in (1/6,1)$, the linear forms $\overline{\ell^0}$ and $\overline{\ell^1}$ admit continuous extensions to $\iX_\sigma$ and
		\begin{equation*}
			[\iX_0,\iXL]_\sigma = \left\{ f \in \iX_\sigma ; \enskip \overline{\ell^0}(f,0,0) = \overline{\ell^1}(f,0,0) = 0 \right\},
		\end{equation*}
		endowed with the norm of $\iX_\sigma$.
	\end{itemize}
\end{lem}

\begin{rmk}
	The threshold at $\sigma = 1/6$ is consistent with the observation of \cref{rmk:16-shear} that the maps $\overline{\ell^j}(\cdot,0,0)$ are bounded on $H^\sigma_x L^2_z$ for every $\sigma > 1/6$.
\end{rmk}

For $\tau \in (0,1)$, we use the notations of the previous paragraph, in particular the norm $\|\cdot\|_\tau$ of~\eqref{eq:norm-t} and the function $\iN(\tau,\cdot)$ of \eqref{eq:N}, with $\iX_0$ and $\iX_1$ defined as above.

To derive the estimates required to apply \cref{lem:lofstrom}, two strategies would be possible.
Both rely on the explicit knowledge of the singular radial solutions constructed in \cref{sec:radial}, which are involved in the orthogonality conditions. 
First, one could impose periodic boundary conditions on~$f$, compute a 2D Fourier-series representation of (an extension by parity of) the singular profiles and estimate the functions $\iN$ working in the Fourier space.
Such a frequency-domain approach is carried out in \cite{MR2387839}, assuming some appropriate asymptotic decay of the Fourier transform of the profile defining the orthogonality condition.
We choose a second strategy, which stays in the spatial domain, and involves estimates using cut-off functions whose space-scale are linked with the parameter $\tau$. 
This strategy is related to the one used in \cite{MR4431930} and inspired by the links between the $K$ functional of real interpolation theory and the notions of modulus of continuity and modulus of smoothness of functions (see e.g.\ \cite{johnen1977equivalence}).

To prove \cref{lem:X0X1-shear}, we intend to apply \cref{lem:lofstrom}.
Hence, we need to bound from below and from above the functions $\iN(\tau,\overline{\ell^j})$.
By \cref{def:ell-shear}, $\overline{\ell^j}(f,0,0) = \int_\Omega \p_x f \overline{\Phi^j}$.
As highlighted in \cref{cor:decomp-Phij}, the profiles $\overline{\Phi^j}$ can be decomposed as the sum of a singular radial part, an $x$-independent part, and a regular part.
The singular radial part is the one that will be dominating the behavior of the orthogonality conditions.
Thus, we start by two lemmas concerning estimates from above and from below for integrals of the form $\int_\Omega (\p_x f) \busing^i$, before moving to the general case.

\begin{lem}
	\label{lem:dxf-using-above}
	Let $h \in H^1_x L^2_z$ such that $h = 0$ on $\Sigma_0 \cup \Sigma_1$.
	Then, for $\tau \in (0,1)$,
	\begin{equation*}
		\left| \int_\Omega (\p_x h(x,z)) \busing^i(x,-z) \dd x \dd z
		\right| \lesssim \tau^{-1/6} \left(\|h\|_{L^2} + \tau \|\p_x h\|_{L^2}\right),
	\end{equation*}
	where $\busing^i$ is defined in \cref{def:sing-profiles}.
\end{lem}

\begin{proof}
	By symmetry, it is sufficient to prove the result with $i = 0$, which we assume from now on, and we drop the indexes $i=0$ on $r_i$ and $t_i$ involved in \cref{def:sing-profiles}.
	We also let $\chi(x,z) := \chi_i(x,-z)$ of \cref{def:sing-profiles} and $\gl(t) := \gnot(-t)$, where $\gnot$ is defined in \cref{prop:rGk}.
	With these notations
	\begin{equation} \label{eq:busing0-rgchi}
		\busing^0(x,-z) = r^{\frac 12} \gl(t) \chi(x,z).
	\end{equation}
	In particular, since $\gnot(+\infty) = 0$, $\busing^0(x_0,-z) = 0$ for $z \in (-1,0)$. 
	We split the integral to be estimated depending on whether $r \leq \thr$ or $r \geq \thr$, where $\alpha > 0$ is to be chosen later.
	Let $\eta \in C^\infty(\R;[0,1])$ such that $\eta(s) \equiv 1$ for $s \leq 1$ and $\eta(s) \equiv 0$ for $s \geq 2$.
	
	\step{Estimate in the region: $r \leq \thr$.}
	By Cauchy--Schwarz,
	\begin{equation*}
		\left| \int_\Omega \p_x h \cdot r^{\frac12} \gl(t) \chi \cdot \eta(r/\thr) \right|
		\leq \|\chi\|_\infty \|\gl\|_\infty \|\p_x h\|_{L^2}
		\left(\int_\Omega r \eta^2(r/\thr) \right)^{\frac 12}.
	\end{equation*}
	Using the polar-like change of coordinates of \eqref{eq:r-t} and \eqref{eq:det-J}, one has
	\begin{equation*}
		\int_\Omega r \eta^2(r/\thr)
		= \int_0^\infty \int_\R \frac{3r^3}{(1+t^2)^2} r \eta^2(r/\thr) \dd r \dd t
		\lesssim (\thr)^5.
	\end{equation*}
	Hence, in this region,
	\begin{equation*}
		\left| \int_\Omega \p_x h \cdot r^{\frac12} \gl(t) \chi \cdot \eta(r/\thr) \right| \lesssim (\thr)^{5/2} \| \p_x h \|_{L^2}.
	\end{equation*}
	
	\step{Estimate in the region: $r \geq \thr$.}
	We intend to integrate by parts in $x$.
	At $x = x_1$, $\busing^0(x,-z) = 0$ for $z \in (-1,1)$ because $\chi = 0$.
	At $x = x_0$, when $z > 0$, $h = 0$ by assumption, and, when $z < 0$, $\busing^0(x,-z) = 0$ as recalled above.
	Hence, there is no boundary term and
	\begin{equation*}
		\int_\Omega \p_x h \cdot r^{\frac12} \gl(t) \chi \cdot (1-\eta(r/\thr))
		=
		- \int_\Omega h \p_x \left( \chi \cdot r^{\frac12} \gl(t) (1-\eta(r/\thr)) \right).
	\end{equation*}
	First, one easily bounds
	\begin{equation*}
		\left| \int_\Omega h \p_x \chi \cdot r^{\frac12} \gl(t) (1-\eta(r/\thr)) \right|
		\leq \| h \|_{L^2} \|\gl\|_\infty \|\p_x\chi\|_{L^2}\max_\Omega r^{\frac12}
		\lesssim \| h \|_{L^2}.
	\end{equation*}
	For the second term, when $\p_x$ hits on the function expressed in $(r,t)$ coordinates, we use the derivative formula \eqref{eq:px-pr-pt}:
	\begin{equation*}
		\begin{split}
			\int_\Omega h \chi \p_x \left( r^{\frac12} \gl(t) (1-\eta(r/\thr)) \right)
			& = \int_\Omega h \chi \frac{(1+t^2)^{\frac12}}{3r^2} \gl(t) \p_r \left(r^{\frac12} (1-\eta(r/\thr))\right) \\
			& - \int_\Omega h \chi \frac{t(1+t^2)^{\frac32}}{3r^3} r^{\frac12} (1-\eta(r/\thr)) \p_t \gl(t).
		\end{split}
	\end{equation*}
	We bound both terms using the Cauchy--Schwarz inequality and the polar-like change of coordinates \eqref{eq:r-t} with Jacobian determinant \eqref{eq:det-J}.
	In particular, on the one hand,
	\begin{equation*}
		\begin{split}
			\int_\Omega \frac{1+t^2}{9r^4} \gl^2(t) & \left(\p_r \left(r^{\frac12} (1-\eta(r/\thr))\right)\right)^2
			\\&= \int_0^\infty \int_\R \frac{3r^3}{(1+t^2)^2} \frac{1+t^2}{9r^4} \gl^2 \left(\p_r \left(r^{\frac12} (1-\eta(r/\thr))\right)\right)^2 \dd t\dd r\\
			& \lesssim \int_0^\infty \left( r^{-1} (1-\eta(r/\thr))^2 + r (\eta'(r/\thr))^2 / (\thr)^2 \right) \frac{\dd r}{r} \\
			& = (\thr)^{-1} \int_1^{\infty} \left(s^{-1} (1-\eta(s))^2 + s (\eta'(s))^2\right) \frac{\dd s}{s} \lesssim (\thr)^{-1}.
		\end{split}
	\end{equation*}
	On the other hand,
	\begin{equation*}
		\begin{split}
			\int_\Omega \frac{t^2(1+t^2)^3}{9r^6} r & (1-\eta(r/\thr))^2 (\p_t \gl(t))^2
			\\&= \int_0^\infty \int_\R \frac{3r^3}{(1+t^2)^2} \frac{t^2(1+t^2)^3}{9r^6} r (1-\eta(r/\thr))^2 (\p_t \gl(t))^2 \dd t\dd r \\
			& \lesssim \left( \int_\R t^2(1+t^2) (\p_t\gl(t))^2 \dd t\right) \int_0^\infty \frac{1}{3r^2} (1-\eta(r/\thr)) \dd r\\
			& \lesssim (\thr)^{-1}
		\end{split}
	\end{equation*}
	by the integrability property $t^3 \p_t \gl(t) = O(1)$ of \cref{lem:G-decay}.
	
	Thus, gathering the estimates in this region proves that
	\begin{equation*}
		\left| \int_\Omega \p_x h \cdot r^{\frac12} \gl(t) \chi \cdot (1-\eta(r/\thr)) \right|
		\lesssim (\thr)^{-\frac 12} \|h\|_{L^2}.
	\end{equation*}
	Gathering the estimates in both regions and choosing $\alpha = 1/3$ concludes the proof. 
\end{proof}

\begin{lem}
	\label{lem:dxf-using-below}
	There exists a family $(h^i_\tau)_{\tau \in (0,1)}$ of non-zero, smooth, compactly supported functions on $\Omega$ such that, as $\tau \to 0$, 
	\begin{equation*}
		\left| \int_\Omega (\p_x h^i_\tau(x,z)) \busing^i(x,-z) \dd x \dd z
		\right| \gtrsim \tau^{-1/6} \left(\|h^i_\tau\|_{L^2} + \tau \|\p_x h^i_\tau\|_{L^2} 
\right),
	\end{equation*}
	and $\int_\Om \p_x h^j_\tau \busing^i=0$ for $j \neq i$, where $\busing^i$ is defined in \cref{def:sing-profiles}.
\end{lem}

\begin{proof}
	As in the previous lemma, by symmetry, it is sufficient to prove the result with $i = 0$, which we assume from now on, and we drop the indexes $i=0$ on $r_i$ and $t_i$ involved in \cref{def:sing-profiles}.
	We also let $\chi(x,z) := \chi_i(x,-z)$ of \cref{def:sing-profiles} and $\gl(t) := \gnot(-t)$, where $\gnot$ is defined in \cref{prop:rGk}.
	With these notations, one has \eqref{eq:busing0-rgchi}.
	
	Let $\alpha > 0$.
	Let $H \in C^\infty_c(\R;[-1,1])$ and $\eta \in C^\infty_c(\R;[-1,1])$ such that $\supp \eta \subset (1/2,3/2)$.
	For $\tau \in (0,1)$, we define
	\begin{equation*}
		h_\tau := \eta(r/\thr) H(t).
	\end{equation*}
	By the support properties of $H$ and $\eta$, one checks that $h_\tau$ is both smooth and compactly supported in $\Omega$.
	Moreover, it is non-zero if $H \neq 0$ and $\eta \neq 0$.
	
	Let $\tau>0$ be sufficiently small such that the support of $h_\tau$ is included in the region where $\chi \equiv 1$. Note that with this choice, we also have $\int_\Om \p_x h_\tau \busing^1=0$. 
	Then, using the formula \eqref{eq:px-pr-pt} for $\p_x$ and the determinant \eqref{eq:det-J},
	\begin{equation*}
		\begin{split}
			 \int_\Omega & \p_x h_\tau \busing(x,-z)\dd x\dd z \\
			& = 
			\int_0^\infty \int_\R r^{\frac12} \gl(t) \left( 
			\frac{(1+t^2)^{\frac 12}}{3r^2} (\thr)^{-1} \eta'(r/\thr) H(t) - \frac{t(1+t^2)^{\frac32}}{3r^3} \eta(r/\thr) H'(t) \right)  \frac{3r^3}{(1+t^2)^2}\dd t\; \dd r\\
			& = \int_0^\infty \int_\R \frac{r^{\frac12} \gl(t)}{(1+t^2)^{\frac32}}  \left( 
			r (\thr)^{-1} \eta'(r/\thr) H(t) - t(1+t^2) \eta(r/\thr) H'(t)
			\right) \dd t\; \dd r \\
			& = (\thr)^{3/2} \int_\R \frac{H(t) \gl(t)}{(1+t^2)^{\frac32}} \dd t \int_0^\infty s^{3/2} \eta'(s) \dd s
			- (\thr)^{3/2} \int_\R \frac{tH'(t) \gl(t)}{(1+t^2)^{\frac12}} \dd t \int_0^\infty s^{1/2} \eta(s) \dd s.
		\end{split}
	\end{equation*}
	Note that since $\eta\in C^\infty_c((0, + \infty))$, 
	\begin{equation*}
	\int_0^\infty s^{3/2} \eta'(s) \dd s = - \frac{3}{2} \int_0^\infty s^{1/2} \eta(s)\dd s.
	\end{equation*}
	We claim that we may choose  $\eta$ and $H$ such that
	\begin{equation*}
	\int_0^\infty s^{1/2} \eta(s)\dd s=\frac{3}{2}\int_\R \frac{H(t) \gl(t)}{(1+t^2)^{\frac32}} \dd t +  \int_\R \frac{tH'(t) \gl(t)}{(1+t^2)^{\frac12}} \dd t=1.
	\end{equation*}
	The claim for $\eta$ is obvious. As for $H$, we assume that $\supp H \subset (0, + \infty)$ and we write the sum of integrals as
	\begin{equation*}
	\begin{aligned}
	&\int_\R H(t) \left[ \frac{3}{2} \frac{ \gl(t)}{(1+t^2)^{\frac32}}  - \frac{d}{dt}\left(\frac{t\gl(t)}{(1+t^2)^{\frac12}}\right) \right]  \dd t\\
	=& -\int_0^\infty H(t) \frac{d}{dt}\left( \frac{t\gl(t)}{(1+t^2)^{1/2}} t^{-3/2} (1+t^2)^{3/4}\right) t^{3/2} (1+t^2)^{-3/4} \dd t.
	\end{aligned}
	\end{equation*}
	Since $\gl(t)\neq C t^{1/2} (1+t^2)^{-1/4}$ on $\R_+$, the claim for $H$ follows.

	The above choice of $\eta$ and $H$  ensures that
	\begin{equation*}
		\int_\Omega \p_x h_\tau \cdot r^{\frac12} \gl(t) \chi =- (\thr)^{3/2}.
	\end{equation*}
	Using once again the formula \eqref{eq:px-pr-pt} for $\p_x$ and the change of coordinates of Jacobian \eqref{eq:det-J}, one obtains
	that $\|h_\tau\|_{L^2} \lesssim (\thr)^2$ and $\|\p_x h_\tau\|_{L^2} \lesssim 1/\thr$.
	Similarly, using \eqref{eq:pz-pr-pt} to compute $\p_z^3 h_\tau$ and the same technique, $\|\p_z^3 h_\tau\|_{L^2} \lesssim 1 / \thr$.
	Thus, choosing $\alpha = 1/3$ leads to
	\begin{equation*}
		\|h_\tau\|_{L^2} + \tau \|\p_xh_\tau\|_{L^2} + \tau \|\p_z^3 h_\tau\|_{L^2}
		\lesssim (\thr)^2,
	\end{equation*}
	which concludes the proof.
\end{proof}

We are now ready to prove \cref{lem:X0X1-shear}.

\begin{proof}[Proof of \cref{lem:X0X1-shear}]
	This is an application of \cref{lem:lofstrom} with $\bar{\sigma} = 1/6$.
	Therefore, we need to find constants $C_\pm > 0$ such that, for every $\tau \in (0,1)$ and $(c_0,c_1) \in \mathbb{S}^1$,
	\begin{equation*}
		C_- \tau^{-1/6} \leq \iN(\tau, c_0 \overline{\ell^0} + c_1 \overline{\ell^1}) \leq C_+ \tau^{-1/6}.
	\end{equation*}
	Let $(c_0,c_1) \in \mathbb{S}^1$ and $f \in \iX_1$.
	By \cref{def:ell-shear},
	\begin{equation*}
		\overline{\ell^j}(f,0,0) = \int_\Omega \p_x f \overline{\Phi^j},
	\end{equation*}
	where $\overline{\Phi^j}$ is the solution to \eqref{eq:Phij-shear}.
	
	By \cref{cor:decomp-Phij}, there exists $(d_0,d_1) \in \R^2 \setminus \{ 0 \}$ such that
	\begin{equation} \label{eq:c0l0c1l1-pxf}
		\begin{split}
			c_0 \overline{\ell^0}(f,0,0) + c_1 \overline{\ell^1}(f,0,0) 
			& = \int_\Omega \p_x f
			(d_0 \busing^0(x,-z) + d_1 \busing^1(x,-z)) \\
			& + \int_\Omega \p_x f \left( \Phi_\reg + (c_1 - z c_0) \chi(z) \mathbf{1}_{z>0} \right),
		\end{split}
	\end{equation}
	where $\Phi_\reg \in Z^1$.
	By linearity, $d_0$, $d_1$ and $\Phi_\reg$ are uniformly bounded for $(c_0,c_1) \in \mathbb{S}^1$.
	The first term corresponds to the one studied in \cref{lem:dxf-using-above} and \cref{lem:dxf-using-below}.
	We want to integrate by parts in the second term.
	Since $f \in \iX_1$, $f_{\rvert \Sigma_0 \cup \Sigma_1} = 0$.
	At $x = x_0$ and $z \in (-1,0)$, $\overline{\Phi^j} = 0$ by \eqref{eq:Phij-shear}.
	Moreover, $\busing^0(x_0,-z) = 0$ because $\gnot(+\infty) = 0$ and $\busing^1(x_0,-z) = 0$ because $\busing^1$ is compactly supported near $(x_1,0)$.
	Hence, $\Phi_\reg(x_0,z) + (c_1 - z c_0) \chi(z) \mathbf{1}_{z>0} \equiv 0$ on $(-1,0)$.
	The same conclusion holds at $x = x_1$ and $z \in (0,1)$.
	Thus, we can integrate by parts with no boundary term and the second term is estimated as
	\begin{equation}\label{bound-interpol-reg}
		\left| \int_\Omega \p_x f \left( \Phi_\reg + (c_1 - z c_0) \chi(z) \mathbf{1}_{z>0} \right) \right| \leq \| f \|_{L^2} \| \p_x \Phi_\reg \|_{L^2}.
	\end{equation}
	
	\step{Bound from above.}
	For $\tau \in (0,1)$, using \cref{lem:dxf-using-above} and \eqref{bound-interpol-reg},
	\begin{equation*}
		|c_0 \overline{\ell^0}(f,0,0) + c_1 \overline{\ell^1}(f,0,0)|
		\lesssim \tau^{-1/6} \left( \|f\|_{L^2} + \tau \|f\|_{\iX_1} \right).
	\end{equation*}
	
	\step{Bound from below.}
	For $\tau \in (0,1)$, let $f_\tau := h_\tau$, where $h_\tau$ is constructed in \cref{lem:dxf-using-below}, which ensures that $f_\tau$ is compactly supported in $\Omega$ so satisfies $(f_\tau)_{\rvert \Sigma_0 \cup \Sigma_1} = 0$.
	Substituting in~\eqref{eq:c0l0c1l1-pxf} and integrating by parts yields
	\begin{equation*}
		c_0 \overline{\ell^0}(f_\tau,0,0) + c_1 \overline{\ell^1}(f_\tau,0,0)
		= -\int_\Omega h_\tau \p_x \Phi_\reg + \sum_{i \in \{0,1\}} d_i \int_\Omega (\p_x h_\tau) \busing^i(x,-z).
	\end{equation*}
	By \cref{cor:decomp-Phij} and linearity, $\min( |d_0|, |d_1|)$ is uniformly bounded from below.
	We choose $h_\tau$ as either $h^0_\tau$ or $h^1_\tau$ of \cref{lem:dxf-using-below} accordingly.
	Thus, by \cref{lem:dxf-using-below}, as $\tau \to 0$,
	\begin{equation*}
		\begin{split}
			|c_0 \overline{\ell^0}(f_\tau,0,0) + c_1 \overline{\ell^1}(f_\tau,0,0)| 
			& \gtrsim 
			\tau^{-1/6} \left( \| h_\tau \|_{L^2} + \tau \| h_\tau \|_{\iX_1} \right) - C
			\|h_\tau\|_{L^2} \| \p_x \Phi_\reg \|_{L^2} \\
			& \gtrsim 
			\tau^{-1/6} \left( \| h_\tau \|_{L^2} + \tau \| h_\tau \|_{\iX_1} \right) \\
			& = \tau^{-1/6} \left( \| f_\tau \|_{L^2} + \tau \| f_\tau \|_{\iX_1} \right)
		\end{split}
	\end{equation*}
	for $\tau > 0$ sufficiently small.
	This concludes the proof.
\end{proof}

To conclude this section, we turn towards the proof of \cref{lem:shear-Hs}.

\begin{proof}[Proof of \cref{lem:shear-Hs}]
	\step{Case $\delta_0=\delta_1=0$ and $\mathbf 1_{\sigma>1/2}f_{|\Sigma_i}=0$.}
	By \cref{p:pagani-shear}, for every $f \in L^2(\Omega)$, there exists a unique solution $u \in Z^0(\Omega)$ to \eqref{eq:shear-z} with $\delta_0 = \delta_1 = 0$ and $\|u\|_{Z^0} \lesssim \|f\|_{L^2}$.
	By \cref{p:shear-WP-Z1} and \cref{p:shear-apriori-Z1}, for every $f \in H^1_x L^2_z$ such that $f_{\rvert \Sigma_0 \cup \Sigma_1} = 0$ (so that $\Delta_0 = \Delta_1 = 0$) and $\overline{\ell^0}(f,0,0) = \overline{\ell^1}(f,0,0) = 0$, this solution satisfies $u \in Z^1(\Omega)$ with $\|u\|_{Z_1} \lesssim \|f \|_{H^1_x L^2_z}$.
	Hence, by interpolation, the mapping $f \mapsto u$ is bounded from $[\iX_0,\iXL]_\sigma$ to $Z^\sigma(\Omega)$.
	Moreover, by \cref{lem:x-theta} and \cref{lem:X0X1-shear}, when $\sigma \in (0,1) \setminus \{1/6,1/2\}$, $[\iX_0,\iXL]_\sigma = H^\sigma_x L^2_z$ (with null boundary conditions on $\Sigma_0 \cup \Sigma_1$ when $\sigma > 1/2$, and null linear forms constraints when $\sigma > 1/6$).
    This proves estimate \eqref{eq:estimate-zsigma-shear} in the case of vanishing boundary data.
	
	\step{Arbitrary boundary data.} 
    When $\delta_0$ and $\delta_1$ are arbitrary, we extend them to $(-1,1)$ in such a way that the extension belongs to $H^2_0(-1,1)$.
    We then lift the boundary data by setting $u_l(x,z)=\chi(x-x_0) \delta_0 + \chi(x-x_1)\delta_1$, with $\chi\in C^\infty_c(\R)$, supported in $B(0, (x_1-x_0)/2)$, and equal to 1 in a neighborhood of zero. 
    This introduces a source term $f_l=z\p_x u_l -\p_{zz} u_l\in H^1_xL^2_z$ in the equation, whose trace on $\Sigma_i$ is $-\delta_i''$, so that the trace of $f-f_l$ on $\Sigma_i$ is $z\Delta_i$. 
    When $\sigma<1/2$, we immediately obtain the desired result thanks to the previous step.
    
	For $\sigma> 1/2$, we first note that, since $u, u_l \in Z^1(\Om)$, by \cref{p:shear-WP-Z1},
	\begin{equation*}
	\overline{\ell^j}(f- f_l,0,0)=0.
	\end{equation*}
	We further decompose $f-f_l$ into $f-f_l = z \Delta_0 \chi(x-x_0) + z \Delta_1 \chi(x-x_1) + g_l$, where $g_l \in H^\sigma_x L^2_z$ is such that $g_{l|\Sigma_0\cup \Sigma_1}=0$.
    Using \cref{free:f} we construct $h_l \in C^\infty_c(\Om)$ such that
    \begin{equation*}
        \|h_l\|_{H^1_xL^2_z} \lesssim \|\Delta_0\|_{\mathscr H^1(\Sigma_0)} + \|\Delta_1\|_{\mathscr H^1(\Sigma_1)}
    \end{equation*}
    and 
    \begin{equation*}
        \overline{\ell^j}(z \Delta_0 \chi(x-x_0) + z \Delta_1 \chi(x-x_1) + h_l,0,0)=\overline{\ell^j}(g_l- h_l,0,0)=0.
    \end{equation*}
	We then apply the result of the first step to the system with source term $g_l - h_l$ (which vanishes on $\Sigma_0\cup \Sigma_1$) and homogeneous boundary data, and the result of \cref{p:shear-WP-Z1} to the system with source term $z \Delta_0 \chi(x-x_0) + z \Delta_1 \chi(x-x_1) + h_l$ and homogeneous boundary data, using the conditions $\Delta_0(1) = \Delta_1(-1) = 0$.
    This concludes the proof.
\end{proof}


\newpage
\appendix

\addtocontents{toc}{\protect\setcounter{tocdepth}{1}}

\section{Uniqueness of weak solutions for linear problems}
\label{sec:proof-uniqueness}

The purpose of this section is to prove the following uniqueness result, which is a slight generalization to the case of variable coefficients of the uniqueness result of \cite[Section 5]{BG} for \eqref{eq:shear-z}.

\begin{lem}
    \label{lem:uniqueness-BG}
    Let $\Om=(x_0,x_1)\times (z_b, z_t)$, where $x_0 < x_1$ and $z_b<0<z_t$.
    Let $\alpha \in C^2(\overline{\Om})$ such that $\inf \alpha > 0$ and $\beta \in L^\infty(\Om)$.
    Assume that one of the two following  conditions is satisfied:
    \begin{itemize}
    \item either $\|\beta\|_\infty\ll 1$ and $\|\p_z \alpha\|_\infty \ll 1$,
    \item or $|z_b|, z_t \leq z_0$, for some small constant $z_0$ depending only on $\alpha.$
    \end{itemize}
    Let $g\in L^2_xH^{-1}_z$, $\delta_0\in \mathscr L^2_z(\Sigma_0)$, $\delta_1\in \mathscr L^2_z(\Sigma_1)$.
    There exists at most one weak solution $U\in L^2_x H^1_0$ to
    \begin{equation} \label{eq:U}
    \begin{cases}
        z\p_x U + \beta \p_z U - \p_{zz}(\alpha U)=g, \\
        U\vert_{\Sigma_0}= \delta_0,\\
          U\vert_{\Sigma_1}= \delta_1,\\
        U\vert_{z=z_t}=U\vert_{z=z_b}=0.
    \end{cases}
    \end{equation}
\end{lem}

The proof follows the arguments of Baouendi and Grisvard in~\cite{BG}, which concern the case of the model equation \eqref{eq:shear-z}.
For the reader's convenience, we recall the main steps of the proof here, and adapt them to the present (slightly different) context.
The proof involves the spaces $\cB$ defined in \eqref{def:cB} and $\cA := \cB \cap H^1(\Om)$.

Note that if $U\in L^2((x_0,x_1) , H^1_0(z_b,z_t))$ is a weak solution to \eqref{eq:U}, then $U\in \cB$. Indeed, it follows from the weak formulation  that for any $V\in H^1_0(\Om)$,
\begin{equation*}
\left\langle z\p_x U, V\right\rangle_{L^2(H^{-1}), L^2(H^1_0)}= - \int_\Om \p_z (\alpha U) \p_z V - \int_\Om \beta \p_z U V +\left\langle g, V\right\rangle_{L^2(H^{-1}), L^2(H^1_0)}.
\end{equation*}
By density, this formula still holds for $V\in L^2_x(H^1_0)$, and therefore $z\p_x U\in L^2_x(H^{-1}_z)$.

We then recall the following result from \cite{BG}:
\begin{lem}
	\label{lem:tracesBG}
	The set $\cA$ is dense in $\cB$.
	Furthermore, there exists a  constant $C$ depending only on~$\Om$, such that for $i \in \{0,1\}$,
	\begin{equation*}
	\forall v\in \cA,\quad \int_{z_b}^{z_t} |z| \; |v(x_i, y)|^2\dd y\leq C \|v\|_{\cB}^2.
	\end{equation*}
	As a consequence, the applications
	\begin{equation*}
	v\in \cA\mapsto v_{|x=x_i}\in \mathscr{L}^2_z(z_b,z_t)
	\end{equation*}
	can be uniquely extended into continuous applications on $\cB$.
\end{lem}

As a consequence, Baouendi and Grisvard \cite{BG} obtain the following corollary:
\begin{coro}
	\label{coro:trace-BG}
	For all $u,v\in \cB$,
	\begin{equation*}
		\langle z \p_x u, v\rangle_{L^2(H^{-1}), L^2(H^1_0)} +  \langle z \p_x v, u\rangle_{L^2(H^{-1}), L^2(H^1_0)}     =\int_{z_b}^{z_t} (z u v)\vert _{x=x_1} - \int_{z_b}^{z_t} (z u v)\vert _{x=x_0}. 
	\end{equation*}
\end{coro}

\begin{proof}
	Thanks to \cref{lem:tracesBG}, it suffices to prove the identity when $u,v\in \cA$. In that case, the left-hand side is simply
	\begin{equation*}
	\int_\Om z \p_x u v + z u \p_x v = \int_\Om \p_x(z u v).
	\end{equation*}
	The result follows by integration.
\end{proof}

\begin{proof}[Proof of \cref{lem:uniqueness-BG}]
	Let $U\in L^2_x(H^1_0)$ be a weak solution to \eqref{eq:U} with $g=0$ and ${\delta_i}=0$. 
	As mentioned above, $U\in \cB$. 
	According to \cref{coro:trace-BG}, for any $V\in \cB$ such that $V=0$ on $\p\Om \setminus (\Sigma_0\cup \Sigma_1)$,
	\begin{equation*}
	-\langle z \p_x V, U\rangle_{L^2(H^{-1}), L^2(H^1_0)} + \int_\Om (\beta \p_z U V + \alpha_z U \p_z V+ \alpha \p_z U \p_z V)=0.
	\end{equation*}
	Now, let $h\in C^\infty_c (\Om)$ be arbitrary, and let $V\in L^2(H^1_0)$ be a weak solution to
	\begin{equation*}
    \begin{cases}
		-z \p_x V -\p_z(\beta V) -\alpha \p_{zz}V=h,\\
		V_{|\p\Om \setminus (\Sigma_0\cup \Sigma_1)}=0.
	\end{cases}
	\end{equation*}
	(The existence of weak solutions for this adjoint problem is proved in the same way as existence for the direct problem in \cref{p:shear-X0} in the case $\|\beta\|_\infty\ll 1$, $\|\alpha_z\|_\infty\ll 1$,
  and \cref{lem:WP-Z0-vorticity} in the case $|z_b|, z_t$ small).
	
	Then $V\in \cB$, and choosing $U$ as a test function in the variational formulation for $V$,  we obtain
	\begin{equation*}
	\int_{\Om} hU=0.
	\end{equation*}
	Thus $U=0$.
    Uniqueness of weak solutions to \eqref{eq:U} follows.
\end{proof}


\section{Proofs of functional analysis results}
\label{sec:proofs-spaces}

\subsection{An abstract existence principle}

As Fichera in \cite{MR0111931}, we use the following abstract existence principle (see \cite[Theorem~1]{douglas1966majorization}), which allows skipping a viscous regularization scheme.

\begin{lem} \label{douglas}
	Let $\mathscr{H}_1$, $\mathscr{H}_2$ and $\mathscr{H}$ be three Hilbert spaces.
	Let $F_i \in \mathcal{L}(\mathscr{H}_i;\mathscr{H})$ for $i \in \{ 1, 2 \}$.
	Then the following statements are equivalent:
	\begin{itemize}
		\item $\operatorname{range} F_1 \subset \operatorname{range} F_2$,
		\item There exists a constant $C > 0$ such that 
		\begin{equation} \label{coercive}
			\forall h \in \mathscr{H}', 
			\quad 
			\| F_1^* h \|_{\mathscr{H}_1'} \leq C \| F_2^* h \|_{\mathscr{H}_2'}.
		\end{equation}
		\item There exists $G \in \mathcal{L}(\mathscr{H}_1;\mathscr{H}_2)$ such that $F_1 = F_2 G$.
	\end{itemize}
	Moreover, when these hold, there exists a unique $G \in \mathcal{L}(\mathscr{H}_1;\mathscr{H}_2)$ such that $\ker G = \ker F_1$, $\operatorname{range} G \subset (\operatorname{range} F_2^*)^\perp$ and $\| G \| = \inf \{ C > 0 ; \enskip \eqref{coercive} \text{ holds} \}$.
\end{lem}

Indeed, this yields the following weak Lax-Migram result, where the linear right-hand side is assumed to be continuous for the weaker norm.

\begin{lem} \label{lax-weak}
	Let $\mathscr{U}$ and $\mathscr{V}$ be two Hilbert spaces with $\mathscr{V}$ continuously embedded in $\mathscr{U}$.
	Let $a$ be a continuous bilinear form on $\mathscr{U} \times \mathscr{V}$
	and $b$ be a continuous linear form on $\mathscr{U}$.
	Assume that there exists a constant $c > 0$ such that, for every $v \in \mathscr{V}$,
	\begin{equation*}
		a(v,v) \geq c \| v \|_{\mathscr{U}}^2.
	\end{equation*}
	Then, there exists $u \in \mathscr{U}$ such that $\|u\|_\mathscr{U} \leq \frac{1}{c} \|b\|_{\mathcal{L}(\mathscr{U})}$ and, for every $v \in \mathscr{V}$, $a(u,v) = b(v)$.
\end{lem}

\begin{proof}
	Set $\mathscr{H} := \mathcal{L}(\mathscr{V})$, $\mathscr{H}_1 := \mathcal{L}(\mathscr{U})$, $F_1 := \operatorname{Id}$ (from $\mathcal{L}(\mathscr{U})$ to $\mathcal{L}(\mathscr{V})$),  $\mathscr{H}_2 := \mathscr{U}$ and $F_2 : \mathscr{U} \to \mathcal{L}(\mathscr{V})$ defined by $F_2 u := a(u, \cdot)$.
	Then $F_1^* = \operatorname{Id}$ (from $\mathscr{V}$ to $\mathscr{U}$) and $F_2^* v = a(\cdot, v)$.
	Moreover 
	\begin{equation*}
		\| F_2^* v \|_{\mathcal{L}(\mathscr{U})} \geq |a(v,v)| / \| v \|_{\mathscr{U}} \geq c \| v \|_{\mathscr{U}} = c \| F_1^* v \|_{\mathscr{U}}.
	\end{equation*}
	So \eqref{coercive} holds with $C = 1/c$ and \cref{douglas} yields the existence of $G \in \mathcal{L}(\mathcal{L}(\mathscr{U});\mathscr{U})$ such that $F_1 = F_2 G$ and $\| G \| \leq \frac{1}{c}$.
	The conclusions follow by setting $u := G b$.
\end{proof}

\subsection{Product and composition rules in Sobolev spaces}

\begin{lem}[Pointwise multiplication]
	\label{lem:prod-hs}
	Pointwise multiplication is a continuous bilinear map
	
	\begin{itemize}
		\item from $H^{3/2}(-1,1) \times H^{3/2}(-1,1)$ to $H^{3/2}(-1,1)$,
		
		\item from $H^{1/2}(x_0,x_1) \times H^{s}(x_0,x_1)$ to $H^{1/2}(x_0,x_1)$ for any $s > 1/2$,
		
		\item from $H^{1/2}(x_0,x_1)\times H^s(x_0,x_1)$ to $H^{s'}(x_0,x_1)$ for any $s'<\min (s, 1/2)$.
		
		\item from $H^{s}(x_0,x_1)\times H^{s'}(x_0,x_1)$ to $H^{s'}(x_0,x_1)$ for any $s>1/2$, $s\geq s'$.
	\end{itemize}
\end{lem}

\begin{proof}
	These are particular cases of \cite[Theorem~7.4]{zbMATH07447116}. 
\end{proof}

\begin{lem}[Composition of $H^\sigma$ functions]
    \label{lem:composition-bis}
    Let $\sigma \in (0,1/6)$, and let $\Omega_y=(x_0,x_1)\times \R$, $\Omega_z=(x_0,x_1)\times (z_b,z_t)$.
    For $f\in H^\sigma_x H^1_y \cap L^4_x H^1_y(\Omega_y)$ and $Y\in H^{\frac{2}{3}+ \sigma}_x H^1_z(\Omega_z)$, such that $\lambda \leq \p_z Y\leq \lambda^{-1}$ for some positive constant $\lambda$,
     \begin{equation*}
        \| f(x, Y(x,z))\|_{H^\sigma_x L^2_z(\Omega_z)}\lesssim C_{\|Y\|}\left(\|f\|_{H^\sigma_x L^2_y} + \| f\|_{L^4_x H^1_y}\right).
    \end{equation*}
    and
    \begin{equation*}
        \| f(x, Y(x,z))\|_{L^\infty ((z_b,z_t) , H^\sigma(x_0,x_1))}\lesssim C_{\|Y\|}\left(
        \|f\|_{H^\sigma_x H^1_z} + \| f\|_{L^4_x H^1_y}\right) .
    \end{equation*}
In a similar way, if $f\in H^{\frac{1}{2}+\sigma}_x L^2_y\cap L^\infty_x W^{1,\infty}_y (\Omega_y)$,
\begin{equation*}
\|f(x, Y(x,z))\|_{H^{\frac{1}{2}+\sigma}_x L^2_z(\Omega_z)}\lesssim C_{\|Y\|}\left( \|f\|_{H^{\frac{1}{2}+\sigma}_x L^2_z} + \|f\|_{L^\infty_x W^{1,\infty}_z}\right).
\end{equation*}

\end{lem}

\begin{proof}
    Using the classical definition of fractional Sobolev spaces, for all $z\in (z_b, z_t)$,
    \begin{equation*}
        \| f(\cdot, Y(\cdot, z))\|_{H^\sigma(x_0, x_1)}^2 =  \| f(\cdot, Y(\cdot, z))\|_{L^2(x_0, x_1)}^2 + \int_{x_0}^{x_1}  \int_{x_0}^{x_1}  \frac{|f(x,Y(x,z)) - f(x', Y(x',z))|^2}{|x-x'|^{1+2\sigma}}\dd x \dd x'.
    \end{equation*}
We start with the $H^\sigma_x L^2_z$ estimate. Integrating with respect to $z$, the norm of the first term
     is bounded by the square of the $L^2$ norm of $f$ after a change of variable with bounded jacobian. We then decompose the second integral into
\begin{equation*}
\int_{x_0}^{x_1}\!  \int_{x_0}^{x_1}   \frac{|f(x,Y(x,z)) - f(x', Y(x,z))|^2}{|x-x'|^{1+2\sigma}} \mathrm{d} x \, \mathrm{d} x' + \int_{x_0}^{x_1}  \!\int_{x_0}^{x_1}  \frac{|f(x',Y(x,z)) - f(x', Y(x',z))|^2}{|x-x'|^{1+2\sigma}} \mathrm{d} x \, \mathrm{d} x'.
\end{equation*}
Once again, the first integral is bounded by $\|f\|_{H^\sigma_x L^2_z}^2$ after vertical integration. 
As for the second one, using the embedding $H^1(-1,1)\hookrightarrow C^{1/2}$, we have
\begin{equation*}
|f(x',Y(x,z)) - f(x', Y(x',z))|^2\lesssim \|\p_y f(x',\cdot )\|_{L^2_y}^2 | Y(x,z)-Y(x',z)|.
\end{equation*}
Since $Y\in H^{\frac{2}{3}+ \sigma}_x H^1_z$,
\begin{equation*}
    \begin{split}
         \int_{x_0}^{x_1} \int_{x_0}^{x_1} \|\p_y f(x',\cdot )\|_{L^2_y}^2 & \frac{| Y(x,z)-Y(x',z)|}{|x-x'|^{1+2\sigma}}\mathrm{d} x \mathrm{d} x'\\
         & \lesssim \|Y(\cdot, z)\|_{H^{\frac{2}{3}+ \sigma}_x }\left(\int_{x_0}^{x_1}  \int_{x_0}^{x_1} \|\p_y f(x',\cdot )\|_{L^2_y}^4|x-x'|^{\frac{1}{3}-3\sigma} \dd x \dd x'\right)^{1/2}\\
         & \lesssim \|Y\|_{H^{\frac{2}{3}+ \sigma}_x H^1_z} \|f\|_{L^4_x H^1_y}^2.
    \end{split}
\end{equation*}
The first estimate follows. The other ones go along the same lines and are left to the reader.
\end{proof}

\begin{lem}[Composition with a $\qone$ function]
    \label{lem:composition-Q1}
    Let $\phi \in \qone(\Om)$ such that $\phi(x,\pm1) = \pm 1$.
    Assume that there exists $m>0$ such that $\p_z \phi(x,z)\in [m^{-1},m]$.
    Let $\sigma \in [0,1]$.
    There exists $C(m,\sigma)$ such that, for any $g\in H^\sigma_x L^2_y \cap L^2_x(W^{\sigma,4}_y)$,
    \begin{equation}
        \label{in:G-interpol}
        \| g(x,\phi(x,z)) \|_{H^\sigma_x L^2_z} \leq C \left( \| g \|_{H^\sigma_x L^2_z} + (1+\|\phi\|_{\qone}^\sigma) \| g \|_{L^2_x(W^{\sigma,4}_y)} \right).
    \end{equation}
\end{lem}

\begin{proof}
    Throughout the proof, we set $G(x,z) := g(x,\phi(x,z))$.
    First, note that, since the Jacobian of the change of variable $z \mapsto \phi(x,z)$ is bounded from below, for any $p,q \in [1,\infty]$,
    \begin{equation}
        \label{eq:Gpq-gpq}
        \| G \|_{L^p_x L^q_z} \leq m^{\frac{1}{q}} \| g \|_{L^p_x L^q_z}.
    \end{equation}
    In particular $\| G \|_{L^2} \leq m^{\frac{1}{2}} \| g \|_{L^2}$.
    Furthermore, for $\sigma=1$,
    \begin{equation*}
        \p_x G(x,z)= \p_x g(x,\phi(x,z)) + \p_x \phi(x,z) \p_y g(x,\phi(x,z)).
    \end{equation*}
    Hence,
    \begin{equation} \label{eq:pxG}
        \| \p_x G \|_{L^2} \leq \| (\p_x g) \circ \phi \|_{L^2} + \| \p_x \phi \|_{L^\infty_x L^4_z} \| (\p_y g) \circ \phi \|_{L^2_x L^4_z}.
    \end{equation}
    By the ``fractional trace theorem'' \cite[Equation~(4.7), Chapter~1]{LM68},
    \begin{equation*}
        \| \p_x \phi \|_{L^\infty_x H^{1/2}_z}
        \lesssim
        \| \p_x \phi \|_{H^{2/3}_x L^2_z} + \| \p_x \phi \|_{L^2_x H^2_z}
        \lesssim \| \phi \|_{\qone}.
    \end{equation*}
    Hence, we obtain from \eqref{eq:Gpq-gpq} and \eqref{eq:pxG} that
    \begin{equation*}
        \| \p_x G \|_{L^2} \lesssim \| \p_x g \|_{L^2} + \| \phi \|_{\qone} \| g \|_{L^2_x (W^{1,4}_y)}.
    \end{equation*}
    Now, note that the application $g\mapsto G$ is linear.
    By interpolation, we obtain, for any $\sigma\in (0,1)$,
    \begin{equation*}
        \|G\|_{H^\sigma_x L^2_z} \lesssim \|g\|_{H^\sigma_x L^2_z} + (1+\|\phi\|_{\qone}^\sigma) \|g\|_{L^2_x (W^{\sigma,4}_y)},
    \end{equation*}
    which concludes the proof.
\end{proof}

\begin{coro}
    \label{lem:inverse-qone}
    Let $\phi \in \qone \cap L^2_x H^4_z$ such that $\phi(x,\pm1) = \pm 1$ and $\| \phi - z \|_{\qone} + \| \phi - z \|_{L^2_x H^4_z} \ll 1$.
    Let $\psi(x,y)$ be such that $\psi(x,\phi(x,z))=z$ for all $(x,z)\in\Om$.
    Then $\psi \in \qone \cap L^2_x H^4_y$ and 
    \begin{equation*}
        \| \psi - y \|_{\qone} + \| \psi - y \|_{L^2_x H^4_y} \lesssim
        \| \phi - z \|_{\qone} + \| \phi - z \|_{L^2_x H^4_z}.
    \end{equation*}
\end{coro}

\begin{proof}
    In this statement and this proof, we use the variable $y$ as second argument for $\psi$ and $z$ as second argument for $\phi$.
    First, observe that $\p_y \psi(x,y) = 1 / (\p_z \phi(x,\psi(x,y)))$, so that $\| \p_y \psi -1 \|_{L^\infty} \ll \| \p_z \phi - 1 \|_{L^\infty}$.
    In particular the associated changes of vertical variables are well-defined and bounded so that estimates such as \eqref{eq:Gpq-gpq} hold and will be used abundantly.

    \step{Vertical regularity of $\psi$.}
    By the ``fractional trace theorem'' \cite[Equation~(4.7), Chapter~1]{LM68}, for $\phi \in \qone \cap L^2_x H^4_z$, $\phi \in H^1_x H^2_z\cap L^2_x H^4_z\hookrightarrow C^0_x(H^3_z)$.
    In particular $\p_{z}^2 \phi \in L^\infty$. 
    Differentiating the definition $\psi(x,\phi(x,z)) = z$, we obtain the following relations and estimates.
    First, we already said that $\p_y \psi \in L^\infty$.
    Second, $\p_y^2 \psi\in L^2$ since
    \begin{equation*}
        - (\p_z \phi)^2 \p_y^2 \psi\circ \phi = 
        \underbrace{(\p_y \psi\circ \phi )}_{L^\infty} \underbrace{(\p_{zz} \phi)}_{L^\infty}. 
    \end{equation*}
    Third, $\p_y^3 \psi\in L^2$ since
    \begin{equation*}
        - (\p_z \phi)^3 \p_y^3 \psi\circ \phi =
        3 \underbrace{\p_y^2 \psi\circ \phi}_{L^2} \underbrace{\p_z \phi \p_z^2 \phi}_{L^\infty} +
        \underbrace{\p_y \psi\circ \phi}_{L^\infty} \underbrace{\p_z^3 \phi}_{L^2}.
    \end{equation*}
    Fourth, omitting the composition with $\phi$ in every occurence of $\psi$ in order to alleviate the notation,
    \begin{equation*}
        - (\p_z \phi)^4 \p_y^4 \psi = 
        6 \underbrace{\p_y^3 \psi}_{L^2} \underbrace{(\p_z \phi)^2 \p_z^2 \phi}_{L^\infty} 
        + \underbrace{\p_y^2 \psi}_{L^2_x H^1_y} \big( 3 \underbrace{(\p_z^2 \phi)^2}_{L^\infty} + 4 \underbrace{(\p_z \phi)}_{L^\infty} \underbrace{(\p_z^3 \phi)}_{L^\infty_x L^2_z} \big)
        + \underbrace{\p_y \psi}_{L^\infty} \underbrace{\p_z^4 \phi}_{L^2}.
    \end{equation*}
    Remembering that $1/(\p_z \phi) \in L^\infty$, we conclude that
    \begin{equation*}
        \| \psi(x,y) - y \|_{L^2_x H^4_y} \lesssim \| \phi(x,z) - z \|_{\qone} + \| \phi(x,z) - z \|_{L^2_x H^4_z}.
    \end{equation*}

    \step{Integer horizontal regularity of $\psi$.}
    This step uses that $\p_x \phi \in L^\infty_x L^2_z$ which follows from $\phi \in H^{5/3}_x L^2_z$.
    Note however that, even for $\phi \in \qone \cap L^2_x H^4_y$, one does not have $\p_x \phi \in L^\infty$.
    We proceed similarly for the integer horizontal regularity.
    First,
    \begin{equation*}
        - \p_x \psi \circ \phi = \underbrace{\p_y \psi\circ \phi}_{L^\infty} \underbrace{\p_x \phi}_{L^2}. 
    \end{equation*}
    Second,
    \begin{equation*}
        - (\p_z \phi) \p_{xy} \psi\circ \phi
        = 
        \underbrace{\p_y^2 \psi\circ \phi}_{L^2_x H^1_y} \underbrace{\p_z \phi}_{L^\infty} \underbrace{\p_x \phi}_{L^\infty_x L^2_z} + \underbrace{\p_y \psi\circ \phi}_{L^\infty} \underbrace{\p_{xz} \phi}_{L^2}.
    \end{equation*}
    Third,
    \begin{equation}
        \label{eq:pxyy-psi}
        - (\p_z \phi)^2 \p_{xyy} \psi\circ \phi
        = 2 \p_y^2 \psi \circ \phi\; \p_z \phi \; \p_{xz} \phi + h
    \end{equation}
    where, omitting once again the composition with $\phi$ in the derivatives of $\psi$,
    \begin{equation*}
        h = \underbrace{\p_y^3 \psi}_{L^2_x H^1_y} \underbrace{\p_x \phi}_{L^\infty_x L^2_z} \underbrace{(\p_z \phi)^2}_{L^\infty}
        + \underbrace{\p_{xy} \psi}_{L^2} \underbrace{\p_z^2 \phi}_{L^\infty}
        + \underbrace{\p_y^2 \psi}_{L^2_x H^1_y} \underbrace{\p_x \phi}_{L^\infty_x L^2_z} \underbrace{\p_z^2 \phi}_{L^\infty}
        + \underbrace{\p_z \psi}_{L^\infty} \underbrace{\p_{xzz} \phi}_{L^2}.
    \end{equation*}
    From \eqref{eq:pxyy-psi} and $h \in L^2$, we obtain that
    \begin{equation*}
        \begin{split}
        \| \p_{xyy} \psi \|_{L^2} & \lesssim \|h \|_{L^2} + \|\p_{xz} \phi\|_{L^2_x H^1_z} \| \p_{yy} \psi \|_{L^\infty_x L^2_z}
        \\ & \lesssim \|h\|_{L^2} + \| \phi(x,z) - z \|_{H^1_x H^2_z} \| ( \| \p_{yy} \psi \|_{L^2} + \| \p_{xyy} \psi \|_{L^2} ).
        \end{split}
    \end{equation*}
    Hence, using the smallness of $ \| \phi(x,z) - z \|_{H^1_x H^2_z}$, we conclude that
    \begin{equation*}
        \| \psi(x,y) - y \|_{H^1_x H^2_y} \lesssim \| \phi(x,z) - z \|_{\qone} + \| \phi(x,z) - z \|_{L^2_x H^4_z}.
    \end{equation*}

    \step{Fractional horizontal regularity of $\psi$.}
    Eventually, to obtain the $H^{5/3}_x L^2_y$ regularity, we write
    \begin{equation*}
        \p_x \psi(x,y) = - \frac{\p_x \phi}{\p_z \phi} (x, \psi(x,y))
    \end{equation*}
    and we apply \cref{lem:composition-Q1} with $\sigma = 2/3$.
    Let us first assume that $\phi$ is smooth (so that $\psi$ is smooth as well by usual results) and then argue by density.
    Estimate \eqref{in:G-interpol} yields
    \begin{equation*}
        \| \psi_x \|_{H^{2/3}_x L^2_y} 
        \lesssim \| \p_x \phi / \p_z \phi \|_{H^{2/3}_x L^2_z} + (1 + \| \psi \|_{\qone}^{2/3}) \| \p_x \phi / \p_z \phi \|_{L^2_x H^2_z}.
    \end{equation*}
    Since we already know that $\psi$ can be estimated in $H^1_x H^2_y$, we can use (the Peter--Paul version) of Young's inequality to obtain
    \begin{equation*}
        \| \psi_x \|_{H^{2/3}_x L^2_y} 
        \lesssim \| \p_x \phi / \p_z \phi \|_{H^{2/3}_x L^2_z} + \| \p_x \phi / \p_z \phi \|_{L^2_x H^2_z}
        + \| \p_x \phi / \p_z \phi \|_{L^2_x H^2_z}^{3/2}.
    \end{equation*}
    Moreover, one easily proves, using standard product rules, that
    \begin{equation*}
        \| \p_x \phi / \p_z \phi \|_{H^{2/3}_x L^2_z} + \| \p_x \phi / \p_z \phi \|_{L^2_x H^2_z}
        \lesssim \| \phi(x,z) - z \|_{\qone} + \| \phi(x,z) - z \|_{L^2_x H^4_z} \ll 1.
    \end{equation*}
    Hence we obtain
    \begin{equation*}
        \| \psi_x \|_{H^{2/3}_x L^2_y} 
        \lesssim \| \phi(x,z) - z \|_{\qone} + \| \phi(x,z) - z \|_{L^2_x H^4_z}
    \end{equation*}
    when $\phi$ is smooth and $\|\phi(x,z)-z\|_{\qone}+\|\phi(x,z)-z\|_{L^2_x H^4_z} \ll 1$. 
    We conclude by density.
\end{proof}

\subsection{Extension operators}

We start with \cref{p:extension}, which allows extending functions from $Z^0(\Omega)$ to $Z^0(\R^2)$.

\begin{proof}[Proof of \cref{p:extension}]
	Up to translation and rescaling, we can assume that $(x_0,x_1) = (0,1)$.
	
	We start by constructing a continuous horizontal extension operator $P_x$ from $Z^0((0,1)\times(-1,1))$ to $Z^0(\R \times (-1,1))$.
	Let $\chi \in C^\infty(\R;[0,1])$ such that $\chi \equiv 1$ on $(0,1)$ and $\supp \chi \subset (-1,2)$.
	Let $\phi \in Z^0((0,1) \times (-1,1))$.
	For $x \in (-1,2)$ and $z \in (-1,1)$, let
	\begin{align}
		(Q_x \phi)(x,z) & :=
		\begin{cases}
			\phi(- x,z) & \text{if } x \in (-1,0) \\
			\phi(x,z) & \text{if } x \in (0,1), \\
			\phi(2-x,z) & \text{if } x \in (1,2),
		\end{cases} 
		\\
		(P_x \phi)(x,z) & := \chi(x) (Q_x \phi)(x,z).
	\end{align}
	First, $\| P_x \|_{L^2_{x,z} \to L^2_{x,z}} \leq 3$.
	Moreover, $\p^k_z (P_x \phi) = P_x \p^k_z \phi$ for $k \in \{ 1,2 \}$.
	Hence $\| P_x \|_{L^2_x H^2_z \to L^2_x H^2_z} \leq~3$.
	Eventually, $\|z \p_x (Q_x \phi) \|_{L^2((-1,2)\times (-1,1)}\leq 3 \|z\p_x \phi\|_{L^2((0,1)\times(-1,1)}$, so that
 \begin{equation*}\| z \p_x (P_x \phi) \|_{L^2} \leq 3 \| z \p_x \phi \|_{L^2} + 2 \| \chi' \|_{L^\infty} \| \phi \|_{L^2}.\end{equation*}
	Thus $P_x$ defines a continuous extension operator from $Z^0((0,1)\times(-1,1))$ to $Z^0(\R \times (-1,1))$.
	
	We now construct a continuous upwards vertical extension operator $P_+$ from $Z^0(\R \times (-1,1))$ to $Z^0(\R \times (-1,+\infty))$.
	We proceed, as classical (see e.g.\ \cite{babich1953extension}), by considering a weighted linear combination of rescaled reflections.
	For $\phi \in Z^0(\R \times (-1,1))$, $x \in \R$ and $z \in (-1,\infty)$, let
	\begin{align}
		(Q_+ \phi)(x,z) & :=
		\begin{cases}
			\phi(x,z) & \text{if } z \in (-1,1), \\
			3 \phi(x,2-z) - 2 \phi(x,3-2z) & \text{if } z \in (1,2),
		\end{cases} 
		\\
		(P_+ \phi)(x,z) & := \chi_+(z) (Q_+ \phi)(x,z),
	\end{align}
	where $\chi_+ \in C^\infty(\R;[0,1])$ is such that $\chi_+ \equiv 1$ on $(-1,1)$ and $\supp \chi_+ \subset (-2,1 + \frac{1}{4})$.
	The chosen coefficients ensure that both $Q_+ \phi$ and $\p_z (Q_+ \phi)$ are continuous at $z = 1$.
	Hence $P_+ \phi \in L^2_x H^2_z$ and
	\begin{equation*}
		\| P_+ \phi \|_{L^2_x(\R; H^2_z(-1,+\infty))}
		=
		\| P_+ \phi \|_{L^2_x(\R; H^2_z(-1,1))} +
		\| P_+ \phi \|_{L^2_x(\R; H^2_z(1,+\infty))}
		\leq C_+ \| \phi \|_{L^2_x H^2_z},
	\end{equation*}
	for some constant $C_+$ depending only on $\| \chi_+ \|_{W^{2,\infty}}$.
	Moreover, using that $\chi(z) = 0$ for $z > 1 + \frac{1}{4}$,
	\begin{equation*}
		\begin{split}
			\| z \p_x (P_+ \phi) \|_{L^2_x(\R;L^2(1,+\infty))}
			& = 
			\| z \p_x (P_+ \phi) \|_{L^2_x(\R;L^2(1,1+\frac{1}{4}))} \\
			& \lesssim
			\| \p_x \phi \|_{L^2_x(\R;L^2(\frac{1}{2},1))} \\
			& \lesssim
			\| z \p_x \phi \|_{L^2_x(\R;L^2(\frac{1}{2},1))}.
		\end{split}
	\end{equation*}
	Hence $P_+$ is a continuous extension operator from $Z^0(\R \times (-1,1))$ to $Z^0(\R \times (-1,+\infty))$.
	
	The extension for $z < -1$ is performed in a similar fashion and left to the reader.
\end{proof}

\subsection{Embeddings}

We collect in this paragraph various embedding results used throughout the paper.

\subsubsection{Embedding of the Pagani space \texorpdfstring{$Z^0$}{Z0} in \texorpdfstring{$H^{2/3}_xL^2_z$}{H23L2}}

We start with an easy one dimensional inequality.
\begin{lem} \label{lem:psi-zpsi-psi''}
	For $\psi \in C^\infty_c(\R)$,
	\begin{equation*}
		\| \psi \|_{L^2} \lesssim \| z \psi \|_{L^2} + \| \p_{zz} \psi \|_{L^2}.
	\end{equation*}
\end{lem}

\begin{proof}
	On the one hand, for $|z| \geq 1$, 
	\begin{equation*}
		\int_{|z| \geq 1} \psi^2 \leq \| z \psi \|_{L^2}^2.
	\end{equation*}
	On the other hand, for every $(z_0,z) \in (-2,2)$,
	\begin{equation} \label{eq:psi'-zz0}
		|\p_z \psi(z)| \leq |\p_z \psi(z_0)| + 2 \| \p_{zz} \psi \|_{L^2}.
	\end{equation}
	Moreover, by classical Sobolev embeddings,
	\begin{equation*}
		\| \p_z\psi \|_{L^2(1,2)} \lesssim \| \psi \|_{L^2(1,2)} + \|\p_{zz}  \psi \|_{L^2(1,2)}
		\leq \| z \psi \|_{L^2(\R)} + \| \p_{zz} \psi \|_{L^2(\R)}.
	\end{equation*}
	Thus, integrating \eqref{eq:psi'-zz0} for $z_0 \in (1,2)$,
	\begin{equation*}
		\|\p_z\psi\|_{L^\infty(-2,2)} \lesssim \| z \psi \|_{L^2(\R)} + \| \p_{zz} \psi \|_{L^2(\R)}.
	\end{equation*}
	Now, writing $\psi(z) = \psi(z_0) + \int_{z_0}^z \psi'$ and integrating for $z_0 \in (1,2)$ yields
	\begin{equation*}
		\| \psi \|_{L^2(-1,1)} \lesssim \| \psi \|_{L^2(1,2)} + \| z \psi \|_{L^2(\R)} + \| \p_{zz} \psi \|_{L^2(\R)}
		\lesssim \| z \psi \|_{L^2(\R)} + \| \p_{zz} \psi \|_{L^2(\R)},
	\end{equation*}
	which concludes the proof.
\end{proof}

We then turn towards the proof of the key result $Z^0(\R^2)\hookrightarrow H^{2/3}_x L^2_z(\R^2)$.

\begin{proof}[Proof of \cref{p:z0-l2-h23}]
	Let $\psi \in C^\infty_c(\R)$.
	By \cref{lem:psi-zpsi-psi''}, one has
	\begin{equation}
		\label{eq:psi-vert-2-copy}
		\| \psi \|_{L^2} \lesssim \| z \psi \|_{L^2} + \| \p_{zz} \psi \|_{L^2}.
	\end{equation}
	Using standard dimensional analysis arguments (e.g.\ by introducing the rescaled function $\psi_\lambda : z \mapsto \psi(\lambda z)$ for $\lambda > 0$ and optimizing the choice of~$\lambda$), one deduces from~\eqref{eq:psi-vert-2-copy} that
	\begin{equation} \label{eq:psi-vert-2-prod}
		\| \psi \|_{L^2} \lesssim \| z \psi \|_{L^2}^{\frac 23} \| \p_{zz} \psi \|_{L^2}^{\frac 13}.
	\end{equation}
	Let $\phi \in C^\infty_c(\R^2)$.
	Let $\hat{\phi}(\xi, z)$ denote the Fourier-transform of $\phi$ in the horizontal direction.
	Then using \eqref{eq:psi-vert-2-prod} and Hölder's inequality,
	\begin{equation*}
		\begin{split}
			\| \phi \|_{H^{2/3}_x L^2_z}^2
			& = \int_{\R^2} (1+|\xi|^2)^{\frac 23} |\hat{\phi}(\xi, z)|^2 \dd \xi \dd z \\
			& \lesssim \| \phi \|_{L^2}^2 +
			\int_\R |\xi|^{\frac 43}
			\| z \hat{\phi}(\xi,z) \|_{L^2_z}^{\frac 43} \| \p_{zz} \hat{\phi}(\xi,z) \|_{L^2_z}^{\frac 23} \dd \xi \\
			& \lesssim \| \phi \|_{L^2}^2 +
			\left( \int_{\R^2} |\xi|^2 z^2 |\hat{\phi}(\xi,z)|^2 \dd z \dd \xi \right)^{\frac 23} 
			\left( \int_{\R^2} |\p_{zz} \hat{\phi}(\xi,z)|^2 \dd z \dd \xi \right)^{\frac 13} \\
			& \lesssim \| \phi \|_{L^2}^2 +
			\| z \p_x \phi \|_{L^2}^{\frac 43} \| \p_{zz} \phi \|_{L^2}^{\frac 23}.
		\end{split}
	\end{equation*}
	Hence $\| \phi \|_{H^{2/3}_x L^2_z} \lesssim \| \phi \|_{Z^0}$.
	This concludes the proof, by density of $C^\infty_c(\R^2)$ in $Z^0(\R^2)$.
\end{proof}

\subsubsection{Embedding of the Baouendi--Grisvard space $\cB$ in $H^{1/3}_x L^2_z$}

Once again, we start with a one-dimensional inequality of Hardy type.
\begin{lem}
    \label{lem:Hardy-45}
    For $\phi \in L^2(0,1)$,
    \begin{equation}
        \label{eq:Hardy-45}
        \int_0^1 \left( \frac{1}{z^2} \int_0^z s \phi(s) \dd s \right)^2 \dd z \leq \frac 4 5 \|\phi\|_{L^2}^2.
    \end{equation}
\end{lem}

\begin{proof}
    For $z \in (0,1)$, by the Cauchy--Schwarz inequality
    \begin{equation*}
        \left( \int_0^z s \phi(s) \dd s \right)^2
        \leq \left(\int_0^z \phi^2(s) s^{\frac 12} \dd s \right)
        \left(\int_0^z s^{2-\frac 12} \dd s \right)
        = \frac 2 5 z^{\frac 5 2} \left(\int_0^z \phi^2(s) s^{\frac 12} \dd s \right).
    \end{equation*}
    Hence, by Fubini,
    \begin{equation*}
        \begin{split}
            \int_0^1 \left( \frac{1}{z^2} \int_0^z s \phi(s) \dd s \right)^2 \dd z  
            & \leq \frac 2 5 \int_{0}^1 z^{-4+\frac 52} \left(\int_{0}^z \phi^2(s) s^{\frac 12} \dd s \right) \dd z \\
            & = \frac 25 \int_{0}^1 \phi^2(s) s^{\frac 12} \left(\int_{s}^1 z^{-\frac 32} \dd z\right) \dd s \\
            & = \frac 25 \int_{0}^1 \phi^2(s) s^{\frac 12} \left(2 (s^{-\frac 12} - 1) \right)\dd s \\
            & = \frac 4 5 \int_0^1 \phi^2(s) (1- s^{\frac 12})\dd s,
        \end{split}
    \end{equation*}
    which implies \eqref{eq:Hardy-45}.
\end{proof}

We then turn towards the proof of the embedding.

\begin{proof}[Proof of \cref{lem:embed-cB-H13x}]
    \step{Extension to a compactly supported function in $(x_0, x_1)\times \R$.}
    
    Let $u\in L^2((x_0,x_1), H^1_0(-1,1))$ such that $z\p_x u\in L^2((x_0,x_1)), H^{-1}(-1,1))$.
    We first extend $u$ to $(x_0,x_1)\times (-3,3)$ by setting, for all $x\in (x_0,x_1)$ and $z'\in (0,2)$,
    \begin{equation*}\begin{aligned}
    u(x,1+z')&=&-u(x,1-z'),\\
    u(x,-1-z')&=&-u(x,-1+z').
    \end{aligned}
    \end{equation*}
    It is clear that the above extension belongs to $ L^2((x_0,x_1), H^1_0(-3,3))$, and we further extend $u$ by zero on $(x_0,x_1)\times \{z\in \R, |z|\geq 3\}$.
    We then take $\chi\in C^\infty_c(\R)$ such that $\supp \chi\subset (-3/2, 3/2)$, and $\chi\equiv 1$ on $(-1,1)$, and we prove that $u \chi \in \cB ((x_0,x_1)\times \R)$.
    Using a partition of unity, we write $\chi=\chi_{-1} + \chi_0 + \chi_1$, where $\supp \chi_{\pm 1} \subset (\pm 1/2, \pm 3/2)$, and $\supp \chi_0\subset (-3/4, 3/4) $. It is clear that $\chi_0 u \in \cB ((x_0,x_1)\times \R)$, and therefore by symmetry is is sufficient to prove the result for  $\chi_1 u$.
    
    Let us take $\phi \in H^1_0((x_0, x_1)\times \R)$ be arbitrary, and compute
    \begin{equation*}
    -\int_{x_0}^{x_1} \int_{\R} z \chi_1 u \p_x \phi.
    \end{equation*}
    By definition of $u$ on $(x_0,x_1)\times (1,2)$,
    \begin{align*}
    -\int_{x_0}^{x_1} \int_{\R} z \chi_1 u \p_x \phi=&-\int_{x_0}^{x_1}\int_0^1 z u(x,z) \chi_1(z) \p_x \phi(x,z)\dd x\dd z\\
    &+\int_{x_0}^{x_1}\int_0^1 (1+z')u(x,1-z') \chi_1(1+z') \p_x \phi(x,1+z')\dd x\dd z'\\
    =&- \int_{x_0}^{x_1}\int_0^1 zu(x,z) \chi_1(z) \p_x \phi(x,z)\dd x\dd z\\
    &+\int_{x_0}^{x_1}\int_0^1 \frac{(1+z')}{1-z'}  (1-z') u(x,1-z') \chi_1(1+z') \p_x\phi(x,1+z')\dd x\dd z'.
    \end{align*}
    Since $z\p_x u\in L^2((x_0,x_1), H^{-1}(-1,1))$, we may write $z\p_x u = f + \p_z g$, with $f,g\in L^2((x_0,x_1)\times (-1,1))$.
    Then
    \begin{align*}
    &-\int_{x_0}^{x_1} \int_{\R} z \chi_1 u \p_x \phi\\
    =&\int_{x_0}^{x_1}\int_0^1 (f + \p_z g)(x,z) \left[\chi_1(z) \phi(x,z) -   \frac{2-z}{z}\chi_1(2-z) \phi(x,2-z)\right]\dd x\dd z.
    \end{align*}
    The assumptions on $\supp \chi_1$ ensure that the function in brackets belongs to $L^2((x_0,x_1), H^1_0(0,1))$. We conclude that for all $\phi \in H^1_0((x_0, x_1)\times \R)$,
    \begin{equation*}
    \left| -\int_{x_0}^{x_1} \int_{\R} z \chi_1 u \p_x \phi\right| \lesssim \| u\|_{\cB((x_0,x_1)\times (0,1))} \|\phi\|_{L^2_x H^1_z}.
    \end{equation*}
    It follows that $z\chi_1\p_x u \in L^2((x_0,x_1), H^1(\R))$, and
    \begin{equation*}
    \| \chi u \|_{\cB((x_0,x_1)\times \R)}\lesssim \| u\|_{\cB((x_0,x_1)\times (0,1))}.
    \end{equation*}
    
    \step{The vertical anti-derivative of $\chi u$ belongs to $Z^0$.}
    
    We now work with the extension of the previous step, and we set $U:=-\int_{z}^\infty \chi u$. 
    Let us  prove that $U\in Z^0((x_0,x_1)\times \R_+)$. 
    Since $\p_z^2 U= \p_z(\chi u)\in L^2((x_0,x_1)\times \R_+)$, 
    it suffices to prove that $z\p_x U\in L^2((x_0,x_1)\times \R_+)$. Hence we take $\phi\in L^2((x_0,x_1)\times \R_+)$ arbitrary, and we compute, after observing that $U$ is supported in $\{z\leq 3/2\}$,
    \begin{align*}
    \int_{x_0}^{x_1}\int_0^\infty s\p_x U(x,s) \phi(x,s)\dd x\; \dd s 
    &=-\int_{x_0}^{x_1}\int_0^\infty s \left(\int_s^{3/2}\p_x \chi u (x,z)\dd z\right) \phi(x,s)\dd s \dd x\\
    = &- \int_{x_0}^{x_1}\int_0^\infty  \frac{1}{z}\left(\int_0^z s \mathbf 1_{0<s<3/2}\phi(x,s)\dd s\right)  z\p_x \chi u(x,z)\dd x\dd z.
    \end{align*}
    Therefore
    \begin{equation*}
    \left\| s\p_x U\right\|_{L^2((x_0,x_1)\times \R_+)} \lesssim \| \chi u \|_{\cB} \sup_{\substack{\phi \in L^2,\\ \|\phi\|_{L^2}\leq 1} }
    \left\| \frac{1}{z}\left(\int_0^z s \mathbf 1_{0<s<3/2}\phi(x,s)\dd s\right)\right\|_{L^2((x_0,x_1), H^1_0(0, +\infty))}.
    \end{equation*}
    The claim therefore follows from the following result, which is postponed to the third and last step:
    \begin{lem}
        For all $z_0>0$, there exists a constant $C_{z_0}$ such that for all $\psi \in L^2(\R)$,
        \begin{equation} \label{eq:hardy2}
          \left\| \frac{1}{z}\left(\int_0^z s \mathbf 1_{0<s<z_0}\psi (s)\dd s\right)\right\|_{H^1_0(0, +\infty)} \leq C_{z_0}\|\psi\|_{L^2(\R)}.
        \end{equation}
    \end{lem}
    From there, we infer that $U\in Z^0((x_0,x_1)\times \R_+)$, and $\|U\|_{Z^0}\lesssim \|u\|_{\cB}$. 
    Using the embedding $Z^0\hookrightarrow H^{1/3}_x H^1_z$, we deduce that $\p_z U=\chi u\in H^{1/3}_x L^2_z((x_0,x_1)\times \R_+)$. Since $\chi\equiv 1$ on $(-1,1)$, we obtain the desired result.

    \step{Proof of \eqref{eq:hardy2}.}

    First, note that for all $s\in (0, +\infty)$,
    \begin{equation*}
    \left| \int_0^z s \mathbf 1_{0<s<z_0}\psi (s)\dd s\right| \leq C_{z_0} \inf(s^{3/2} , 1) \|\psi\|_{L^2}. 
    \end{equation*}
    Thus we only need to prove that
    \begin{equation*}
    \int_0^{z_0} \left( \frac{1}{z^2} \int_0^z s \psi(s) \dd s \right)^2 \dd z\leq C_{z_0}  \|\psi\|_{L^2}^2.
    \end{equation*}
    This is a rescaling of inequality \eqref{eq:Hardy-45} of \cref{lem:Hardy-45}.
\end{proof}

\section{Unconditional regularity away from lateral boundaries}
\label{sec:proof-lem:reg-away-sing-pts}

\begin{proof}[Proof of \cref{lem:reg-away-sing-pts}]
    We proceed by (horizontal) viscous regularization to obtain uniform estimates which pass to the limit.

    Let us extend the functions $w_0, w_1$ into $H^2$ functions on the whole interval $(z_b,z_t)$, such that $w_0(z_b)=w_b(x_0)$ and $w_1(z_t)=w_t(x_1)$. 
    For $\eps>0$, consider the solution to the elliptic equation
    \begin{equation} \label{eq:W-eps}
    \begin{cases}
    	- \eps \p_x^2 W^\eps - \p_z^2 (\alpha  W^\eps) + z \p_x W^\eps + \beta \p_z W^\eps = h,\\
    	W^\eps\vert_{x=x_i}= w_i,\\
    	W^\eps\vert_{z=z_j}= w_j.
    \end{cases}
    \end{equation}
    Let us recall that $|z_b|, z_t\leq z_0$ for some small constant $z_0$ depending only on $\alpha$.
    Classical results on elliptic equations ensure that if $z_0$ is small enough, \eqref{eq:W-eps} has a unique solution in $H^1(\Om)$ for all $\eps>0$, which satisfies the energy estimate
    \begin{equation} \label{eq:estimate-Weps-1}
        \sqrt{\eps} \|\p_x W^\eps\|_{L^2} + \| \p_z W^\eps\|_{L^2} \lesssim \|h\|_{L^2} + \| w_0\|_{H^2_z} +  \| w_1\|_{H^2_z} + \|w_t\|_{H^2_x} +  \|w_b\|_{H^2_x}.
    \end{equation}
    Hence $W^\eps$ is uniformly bounded in $L^2_x H^1_z$. It follows that $W^\eps\rightharpoonup W$ in $L^2_x H^1_z$, where $W\in Z^0$ is the unique solution to \eqref{vorticity-model}.
    
    Furthermore, since $h\in L^2(\Om)$, the compatibility conditions in the corners of the domain and the fact that $\Om$ is a rectangle ensure that $W^\eps\in H^2(\Om)$ (see \cite[Chapter 4]{MR775683}). 
    Hence $\p_x W^\eps\in H^1(\Om)$ is a weak solution to
    \begin{equation} \label{eq:pxWeps}
        - \eps \p_x^2 \p_xW^\eps - \p_z^2 (\alpha \p_x W^\eps) + z \p_x^2 W^\eps + \beta \p_z\p_x  W^\eps = \p_x h + \p_z^2 (\p_x \alpha  W^\eps) - \p_x \beta \p_z W^\eps.
    \end{equation}
    Without loss of generality, we assume from now on that $w_t=w_b=0$ in order to simplify the computations. 
    This condition can always be satisfied up to a lift of the boundary conditions.
    Let $\rho(x) := (x-x_0) (x_1-x)$.
    We multiply \eqref{eq:pxWeps} by $\rho^2 \p_x W^\eps$, integrate by parts and obtain
    \begin{equation*}
        \begin{split}
        	\eps \int_\Om \rho^2 (\p_x^2 W^\eps)^2 + \int_\Om \alpha \rho^2 (\p_{xz} W^\eps)^2
        	\leq \eps \int_\Om & \p_x (\rho \p_x \rho) (\p_x W^\eps)^2 - \int_\Om \rho^2 \p_x^2 W^\eps (z\p_x W^\eps)\\
        	&+\left( \| \rho \p_x h\|_{L^2} + \|\rho \p_x \beta \p_z W^\eps\|_{L^2}\right) \| \rho \p_x W^\eps\|_{L^2}\\
        	&+ \| \rho \p_z(\p_x \alpha W^\eps)\|_{L^2} \| \rho \p_{xz}  W^\eps\|_{L^2}\\
        	&+ \|\p_z \alpha\|_\infty \| \rho \p_x W^\eps\|_{L^2}  \| \rho \p_{xz} W^\eps\|_{L^2} .
        \end{split}
    \end{equation*}
    The first integral in the right-hand side is uniformly bounded thanks to \eqref{eq:estimate-Weps-1}.
    Let us focus momentarily on the second integral in the right-hand side. Using the equation satisfied by $W^\eps$, we infer
    \begin{equation*}
        \int_\Om \rho^2 \p_x^2 W^\eps (z\p_x W^\eps)
        = 
        \int_\Om \rho^2 \p_x^2 W^\eps \left[h-\beta \p_z W^\eps + \p_z^2 (\alpha W^\eps)+ \eps \p_x^2 W^\eps \right].
    \end{equation*}
    We then perform integration by parts in the right-hand side, which is equal to
    \begin{equation*}
        \begin{split}
            \eps \int_\Om \rho^2 (\p_x^2 W^\eps)^2 & + \int_\Om \rho^2 \alpha (\p_{xz} W^\eps)^2\\
        	&-2 \int \rho \p_x \rho \p_x W^\eps ( h -\beta \p_z W^\eps ) - \int \rho^2 \p_x W^\eps \p_x (h-\beta \p_z W^\eps)\\
        	&+ 2 \int \rho \p_x \rho \p_{xz} W^\eps \p_z(\alpha W^\eps) + \int \rho^2 \p_{xz} W^\eps \left(\p_x (\p_z \alpha W^\eps) + \p_x \alpha \p_z W^\eps\right).
        \end{split}
    \end{equation*}
    Since $\p_z \beta=0$ and $w_t=w_b=0$, we have 
    \begin{equation*}
        \int \rho^2 \p_x W^\eps \beta \p_{xz} W^\eps=0.
    \end{equation*}
    We also recall that $\| \rho \p_x W^\eps\|_{L^2} \leq z_0 \| \rho \p_x \p_z W^\eps\|_{L^2}$.
    Therefore, provided that $z_0$ is small enough, there exists $C > 0$ such that
    \begin{equation*}
       \begin{aligned} 
            \int_\Om \rho^2 \p_x^2 W^\eps (z\p_x W^\eps)
    	    \geq & \eps \int_\Om \rho^2 (\p_x^2 W^\eps)^2  +\frac{1}{2} \int_\Om \rho^2 \alpha (\p_{xz} W^\eps)^2 \\
            & - C  \left( \|h\|_{L^2}^2 + \|W^\eps \|_{L^2_x H^1_z}^2 + \|\rho \p_x h\|_{L^2}\right).
  \end{aligned}
    \end{equation*}
    Gathering all the terms and using the $L^2_x H^1_z$ estimate on $W^\eps$ of \eqref{eq:estimate-Weps-1}, for $z_0$ small enough,
    \begin{equation*}
        \eps \int_\Om \rho^2 (\p_x^2 W^\eps)^2 + \int_\Om \alpha \rho^2 (\p_x \p_z W^\eps)^2
    	\lesssim \|h\|_{L^2}^2 +  \|\rho \p_x h\|_{L^2}^2 +\sum_{i\in \{0,1\}}\| w_i\|_{H^2_z}^2+\sum_{j\in \{t,b\}} \|w_j\|_{H^2_x}^2.
    \end{equation*}
    Hence $\rho \p_{xz} W^\eps$ is uniformly bounded in $L^2$. Passing to the limit, we obtain
    \begin{equation*}
    	\|\rho \p_x W\|_{L^2_x H^1_z} \lesssim \|h\|_{L^2} +  \|\rho \p_x h\|_{L^2} + \sum_{i\in \{0,1\}}\| w_i\|_{H^2_z}+\sum_{j\in \{t,b\}} \|w_j\|_{H^2_x}.
    \end{equation*}
    It follows that $\rho \p_x W$ is a weak $L^2_x H^1_z$ solution to
    \begin{equation*}
        \begin{cases}
            z \p_x (\rho \p_x W) + \beta \p_z (\rho \p_x W) - \p_z^2 (\alpha \rho \p_x W) = g, \\
            \rho \p_x W \vert_{\Sigma_i} = 0, \\
            \rho \p_x W \vert_{z = z_j} = \rho \p_x w_j
        \end{cases}
    \end{equation*}
    where
    \begin{equation*}
        g := \rho \p_x h - \rho \p_x \beta \p_z W + \rho \p_{zz} (\p_x \alpha W) + (x_0+x_1-2x) z \p_x W.
    \end{equation*}
    Since $\alpha \in C^3(\overline{\Om})$ and $W \in Z^0$, $g \in L^2$.
    Hence, according to \cref{lem:WP-Z0-vorticity}, we obtain $\rho \p_x W \in Z^0$.
\end{proof}


\newpage 

\renewcommand{\nomname}{List of notations
	\label{sec:notations}}
    
\printnomenclature


\section*{Acknowledgements}

The authors thank Jérémie Szeftel and Frédéric Hérau for nice discussions about this problem.
We also warmly thank Felix Otto, Yann Brenier and an anonymous referee for suggesting in various ways to perform the nonlinear change of unknown of \cref{sec:Burgers-change} before performing the iterative scheme; an idea which greatly simplified the first version \cite{DMR-3} of our paper and allowed to handle the Prandtl system in \cref{sec:Prandtl}.

This project has received funding from the European Research Council (ERC) under the European Union's Horizon 2020 research and innovation program Grant agreement No 637653, project BLOC ``Mathematical Study of Boundary Layers in Oceanic Motion''. 
This work was supported by the SingFlows project, grant ANR-18-CE40-0027   and by the BOURGEONS project, grant ANR-23-CE40-0014-01 of the French National Research Agency (ANR).
A.-L.\ D.\ acknowledges the support of the Institut Universitaire de France.
This material is based upon work supported by the National Science Foundation under Grant No. DMS-1928930 while A.-L.\ D.\ participated in a program hosted by the Mathematical Sciences Research Institute in Berkeley, California, during the Spring 2021 semester.


\bibliographystyle{plain}
\bibliography{biblio}

\end{document}